\documentclass[10pt,reqno]{amsart}
\usepackage{amsmath}
\usepackage{color}
\usepackage[dvipsnames]{xcolor}
\usepackage{caption}  
\usepackage{makecell}
\usepackage{float}
\restylefloat{table}

\usepackage{tablefootnote}
\usepackage{mathrsfs}
\usepackage{amssymb}
\usepackage{amsthm}
\usepackage{mathtools}
\usepackage{mathabx}
\usepackage{url}

\usepackage{enumitem}
\usepackage{bbm}
\usepackage{times}
\usepackage[colorinlistoftodos]{todonotes}
\usepackage{csquotes}
\usepackage{cancel}
\usepackage{subcaption}
\usepackage{lipsum}
\usepackage[colorlinks=true,linkcolor=blue,citecolor=blue,linktocpage=true,urlcolor=black]{hyperref} 
\usepackage{marginnote}

\usepackage[alphabetic, abbrev,nobysame]{amsrefs}
\usepackage[capitalize]{cleveref}
\usepackage[toc, title]{appendix}
\usepackage{caption}
\usepackage[labelformat=simple]{subcaption}

\usepackage{setspace}     
\usepackage{centernot,cancel}

\allowdisplaybreaks

\renewcommand{\AA}{\mathbb{A}}

\newcommand{\NN}{\mathbb{N}}

\newcommand{\PP}{\mathbb{P}}

\newcommand{\RR}{\mathbb{R}}

\newcommand{\TT}{\mathbb{T}}

\newcommand{\ZZ}{\mathbb{Z}}

\newcommand{\cA}{\mathcal{A}}
\newcommand{\cB}{\mathcal{B}}
\newcommand{\cC}{\mathcal{C}}

\newcommand{\cG}{\mathcal{G}}

\newcommand{\cU}{\mathcal{U}}

\newcommand{\cW}{\mathcal{W}}

\newcommand{\sP}{\mathsf{P}}

\newcommand{\bC}{\mathbf{C}}
\newcommand{\bD}{\mathbf{D}}

\newcommand{\bw}{\overline{\sigma}}

\newcommand{\Bcal}{\mathcal{B}}

\newcommand{\Nbb}{\mathbb{N}}
\newcommand{\Zbb}{\mathbb{Z}}
\newcommand{\Rbb}{\mathbb{R}}

\newcommand{\Leb}{\mathrm{Leb}}

\newcommand{\supp}{\mathrm{supp}}

\newcommand{\CS}[2]{\hypertarget{#1#2}{#1_{#2}}}
\newcommand{\C}[2]{\hyperlink{#1#2}{#1_{#2}}}

\newcommand{\OLD}[1]{}
\newcommand{\curv}{\mathrm{curv}}
\newcommand{\Jac}{\mathrm{Jac}}

\def\be#1\ee{\begin{align}\begin{split} #1 \end{split}\end{align}}
\def\beq#1\eeq{\begin{align*}\begin{split} #1 \end{split}\end{align*}}

\newlist{enumlemma}{enumerate}{3} 
\setlist[enumlemma]{label*={ (\alph*)}, ref= {(\alph*)} }
\newlist{enumcount}{enumerate}{3} 
\setlist[enumcount]{label*={ (\arabic*)}, ref= {(\arabic*)} }

\long\def\symbolfootnote[#1]#2{\begingroup\def\thefootnote{\fnsymbol{footnote}}
\footnote[#1]{#2}\endgroup}

\DeclarePairedDelimiterX{\inner}[2]{\langle}{\rangle}{#1, #2}

\DeclareFontFamily{U}{wncy}{}
\DeclareFontShape{U}{wncy}{m}{n}{<->wncyr10}{}
\DeclareSymbolFont{mcy}{U}{wncy}{m}{n}
\DeclareMathSymbol{\Sh}{\mathord}{mcy}{"58}

\renewcommand{\phi}{\varphi}

\long\def\symbolfootnote[#1]#2{\begingroup\def\thefootnote{\fnsymbol{footnote}}
\footnote[#1]{#2}\endgroup}

\newcommand{\SL}{{\rm SL}}

\newcommand{\Diff}{{\rm Diff}}

\renewcommand{\AA}{\mathbb{A}}

\newcommand{\Ucal}{ {\mathcal U}}

\newcommand{\Tbb}{{\mathbb T}}

\newcommand{\matii}[2]{\left(\begin{array}{cc} #1\\#2\end{array}\right)}

\newcommand{\good}{\mathrm{good}}
\newcommand{\bad}{\mathrm{bad}}
\newcommand{\llong}{\mathrm{long}}
\newcommand{\sshort}{\mathrm{short}}

		\theoremstyle{Theorem}
\newtheorem{theorem}{Theorem} [section]
\newtheorem{thm}{Theorem} [section]
\newtheorem{ltheorem}{Theorem}

	\newtheorem{prop}[theorem]{Proposition} 
\newtheorem{claim}[theorem]{Claim}

\newtheorem{corollary}[theorem]{Corollary}
\newtheorem{cor}[theorem]{Corollary}

\newtheorem{lem}[theorem]{Lemma}

\newtheorem*{theorem*}{Theorem}
\newtheorem*{COR*}{Corollary}

	\newtheorem{quest}[theorem]{Question}
	
\theoremstyle{definition}
\newtheorem{definition}[theorem]{Definition}
\newtheorem{question}[theorem]{Question}

\newtheorem{example}[theorem]{Example}

\newtheorem{remark}[theorem]{Remark}

	\newtheorem{dfn}[theorem]{Definition}

	\newtheorem*{Not*}{Notation}

	\numberwithin{equation}{section}
	\numberwithin{theorem}{section}
	\makeindex

		\theoremstyle{remark}

\newtheorem{rmk}[theorem]{\textbf{Remark}}

		\theoremstyle{remark}
  \theoremstyle{theorem}

\title[Absolute continuity of stationary measures]{Absolute continuity of stationary measures for random surface dynamics}

\author{Aaron Brown}
\address{Northwestern University, 2033 Sheridan Road, Evanston, IL 60208}
\email{\href{mailto:awb@northwestern.edu}{awb@northwestern.edu}}

\author{Homin Lee}
\address{Korea Institute for Advanced Study, 85 Hoegiro Dongdaemun-gu, Seoul, Republic of Korea}
\email{\href{mailto:hominlee@kias.re.kr}{hominlee@kias.re.kr}}
\author{Davi Obata}
\address{Brigham Young University,  275 TMCB Brigham Young University Provo, UT 84602 }
\email{\href{mailto:davi.obata@mathematics.byu.edu}{davi.obata@mathematics.byu.edu}}

\author{Yuping Ruan}
\address{Northwestern University, 2033 Sheridan Road, Evanston, IL 60208}
\email{\href{mailto:ruanyp@northwestern.edu; ruanyp@umich.edu}{ruanyp@northwestern.edu; ruanyp@umich.edu}}

\begin{document}

\maketitle\symbolfootnote[0]{\it  Updated \today.}

\begin{abstract}
We find conditions for stationary measures of random dynamical systems on surfaces having dissipative diffeomorphisms to be absolutely continuous.  These conditions involve an uniformly expanding on average property in the future (UEF) and past (UEP). Our results can cover random dynamical systems generated by ``very dissipative" diffeomorphisms and perturbations of volume preserving surface diffeomorphisms. For example, we can consider a random dynamical system on $\Tbb^2$ generated by perturbations of a pair of non-commuting infinite order toral automorphisms with any arbitrary single diffeomorphism. In this case, we conclude that stationary measures are either atomic or absolutely continuous.  We also obtain an orbit classification and equidistribution result.
\end{abstract}

\let\oldtocsection=\tocsection
\let\oldtocsubsection=\tocsubsection 
\let\oldtocsubsubsection=\tocsubsubsection
 
\renewcommand{\tocsection}[2]{\hspace{0em}\oldtocsection{#1}{#2}}
\renewcommand{\tocsubsection}[2]{\hspace{1em}\oldtocsubsection{#1}{#2}}
\renewcommand{\tocsubsubsection}[2]{\hspace{2em}\oldtocsubsubsection{#1}{#2}}
\setcounter{tocdepth}{2}

\tableofcontents

\section{Introduction}

Given a smooth action of a group $\Gamma$ on a closed manifold $M$, many natural questions arise including the extent to which it is possible to classify all orbit closures and all invariant or stationary measures. One way to understand the group action is to understand measures that encodes information about the action, for instance, invariant measures. However, in general, if $\Gamma$ is not amenable, then there is no reason to exist an $\Gamma$-invariant probability measure on $M$. On the other hand, it is known that, for every probability measure $\mu$ on $\Gamma$, there is always a $\mu$-stationary measure on $M$. Furthermore, since there is a rough correspondence between stationary measures and closed invariant sets, classifying stationary measures played an important role in understanding group actions so far.

Several important works, including those by \cite{MR2831114, MR3814652, emm, el1, el2, Brown-Hertz, befh} have classified stationary measures in various settings and conclude strong consequences including the classification of orbit closures and equidistribution results.  Let us emphasize that most of the works, including the papers we mentioned,  are in conservative (volume-preserving) settings. Otherwise, as in \cite{Brown-Hertz}, one  can only conclude the SRB property of stationary measures under the certain dynamical assumptions. Since the SRB property does not give absolute continuity, one can naturally ask  when an stationary measure is absolutely continuous. In this paper, we find sufficient conditions for absolute continuity of stationary measures and deduce several consequences including orbit classifications and equidistributions in dissipative settings.

\subsection{Motivation and main results}

Let $S$ be a surface. Given a probability measure $\mu$ supported on  $\textrm{Diff}(S)$,  we say that a probability measure $\nu$ on $S$ is \emph{$\mu$-stationary} if $\int f_*\nu d\mu(f) = \nu$.

As a motivating result, we start with an example of the main results of the seminal work  Benoist and Quint, \cite{MR2831114} in $\Tbb^2$.  As formulated, this also follows from the main result in the work of Bourgain, Furman, Lindenstrauss, and Mozes, \cite{MR2340439};

Let us consider the matrices 
\begin{equation}\label{eq.AandB}
A = \begin{pmatrix}
1 & 1 \\ 0 & 1
\end{pmatrix} \textrm{ and }
 B = \begin{pmatrix}
1 & 0\\ 1 & 1
\end{pmatrix}.
\end{equation} We naturally think of $A$ and $B$ as toral automorphisms, hence, diffeomorphisms of $\Tbb^2$ which we still denote by $A$ and $B$. For a fixed $p\in (0,1)$, let $\mu=p\delta_A +(1-p)\delta_B$.

\begin{theorem*}[\cite{MR2831114,MR2340439}]\label{thm:BenoistQuint}
For $A,B$ and ${\mu}$ as above, every ergodic $\mu$-stationary measure $\nu$ on $\Tbb^2$ is either Haar (Lebesgue) measure or finitely supported.
\end{theorem*}

A natural question is what happens when we leave the setting of linear actions. One simple corollary of our main results in this paper is given by following: 

\begin{corollary}\label{cor:main} Let $A$ and $B$ be as above. For every $p\in (0,1)$, there exist open neighborhoods $\Ucal_A,\Ucal_B\subset \Diff^2 (\Tbb^2)$ of $A$ and $B$, respectively, such that, for any probability measure $\mu$ on $\Diff^2(\Tbb^2)$ with $\mu(\Ucal_A)=p, \mu(\Ucal_B)=1-p$, every ergodic $\mu$-stationary measure $\nu$ is either fully supported and absolutely continuous with respect to the Lebesgue measure or finitely supported.
\end{corollary}

More generally,  let $\mu$ be a finitely supported probability measure supported on $\SL(2,\Zbb)$ such that $\supp(\mu)$ generated a Zariski dense sub-semigroup in $\SL(2,\Rbb)$. Then, \cite{MR2831114,MR2340439} showed that every $\mu$-stationary measure is either the Haar (Lebesgue) measure or finitely supported in  \cite{MR2831114}. We have the following corresponding theorem for non-linear perturbations.

\begin{ltheorem}\label{simplifiedmaintheorem1}
Let $\mu$ be a finitely supported probability measure on $\mathrm{SL}(2,\mathbb{Z})$ such that $\Gamma_\mu=\langle \supp(\mu)\rangle $ is Zariski dense in $\SL(
2,\Rbb)$. 
For every $L>\max_{A\in\supp (\mu)}\{\|A\|,\|A^{-1}\|\}$, there exists a weak* open neighborhood $\mathcal{V}$ of $\mu$ such that, for every $\mu'\in \mathcal{V}$ with \[\max_{f\in \supp(\mu')} \{\|f\|_{C^2}, \|f^{-1}\|_{C^2}\} < L,\] every ergodic $\mu'$-stationary measure $\nu$ is either
\begin{enumerate}
\item fully supported and absolutely continuous with respect to the Lebesgue measure, or 
\item finitely supported.
\end{enumerate}
Furthermore, there exists a unique ergodic $\mu'$-stationary measure $\nu_{abs}$ which is fully supported and absolutely continuous with respect to  the Lebesgue measure.
\end{ltheorem}

For actions of $C^2$-diffeomorphisms on surfaces, generally, the first author of this paper with Rodriguez Hertz obtained a measure classification result of stationary measures \cite{Brown-Hertz}. A consequence of their result is that, in the setting of \cref{simplifiedmaintheorem1}, every ergodic $\mu'$-stationary measure is either an SRB stationary measure or finitely supported. See \Cref{thm:brownhertz} for the general statement. The definition of SRB stationary measures can be found in \Cref{section.preliminariesrandom}.  The results in \cite{Brown-Hertz} builds on important techniques for obtaining measure classification by Eskin-Mirzakhani and Benoist-Quint \cite{MR3814652,MR2831114,MR3037785}. The SRB property gives some information about the geometry of the measure, but this property alone does not imply that the stationary measure is absolutely continuous with respect to the Lebesgue measure on $\mathbb{T}^2$.

One of the novelties of the present paper is to provide a mechanism to improve the SRB condition to absolutely continuous as in \cref{cor:main,simplifiedmaintheorem1}.  Observe that, in \Cref{simplifiedmaintheorem1}, we allow the presence of some ``very dissipative'' diffeomorphisms as long as they appear with small probability.  The other novelty of the paper is that, with ``very dissipative" diffeomorphisms, we indeed show the existence of an absolutely continuous stationary measure, hence, the existence of SRB stationary measure which, as far as we know, wasn't known in the setting of \Cref{simplifiedmaintheorem1}.

Let us mention one example of  an application of \Cref{simplifiedmaintheorem1}. Let $A$ and $B$ be as in  \cref{cor:main}. Fix  $L>\max\{\|A\|, \|B\|\}$,  there exist $\epsilon>0$ and open neighborhoods $\Ucal_A,\Ucal_B\subset \Diff^2 (\Tbb^2)$ of $A$ and $B$, respectively, such that, for every $f\in\Diff^{2}(\Tbb^2)$ with $\|f\|_{C^{2}}<L$ and for every $0\le \epsilon_0 <\epsilon$, if we consider a probability measure $\mu'$ on $\Diff^2(\Tbb^2)$ with $\mu'(\Ucal_A)=(1-\epsilon_0)/2, \mu'(\Ucal_B)=(1-\epsilon_0)/2$, and $\mu'(\{f\})=\epsilon_0$, then every ergodic $\mu'$-stationary measure $\nu$ is either
\begin{enumerate}
\item  fully supported and  absolutely continuous with respect to the Lebesgue measure, or 
\item finitely supported.
\end{enumerate}
Furthermore, there exists a unique ergodic $\mu'$-stationary measure which is fully supported and absolutely continuous with respect to the Lebesgue measure.

Our results apply to much more generality. Indeed, we don't need to start with linear actions.  The main technical theorem, \cref{thm:maintechnicaltheorem},   gives conditions on $\mu$ for the existence of ergodic $\mu$-stationary measures which are  absolutely continuous  to the Lebesgue measure.  \cref{thm:maintechnicaltheorem} can be applied to non-perturbative settings, too. However,  for simplicity,  in the statements in the introduction we only focus on random dynamical systems that are generated by a measure that is weak*-close to a measure supported on  volume preserving diffeomorphisms. 

Let us make a few definitions in order to give a general result. Let $S$ be a closed surface. Let $\mathrm{Diff}^r(S)$ be the set of $C^r$-diffeomorphisms of $S$. Given $N\in \mathbb{N}$,  $\omega^N  = (f_0, \cdots, f_{N-1}) \in \left(\mathrm{Diff}^r(S) \right)^N$, and for $k \in \{1, \cdots, N\}$, we write $f^k_{\omega^N} = f_{k-1} \circ \cdots \circ f_0$. We will come back to these notations in \Cref{section.preliminaries}. Given a probability measure $\mu$ on $\mathrm{Diff}^r(S)$,  we say that $\mu$ is \emph{uniformly expanding on average in the future (UEF)} if there are constants $C\geq 1$ and $N\in \mathbb{N}$ such that for any $x\in S$ and $v\in T^1S$,  we have
\[
\displaystyle \int_{\mathrm{Diff}(S)^N} \ln \|Df_{\omega^N}^N(x)v\| d\mu^N(\omega^N) > C. 
\]
We say that $\mu$ is \emph{uniformly expanding on average in the past (UEP)} if there are constants $C>0$ and $N\in \mathbb{N}$ such that for any $x\in S$ and $v\in T^1S$, we have
\[
\displaystyle \int_{\mathrm{Diff}(S)^N} \ln \left\| \left(Df^N_{\omega^N}((f^N_{\omega^N})^{-1}(x))\right)^{-1} v\right\| d\mu^N(\omega^N) >C.
\]
We denote the sub-semigroup generated by $\supp(\mu)$ in $\mathrm{Diff}^r (S)$ by $\Gamma_\mu$.
Let us denote a volume measure on $S$ by $m$ and the group of volume preserving diffeomorphisms by $\Diff^2 _m(S)$.
\begin{ltheorem}\label{simplifiedmaintheorem2}
Suppose that $\mu$ is a probability measure supported on $\mathrm{Diff}^2_m(S)$ with bounded support which is uniformly expanding on average in the future and past.   Let $L>0$ be a constant such that 
\[
\max_{f\in\supp(\mu)} \{\|f\|_{C^2}, \|f^{-1}\|_{C^2}\} < L,
\]
Then,  there exists a weak*-neighborhood $\mathcal{V}$ of $\mu$ such that, for every $\mu' \in \mathcal{V}$ with \[\max_{f\in \supp(\mu')} \{\|f\|_{C^2}, \|f^{-1}\|_{C^2}\} < L,\] 
every ergodic $\mu'$-stationary measure $\nu$ is either
\begin{enumerate}
\item absolutely continuous with respect to $m$, and, in this case, the measure is SRB, or
\item finitely supported.
\end{enumerate}
Moreover, there exists at least one but at most finitely many absolutely continuous ergodic $\mu'$-stationary measures. Furthermore, after taking $\mathcal{V}$ smaller, we can ensure that there is a unique $\mu'$-stationary SRB measure.  
\end{ltheorem}

The UEF property is somehow typical among driving measures (see \cite{PotrieUE, RoseUE}).  It is also expected to be typical even among $k$-tuples of volume preserving diffeomorphisms. Indeed,  Dolgopyat-Krikorian conjectured in \cite{DK} that stable ergodicity is generic among $k$-tuples of volume preserving diffeormorphisms. The mechanism they had in mind to show ergodicity comes from UEF, and to show the stable ergodicity conjecture one could show that UEF is typical among $k$-tuples of volume preserving systems. Hence, we expect our assumption is verified generically nearby $\Diff^2_m(S)$. 

In some settings, we obtain equidistribution results. 

\begin{ltheorem}\label{simplifiedmaintheorem3}
Let us retain the notations and assumptions on $\mu$ from \cref{simplifiedmaintheorem2}. Then, for all $L$, as in the previous theorem, there exists a weak*-neighborhood $\mathcal{V}$ of $\mu$ such that, for every $\mu' \in \mathcal{V}$ with \[\max_{f\in \supp(\mu')} \{\|f\|_{C^2}, \|f^{-1}\|_{C^2}\} < L,\]  we have
\[
\displaystyle \lim_{n\to +\infty} \frac{1}{n} \sum_{j=0}^{n-1} (\mu')^j_* \delta_x = \nu_{\mathrm{abs}}',
\]
for all $x\in S$ which has an infinite $ \Gamma_{\mu'}$ orbit, where $\Gamma_{\mu'}$ is the sub-semigroup generated by $\supp(\mu')$. Here, $\nu_{\mathrm{abs}}'$ is the unique $\mu'$-stationary measure that is absolutely continuous with respect to $m$.  
\end{ltheorem}

In \Cref{simplifiedmaintheorem3},  the existence of the unique absolutely continuous stationary measure $\nu_{abs}$  is guaranteed by \cref{simplifiedmaintheorem2}.

Naturally, one can ask the equidistribution of the orbit without averaging. For instance;
\begin{question}
In the setting of \cref{simplifiedmaintheorem1}, for all $x\in \Tbb^2$ with infinite $\Gamma_{\mu'}$ orbit, does the weak*-limit of $\{({\mu'})^n *\delta_x\}_{n\in\Nbb}$ exist and equal to $\nu_{abs}$?
\end{question}
\begin{remark} 
In the proof of \cref{thm:maintechnicaltheorem}, we indeed showed that, for generic $C^1$ embedded curves $\gamma$ in $S$, any weak*-limit of $\{\mu^n * m_{\gamma}\}_{n\in\Nbb}$ is an absolutely continuous measure on $S$, where $m_{\gamma}$ is the Lebesgue measure on $\gamma$.  However, we cannot guarantee that such limits will be stationary measures. 
\end{remark}
With some additional assumption,  \Cref{simplifiedmaintheorem3} allows us to classify all orbit closures;

\begin{corollary}\label{cor.orbitclassification}
Let $\mu$, $L>0$ and $\mathcal{V}$  be as in \Cref{simplifiedmaintheorem2}. Let $\mu'\in \mathcal{V}$ be a probability measure that verifies
\[\max_{f\in \supp(\mu')} \{\|f\|_{C^2}, \|f^{-1}\|_{C^2}\} < L,\]
and suppose that the unique $\mu'$-stationary SRB measure $\nu_{\mathrm{abs}}$ is fully supported. Then, for any $x\in S$,  the $\Gamma_{\mu'}$-orbit of $x$ is either finite or dense. 
\end{corollary}

  Another (simplified) corollary about orbit closures is the following: 

\begin{cor}\label{cor.genericminimality}
Let $f,g\in \Diff^2 _m (\mathbb{T}^2)$. Assume that the sub-semigroup generated by $\{f,g\}$ contains an Anosov diffeomorphism, and that there is a probability measure 
$\mu$ on $\{f,g\}$ that verifies UEF and UEP.  Then, there exist $C^2$-open neighborhoods $\Ucal_f$ and $\Ucal_g$ in $\Diff^2(\mathbb{T}^2)$ of $f$ and $g$, respectively,  with the following property: for any $(\widehat{f}, \widehat{g}) \in \mathcal{U}_f \times \mathcal{U}_g$, let $\Gamma_{(\widehat{f}, \widehat{g})}$ be the sub-semigroup generated by them. Then, every $\Gamma_{(\widehat{f}, \widehat{g})}$-orbit is either finite or dense.

Moreover,  there exists  a dense $\text{G}_\delta$ subset of  $\mathcal{U}_f \times \mathcal{U}_g$,  $\mathcal{R}$, with the following property. For any  $(\widehat{f}, \widehat{g}) \in \mathcal{R}$,  the $\Gamma_{(\widehat{f}, \widehat{g})}$-action is minimal, that is, every $\Gamma_{(\widehat{f}, \widehat{g})}$-orbit is dense. 
\end{cor}

For example,  starting with the matrices $A$ and $B$ from  \eqref{eq.AandB},  there are $C^2$ neighborhoods $\mathcal{U}_A$ and $\mathcal{U}_B$ of $A$ and $B$, respectively, such that for any $(\widehat{f},\widehat{g}) \in \mathcal{U}_A \times \mathcal{U}_B$,  every $\Gamma_{(\widehat{f},\widehat{g})}$-orbit is either finite or dense.

The following corollary can be seen as stating that the only measures that can be invariant by every system in the support of $\mu'$ are the ``natural'' ones, meaning, either atomic or absolutely continuous with respect to $m$. 

\begin{corollary}\label{coro:commoninv}
Let $f,g\in \Diff^2 _m (S)$. Assume that there is a probability measure 
$\mu$ on $\{f,g\}$ that verifies UEF and UEP. Then, there exist  $C^2$-open neighborhood $\Ucal_f$ and $\Ucal_g$ in $\Diff^2(S)$ of $f$ and $g$, respectively,  such that for all pairs of $(\widehat{f},\widehat{g})\in\Ucal_f\times \Ucal_g$, if $\widehat{f}$ and $\widehat{g}$ have a common invariant non-atomic measure $\nu$, then $\nu$ is absolutely continuous with respect to $m$.
\end{corollary}

The absolute continuity part of the statement in our main theorems will be a consequence of a more technical result given by \Cref{thm:maintechnicaltheorem}, which is stated in \Cref{section.preliminaries}. Most of this paper is devoted to prove \Cref{thm:maintechnicaltheorem}.

Let us mention that recently, DeWitt and Dolgopyatt obtained the existence of absolutely continuous stationary measures for dissipative perturbations of volume preserving random dynamical systems having UEF and UEP (see \cite{JonDimaCo}). Indeed, their result holds in any dimension using a condition called \emph{coexpanding on average}.  However, their result only applies to random dynamical systems generated by systems that are very close to volume preserving ones.  Our main theorems allow the existence of ``very dissipative'' system. Let us also mention that the approach in this paper and DeWitt-Dolgopyatt's approach are very different.

\subsection{The uniformly hyperbolic case}

Our main results generalizes the main results in our previous paper \cite{BLOR}.  In what follows, we will explain the setting of our previous work, and state a slightly simplified version of our main theorem in there.  We explain how our main result here implies the main result in our previous work. 

 Fix two diffeomorphisms $f,g\in \mathrm{Diff}^2_m(\mathbb{T}^2)$. Consider the following conditions: 
\begin{enumerate}
\item[\textbf{(C1)}] $f$ and $g$ are Anosov diffeomorphisms having a splitting $T\mathbb{T}^2  = E^s_{\star} \oplus E^u_\star$, for $\star= f,g$. 
\item[\textbf{(C2)}] There exist continuous cone fields $x \mapsto \mathcal{C}^s_x$ and $x\mapsto \mathcal{C}^u_x$,  such that
\begin{itemize}
\item $Df^{-1}(x) \mathcal{C}^s_x \subset \mathcal{C}^s_{f^{-1}(x)}$  and  $Df^{-1}$ expands vectors uniformly in $\mathcal{C}^s$; 
\item $Dg^{-1}(x) \mathcal{C}^s_x \subset \mathcal{C}^s_{g^{-1}(x)}$ and $Dg^{-1}$ expands vectors uniformly in $\mathcal{C}^s$;
\item $Df(x) \mathcal{C}^u_x \subset \mathcal{C}^u_{f(x)}$ and $Df$ expands vector uniformly in $\mathcal{C}^u$; and 
\item $Dg(x) \mathcal{C}^u_x \subset \mathcal{C}^u_{g(x)}$ and $Dg$ expands vectors uniformly in $\mathcal{C}^u$.
\end{itemize}
\item[\textbf{(C3)}] For every $x\in \mathbb{T}^2$,  $ E^u_f(x) \cap E^u_g(x) = \{0\}$.
\item[\textbf{(C4)}] For every $x\in \mathbb{T}^2$,  $ E^s_f(x) \cap E^s_g(x) = \{0\}$.
\end{enumerate}
\begin{example} Two non-commuting hyperbolic matrices in $\SL(2,\Zbb)$ with positive entries satisfy the conditions \textbf{(C1)}-\textbf{(C4)}. Moreover, sufficiently small perturbations of such pair in $\Diff^2(\Tbb^2)\times \Diff^2(\Tbb^2)$ also satisfy \textbf{(C1)}-\textbf{(C4)} since they are all open conditions.
\end{example}
The following is a slightly simplified version of the  main theorem in \cite{BLOR}. 

\begin{ltheorem}\label{thm:UHcase}
Let $f$ and $g$ verify the conditions \textbf{(C1)}-\textbf{(C3)} above. For any $\beta \in (0, 1)$, there exist $C^2$-neighborhoods of $f$ and $g$ in $\mathrm{Diff}^2(\mathbb{T}^2)$, $\mathcal{U}_f$ and $\mathcal{U}_g$,  and $\mathcal{V}$ a weak*-neighborhood of $\mu_\beta:= \beta \delta_f + (1-\beta) \delta_g$ in the space of probability measures of $\mathrm{Diff}^2(\TT^2)$, with the following property: let $\mu \in \mathcal{V}$ be a probability measure such that $\mu(\mathcal{U}_f \cup \mathcal{U}_g) = 1$. Then,  there is a unique $\mu$-stationary SRB measure $\nu$ and this measure is absolutely continuous with respect to $m$.

Furthermore,  if we added the assumption that $f$ and $g$ also verify condition \textbf{(C4}),  then for every measure $\mu$ as above, every ergodic $\mu$-stationary measure  $\nu$ is either
\begin{enumerate}
\item absolutely continuous with respect to $m$, and, in this case, there is only one such measure; or 
\item  finitely supported.
\end{enumerate}
\end{ltheorem}

\Cref{thm:UHcase} will be again a consequence of our main result \Cref{thm:maintechnicaltheorem} below.  In \Cref{thm:UHcase}, one can actually show that the density of the absolutely continuous measure belongs to $\mathrm{L}^2(m)$. A natural question is the following.

\begin{quest}
Under what additional assumptions can we guarantee that the absolutely continuous measure obtained in \Cref{thm:UHcase} has a continuous density? What about smooth density? 
\end{quest}

As we will explain in \Cref{subsection.comparingtsujii}, our approach to prove the absolute continuity of stationary measures is inspired in Tsujii's work \cite{Tsujii-bigpaper}. Tsujii had a previous simpler paper \cite{Tsujii-Fat} in which he shows the absolute continuity of certain invariant measures for non invertible skew products in dimension $2$ under some transversality condition.  The density he obtains is in $\mathrm{L}^2(m)$. In \cite{AGT}, the authors show that under stronger transversality assumptions in this skew product setting, the density of the absolutely continuous measures could be improved to continuous.  Our question above is inspired in this latter result for endomorphisms.

\subsection{Examples}

Let us mention some examples of systems with UEF and UEP properties. In particular, our main results apply to perturbations of them.  For homogeneous systems,  consider any measure $\mu$ finitely supported on $\mathrm{SL}(2, \mathbb{Z})$ such that $\mu$ is strongly irreducible, that is, there are no finite union of subspaces in $\mathbb{R}^2$ which is preserved by every element in $\mathrm{supp}(\mu)$. Then, Furstenberg's Theorem \cite{Furstenberg} implies that there is a positive random Lyapunov exponent.  Furstenberg's proof also shows that the stable direction is random. In particular, the corresponding random dynamical system in $\mathbb{T}^2$ is UEF. Applying the same argument in the past, one obtains that $\mu$ is UEP.  So our theorem applies for dissipative perturbations of such random dynamical systems. 

Another, very interesting, source of examples is the following. Consider $S = S^2$, the sphere.  Let $R_1, \cdots, R_k$ be rotations generating $\mathrm{SO}(3)$.  Of course, since they are all isometries, any random dynamical systems supported on them is not hyperbolic, in particular, it is not UEF or UEP.  In \cite{DK}, the authors prove the following beautiful result.

\begin{theorem}\label{thm:DK}
There exists a number $r$ such that for any $k \geq 2$,  and any set of rotations $R_1, \cdots, R_k$ generating $\mathrm{SO}(3)$, there exists a number $\varepsilon>0$ such that for any $k$ diffeomorphisms $f_1, \cdots, f_k$ in $\mathrm{Diff}^r_m(S)$ such that $d_{C^r}(f_i, R_i)< \varepsilon$, then, either
\begin{enumerate}
\item The measure $\mu = \frac{1}{k} \sum_{j=1}^k \delta_{f_j}$ is UEF; or
\item There are rotations $R_1', \cdots, R_k'$ and a smooth diffeomorphism $h:S \to S$ such that $f_j \circ h = h \circ R_j'$, for every $j=1, \cdots, k$. In this case, we say that $f_1, \cdots, f_k$ are simultaneously conjugated to $R_1', \cdots, R_k'$. 
\end{enumerate}
\end{theorem} 

\begin{remark}
In \cite{DK}, the authors don't state their result exactly in this form for alternative (1). However, their proof actually gives the UEF condition. 
\end{remark}

With this result in mind, if one start with rotations $R_1, \cdots, R_k$, it is enough just to perturb one of them to break smooth conjugacy with any other rotation, and then obtaining a random dynamical system with UEF. Of course, this system will also have the UEP property, because if one runs \Cref{thm:DK} in the past, one would get smooth conjugacy with some rotations, which violates the UEF property.

\subsection{Relation with other works}\label{subsection.comparingtsujii}

As we mentioned, the main result in this work is about improving the regularity of an SRB measure to absolutely continuous.  In our setting,  the stationary measures will be hyperbolic. In random dynamics, unstable manifolds are not uniquely defined. Indeed, a point might have several different unstable manifolds, depending on different choices of possible pasts.  Such manifolds might ``oscillate'' depending on the choice of past.  The SRB property states that the stationary measure can be written as a convex combination of measures which are supported on unstable manifolds, each of which being equivalent to the arclength measure along the corresponding unstable manifold.  

The ``picture'' of the proof of our main technical theorem is to explore the oscillation of these unstable  manifolds to guarantee that there is enough overlap of them to ``cover'' something with positive area, thus guaranteeing the absolute continuity of the measure we are considering.  This idea is inspired in Tsujii's work \cite{Tsujii-bigpaper} in which he proves one of the Palis' conjectures for partially hyperbolic endomorphisms on $\TT^2$.  Let us note that, previously, there is a work by Bortolotti \cite{Bortolotti} which is inspired by Tsujii's work. Note that \cite{Bortolotti} considered a partially hyperbolic setting in dimension $3$ having smooth stable holonomies and assumed a  strong transversality condition.

Let us expand a little bit more on Tsujii's work so that we can compare it with ours and point out to the main innovations of our paper.

Tsujii considers the set of partially hyperbolic endomorphisms, system $f:\TT^2 \to \TT^2$ admiting a center direction $E^c$ and an unstable cone $\mathcal{C}^u$. Since, these systems are not invertible, there is not a well defined unstable direction. Indeed, unstable directions will depend on the choices of inverse branches of the system.  Tsujii shows that for a generic partially hyperbolic endomorphism, there are finitely many ergodic physical measures whose union of its basins cover a set of full Lebesgue measure (see \cite{Tsujii-bigpaper}  for details). The main novelty in his work comes from understanding the part of the dynamics having zero center Lyapunov exponents.  For simplicity, let us assume that Lebesgue almost every point has a zero Lyapunov exponent in the center direction.

Consider a curve $\gamma$ tangent to an unstable cone and let $m_\gamma$ be the normalized arclength measure on $\gamma$.  Consider the sequence
\[
m_n:= \displaystyle \frac{1}{n} \sum_{j=0}^{n-1} f^j_* m_\gamma.
\]
It is known that every accumulation measure of the sequence $m_n$ will have the property that it can be written as a convex combination of measures supported on unstable manifolds, each of which being equivalent to the arclength measure on the corresponding manifold.  In order to understand the absolute continuity of limits of $m_n$, Tsujii studies the so called $\rho$-norm of $m_n$, for some well chosen $\rho = \rho(n)$ that goes to $0$ as $n$ increases (see \Cref{Sec:semi} for the definition of the norm).  He introduces a notion of transversality condition, where transversality refers to a quantified ``oscillation'' of the unstable direction for the different choices of past.  This portion of his work can be divided into two parts:
\begin{itemize}
\item[(a)] The transversality condition implies that any limit of $m_n$ will have a part that is absolutely continuous.
\item[(b)] The transversality condition is generic.
\end{itemize}

Let us remark that given a partially hyperbolic endomorphism it is very hard to check if it verifies the transversality condition or not. So his argument to verify the transversality condition uses perturbative arguments.

For part (a), let us highlight a few features of his setting. 
\begin{itemize}
\item[\hypertarget{Tsu1}{(\textbf{T1})}] Estimates for the direction $E^c$ are not uniform, since we are assuming only that center Lyapunov exponents are zero (or close to zero).  So he has to work with some version of finite time ``Pesin blocks''.
\item[\hypertarget{Tsu2}{(\textbf{T2})}] The existence of unstable cones allows one to obtain uniform control on expansion of iterates of the curve $\gamma$. It also allows one to obtain uniform estimates on distortion for the density of iterates of $m_\gamma$.
\end{itemize}
For \hyperlink{Tsu2}{(\textbf{T2})},  the existence of uniform unstable cones, in his setting, is very convenient for the proof, as the computations are using this type of control all the time. 

In our work,  we introduce a notion of transversality that we want to explore to prove absolute continuity.  The main differences that appear are:
\begin{enumerate}
\item The presence of nonuniform stable and unstable directions.
\item Our result is non perturbative.
\end{enumerate}
The second item means that we start with a driving measure $\mu$ verifying some properties and we conclude absolute continuity of SRB measures which are stationary for any $\mu'$ sufficiently close to $\mu$.  This requires us to verify transversality without perturbing, which is typically a delicate problem. Moreover, we want that this notion of transversality persists for measures $\mu'$ nearby.  The amount of transversality we obtain is given  by some very quantified estimate on the oscillation of unstable directions, this is where UEP is used (see \Cref{Sec:someestimates}).  

The first part leads to the main challenge of this paper. For any curve in $\TT^2$ and under most words, the curvature of its random iteration may blow up. This is worse than \hyperlink{Tsu2}{(\textbf{T2})} in Tsujii's setting. Fortunately, the uniform expansion properties \hyperlink{A3}{(\textbf{A3})}-\hyperlink{A4}{(\textbf{A4})} imply the following:
\begin{itemize}
\item Most words are ``uniformly hyperbolic''. Their unstable directions are random enough to have lots of transversality. The stable directions within a non-uniform cone field in \hyperlink{A2}{(\textbf{A2})} are also random. (In the setting of Theorems \ref{simplifiedmaintheorem1}-\ref{simplifiedmaintheorem3}, this non-uniform cone field is the entire tangent bundle $T\TT^2$.)
\item For any $n\in\ZZ_+$, for any short enough $C^2$ curve $\gamma$ in $\TT^2$ tangent to the non-uniform cone field in \hyperlink{A2}{(\textbf{A2})} and for most words $\omega_n=(f_1,\dots, f_n)$ of length $n$, the curvatures of the curves $\gamma,f_1\circ\gamma,f_2\circ f_1\circ\gamma,\dots, f_n\circ\dots \circ f_1\circ\gamma$ are bounded by a slowly increasing exponential function in $n$. If the curve carries a measure which is absolute continuous with respect to arclength measure and has log-Lipschitz density, a similar property holds for the log-Lipschitz constant of the densities.
\end{itemize}
The benefits of the first bullet point win over the bad effects caused by (1) dissipativeness of the random dynamics \hyperlink{A5}{(\textbf{A5})}-\hyperlink{A6}{(\textbf{A6})}, and (2) the increasing curvatures and log-Lipschitz constants of densities in the second bullet point. Hence, we can show that for any $n\in\ZZ_+$ and for any $C^2$-curve $\gamma$ tangent to the non-uniform cone field in \hyperlink{A2}{(\textbf{A2})}, if we let $m_\gamma$ be the arclength measure, then the measure 
$$m_n:=\frac{1}{n}\sum_{j=0}^{n-1}\mu^{*j}*m_\gamma.$$
has a large part $m_n'$ such that any limit measure of the sequence $\{m'_n\}_{n\in\ZZ_+}$ is absolutely continuous with respect to the Lebesgue measure. Moreover, $m_n'(\TT^2)$ can be chosen to be arbitrarily close to $m_\gamma(\TT^2)$ independent of $n$. Roughly speaking, the measure $m'_n$ is obtained by discarding some parts of the $\mu$-convolution in $m_n$ to avoid rapid blow-up of curvature and density. (See the second bullet point in the above.) $m_\gamma$ can also be replaced by a convex combination of measures supported on curves tangent to the non-uniform cone field in \hyperlink{A2}{(\textbf{A2})} which are (1) absolute continuous with respect to the arclength measure, and (2) has log-Lipschitz density. The details of the above statement is in Theorem \ref{thm:main.detailed}, which is the core of Theorems \ref{simplifiedmaintheorem1}-\ref{thm:maintechnicaltheorem}.

\subsection{Organization of this paper}

In \Cref{section.preliminariesrandom}, we give some basic results and constructions for random dynamical systems.  In  \Cref{section.preliminaries},  we establish the setting to state our main result \Cref{thm:maintechnicaltheorem}. Some basic estimates that we will need are also in there.  \Cref{Sec:someestimates} gives some crucial estimates which are consequences of UEF and UEP, and that will be used throughout the rest of the paper. In particular,  quantitative estimates on the ``oscillation'' of finite time unstable and stable directions.

 In Sections \ref{Sec.curvaturedensity} and \ref{Sec.Distortion}, we establish estimates on curvature, density and distortion on iterates of curves.  We introduce the notion of $\rho$-norm of a measure in \Cref{Sec:semi}. This is the main tool we will use to ``detect'' absolute continuity.  We define our notion of admissible measures and obtain some basic estimates about them in \Cref{sect:ACAM}.  

The most important estimate for our proof is done in \Cref{sec.lasotayorke}, in which we prove a Lasota-Yorke type of estimate for the $\rho$-norm. We then use this estimate to show absolute continuity of stationary measures in \Cref{sec.abs} and finish the proof of \Cref{thm:maintechnicaltheorem} in \Cref{sec.mainthm}.  Finally, we prove our main theorems that appeared in the introduction  in \Cref{Sec.proofmaintheorems}.

\subsection*{Acknowledgments}

A.\ B.\ was partially supported by the  National Science Foundation under Grant Nos.\ DMS-2020013 and DMS-2400191.  H.\ L.\ was supported in part by AMS-Simons travel grant, a KIAS Individual Grant (HP104101) via the June E Huh Center for Mathematical Challenges at Korea Institute for
Advanced Study, and Sang-hyun Kim's 
Mid-Career Researcher Program (RS-2023-00278510) through the National Research Foundation funded by the government of Korea. D.\ O.\ was partially supported by the National Science Foundation under Grant No. \ DMS-2349380.

\section{Preliminaries on random dynamical systems}\label{section.preliminariesrandom}
In this section, we collect some preliminaries on random dynamical systems. Most of the arguments can be found in many literatures, such as \cite{Liu-Qian-book}.  
\subsection{$\Diff^2 (S)$ and Borel probability measures on $\Diff^2 (S)$}
On a $C^\infty$ closed connected surface $S$, the group of $C^2$ diffeomorphisms $\Diff^{2}(S)$ is an open subset of $C^2 (S,S)$ which is the $C^2$ maps from $S$ to itself. $C^2(S,S)$, with $C^2$ topology on it, can be metrized to a Polish space, that is, complete and separable space. Hence, $\Diff^2 (S)$ can be metrized to a separable metric space with $C^2$ topology, so that, the space of probability measures on $\Diff^2(S)$, with weak*-topology, can be metrized to a separable metric space.

 \begin{remark}
Note that, even if $\mu$ is weak*-close to $\overline{\mu}$, $\supp \mu$ may not be contained in $C^2$ neighborhood of the support of $\overline{\mu}$. 
 \end{remark}

\subsection{Skew extension and stationary measure} We recall facts on random dynamical systems on smooth manifolds.  For the simplicity of notations, we denote the product measure $\mu^{\otimes N}$ by $\mu^ N$ for $N=1,\dots, \Nbb$ or $\Zbb$.

Let $M$ be a smooth closed manifold.  With $C^{2}$-topology on $\Diff^{2}(M)$, $\Diff^{2}(M)$ is Polish space. Let $\mu$ be a probability measure on $(\Diff^{2}(M) ,\Bcal(\Diff^{2}(M) )$ where $\Bcal(\Diff^{2}(M))$ is a Borel $\sigma$-algebra of $\Diff^{2}(M)$. When we have a probability measure on space, we always consider the completion of the $\sigma$-algebra with respect to the measure and still denote the completion of $\sigma$-algebra by the same notation.
\begin{dfn}\label{dfn.stationary}
A probability measure $\nu$ on $M$ is called a \emph{stationary measure with respect to $\mu$ (or $\mu$-stationary) measure}, if 
\[\nu=\int_{\Omega} \left(f_{*}\nu\right) d\mu(f).\]
\end{dfn}

Let $\Sigma^+=(\Diff^{2}(M))^{\Nbb}$. $\Sigma^+$ equipped with the Borel probability measure $\mu^{\Nbb}$ which is an infinite product of $\mu$. We sometimes denote $\mu^+$ in order to emphasize that it is Bernoulli measure for the forward words. For each $\omega\in \Sigma^+$, $\omega=\left(f_0,f_1, \dots\right)$, we define 
\[f_{\omega}^{0}=id, f_{\omega}^{n}=f_{n-1}(\omega)\circ\dots\circ f_{0}(\omega)\quad\textrm{ for } n\ge 1.\]

Naturally, we can consider a skew product related to the random dynamical system $F^{+}\colon\Sigma^{+}\times M \to \Sigma^{+}\times M$ as \[F^{+}\colon\left(\omega, x\right)\mapsto \left(\sigma(\omega),f_{\omega}(x)\right)\] where $\sigma:\Sigma^{+}\to \Sigma^{+}$ is the (left) shift map and $f_{\omega}=f_{\omega}^{1}$.
\begin{claim}[\cite{Liu-Qian-book}]
$\nu$ is $\mu$ stationary measure if and only if $\mu^{\Nbb}\otimes \nu$ is $F^{+}$ invariant. Furthermore, $\nu$ is $\mu$ ergodic stationary measure if and only if $\mu^{\Nbb}\otimes \nu$ is $F^{+}$ ergodic invariant measure.
\end{claim}

We also can consider the natural extension. Let $\Sigma = (\Diff^{2}(M))^{\Zbb}$. $\Sigma$ equipped with the Borel probability measure $\mu^{\Zbb}$ which is an infinite produce of $\mu$. 
We define, for each $\omega=(\dots, f_{-1}(\omega),f_{0}(\omega),f_{1}(\omega),\dots)\in\Sigma$, 
\[\begin{cases}f^{n}_{\omega}=f_{n-1}(\omega)\circ\dots\circ f_{0}(\omega)\quad \textrm{ for } n\ge 1\\ f^{0}_{\omega}=id \\ f^{n}_{\omega}=\left(f_{-n}(\omega)\right)^{-1}\circ\dots\circ \left(f_{-1}(\omega)\right)^{-1}\quad\textrm{ for } n<0 \end{cases}.\] Then, we have the natural extension of $F^+$ as   \[F\colon \Sigma\times M\to \Sigma\times M, \quad F(\omega,x)=(\sigma(\omega),f_{\omega}(x))\] where $f_{\omega}=f^{1}_{\omega}$. We denote $\omega^{+}$ and $\omega^{-}$ by $\omega^{+}=\left(\omega_{0},\omega_{1},\dots\right)$ and $\omega^{-}=\left(\dots, \omega_{-2},\omega_{-1}\right)$. Note that $\Omega$ has a probability measure $\mu^{\mathbb{Z}}$ which is an infinite product measure $\mu^{\mathbb{Z}}$.

\begin{prop}\label{prop2.3.}
There exists a unique Borel probability measure $\widehat{\nu}$ on $\Sigma\times M$ such that 
\begin{enumerate}
\item $\widehat{\nu}$ is $F$ invariant, and
\item $P_{*}^{+}\left(\widehat{\nu}\right)=\mu^{\Nbb}\times \nu$ where $P^{+}\colon \Sigma\times M\to \Sigma^{+}\times M$ is the projection.
\end{enumerate}
Furthermore, if we disintegrate $\widehat{\nu}$ with respect to $P\colon \Sigma\times M\to \Sigma$, there is a family of Borel probability measure $\left\{\nu_{\omega}\right\}_{\omega\in \Sigma}$ such that \[\widehat{\nu}=\int_{\Sigma} \nu_{\omega} d\mu^{\mathbb{Z}}(\omega).\] Furthermore, for $\mu^{\Zbb}$ almost every $\omega=(\dots, f_{-1}(\omega), f_{0}(\omega),f_{1}(\omega),\dots)$, $\nu_{\omega}$ only depends on $\omega^{-}=\left(\dots, f_{-2}(\omega),f_{-1}(\omega)\right)$.
\end{prop}
Let $\mu^-$ be the Bernoulli measure on the set of $\omega^-$.
Note that $\nu_{\omega}=\nu_{\omega^{-}}$ since it only depends on $\omega^{-}$. We call a probability measure $\nu_{\omega}=\nu_{\omega^{-}}$ on $M$ a \emph{sample measure} with respect to $\omega$. From the above discussions, we can write \[\int g d\widehat{\nu}=\int g(\omega^{-},\omega^{+},x) d\nu_{\omega^{-}}(x)d\mu^{-}(\omega^{-}) d \mu^{+}(\omega^{+}) \] for all bounded measurable function $g\colon\Sigma\times M\to \Rbb$.

\subsection{Lyapunov exponents, stable manifolds and unstable manifolds}\label{sec:stablemfd} 

From this section, we assume that $\dim M=2$. Furthermore, we assume that $\mu$ is a probability measure on $\Diff^{2}(M)$ concentrated on a $C^2$ open set $\Ucal$ and the following moment condition holds: 
\begin{equation}\tag{$L^{1}$}
\int_{\Omega}\left( \ln^{+}||f||_{C^{2}} +\ln^{+}||f^{-1}||_{C^{2}} \right) d\mu(f) <\infty
\end{equation}
where $\ln^{+}(x)=\max\{x,0\}$ and $||\cdot||_{C^{2}}$ is the $C^{2}$-norm of a diffeomorphism.

As before, consider  $\Sigma= \Ucal^{\mathbb{Z}}$.   
We defined the following skew product
\[
\begin{array}{rcl}
\Sigma \times  M& \longrightarrow & \Sigma \times M\\
\left(\omega, x\right) & \mapsto & \left( \omega, f_{\omega_{0}}(x) \right),
\end{array}
\]
where $\sigma$ is the left shift.   
We also define the set $\Sigma^+ = \Ucal^{\mathbb{N}}$ and a similar skew product on $\Sigma^+ \times M$. 

Let $\nu$ be a $\mu$-ergodic stationary measure on $M$ and $\widehat{\nu}$ be the induced measure on $\Sigma\times M$ as in \Cref{prop2.3.}.\begin{prop}\label{prop:UH}
Under the above setting, for $\widehat{\nu}$-almost every  $(\omega,x)\in \Sigma\times M$, there exists $\lambda_1\le \lambda_2$ and a splitting $T_x M=E^{1}_{\omega}(x)\oplus E^{2}_{\omega}(x)$ such that 
\begin{enumerate}
\item $D_{x}f_{\omega_{0}}E^{i}_{\omega}(x)=E^{i}_{\sigma(\omega)}(f_{\omega}(x))$ for $i=1,2$
\item For $v^{i}\in E^{i}_{\omega}(x)$, we have 
\[\lim_{n\to \pm\infty}\frac{1}{|n|} \log \|D_{x}f_{\omega}^{(n)}v^{i}\|=\lambda_i\] for $i=1,2$.
\end{enumerate}
\end{prop}

If $\lambda_1 <0 <\lambda_2$ then we call $\nu$ is \emph{hyperbolic}. From now on, we assume that $\nu$ is hyperbolic $\mu$-ergodic stationary measure on $M$. We also write $E^{1}_{\omega}(x)=E^{s}_{\omega}(x)$ and $E^{2}_{\omega}(x)=E^{u}_{\omega}(x)$.

Let \[W_{\textrm{loc}, r}^{s}(\omega,x)=\left\{y\in \Tbb^{2}: d(f^{n}_{\omega}(y),f^{n}_{\omega}(x))\le r, \textrm{ for all $n\ge 0$ and }\lim_{n\to\infty}d(f^{n}_{\omega}(y),f^{n}_{\omega}(x))=0\right\}\] and
\[W_{\textrm{loc}, r}^{u}(\omega,x)=\left\{y\in \Tbb^{2}: d(f^{-n}_{\omega}(y),f^{-n}_{\omega}(x))\le r,  \textrm{ for all $n\ge 0$ and } \lim_{n\to\infty}d(f^{-n}_{\omega}(y),f^{-n}_{\omega}(x))=0\right\}\]
\begin{prop}\label{prop:stablemfd}
Under the above setting, for $\widehat{\nu}$ almost every $(\omega,x)\in \Sigma\times M$, there exists $r(\omega,x)>0$ such that 
\begin{enumerate}
\item $(\omega,x)\mapsto r(\omega,x)$ is measurable and,
\end{enumerate}
for all $r\in (0,r(\omega,x))$ and $*=s,u$, 
\begin{enumerate}[resume]
\item $W_{r}^{s}(\omega,x)$ and $W_{r}^{u}(\omega,x)$ is a $C^{2}$ embedded curve tangent to $E^{*}_{\omega}(x)$. Indeed, there exists $C^{1}$
\item $W_{r}^{*}(\omega,x)$ is continuous in $x$ with respect to the $C^{2}$ topology.
\item There exists $C\ge 1$ and $0<\lambda<1$ such that $W_{r}^{*}(\omega,x)$ can be characterized by 
\[
\begin{aligned}
W_{r}^{s}(\omega,x)
= \Bigl\{ y \in \Tbb^{2} : {} & 
d(f^{n}_{\omega}(y), f^{n}_{\omega}(x)) \le r \text{ and} \\
& d(f^{n}_{\omega}(y), f^{n}_{\omega}(x))
   \le C \lambda^{n} d(f^{n}_{\omega}(y), f^{n}_{\omega}(x))
   \text{ for all } n \ge 0 \Bigr\},
\end{aligned}
\]

\[
\begin{aligned}
W_{r}^{u}(\omega,x)
= \Bigl\{ y \in \Tbb^{2} : {} & 
d(f^{-n}_{\omega}(y), f^{-n}_{\omega}(x)) \le r \text{ and} \\
& d(f^{-n}_{\omega}(y), f^{-n}_{\omega}(x))
   \le C \lambda^{n} d(f^{-n}_{\omega}(y), f^{-n}_{\omega}(x))
   \text{ for all } n \ge 0 \Bigr\}.
\end{aligned}
\]
\end{enumerate}
\end{prop}

Indeed, for $\widehat{\nu}$ almost every $(\omega,x)$ in $\Sigma\times M$, there exists a $C^{2}$ function $\phi_{\omega,x}^{s}:E^{s}_{\omega,x}(r)\to E^{u}_{\omega,x}$ so that $\phi_{\omega,x}^{s}(0)=0$, $D_{0}\phi_{\omega,x}^{s}(0)=0$ and $W_{r}^{s}(\omega,x)=\exp(\textrm{graph}\phi_{\omega,x}^{s})$ where $E^{s}_{\omega,x}(r)=\{v\in E^{s}_{\omega}(x):||v||<r\}$. The $E^{u}_{\omega,x}(r)$ is defined similarliy. Global stable and unstable manifolds are defined by 
\begin{align*}
W^{s}_{\omega}(x)&=\bigcup_{n\ge 0}f_{\omega}^{-n}W^{s}(\sigma^{n}\omega, f^{n}_{\omega}(x))=\left\{y\in \Tbb^{2}:	\lim_{n\to\infty} d(f^{n}_{\omega}(y),f^{n}_{\omega}(x))=0	\right\}\end{align*}

\begin{align*}
W^{u}_{\omega}(x)&=\bigcup_{n\ge 0}f_{\omega}^{n}W^{u}(\sigma^{-n}\omega, f^{-n}_{\omega}(x))=\left\{y\in \Tbb^{2}:	\lim_{n\to-\infty} d(f^{n}_{\omega}(y),f^{n}_{\omega}(x))=0	\right\}\end{align*}

\subsection{Random SRB measures}\label{sec:SRBdef} 
Recall that we fix a hyperbolic $\mu$-ergodic stationary measure $\nu$ on $M$. 

\begin{definition}[Subordinated partition]
A measurable partition $\eta$ of $\Sigma\times M$ is said to be \emph{subordinated to $W^{u}$ manifolds} if 
for $\widehat{\nu}$ almost every $(\omega,x)$, the set $\{y\in M:(\omega,y)\in \eta(\omega,x)\}$ is 
\begin{enumerate}
\item precompact in $W^{u}_{\omega}(x)$,
\item contained in $W^{u}_{\omega}(x)$, and
\item  contains an open neighborhood of $x$ in $W^{u}_{\omega}(x)$.
\end{enumerate}
\end{definition}
\begin{theorem}[Existence of expanding subordinated partition]
There exits $W^{u}$-subordinated measurable partition $\eta$ such that $\eta\prec F^{-1}\eta$.
\end{theorem}

For a $W^{u}$-subordinated measurable partition $\eta$, we denote by $\eta_{\omega}(x)$ the set $\{y\in M:(\omega,y)\in \eta(\omega,x)\}$. Note that for each $\omega$, $\eta_{\omega}(x)$ forms a measurable partition on $M$. Let $\widehat{\nu}_{(\omega,x)}^{\eta}$ be a system of conditional measures with respect to $\eta$. By transitivity of conditional measures, we have $(\nu_{\omega})^{\eta_{\omega}}_{x}=\widehat{\nu}_{(\omega,x)}^{\eta}$ for $\widehat{\nu}$ almost every $(\omega,x)$. Here $(\nu_{\omega})^{\eta_{\omega}}_{x}$ be a system of conditional measures of $\nu_{\omega}$ associated to the measurable partition $\eta_{\omega}$.
\begin{definition}[Random SRB]We say that  a stationary measure $\nu$ has the \emph{SRB property} if, for every $W^{u}$-subordinated measurable partition $\eta$, $(\nu_{\omega})^{\eta_{\omega}}$ is absolutely continuous with respect to the Lebesgue measure on $W^{u}(\omega,x)$ inherited by immersed Riemannian submanifold structure on $W^{u}(\omega,x)$.
\end{definition}
\begin{remark} Every hyperbolic $\mu$-stationary measure which is absolutely continuous with respect to the Lebesgue measure has SRB property.
\end{remark}
\begin{theorem}[Log-Lip regularity of the density]
Let $\nu$ be a stationary measure on $M$ with respect to $\mu$. Assume that $\nu$ has a SRB property. Let $m^{u}_{(\omega,x)}$ be the Lebesgue measure on $W^{u}_{\omega}(x)$ inherited by immersed Riemannian structure on $W^{u}_{\omega}(x)$. Fix a  $W^{u}$-subordinated measurable partition $\eta$ such that $\eta\prec F^{-1}\eta$. 

Then, for $\widehat{\nu}$ almost every $(\omega,x)$ in $\Sigma\times M$, \[g_{\omega}(y)=\frac{d(\nu_{\omega})^{\eta_{\omega}}}{dm^{u}_{(\omega,x)}}(y)\]  is strictly positive and locally Log-Lipschitz.
\end{theorem}

\section{Technical assumptions, basic estimates and statement of \Cref{thm:maintechnicaltheorem}} \label{section.preliminaries}

In this section, we will establish the setting  and state our main technical theorem. We also prove some basic estimates coming from uniform expansion that are going to be crucial for us. 

\subsection{Setting and the statement of the main technical theorem}

 Let $\cU'\subset \cU$ be  $C^2$-open subsets  of  $\mathrm{Diff}^2(\TT^2)$. We write
$$\Sigma=\cU^\ZZ,~\Sigma'=(\cU')^\ZZ,~\Sigma_n'=(\cU')^n\text{ and } \Sigma_n=\cU^n \text{ for any }n\in\ZZ_+.$$
For any $x\in\TT^2$,   $n\in\ZZ_+$,  $\omega^n=(f_0,\dots,f_{n-1})\in\Sigma_n$ and  $k\in\{1,\dots,n\}$, we let 
\[f^k_{\omega^n}=f_{k-1}\circ\dots\circ f_0.\] Similarly, for any $\omega=(\dots, f_{-1},f_0,f_1,\dots)\in\Sigma$ and  $k\in\ZZ_+$, we let $f^k_{\omega}=f_{k-1}\circ\dots\circ f_0$.

For  $F\in\mathrm{Diff}^2(\TT^2)$ and $x\in\TT^2$, we write $\Jac F(x):=|\det DF(x)|$. Observe that if $\|DF(x)\|^2>\Jac F(x)$,  then, the singular values of $DF(x)$ are different.  In this case,  let $E^u_F(x), E^s_F(x) \in\PP T_x\TT^2$  be the mostly expanding and mostly contracting directions, that is, 
\begin{align}\label{notation:EuEs}
\left\|DF(x)|_{E^u_F(x)}\right\|=\|DF(x)\|\text{ and }\left\|DF(x)|_{E^s_F(x)}\right\|=\frac{\Jac F(x)}{\|DF(x)\|}.
\end{align}

We will make several assumptions on $\mathcal{U}$ and $\mathcal{U}'$.  We will need Lemmas \ref{lem.DK.trick}, \ref{lem:VAP} and Corollary \ref{cor.mostly expand/contract} below. To increase readability, we will state them in this subsection and postpone the proofs to \Cref{subsection:proofslemmas}.  Our first assumption is the following. 
\begin{enumerate}
\item[\hypertarget{A1}{(\textbf{A1})}] \textbf{Uniformly bounded $C^2$-norm on $\cU$}: There exists $\CS{C'}{0}>2$  such that for any $F\in\cU$, we have 
\begin{align}\label{eqn:unif.C2.norm}
\max\{\|F\|_{C^2}, \|F^{-1}\|_{C^2}\}\leq e^{\C{C'}{0}}.
\end{align}
\end{enumerate}

Fix an open subset $\cC\subset T\TT^2\setminus \{0\}$ invariant under $\RR\setminus\{0\}$-scalar multiplication. For simplicity, we write 
$$\cC^1:=\cC\cap T^1\TT^2,~\cC_x:=\cC\cap T_x\TT^2\text{ and }\cC^1_x:=\cC\cap T^1_x\TT^2\text{ for any }x\in\TT^2.$$ 
Throughout this paper, we let $\mu\in\mathrm{Prob}(\cU)$ be a probability measure such that there exist $\CS{C}{1}>0$ and $N\in\NN$ verifying the following: 
\begin{enumerate}
\item[\hypertarget{A2}{(\textbf{A2})}] ($\cC$-invariance in the future) For any $F\in\supp \mu$, we have $DF(\cC)\subset \cC$.
\item[\hypertarget{A3}{(\textbf{A3})}] ($(\cC;N,\C{C}{1})$-\textbf{UEF}, uniform expansion on average in the future for vectors in $\cC$) For any $x\in\TT^2$ and for any $v\in\cC^1_x$, we have
$$\int_{\Sigma_N}\ln\|Df^N_{\omega^N}(x)v\|d\mu^N(\omega^N)>\C{C}{1}.$$
\item[\hypertarget{A4}{(\textbf{A4})}] (\textbf{$(N,\C{C}{1})$-UEP}, uniform expansion on average in the past for all unit vectors) For any $x\in\TT^2$ and any $v\in T^1\TT^2$, we have
$$\int_{\Sigma_N}\ln\left\|\left(Df^{N}_{\omega^{N}}((f^N_{\omega^N})^{-1}(x))\right)^{-1}v\right\|d\mu^{N}(\omega^{N})>\C{C}{1}.$$
\end{enumerate}

\begin{remark}
The set $\mathcal{C}$ can be chosen to be $T\mathbb{T}^2 \setminus \{0\}$. 
\end{remark}

The following two results are essentially contained in \cite{DK}.  They give us some important basic consequences of the UEP and UEF properties.

\begin{lem}\label{lem.DK.trick} Let $\Ucal$ be an $C^{2}$-open set verifying \hyperlink{A1}{\textbf{(A1)}}. Assume that $\mu\in \textrm{Prob}(\Ucal)$ satisfies \hyperlink{A2}{\textbf{(A2)}} - \hyperlink{A4}{\textbf{(A4)}}. Then, there exists a constant $\CS{C}{2}>0$ such that for any $n>0$ and for any $\delta\in(0,\min\{(N\C{C'}{0})^{-1}, (2\C{C}{1})^{-1}, \C{C}{1}(N\C{C'}{0})^{-2}/2\})$, there exists a constant $\chi=\chi(\delta,\C{C}{1},N)>0$ such that 
\begin{enumerate}
\item $\displaystyle \int_{\Sigma_n} \|Df^n_{\omega^n}(x)v\|^{-\delta}d\mu^n(\omega^n)<e^{\C{C}{2}}e^{-n\chi}$ for any $x\in\TT^2$ and $v\in\cC^1_x$.
\item$\displaystyle \int_{\Sigma_n} \left\|\left(Df^{n}_{\omega^{n}}((f^n_{\omega^n})^{-1}(x))\right)^{-1}v\right\|^{-\delta}d\mu^{n}(\omega^{n})<e^{\C{C}{2}}e^{-n\chi}$ for any $x\in\TT^2$ and $v\in T_x^1\TT^2$.
\end{enumerate}
\end{lem}

\begin{remark}
From now on, we fix a choice of $\delta$ and $\chi$ which satisfies Lemma \ref{lem.DK.trick}.
\end{remark}
\begin{cor}\label{cor.mostly expand/contract}
For any $n>0$, for any $0<\overline{\chi}<\chi$ and for any $C\in\RR$, the following holds:
\begin{enumerate}
\item For any $x\in\TT^2$ and for any $v\in\cC^1_x$, we have
$$\mu^n\left(\{\omega^n:\|Df^n_{\omega^n}(x)v\|^{-\delta}\geq e^{-C}e^{-n\overline{\chi}}\}\right)<e^{\C{C}{2}}e^{C}e^{-n(\chi-\overline{\chi})}.$$
\item For any $(x,v)\in T^1\TT^2$, we have
$$\mu^{n}\left(\left\{\omega^{n}:\left\|\left(Df^{n}_{\omega^{n}}((f^n_{\omega^n})^{-1}(x))\right)^{-1}v\right\|^{-\delta}\geq e^{-C}e^{-n\overline{\chi}}\right\}\right)<e^{\C{C}{2}}e^{C}e^{-n(\chi-\overline{\chi})}.$$
\end{enumerate}
\end{cor}

From now on, we fix a choice of $\overline{\chi}\in (0,\chi)$. We assume in addition that $\overline{\chi}$ satisfy
\begin{align}\label{eqn:epsilon.cond.1}
\overline{\chi}\in(0,\min\{\chi,\delta\C{C'}{0}/2\}).
\end{align}
Let us continue with the assumptions we will need for our results. 

\begin{enumerate}
\item[\hypertarget{A5}{(\textbf{A5})}] \textbf{$(\C{C}{0},\CS{\epsilon}{0})$-almost volume preserving on $\cU'$}: There exist $\CS{C}{0}>0$ and 
\[
\CS{\epsilon}{0} = \epsilon_0 ( \C{C}{0}, \C{C'}{0}, \overline{\chi}, \delta, \C{C}{1}),
\] with $0<\C{\epsilon}{0}\ll1$, to be determined later,  such that for any $n\in\ZZ$, $x\in \TT^2$ and  $\omega^n\in\Sigma_n'$, we have 
\begin{align}\label{eqn:partial.VAP}
\Jac f^n_{\omega^n}(x)\in(e^{-\C{\epsilon}{0}|n|-\C{C}{0}}, e^{\C{\epsilon}{0}|n|+\C{C}{0}}).
\end{align}

\end{enumerate}

\begin{remark} 
The condition we will require for ${\C{\epsilon}{0}}$ are that (1) ${\C{\epsilon}{0}}<\min\{\C{C}{0},\C{C'}{0},\overline{\chi}/2\delta\}$, and (2) that it  verifies the condition given by \eqref{eqn:epsilon0}.
\end{remark}

\begin{lem}\label{lem:VAP}
Let $\mathcal{U}' \subset \mathcal{U}$ be  $C^2$-open sets verifying assumptions \hyperlink{A1}{(\textbf{A1})} and \hyperlink{A5}{(\textbf{A5})}. Suppose that $\mu$ is a probability measure verifying assumptions \hyperlink{A2}{(\textbf{A2})} - \hyperlink{A4}{(\textbf{A4})}. Fix any $\CS{\epsilon}{1}\in(0,\frac{\C{\epsilon}{0}}{\C{C}{0}+2\C{C'}{0}})$. If $\CS{\epsilon}{2}>0$ verifies 
\begin{align}\label{eqn:e1.ineq-1}
-\ln(\C{\epsilon}{2})>\frac{1}{\C{\epsilon}{1}}-\ln(\C{\epsilon}{1})-\frac{1-\C{\epsilon}{1}}{\C{\epsilon}{1}}\ln(1-\C{\epsilon}{1}),
\end{align}
and, moreover,  $\mu$ satisfies  $\mu(\mathcal{U}') > 1- \C{\epsilon}{2}$, 
then for any $n\in\ZZ_+$, we have
$$\mu^n\left(\left\{\omega^n:\Jac f^n_{\omega^n}(x)\in(e^{-\C{C}{0}-2n\C{\epsilon}{0}}, e^{\C{C}{0}+2n\C{\epsilon}{0}})\text{ for any }x\in\TT^2\right\}\right)\geq 1-\C{C'}{1}e^{-n}.$$
\end{lem}

Finally, our last assumption is the following.
\begin{enumerate}
\item[\hypertarget{A6}{(\textbf{A6})}] ($\CS{\epsilon}{2}$-almost supported on $\cU'$) Fix $\C{\epsilon}{1}$ and $\C{\epsilon}{2}$ verifying the hypothesis of Lemma \ref{lem:VAP}. We assume that  $\mu(\cU')>1-\C{\epsilon}{2}$.
\end{enumerate}

\begin{remark}\label{remark:constants}
In the rest of this paper, whenever we introduce a new constant, we do not track its dependence on $\C{C}{0},\C{C'}{0},N,\C{C}{1},\delta,\overline{\chi}$ and other constants which only depend on the aforementioned list of constants. In particular, the dependence on $\C{C}{2}$ and $\chi$ will not be tracked since $\C{C}{2}=\ln(4e/3)$ and $\chi=\delta\C{C}{1}/2N$.
\end{remark}

The main technical result in this paper is the following. 

\begin{ltheorem}\label{thm:maintechnicaltheorem}
Suppose that $\mathcal{U}' \subset \mathcal{U}$ are $C^2$-open sets in $\mathrm{Diff}^2(\mathbb{T}^2)$ and $\mu$ a probability measure verifying assumptions \hyperlink{A1}{(\textbf{A1})} - \hyperlink{A6}{(\textbf{A6})}.  Then,  there are only finitely many  $\mu$-stationary measures SRB. Moreover, each $\mu$-stationary SRB measure is absolutely continuous with respect to a volume measure on $\mathbb{T}^2$. 
\end{ltheorem}

\begin{remark}
Our main technical results holds for any surface $S$. However, for clarity in the presentation, we decided to present the proofs working on $\TT^2$.  The proofs can be adapted directly by working in local charts. 
\end{remark}

\begin{remark}
UEF is related to ``oscillation'' of stable directions and this is the key ingredient to apply the measure classification result in \cite{Brown-Hertz}. This alone gives that stationary measures are either atomic or SRB.  Let us remark that SRB property alone does not imply absolute continuity of the measure.  UEP is related to ``oscillation'' of unstables. This is the key property that will improve the regularity of SRB measures. 
\end{remark}

\subsection{Proofs of Lemmas \ref{lem.DK.trick}, \ref{lem:VAP} and Corollary \ref{cor.mostly expand/contract}}\label{subsection:proofslemmas}

\begin{proof}[Proof of \Cref{lem.DK.trick}]

For any $n>0$, we write $n=qN+d$ for some $q\in\NN$ and $d\in\{0,\dots,N-1\}$. For any $\omega^n\in\Sigma_n$,  any $(x,v)\in T^1\TT^2$ and any $k\in\{0,\dots,q\}$, we write 
$$x_{k,\omega^n}:=f^{kN+d}_{\omega^n}(x)\text{ and }v_{k,\omega^n}:=\frac{Df^{kN+d}_{\omega^n}(x)v}{\|Df^{kN+d}_{\omega^n}(x)v\|}.$$ 
If $v\in\cC^1_x$ and $\omega^n\in (\supp\mu)^n$, by \hyperlink{A2}{(\textbf{A2})}, we have $v_{k,\omega^n}\in\cC^1_{x_{k,\omega^n}}$ for any $k\in\{0,\dots,q\}$. 

For simplicity, if $\omega^n=(f_0,\cdots,f_{n-1})$, we let $\omega^d:=(f_0,\dots,f_{d-1})$ and 
$$\omega^N_k:=(f_{kN+d},\dots,f_{(k+1)N+d-1})\text{ for any }k\in\{0,\dots, q-1\}.$$
In particular, $x_{k,\omega^n}$ and $v_{k,\omega^n}$ only depend on $(x,v)$, $\omega^d$ and $\omega^N_0,\dots,\omega^N_{k-1}$.

Recall the following elementary fact
\begin{align}\label{eqn:220-1}
e^{-t}\leq 1-t+t^2,~\text{for any }t\in[-1,1].
\end{align}
Then,  by \hyperlink{A1}{(\textbf{A1})}, \eqref{eqn:220-1}, \hyperlink{A2}{(\textbf{A2})}, \hyperlink{A3}{(\textbf{A3})} and the fact that $\mu^n$ is a product measure, we have
\begin{align*}
&\int_{\Sigma_n} \|Df^n_{\omega^n}(x)v\|^{-\delta}d\mu^n(\omega^n)\\
=&\int_{\Sigma_n}\left(\|Df^d_{\omega^d}(x)v\|^{-\delta}\cdot \prod_{k=0}^{q-1} \|Df^N_{\omega^N_k}(x_{k,\omega^n})v_{k,\omega^n}\|^{-\delta}\right)d\mu^n(\omega^n)\\
\leq& \int_{\Sigma_n} e^{\delta d\C{C'}{0}}\cdot \prod_{k=0}^{q-1} \left(1-\delta\ln\|Df^N_{\omega^N_k}(x_{k,\omega^n})v_{k,\omega^n}\|\right.\\
&\left.+\delta^2\left(\ln\|Df^N_{\omega^N_k}(x_{k,\omega^n})v_{k,\omega^n}\|\right)^2\right)d\mu^n(\omega^n)\\
= &e^{\delta d\C{C'}{0}}\cdot  \prod_{k=0}^{q-1} \left(1-\delta\int_{\Sigma_N}\ln\|Df^N_{\omega^N_k}(x_{k,\omega^n})v_{k,\omega^n}\|d\mu^N(\omega^N_k)\right.\\
&\left.+ \delta^2\int_{\Sigma_N}\left(\ln\|Df^N_{\omega^N_k}(x_{k,\omega^n})v_{k,\omega^n}\|\right)^2d\mu^N(\omega^N_k)\right)\\
\leq & e\cdot \prod_{k=0}^{q-1} \left(1-\delta\C{C}{1}+\delta^2(N\C{C'}{0})^2\right)\\
<&e\left(1-\frac{\delta\C{C}{1}}{2}\right)^q<\frac{e}{1-\frac{\delta\C{C}{1}}{2}}\cdot\left(1-\frac{\delta\C{C}{1}}{2}\right)^{n/N}<\frac{4e}{3}e^{-n\cdot\frac{\delta\C{C}{1}}{2N}}.
\end{align*}  
The first assertion follows from choosing $\C{C}{2}=\ln(4e/3)$ and $\chi=\frac{\delta\C{C}{1}}{2N}$. The second assertion follows from essentially the same proof. 
\end{proof}

\begin{proof}[Proof of \Cref{cor.mostly expand/contract}]
By Markov's inequalty, we have 
\begin{align*}
\mu^n\left(\{\omega^n:\|Df^n_{\omega^n}(x)v\|^{-\delta}\geq e^{-C}e^{-n\overline{\chi}}\}\right)
\leq &e^{C}e^{n\overline{\chi}}\cdot\int_{\Sigma_n} \|Df^n_{\omega^n}(x)v\|^{-\delta}d\mu^n(\omega^n)  \\
<&e^{\C{C}{2}}e^{C}e^{-n(\chi-\overline{\chi})}.
\end{align*}
This proves the first assertion. The second assertion follows in a similar way.
\end{proof}
\begin{proof}[Proof of \Cref{lem:VAP}]
Let $\omega^n=(f_0,\dots,f_{n-1})$. If $k=\#\{j\in\{0,\dots,n-1\}:f_j\in\cU\setminus\cU'\}$, then by \eqref{eqn:partial.VAP} and \hyperlink{A1}{(\textbf{A1})}, we have 
$$\Jac f^n_{\omega^n}(x)\in\left(e^{-\C{C}{0}-n\C{\epsilon}{0}-k(\C{C}{0}+2\C{C'}{0}-\C{\epsilon}{0})}, e^{\C{C}{0}+n\C{\epsilon}{0}+k(\C{C}{0}+2\C{C'}{0}-\C{\epsilon}{0})}\right)\text{ for any }x\in\TT^2.$$
If in addition that $k\leq n\C{\epsilon}{2}$, by the assumption on $\C{\epsilon}{2}$, we have 
$$\Jac f^n_{\omega^n}(x)\in (e^{-\C{C}{0}-2n\C{\epsilon}{0}}, e^{\C{C}{0}+2n\C{\epsilon}{0}})\text{ for any }x\in\TT^2.$$
Hence, since $\mu(\mathcal{U}') >1-\C{\epsilon}{2}$ and by Lemma \ref{lem:stirling-2}, we have 
\begin{align*}
&\mu^n\left(\left\{\omega^n:\Jac f^n_{\omega^n}(x)\in(e^{-\C{C}{0}-2n\C{\epsilon}{0}}, e^{\C{C}{0}+2n\C{\epsilon}{0}})\text{ for any } x\in\TT^2\right\}\right)\\
\geq &\mu^n\left(\left\{\omega^n=(f_0,\dots, f_{n-1}):\#\{j\in\{0,\dots,n-1\}:f_j\in\cU\setminus\cU'\}\leq n\C{\epsilon}{2} \right\}\right)\\
=&\sum_{j=0}^{[n\C{\epsilon}{2}]}{n\choose j}(\mu(\cU'))^{n-j}(1-\mu(\cU'))^{j}\\
=&1-\left(\sum_{j=[n\C{\epsilon}{2}]+1}^{n}{n\choose j}(\mu(\cU'))^{n-j}(1-\mu(\cU'))^{j}\right)\\
>& 1-\left(\sum_{j=[n\C{\epsilon}{2}]+1}^{n}{n\choose j}(1-\mu(\cU'))^j\right)\\
\geq &1-\left(\sum_{j=[n\C{\epsilon}{2}]+1}^{n}{n\choose j}(\C{\epsilon}{2})^j\right)\geq 1-\frac{\C{C''}{4}e^{-n}}{1-e^{-1/\C{\epsilon}{2}}}\geq 1-2\C{C''}{4}e^{-n},
\end{align*}
where $\C{C''}{4}$ is the constant given by \Cref{lem:stirling-2}. Choose $\C{C'}{1}=2\C{C''}{4}$ and the lemma follows.\qedhere
\end{proof}

\section{Some consequences of UEF and UEP}\label{Sec:someestimates}

It will be crucial for our proof to obtain quantitative information about how the finite time mostly contracting and mostly expanding directions oscillate with the choices of words. The oscillation of the mostly contracting direction comes from UEF, and the oscillation of the mostly expanding directions comes from UEP.  This is obtained in \Cref{lem.cone.trans} below. 

For any finite word $\omega^n\in\Sigma_n$, we let
$$\lambda^u_{\omega^{n}}(x):=\|Df^n_{\omega^n}(x)\|\text{ and }\lambda^s_{\omega^{n}}(x):=\frac{\Jac f^n_{\omega^n}(x)}{\|Df^n_{\omega^n}(x)\|}.$$
One can easily check that
\begin{align}\label{eqn:polar.decomp.inverse}
\lambda^s_{\omega^n}(x)=\left\|\left(Df^n_{\omega^n}(x)\right)^{-1}\right\|^{-1}.
\end{align}
If $\lambda^u_{\omega^{n}}(x)\neq \lambda^s_{\omega^{n}}(x)$,  for simplicity, we write $E^*_{\omega^n}(x):=E^*_{f^n_{\omega^n}}(x)$, $*=s,u$. In particular, we have 
$$\left\|\left.Df^n_{\omega^n}(x)\right|_{E^*_{\omega^n}}\right\|=\lambda^*_{\omega^{n}}(x),~*=s,u.$$

\begin{remark}
We emphasize that  the directions $E^*_{\omega^n}(x)$ are not the stable and unstable directions of $x$, but they are the mostly expanding and mostly contracting directions for a finite word $\omega^n$. 
\end{remark}

At many times, we will need to do computations with 
\begin{align*}
Df^n_{\omega^n}((f^n_{\omega^n})^{-1}(x))E^u_{\omega^n}((f^n_{\omega^n})^{-1}(x))\in \PP T_x\TT^2
\end{align*}
and
\begin{align*}
Df^n_{\omega^n}((f^n_{\omega^n})^{-1}(x))E^s_{\omega^n}((f^n_{\omega^n})^{-1}(x))\in \PP T_x\TT^2,
\end{align*}
when the direction  $E^*_{\omega^n}((f^n_{\omega})^{-1}(x))$ is well-defined for  $*=s,u$.  To simplify the  notation later, we let
\begin{align}\label{eqn:Es.past}
V^u_{-\omega^n}(x):= Df^n_{\omega^n}((f^n_{\omega^n})^{-1}(x))E^u_{\omega^n}((f^n_{\omega^n})^{-1}(x))\in \PP T_x\TT^2
\end{align}
and
\begin{align}\label{eqn:Eu.past}
V^s_{-\omega^n}(x):= Df^n_{\omega^n}((f^n_{\omega^n})^{-1}(x))E^s_{\omega^n}((f^n_{\omega^n})^{-1}(x))\in \PP T_x\TT^2,
\end{align}

In particular, $V^{u}_{-\omega^n}(x)$ is the most contracting direction of $(Df^n_{\omega^n}((f^n_{\omega^n})^{-1}(x)))^{-1}$; and $V^{s}_{-\omega^n}(x)$ is the most expanding direction of $(Df^n_{\omega^n}((f^n_{\omega^n})^{-1}(x)))^{-1}$.

{ \begin{lem}\label{lem.cone.trans}
Suppose that $\mathcal{U}' \subset \mathcal{U}$ are $C^2$-open sets in $\mathrm{Diff}^2(\mathbb{T}^2)$ and $\mu$ a probability measure verifying assumptions \hyperlink{A1}{(\textbf{A1})} - \hyperlink{A6}{(\textbf{A6})}.  Then, there exist positive constants $\CS{\beta}{1}\in(0,1)$ and $\CS{C}{3}>0$ such that for any $\eta\geq e^{-n}$, the following holds:
\begin{enumerate}
\item For any $x\in \TT^2$ and for any $v\in \cC^1_x$, we have
$$\mu^{n}\left(\left\{\omega^n\in\Sigma_n: E^s_{\omega^n}(x)\text{ is not well-defined, or }\sphericalangle(E^s_{\omega^n}(x),v)<\eta\right\}\right)\leq \C{C}{3} \eta^{\C{\beta}{1}}$$
\item For any $(x,v)\in T^1\TT^2$, we have
$$\mu^{n}\left(\left\{\omega^n\in\Sigma_n: V^u_{-\omega^n}(x)\text{ is not well-defined, or }\sphericalangle(V^u_{-\omega^n}(x),v)<\eta\right\}\right)\leq \C{C}{3} \eta^{\C{\beta}{1}}$$
\end{enumerate}
\end{lem}
\begin{proof}
\begin{enumerate}
\item To simplify notations, for any $\omega^n=(f_0,\dots,f_{n-1})\in\Sigma_n$ and for any $k\in\{1,\dots,n\}$, we write $A_{k,\omega^n}:=f_{k-1}\circ\dots\circ f_0$. For any $v\in T^1_x\TT^2$, we let  
\begin{align}\label{eqn:cone.trans.notation}
x_{k,\omega^n}=A_{k,\omega^n}(x)\text{ and }v_{k,\omega^n}=\frac{DA_{k,\omega^n}(x)v}{\|DA_{k,\omega^n}(x)v\|}.
\end{align}
For simplicity, we let $x_{0,\omega^n}=x$ and $v_{0,\omega^n}=v$. By \hyperlink{A2}{(\textbf{A2})}, we have 
\begin{align}\label{eqn:remain.in.cone}
v_{k,\omega^n}\in \cC^1_{x_{k,\omega^n}},~\text{for any }k\in\{0,\dots,n\}\text{ and for }\mu^{n}\text{-a.e. }\omega^n\in\Sigma_n.
\end{align}
Choose a unit vector $v^\perp_{0,\omega^n}\in T_x\TT^2$ perpendicular to $v$. For any $k\in\{1,\dots,n\}$, we choose unit vectors $v^\perp_{k,\omega^n}\in T_{x_{k,\omega^n}}\TT^2$ orthogonal to $v_{k,\omega^n}$ such that under the bases $\{v_{k,\omega^n},v^\perp_{k,\omega^n}\}$ in $T_{x_{k,\omega^n}}\TT^2$, we have
$$Df_{k-1}(x_{k-1,\omega^n})=\matii{\lambda_{k}(\omega^n)&c_{k}(\omega^n)}{0&d_{k}(\omega^n)/\lambda_{k}(\omega^n)},$$
where $\lambda_k(\omega^n):=\|Df_{k-1}(x_{k-1,\omega^n})v_{k-1,\omega^n}\|$, $d_{k}(\omega^n):=\Jac f_{k-1}(x_{k-1,\omega^n}))$ and $c_k(\omega^n)=\left\langle Df_{k-1}(x_{k-1,\omega^n})v^\perp_{k-1,\omega^n}, v_{k,\omega^n}\right\rangle$. Write
$$DA_{k,\omega^n}(x)=\matii{\lambda^k_{0}(\omega^n)&c^k_{0}(\omega^n)}{0&d^k_{0}(\omega^n)/\lambda^k_{0}(\omega^n)}.$$ 
For simplicity, we let $\lambda_0(\omega^n)=\lambda_0^0(\omega^n)=1=d_0^0(\omega^n)=d_0(\omega^n)$ and $c_0^0=c_0=0$. It is straightforward to see that
$$\lambda^k_{0}(\omega^n)=\prod_{l=1}^k\lambda_{l}(\omega^n),~d^k_{0}(\omega^n)=\prod_{l=1}^kd_{l}(\omega^n)$$
and
\begin{align}\label{eqn:off.diag.entries}
c^k_{0}(\omega^n)=\sum_{l=1}^k\frac{\lambda^k_{0}(\omega^n)c_{l}(\omega^n)d^{l-1}_{0}(\omega^n)}{\lambda^l_{0}(\omega^n)\lambda^{l-1}_{0}(\omega^n)}.
\end{align}

\hyperlink{A1}{(\textbf{A1})} and \eqref{eqn:off.diag.entries} imply that 
\begin{align}\label{eqn:off.diag.est.2}
|c^n_0(\omega^n)|\leq \frac{\lambda_0^n(\omega^n)}{2\eta(\omega^n)},
\end{align}
where
\begin{align}\label{eqn:eta.nbhd}
\eta(\omega^n):=\left(2e^{2\C{C'}{0}}\left(1+\sum_{l=1}^n\frac{d_0^l(\omega^n)}{\left(\lambda_0^{l}(\omega^n)\right)^2}\right)\right)^{-1}<\frac{1}{2}.
\end{align}
In particular, for any $\theta\in(-\eta(\omega^n),\eta(\omega^n))$ and $v_\theta=\cos(\theta)v_{0,\omega^n}+\sin(\theta)v^\perp_{0,\omega^n}$, by \eqref{eqn:off.diag.est.2} and \eqref{eqn:eta.nbhd}, we have 
\begin{align*}
\|A_{n,\omega^n}v_\theta\|\geq &\cos(\theta)\lambda_0^n(\omega^n)-|\sin(\theta)c_0^n(\omega^n)|\\
\geq &\frac{3}{4}\lambda_0^n(\omega^n)-\frac{|\theta|}{2\eta(\omega^n)}\lambda_0^n(\omega^n)\geq \frac{1}{4}\lambda_0^n(\omega^n).
\end{align*}
Since $v=v_{0,\omega^n}$, the above implies the following claim:
\begin{claim}\label{claim:split.cases.for.bad.stable}
If $\lambda_0^n(\omega^n)\geq 5e^{(\C{C}{0}+2n\C{\epsilon}{0})/2} $ and $|d_0^n(\omega^n)|\leq e^{\C{C}{0}+2n\C{\epsilon}{0}}$, then $E^s_{\omega^n}(x)$ is well-defined and $\sphericalangle(E^s_{\omega^n}(x),v)>\eta(\omega^n).$
\end{claim}
Back to the proof of \Cref{lem.cone.trans}. Recall that by Lemma \ref{lem:VAP}, for any $k\in\{0,\dots,n\}$, we have
\begin{align}\label{eqn:off.diag.est.1}
\mu^n\left(\left\{\omega^n\in\Sigma_n:|d_0^k(\omega^n)|\in(e^{-\C{C}{0}-2k\C{\epsilon}{0}}, e^{\C{C}{0}+2k\C{\epsilon}{0}})\right\}\right)\geq 1-\C{C'}{1}e^{-k}.
\end{align}
Hence by \eqref{eqn:epsilon.cond.1}, the notations near \eqref{eqn:cone.trans.notation}, the first assertion in \Cref{cor.mostly expand/contract}, \Cref{claim:split.cases.for.bad.stable} and \eqref{eqn:off.diag.est.1}, we have
\begin{align}\label{eqn:cone.trans-1}
&\mu^{n}\left(\left\{\omega^n\in\Sigma_n: E^s_{\omega^n}(x)\text{ is not well-defined, or }\sphericalangle(E^s_{\omega^n}(x),v)<\eta\right\}\right)\nonumber\\
\leq&\mu^n(\{\omega^n\in\Sigma_n:\lambda_0^n(\omega^n)<5e^{(\C{C}{0}+2n{\C{\epsilon}{0}})/2}\})\nonumber\\
&+\mu^n(\{\omega^n\in\Sigma_n:d_0^n(\omega^n)>e^{\C{C}{0}+2n\C{\epsilon}{0}}\})\nonumber\\
&+\mu^n(\{\omega^n\in\Sigma_n:\eta(\omega^n)\leq \eta\})\nonumber\\
\leq&\mu^n(\{\omega^n\in\Sigma_n:\lambda_0^n(\omega^n)<5e^{(\C{C}{0}+2n{\C{\epsilon}{0}})/2}\})\nonumber\\
&+\mu^n(\{\omega^n\in\Sigma_n:|d_0^n(\omega^n)|\not\in(e^{-\C{C}{0}-2n\C{\epsilon}{0}},e^{\C{C}{0}+2n\C{\epsilon}{0}})\})\nonumber\\
&+\mu^n(\{\omega^n\in\Sigma_n:\eta(\omega^n)\leq \eta\})\nonumber\\
=&  \mu^n\left(\left\{\omega^n\in\Sigma_n:\left\| Df^n_{\omega^n}(x)v\right\|<5e^{(\C{C}{0}+2n{\C{\epsilon}{0}})/2}\right\}\right)\nonumber\\
&+1- \mu^n\left(\left\{\omega^n\in\Sigma_n:|d_0^n(\omega^n)|\in(e^{-\C{C}{0}-2n\C{\epsilon}{0}}, e^{\C{C}{0}+2n\C{\epsilon}{0}})\right\}\right)\nonumber\\
&+\mu^n(\{\omega^n\in\Sigma_n:\eta(\omega^n)\leq \eta\})\nonumber\\
=& \mu^n\left(\left\{\omega^n\in\Sigma_n:\left\| Df^n_{\omega^n}(x)v\right\|^{-\delta}>5^{-\delta}e^{-\delta(\C{C}{0}+2n{\C{\epsilon}{0}})/2}\right\}\right)\nonumber\\
&+1- \mu^n\left(\left\{\omega^n\in\Sigma_n:|d_0^n(\omega^n)|\in(e^{-\C{C}{0}-2n\C{\epsilon}{0}}, e^{\C{C}{0}+2n\C{\epsilon}{0}})\right\}\right)\nonumber\\
&+\mu^n(\{\omega^n\in\Sigma_n:\eta(\omega^n)\leq \eta\})\nonumber\\
\leq & 5^{\delta}e^{\C{C}{2}}e^{\delta\C{C}{0}/2}e^{-n(\chi-\delta{\C{\epsilon}{0}})}+\C{C'}{1}e^{-n}+\mu^n(\{\omega^n\in\Sigma_n:\eta(\omega^n)\leq \eta\}).
\end{align}
Let 
$$\cB_n(\omega^n):=\left\{l\in\{1,\dots,n\}: \frac{d_0^l(\omega^n)}{(\lambda^l_0(\omega^n))^2}\geq  e^{\C{C}{0}+2\C{\epsilon}{0}l-(2l\overline{\chi}/\delta)}\right\}$$
and 
$$l_n(\omega^n):=\begin{cases}
\displaystyle \sup\cB_n(\omega^n),~&\text{if }\cB_n(\omega^n)\neq\emptyset,\\
\displaystyle 0,~&\text{otherwise.}\\
\end{cases}$$ 
Recall that $\lambda_0^k(\omega^n)=A_{k,\omega^n}(x)v=Df^k_{\omega^n}(x)v$ for any $k\in\{1,\dots n\}$. By \eqref{eqn:off.diag.est.1}, the first assertion in \Cref{cor.mostly expand/contract} and the definition of $l_n(\omega^n)$, for any $l\in\{0,\dots,n\}$, we have
\begin{align}\label{eqn:u.cone.trans-2}
&\mu^n(\{\omega^n\in\Sigma_n:l_n(\omega^n)>l\})\nonumber\\
=& \sum_{k=l+1}^n\mu^n(\{\omega^n\in\Sigma_n:l_n(\omega^n)=k\})\nonumber\\
\leq& \sum_{k=l+1}^n\mu^n\left(\left\{\omega^n\in\Sigma_n:\frac{d_0^k(\omega^n)}{(\lambda^k_0(\omega^n))^2}\geq  e^{\C{C}{0}+2\C{\epsilon}{0}k-(2k\overline{\chi}/\delta)}\right\}\right)\nonumber\\
\leq& \sum_{k=l+1}^n\mu^n\left(\left\{\omega^n\in\Sigma_n:d_0^k(\omega^n)\geq  e^{\C{C}{0}+2\C{\epsilon}{0}k}\right\}\right)\nonumber\\
&+\sum_{k=l+1}^n\mu^n\left(\left\{\omega^n\in\Sigma_n: \lambda^k_0(\omega^n)\leq e^{k\overline{\chi}/\delta}\right\}\right)\nonumber\\
\leq & \sum_{k=l+1}^n\mu^n\left(\left\{\omega^n\in\Sigma_n:|d_0^k(\omega^n)|\not\in( e^{-\C{C}{0}-2\C{\epsilon}{0}k},  e^{\C{C}{0}+2\C{\epsilon}{0}k})\right\}\right)\nonumber\\
&+\sum_{k=l+1}^n\mu^n\left(\left\{\omega^n\in\Sigma_n: \|Df^k_{\omega^n}(x)v\|^{-\delta}\geq e^{-k\overline{\chi}}\right\}\right)\nonumber\\
= & \sum_{k=l+1}^n\mu^n\left(\left\{\omega^n\in\Sigma_n:|d_0^k(\omega^n)|\not\in( e^{-\C{C}{0}-2\C{\epsilon}{0}k},  e^{\C{C}{0}+2\C{\epsilon}{0}k})\right\}\right)\nonumber\\
&+\sum_{k=l+1}^n\mu^k\left(\left\{\omega^k\in\Sigma_k: \|Df^k_{\omega^k}(x)v\|^{-\delta}\geq e^{-k\overline{\chi}}\right\}\right)\nonumber\\
\leq&\sum_{k=l+1}^n\C{C'}{1}e^{-k}+\sum_{k=l+1}^ne^{\C{C}{2}}e^{-k(\chi-\overline{\chi})}\nonumber\\
<&\frac{\C{C'}{1}}{1-e^{-1}}e^{-l}+\frac{e^{\C{C}{2}}}{1-e^{-(\chi-\overline{\chi})}}e^{-l(\chi-\overline{\chi})}.
\end{align}
Notice that by \hyperlink{A1}{(\textbf{A1})} and the definition of $\cB_n(\omega^n)$, we have 
$$\frac{d_0^l(\omega^n)}{(\lambda_0^l(\omega^n))^2}\leq \begin{cases}
\displaystyle  e^{\C{C}{0}+2\C{\epsilon}{0}l-(2l\overline{\chi}/\delta)},~&\text{if }l\in\{1,\dots,n\}\setminus\cB_n(\omega^n),\\
\displaystyle e^{4l\C{C'}{0}},~&\text{if }l\in\cB_n(\omega^n).
\end{cases}$$
It follows from \eqref{eqn:epsilon.cond.1}, \eqref{eqn:eta.nbhd} and the above that
\begin{align}\label{eqn:u.cone.trans-3}
(\eta(\omega^n))^{-1}
=&2e^{2\C{C'}{0}}\left(1+\sum_{l=1}^n\frac{d_0^l(\omega^n)}{\left(\lambda_0^{l}(\omega^n)\right)^2}\right)\nonumber\\
=&2e^{2\C{C'}{0}}\left(1+\left(\sum_{l\in\cB_n(\omega^n)}\frac{d_0^l(\omega^n)}{\left(\lambda_0^{l}(\omega^n)\right)^2}\right)+\left(\sum_{\substack{1\leq l\leq n\nonumber\\
l\not\in\cB_n(\omega^n)}}\frac{d_0^l(\omega^n)}{\left(\lambda_0^{l}(\omega^n)\right)^2}\right)\right)\nonumber\\
\leq& 2e^{2\C{C'}{0}}\left(1+\left(\sum_{l\in\cB_n(\omega^n)}e^{4l\C{C'}{0}}\right)+\left(\sum_{\substack{1\leq l\leq n\nonumber\\
l\not\in\cB_n(\omega^n)}}e^{\C{C}{0}+2\C{\epsilon}{0}l-(2l\overline{\chi}/\delta)}\right)\right)\nonumber\\
\leq&2e^{2\C{C'}{0}}\left(1+\frac{e^{4\C{C'}{0}(l_n(\omega^n)+1)}-1}{e^{4\C{C'}{0}}-1}+\frac{e^{\C{C}{0}}}{1-e^{2{\C{\epsilon}{0}}-2(\overline{\chi}/\delta)}}\right)\nonumber\\
<&e^{\C{C'}{2}}e^{4\C{C'}{0}l_n(\omega^n)},
\end{align}
where
$$\CS{C'}{2}:=\ln\left(2e^{2\C{C'}{0}}\left(1+\frac{e^{4\C{C'}{0}}}{e^{4\C{C'}{0}}-1}+\frac{e^{\C{C}{0}}}{1-e^{-\overline{\chi}/\delta}}\right)\right).$$
By \eqref{eqn:u.cone.trans-2} and \eqref{eqn:u.cone.trans-3}, we have
\begin{align*}
&\mu^n\left(\left\{\omega^n\in\Sigma_n: \eta(\omega^n)\leq \eta\right\}\right)\\
\leq&\mu^n\left(\left\{\omega^n\in\Sigma_n: l_n(\omega^n)> \frac{-\ln(\eta)-\C{C'}{2}}{4\C{C'}{0}}\right\}\right)\\
\leq &\frac{\C{C'}{1}}{1-e^{-1}}e^{-\left[\frac{-\ln(\eta)-\C{C'}{2}}{4\C{C'}{0}}\right]}+\frac{e^{\C{C}{2}}}{1-e^{-(\chi-\overline{\chi})}}e^{-(\chi-\overline{\chi})\cdot \left[\frac{-\ln(\eta)-\C{C'}{2}}{4\C{C'}{0}}\right]}.
\end{align*}
Since $\mu^n$ is a probability measure, the above implies that
\begin{align}\label{eqn:u.cone.trans-4}
\mu^n\left(\left\{\omega^n\in\Sigma_n: \eta(\omega^n)\leq \eta\right\}\right)\leq \C{C'}{3}\eta^{\C{\beta'}{1}},
\end{align}
where
$$\C{C'}{3}=\frac{\C{C'}{1}}{1-e^{-1}}e^{\frac{1\C{C'}{2}}{4\C{C'}{0}}}+\frac{e^{\C{C}{2}}}{1-e^{-(\chi-\overline{\chi})}}e^{\frac{\C{C'}{2}(\chi-\overline{\chi})}{4\C{C'}{0}}}>1$$
and
$$\C{\beta'}{1}=\min\left\{ \frac{1}{4\C{C'}{0}},\frac{\chi-\overline{\chi}}{4\C{C'}{0}}\right\}.$$
Since we assumed that $\eta\geq e^{-n}$, apply \eqref{eqn:epsilon.cond.1} and \eqref{eqn:u.cone.trans-4} to \eqref{eqn:cone.trans-1}, we obtain
\begin{align*}
&\mu^{n}\left(\left\{\omega^n\in\Sigma_n: E^u_{\omega^n,n}(x)\text{ is not well-defined, or }\sphericalangle(E^u_{\omega^n,n}(x),v)<\eta\right\}\right)\\
\leq & 5^{\delta}e^{\C{C}{2}}e^{\delta\C{C}{0}/2}e^{-n(\chi-\delta{\C{\epsilon}{0}})}+\C{C'}{1}e^{-n}+\C{C'}{3}\eta^{\C{\beta'}{1}}\\
\leq &  5^{\delta}e^{\C{C}{2}}e^{\delta\C{C}{0}/2}e^{-n(\chi-\overline{\chi})}+\C{C'}{1}e^{-n}+\C{C'}{3}\eta^{\C{\beta'}{1}}. 
\end{align*}
Again, since $\mu^n$ is a probability measure, the above implies that
$$\mu^{n}\left(\left\{\omega^n\in\Sigma_n: E^u_{\omega^n,n}(x)\text{ is not well-defined, or }\sphericalangle(E^u_{\omega^n,n}(x),v)<\eta\right\}\right) \leq \C{C}{3}\eta^{\C{\beta}{1}}$$
where
$$\C{\beta}{1}:=\min\left\{\C{\beta'}{1},1,\chi-\overline{\chi}\right\}\text{ and }\C{C}{3}:=5^{\delta}e^{\C{C}{2}}e^{\delta\C{C}{0}/2}+\C{C'}{1}+\C{C'}{3}>1.$$
This finishes the proof.
\item The proof of this assertion is essentially the same as the proof of the previous assertion. The only difference is that we use the second assertion in Corollary \ref{cor.mostly expand/contract} instead of the first assertion in Corollary \ref{cor.mostly expand/contract}.\qedhere
\end{enumerate}

\end{proof}

}

\section{Curvature and density estimates}\label{Sec.curvaturedensity}

Since  \Cref{thm:maintechnicaltheorem} is a result about the SRB measures, we will need to control pushforward of measures supported on curves, which are absolutely continuous with respect to the arclength measure. In order to do that, we will need to obtain estimates on how the curvature of a curve, and densities of absolutely continuous measures, change for different iterates. 
 
\subsection{Notations} \label{notations.curvature.measure}
\begin{enumerate}
\item Throughout the paper, for any curve $\gamma$, we always use $m_\gamma$ to denote the arclength measure along $\gamma$.
\item For any curve $\gamma:[a,b]\to \TT^2$ or $\RR^2$, and for any $t\in[a,b]$, we denote by $\curv(t;\gamma)$ the curvature of $\gamma$ at $\gamma(t)$. Namely,
$$\curv(t;\gamma)=\frac{\det(\dot\gamma(t),\ddot\gamma(t))}{\|\dot\gamma(t)\|^3}$$
\end{enumerate}

\subsection{Curvature estimates}
\begin{lem}\label{lem:basic.curv.est}
Let $F:\TT^2\to\TT^2$ be a $C^2$ map and $\gamma:[0,a]\to\TT^2$ be a $C^2$ curve. Assume that $\|F\|_{C^2}\leq A$ on $\TT^2$. Then,
$$|\curv(t;F\circ\gamma)|\leq A\cdot\frac{\|\dot\gamma(t)\|^2}{\|DF(\gamma(t))\dot\gamma(t)\|^2}+\Jac F(\gamma(t))\cdot|\curv(t;\gamma)|\cdot\frac{\|\dot\gamma(t)\|^3}{\|DF(\gamma(t))\dot\gamma(t)\|^3}.$$
\end{lem}
\begin{proof}
For simplicity, we write $\gamma_F=F\circ \gamma$.
\begin{align*}
&|\curv(t;F\circ\gamma)|\\
=&\left|\frac{\det(\dot\gamma_F(t),\ddot\gamma_F(t))}{\|\dot\gamma_F(t)\|^3}\right|\\
\leq& \frac{|\det(\dot{\gamma}_{F}(t), DF(\gamma(t)) \ddot{\gamma}(t))|}{\|\dot{\gamma}_{F}(t)\|^3}+ \frac{|\det(\dot{\gamma}_{F}(t), D^2F(\gamma(t))(\dot{\gamma}(t), \dot{\gamma}(t)))|}{\|\dot{\gamma}_{F}(t)\|^3}\\
=&\frac{|\det(DF(\gamma(t))\dot\gamma(t), DF(\gamma(t)) \ddot{\gamma}(t))|}{\|DF(\gamma(t))\dot\gamma(t)(t)\|^3}+ \frac{|\det(\dot{\gamma}_{F}(t), D^2F(\gamma(t))(\dot{\gamma}(t), \dot{\gamma}(t)))|}{\|\dot{\gamma}_{F}(t)\|^3}\\
\leq & \Jac F(\gamma(t))\cdot|\curv(t;F\circ\gamma)|\cdot\frac{\|\dot\gamma(t)\|^3}{\|DF(\gamma(t))\dot\gamma(t)\|^3}+A\cdot\frac{\|\dot\gamma(t)\|^2}{\|DF(\gamma(t))\dot\gamma(t)\|^2}.\qedhere
\end{align*}
\end{proof}

For the rest of this section we will suppose that $\mathcal{U}' \subset \mathcal{U}$ are $C^2$-open sets in $\mathrm{Diff}^2(\mathbb{T}^2)$ and $\mu$ a probability measure verifying assumptions \hyperlink{A1}{(\textbf{A1})} - \hyperlink{A6}{(\textbf{A6})}.

\begin{dfn}\label{dfn:tempered}
Let $\gamma:[0,a]\to \TT^2$ be a $C^2$ curve. Let $p_0,m\in\ZZ^+$ and $n=mp_0$. Fix an arbitrary $\omega^n=(f_0,\dots,f_{n-1})\in\Sigma_n$ and write $\gamma_k:=f^k_{\omega^n}\circ\gamma=f_{k-1}\circ\dots\circ f_0\circ \omega$. (For simplicity, we write $\gamma_0=\gamma$.) 
\begin{enumerate}
\item  (\textbf{$(p_0;\eta)$-NCT (Nearly Conservative Tails)}) We say that $\omega^n$ has \emph{$(p_0;\eta)$-NCT along $\gamma$} if for any $k\in\{1,\dots, m\}$, we have 
$$\#\left\{j\in\{k+1,\dots,m\}\left|\exists t\in[0,a]\text{ s.t. } \left|\ln\frac{\Jac f^{jp_0}_{\omega^n}(\gamma(t))}{\Jac f^{(j-1)p_0}_{\omega^n}(\gamma(t))}\right|>{\C{C}{0}+2\C{\epsilon}{0}p_0}\right.\right\}<(m-k)\eta.$$
\item (\textbf{$(p_0; c,\eta)$-ET (Expanding Tails)}) For any constants $c\in(0,\C{C'}{0})$ and $\eta>0$, we say that $\omega^n$ has \emph{$(p_0; c,\eta)$-ET along $\gamma$} if for any $k\in\{1,\dots, m\}$, we have 
$$\#\left\{j\in\{k+1,\dots,m\}\left|\exists t\in[0,a]\text{ s.t. } \frac{\|\dot\gamma_{jp_0}(t)\|}{\|\dot\gamma_{(j-1)p_0}(t)\|}<e^{cp_0}\right.\right\}<(m-k)\eta.$$
\end{enumerate}
\end{dfn}
\begin{rmk}
Let $\omega^n=(f_1,\dots,f_n)$. The above notion of NCT and ET are all properties regarding $(f_{n-kp_0+1},\dots, f_n)$ for any $k\in\{1,\dots,m\}$. For simplicity, we call $(f_{n-kp_0+1},\dots, f_n)$ the ``backward subword'' of $\omega^n$ with length $kp_0$.  A direct consequence of NCT is that for any $k\in\{1,\dots,m\}$, the Jacobian of the ``backward subword'' of $\omega^n$ with length $kp_0$ is almost subexponential in $kp_0$.  Similarly, ET implies that for any $k\in\{1,\dots,m\}$, the ``backward subword'' of $\omega^n$ with length $kp_0$ is expanding the curve $\gamma_{(m-k)p_0}$ faster than a uniform exponential rate.
\end{rmk}

Now, let us assume 
\begin{lem}\label{lem:curv.est}
Let $\gamma:[0,a]\to \TT^2$ be a $C^2$ curve such that $|\curv(t;\gamma)|\leq K$ for any $t\in[0,a]$. Let $p_0,m\in\ZZ^+$ and $n=mp_0$. Fix some $\omega^n=(f_0,\dots,f_{n-1})\in\Sigma_n$ and write $\gamma_k:=f^k_{\omega^n}\circ\gamma=f_{k-1}\circ\dots\circ f_0\circ \omega$. (For simplicity, we write $\gamma_0=\gamma$.) 

If $\omega^n$ has $(p_0;\eta)$-NCT and $(p_0; c,\eta)$-ET along $\gamma$ (in the sense of Definition \ref{dfn:tempered}) with $p_0,c,\eta$ satisfying 
$$\C{\epsilon}{0}<\frac{c}{2}<\frac{\C{C'}{0}}{2}\text{ and }0<\frac{\C{C}{0}}{\C{C'}{0}p_0}<\eta<\frac{c}{2c+6\C{C'}{0}}<\frac{1}{8},$$
then, there exists a constant $\C{K}{1}(p_0, c)$, such that for any $j\in\{1,\dots,n\}$,  we have 
$$|\curv(t;\gamma_j)|\leq \C{K}{1}(p_0;c)(K+1)e^{8(n-j)\eta\C{C'}{0}}.$$
\end{lem}
\begin{proof}
Recall that for any $F\in \mathcal{U}$ we have $\|F\|_{C^2} \leq e^{\C{C'}{0}}$. Notice that for any $j\in\{1,\dots,n\}$, Lemma \ref{lem:basic.curv.est}  and the estimate on the $C^2$-norm mentioned above, implies that

\begin{align}\label{eqn:stupid.curv.est}
|\curv(t;\gamma_j)|\leq e^{3\C{C'}{0}}+e^{5\C{C'}{0}}|\curv(t;\gamma_{j-1})|.
\end{align}

By applying \Cref{lem:basic.curv.est} multiple times, we obtain 

\begin{align}
&|\curv(t;\gamma_j)|\nonumber\\
\leq &e^{\C{C'}{0}}\frac{\|\dot\gamma_{j-1}(t)\|^2}{\|\dot\gamma_{j}(t)\|^2}+\Jac f_{j-1}(\gamma_{j-1}(t))\cdot\frac{\|\dot\gamma_{j-1}(t)\|^3}{\|\dot\gamma_{j}(t)\|^3}\cdot|\curv(t;\gamma_{j-1})|\label{eqn:1.step.curv.est}\\
\leq &e^{\C{C'}{0}}\left(\frac{\|\dot\gamma_{j-1}(t)\|^2}{\|\dot\gamma_{j}(t)\|^2}+\Jac f_{j-1}(\gamma_{j-1}(t))\cdot\frac{\|\dot\gamma_{j-1}(t)\|\cdot\|\dot\gamma_{j-2}(t)\|^2}{\|\dot\gamma_{j}(t)\|^3}\right)\nonumber\\
&+|\curv(t;\gamma_{j-2})|\cdot\Jac (f_{j-1}\circ f_{j-2})(\gamma_{j-2}(t))\cdot\frac{\|\dot\gamma_{j-2}(t)\|^3}{\|\dot\gamma_{j}(t)\|^3}\leq\dots\nonumber\\
\leq&e^{\C{C'}{0}}\sum_{l=1}^{j-1}\Jac(f_{j-1}\circ\dots\circ f_{j-l})(\gamma_{j-l}(t))\cdot\frac{\|\dot\gamma_{j-l}(t)\|\cdot\|\dot\gamma_{j-l-1}(t)\|^2}{\|\dot\gamma_{j}(t)\|^3}\nonumber\\
&+e^{\C{C'}{0}}\frac{\|\dot\gamma_{j-1}(t)\|^2}{\|\dot\gamma_{j}(t)\|^2}+K\cdot\Jac (f_{j-1}\circ\dots\circ f_{0})(\gamma_0(t))\frac{\|\dot\gamma_{0}(t)\|^3}{\|\dot\gamma_{j}(t)\|^3}\nonumber\\
=&e^{\C{C'}{0}}\sum_{l=1}^{j-1}\frac{\Jac f^j_{\omega^n}(\gamma(t))}{\Jac f^{j-l}_{\omega^n}(\gamma(t))}\cdot\frac{\|\dot\gamma_{j-l}(t)\|\cdot\|\dot\gamma_{j-l-1}(t)\|^2}{\|\dot\gamma_{j}(t)\|^3}\nonumber\\
&+e^{\C{C'}{0}}\frac{\|\dot\gamma_{j-1}(t)\|^2}{\|\dot\gamma_{j}(t)\|^2}+K\cdot\frac{\Jac f^j_{\omega^n}(\gamma(t))}{\Jac f^0_{\omega^n}(\gamma(t))}\frac{\|\dot\gamma_{0}(t)\|^3}{\|\dot\gamma_{j}(t)\|^3}\nonumber\\
\leq &e^{3\C{C'}{0}}\left(\sum_{l=0}^{j-1}d^j_{j-l}(t)\right)+Kd^j_{0}(t),\label{eqn:gen.curv.est}
\end{align}
where $d^i_i(t)=1$ for any $0\leq i\leq n$ and  
$$d^j_{i}(t):=\frac{\Jac f^j_{\omega^n}(\gamma(t))}{\Jac f^i_{\omega^n}(\gamma(t))}\cdot \frac{\|\dot\gamma_i(t)\|^3}{\|\dot\gamma_j(t)\|^3}, ~\forall 0\leq i<j\leq n.$$

By applying \eqref{eqn:stupid.curv.est}, we know that for any $j=kp_0+d$ with $d\in\{0,\dots,p_0-1\}$, we have
$$|\curv(t;\gamma_{kp_0+d})|\leq e^{5\C{C'}{0}p_0}|\curv(t;\gamma_{kp_0})|+\frac{e^{(5p_0+3)\C{C'}{0}}-e^{3\C{C'}{0}}}{e^{5\C{C'}{0}}-1}.$$
If one can prove this lemma for $j=kp_0$ with $\C{K}{1}(p_0;c)$ replaced by some $\C{K'}{1}(p_0;c)>0$, then by the assumption that $\eta<1/8$, the lemma simply follows from choosing $\C{K}{1}(p_0;c)$ such that 
$$\CS{K}{1}(p_0;c)= e^{6p_0\C{C'}{0}}\C{K'}{1}(p_0;c)+\frac{e^{(5p_0+3)\C{C'}{0}}-e^{3\C{C'}{0}}}{e^{5\C{C'}{0}}-1}.$$

We now prove the lemma for $j=kp_0$ with $\C{K}{1}(p_0;c)$ replaced by a suitable $\C{K'}{1}(p_0;c)>0$. Recall that $\omega^n$ has $(p_0,\eta)$-NCT and $(p_0; c,\eta)$-ET along $\gamma$ (in the sense of Definition \ref{dfn:tempered}). For any $k\in\{1,\dots,m\}$, for any $l\in\{0,\dots,k\}$ and  $d\in\{0,\dots,p_0\}$ such that $lp_0+d\leq k$, by \eqref{eqn:unif.C2.norm} in \hyperlink{A1}{(\textbf{A1})}, Definition \ref{dfn:tempered} and the assumptions on $\eta,p_0,c$, we have
\begin{align*}
&d^{kp_0}_{kp_0-lp_0-d}(t)\\
\leq&e^{5d\C{C'}{0}}\cdot\prod_{s=k-l}^{k-1}d^{(s+1)p_0}_{sp_0}(t)\\
\leq &\begin{cases}
\displaystyle e^{5d\C{C'}{0}}\cdot e^{5lp_0\C{C'}{0}},~&\text{if }l<\frac{(m-k)\eta}{1-\eta},\\
\displaystyle e^{5d\C{C'}{0}}\cdot e^{5(m-k+l)\eta p_0\C{C'}{0}}e^{\left(l-\eta (m-k+l)\right)(\C{C}{0}+(2\C{\epsilon}{0}-3c)p_0)},~&\text{if }l\geq \frac{(m-k)\eta}{1-\eta},
\end{cases}\\
=&\begin{cases}
\displaystyle e^{5d\C{C'}{0}}\cdot e^{5lp_0\C{C'}{0}},~&\text{if }l<\frac{(m-k)\eta}{1-\eta},\\
\displaystyle e^{5d\C{C'}{0}}\cdot e^{l(5\eta\C{C'}{0}p_0+(1-\eta)(\C{C}{0}+(2\C{\epsilon}{0}-3c)p_0))}e^{(m-k)\eta (-\C{C}{0}+(5\C{C'}{0}-2\C{\epsilon}{0}+3c)p_0)},~&\text{if }l\geq \frac{(m-k)\eta}{1-\eta},
\end{cases}\\
\leq&\begin{cases}
\displaystyle e^{5d\C{C'}{0}}\cdot e^{5lp_0\C{C'}{0}},~&\text{if }l<\frac{(m-k)\eta}{1-\eta},\\
\displaystyle e^{5d\C{C'}{0}}\cdot e^{l(5\eta\C{C'}{0}p_0+\C{C}{0}+(1-\eta)(2\C{\epsilon}{0}-3c)p_0)}e^{(m-k)\eta (5\C{C'}{0}+3c)p_0},~&\text{if }l\geq \frac{(m-k)\eta}{1-\eta},
\end{cases}\\
\leq &\begin{cases}
\displaystyle e^{5d\C{C'}{0}}\cdot e^{5lp_0\C{C'}{0}},~&\text{if }l<\frac{(m-k)\eta}{1-\eta},\\
\displaystyle e^{5d\C{C'}{0}}\cdot e^{l(6\eta\C{C'}{0}-2(1-\eta)c)p_0}e^{8\C{C'}{0}(m-k)\eta p_0},~&\text{if }l\geq \frac{(m-k)\eta}{1-\eta},
\end{cases}\\
\leq &\begin{cases}
\displaystyle e^{5d\C{C'}{0}}\cdot e^{5lp_0\C{C'}{0}},~&\text{if }l<\frac{(m-k)\eta}{1-\eta},\\
\displaystyle e^{5d\C{C'}{0}}\cdot e^{-clp_0}e^{8\C{C'}{0}(m-k)\eta p_0},~&\text{if }l\geq \frac{(m-k)\eta}{1-\eta},
\end{cases}\\
\leq& e^{5d\C{C'}{0}}e^{8\C{C'}{0}(m-k)\eta p_0}.
\end{align*}
By  the assumptions on $\eta$ and the above to \eqref{eqn:gen.curv.est} when $j=kp_0$, we have 
\begin{align*}
|\curv(t;\gamma_{kp_0})|
\leq &e^{3\C{C'}{0}}\left(\sum_{l=0}^{k-1}\sum_{d=0}^{p_0-1}d^{kp_0}_{kp_0-lp_0-d}(t)\right)
+Kd^{kp_0}_{0}(t)\\
= &e^{3\C{C'}{0}}\left(\sum_{l=0}^{\left[\frac{(m-k)\eta}{1-\eta}\right]}\sum_{d=0}^{p_0-1}d^{kp_0}_{kp_0-lp_0-d}(t)\right)+Kd^{kp_0}_{0}(t)\\
&+e^{3\C{C'}{0}}\left(\sum_{l=\left[\frac{(m-k)\eta}{1-\eta}\right]+1}^{k-1}\sum_{d=0}^{p_0-1}d^{kp_0}_{kp_0-lp_0-d}(t)\right)\\
\leq &e^{3\C{C'}{0}}\left(\sum_{l=0}^{\left[\frac{(m-k)\eta}{1-\eta}\right]}\sum_{d=0}^{p_0-1}e^{5d\C{C'}{0}}\cdot e^{5lp_0\C{C'}{0}}\right)+Ke^{8\C{C'}{0}(m-k)\eta p_0}\\
&+e^{3\C{C'}{0}}\left(\sum_{l=\left[\frac{(m-k)\eta}{1-\eta}\right]+1}^{k-1}\sum_{d=0}^{p_0-1}e^{5d\C{C'}{0}}\cdot e^{-clp_0}e^{8\C{C'}{0}(m-k)\eta p_0}\right)\\
\leq& \frac{e^{3\C{C'}{0}+5p_0\C{C'}{0}}}{e^{5\C{C'}{0}}-1}\left(\sum_{l=0}^{\left[\frac{(m-k)\eta}{1-\eta}\right]} e^{5lp_0\C{C'}{0}}\right)+Ke^{8\C{C'}{0}(m-k)\eta p_0}\\
&+\frac{e^{3\C{C'}{0}+5p_0\C{C'}{0}}e^{8\C{C'}{0}(m-k)\eta p_0}}{e^{5\C{C'}{0}}-1}\sum_{l=\left[\frac{(m-k)\eta}{1-\eta}\right]+1}^{k-1} e^{-clp_0}\\
\leq& \frac{e^{3\C{C'}{0}+10p_0\C{C'}{0}}}{(e^{5p_0\C{C'}{0}}-1)(e^{5\C{C'}{0}}-1)}e^{5\frac{\eta}{1-\eta}(m-k)p_0\C{C'}{0}} +Ke^{8\C{C'}{0}(m-k)\eta p_0}+\\
&+\frac{e^{3\C{C'}{0}+5p_0\C{C'}{0}}e^{8\C{C'}{0}(m-k)\eta p_0}}{(1-e^{-cp_0})(e^{5\C{C'}{0}}-1)}\\
\leq&\C{K'}{1}(p_0;c)(K+1)e^{8(m-k)p_0\eta\C{C'}{0}},
\end{align*}
where
$$\CS{K'}{1}(p_0;c):= \frac{e^{3\C{C'}{0}+10p_0\C{C'}{0}}}{(e^{5p_0\C{C'}{0}}-1)(e^{5\C{C'}{0}}-1)}+1+\frac{e^{3\C{C'}{0}+5p_0\C{C'}{0}}}{(1-e^{-cp_0})(e^{5\C{C'}{0}}-1)}.$$
This proves the lemma when $j=kp_0$ with $\C{K}{1}$ replaced by $\C{K'}{1}$.
\end{proof}
\begin{lem}\label{lem:length.bd.curv}
Let $K>0$ and $\gamma:[0,a]\to\TT^2$ be a $C^2$ curve such that $|\curv(\cdot;\gamma)|\leq K$. Then for any $z\in\TT^2$, for any $\rho\in(0,\min\{1/4K,1/4\})$ and for any $x\in B(z,\rho)$, the following holds:
\begin{enumerate}
\item The length of any connected component of $\gamma\cap B(x,\rho)$ is at most $4\rho$.
\item Let $J$ be any connected component of $\gamma\cap B(z,2\rho)$, then for any $\rho'\in(0,\rho)$ $J\cap B(x,\rho')$ is connected.
\end{enumerate}
\end{lem}
\begin{proof}
Without loss of generality, we assume that $\gamma$ is parametrized by arclength.  Let $J$ be any connected component of $\gamma\cap B(z,2\rho)$. For any $x\in B(z,\rho)$, we consider the function $E_x(t):=d(\gamma(t),x)^2/2$. Then 
$$\frac{d^2}{dt^2}E_x(t)\geq 1-|\curv(t;\gamma)|d(\gamma(t),x)>0\text{ whenever }\gamma(t)\in J.$$
This implies that $J\cap B(x,\rho')$ is connected for any $\rho'\in(0,\rho)$. Moreover, $\frac{d^2}{dt^2}E_x(t)>1/4$ whenever $\gamma(t)\in J\cap B(x,\rho)$. Hence the length of $ J\cap B(x,\rho)$ is at most $2\cdot(2\rho)=4\rho$.
\end{proof}

\subsection{Density distortion estimates}
\begin{lem}\label{lem:density.est}
Let $\gamma:[0,a]\to \TT^2$ be a $C^2$ curve such that $|\curv(t;\gamma)|\leq K$ for any $t\in[0,a]$, where $K\geq 1$. Let $p_0,m\in\ZZ^+$ and $n=mp_0$. Fix some $\omega^n=(f_0,\dots,f_{n-1})\in\Sigma_n$ and write $\gamma_k:=f^{k}_{\omega^n}\circ\gamma=f_{k-1}\circ\dots\circ f_0\circ \gamma$. For simplicity, we write $\gamma_0=\gamma$. 

Let $\nu$ be a finite measure supported on $\gamma$ which  is absolutely continuous with respect to $m_\gamma$. We assume in addition that $d\nu/dm_\gamma$ is $L$-log-Lipschitz, for some constant $L>0$. For simplicity, we let $\nu_j:=(f^j_{\omega^n})_*\nu$ and $\rho_j:=d\nu_j/dm_{\gamma_j}$ for any $j\in\{1,\dots,n)$. We also let $\nu_0=\nu$ and $\rho_0:=d\nu/dm_\gamma$.

If $\omega^n$ has $(p_0,\eta)$-NCT and $(p_0; c,\eta)$-ET along $\gamma$ (in the sense of Definition \ref{dfn:tempered}) with $p_0,c,\eta$ satisfying 
$$\C{\epsilon}{0}<\frac{c}{2}<\frac{\C{C'}{0}}{2}\text{ and }0<\frac{\C{C}{0}}{\C{C'}{0}p_0}<\eta<\frac{c}{2c+16\C{C'}{0}}<\frac{1}{18},$$ 
then there exists some constant $\C{K}{2}=\C{K}{2}(p_0;c)>0$ such that $\rho_j$ is $\displaystyle \C{K}{2}(p_0;c)(L+1)(K+1)e^{11(n-j)\eta\C{C'}{0}}$-log-Lipschitz.
\end{lem}
\begin{proof}
We first notice that for any $j\in\{0,\dots,n-1\}$.
$$\left\|Df_{j}(\gamma_j(t))|_{\dot\gamma_j(t)}\right\|\cdot\rho_{j+1}(\gamma_{j+1}(t))=\rho_j(\gamma_j(t)).$$
Hence for any $t_1,t_2\in [0,a]$ and for any $j\in\{0,\dots,n-1\}$, we have 
\begin{align}\label{eqn:density.ineq}
\begin{split}
&\left|\ln\rho_{j+1}(\gamma_{j+1}(t_1))-\ln\rho_{j+1}(\gamma_{j+1}(t_2))\right|\\
\leq&\left|\ln\left\|Df_{j}(\gamma_j(t_1))|_{\dot\gamma_j(t_1)}\right\|-\ln\left\|Df_{j}(\gamma_j(t_2))|_{\dot\gamma_j(t_2)}\right\|\right|\\
&+ \left|\ln\rho_{j}(\gamma_{j}(t_1))-\ln\rho_{j}(\gamma_{j}(t_2))\right|.
\end{split}
\end{align}
We first prove the lemma when $j\in p_0\ZZ$ with $\C{K}{2}(p_0;c)$ replaced by some other constant $\C{K''}{2}(p_0;c)$. The constant $\C{K''}{2}(p_0;c)$ is determined right after \eqref{eqn:density.p0.main}.

Iterating \eqref{eqn:density.ineq}, we obtain that for any $k\in\{0,\dots,m-1\}$.
\begin{align}\label{eqn:density.ineq.p0}
\begin{split}
&\left|\ln\rho_{(k+1)p_0}(\gamma_{(k+1)p_0}(t_1))-\ln\rho_{(k+1)p_0}(\gamma_{(k+1)p_0}(t_2))\right|\\
\leq& \sum_{d=0}^{p_0-1}\left|\ln\frac{\left\|Df_{kp_0+d}(\gamma_{kp_0+d}(t_1))|_{\dot\gamma_{kp_0+d}(t_1)}\right\|}{\left\|Df_{kp_0+d}(\gamma_{kp_0+d}(t_2))|_{\dot\gamma_{kp_0+d}(t_2)}\right\|}\right|\\
&+\left|\ln\rho_{kp_0}(\gamma_{kp_0}(t_1))-\ln\rho_{kp_0}(\gamma_{kp_0}(t_2))\right|.
\end{split}
\end{align}
To simplify notations, we let $\lambda_k:=\inf_{t\in[0,a]}\frac{\|\dot\gamma_{kp_0}(t)\|}{\|\dot\gamma_{(k-1)p_0}(t)\|}$. In particular, the $(p_0; c,\eta)$-ET assumption implies that for any $k\in\{0,\dots,m\}$, we have
\begin{align}\label{eqn:density.BUEF}
\#\{j\in\{k+1,\dots,m\}|\lambda_j<e^{cp_0}\}<(m-k)\eta.
\end{align}
For any $j\in\{0,\dots,n\}$, let $L_{(j)}>0$ be a constant such that $\ln\rho_j$ is $L_{(j)}$-Lipschitz. We would like to construct $L_{(kp_0)}>0$ inductively on $k\in\{0,\dots,m\}$ such that $\ln\rho_{kp_0}$ is $L_{(kp_0)}$-Lipschitz. Let $L_{(0)}=L$. Assume that $L_{(kp_0)}$ is contructed for some $k\in\{0,\dots,m-1\}$, then we have 
\begin{align}\label{eqn:density.est.1}
\begin{split}
&\left|\ln\rho_{kp_0}(\gamma_{kp_0}(t_1))-\ln\rho_{kp_0}(\gamma_{kp_0}(t_2))\right|\\
\leq& L_{(kp_0)}d_{\gamma_{kp_0}}(\gamma_{kp_0}(t_1),\gamma_{kp_0}(t_2))\\
\leq &\frac{L_{(kp_0)}}{\lambda_{k+1}}d_{\gamma_{(k+1)p_0}}(\gamma_{(k+1)p_0}(t_1),\gamma_{(k+1)kp_0}(t_2)).
\end{split}
\end{align}
Notice that the assumptions on $p_0,c,\eta$ are stronger than the corresponding assumptions in Lemma \ref{lem:curv.est}. Therefore,  by \eqref{eqn:unif.C2.norm} in \hyperlink{A1}{(\textbf{A1})} and Lemma \ref{lem:curv.est}, we have
\begin{align*}
& \sum_{d=0}^{p_0-1}\left|\ln\frac{\left\|Df_{kp_0+d}(\gamma_{kp_0+d}(t_1))|_{\dot\gamma_{kp_0+d}(t_1)}\right\|}{\left\|Df_{kp_0+d}(\gamma_{kp_0+d}(t_2))|_{\dot\gamma_{kp_0+d}(t_2)}\right\|}\right|\nonumber\\
\leq &\sum_{d=0}^{p_0-1}e^{2\C{C'}{0}}d(\dot\gamma_{kp_0+d}(t_1),\dot\gamma_{kp_0+d}(t_2))\nonumber\\
\leq &\sum_{d=0}^{p_0-1}e^{2\C{C'}{0}}\sqrt{1+\left(\C{K}{1}(p_0;c)(K+1)e^{8(n-kp_0-d)\eta\C{C'}{0}}\right)^2}\cdot d_{\gamma_{kp_0+d}}(\gamma_{kp_0+d}(t_1),\gamma_{kp_0+d}(t_2))\nonumber\\
\end{align*}
Observe that $ \C{K}{1}(p_0;c)(K+1)e^{8(n-kp_0-d)\eta\C{C'}{0}}>1$ and recall that if $a>1$ then $\sqrt{1+a^2} < 2a$.  Therefore,
\begin{align}\label{eqn:density.est.2}
&\sum_{d=0}^{p_0-1}e^{2\C{C'}{0}}\sqrt{1+\left(\C{K}{1}(p_0;c)(K+1)e^{8(n-kp_0-d)\eta\C{C'}{0}}\right)^2}\cdot d_{\gamma_{kp_0+d}}(\gamma_{kp_0+d}(t_1),\gamma_{kp_0+d}(t_2))\nonumber\\
\leq &\sum_{d=0}^{p_0-1}2e^{2\C{C'}{0}}\C{K}{1}(p_0;c)(K+1)e^{8(n-kp_0-d)\eta\C{C'}{0}}d_{\gamma_{kp_0+d}}(\gamma_{kp_0+d}(t_1),\gamma_{kp_0+d}(t_2))\nonumber\\
\leq &\sum_{d=0}^{p_0-1}2e^{(2+p_0-d)\C{C'}{0}}\C{K}{1}(p_0;c)(K+1)e^{8(n-kp_0-d)\eta\C{C'}{0}}d_{\gamma_{(k+1)p_0}}(\gamma_{(k+1)p_0}(t_1),\gamma_{(k+1)p_0}(t_2))\nonumber\\
\leq &\sum_{d=0}^{p_0-1}2e^{(2+p_0)\C{C'}{0}}\C{K}{1}(p_0;c)(K+1)e^{8(n-kp_0)\eta\C{C'}{0}}d_{\gamma_{(k+1)p_0}}(\gamma_{(k+1)p_0}(t_1),\gamma_{(k+1)p_0}(t_2))\nonumber\\
= & \C{K'}{2}(p_0;c)(K+1)e^{8(n-kp_0)\eta\C{C'}{0}}d_{\gamma_{(k+1)p_0}}(\gamma_{(k+1)p_0}(t_1),\gamma_{(k+1)p_0}(t_2)),
\end{align}
where $$\CS{K'}{2}(p_0;c):=2p_0\C{K}{1}(p_0;c)e^{(2+p_0)\C{C'}{0}}.$$
By \eqref{eqn:density.ineq.p0}, \eqref{eqn:density.est.1} and \eqref{eqn:density.est.2}, we can choose $L_{((k+1)p_0)}$ as
\begin{align}\label{eqn:log.lip.induct}
L_{((k+1)p_0)}=\frac{L_{(kp_0)}}{\lambda_{k+1}}+\C{K'}{2}(p_0;c)(K+1)e^{8(n-kp_0)\eta\C{C'}{0}}.
\end{align}
Applying \eqref{eqn:log.lip.induct} inductively, we have 
\begin{align}\label{eqn:log.lip.est}
L_{(kp_0)}=\frac{L}{\prod_{s=1}^k\lambda_s}+\C{K'}{2}(p_0;c)(K+1)\sum_{l=1}^k\frac{e^{8(n-(k-l)p_0)\eta\C{C'}{0}}}{\prod_{s=0}^{l-2}\lambda_{k-s}}.
\end{align}
By \eqref{eqn:density.BUEF} and the assumptions on $p_0,c,\eta$, for any $l\in\{0,\dots,k-1\}$, we have 
\begin{align}
\frac{1}{\prod_{s=0}^{l-1}\lambda_{k-s}}\leq&\begin{cases}
\displaystyle  e^{lp_0\C{C'}{0}},~&\text{if }l\leq\frac{(m-k)\eta}{1-\eta},\\
\displaystyle  e^{(m-k+l)\eta p_0\C{C'}{0}}e^{-\left(l-(m-k+l)\eta\right)cp_0},~&\text{if }l> \frac{(m-k)\eta}{1-\eta},
\end{cases}\nonumber\\
=&\begin{cases}
\displaystyle  e^{lp_0\C{C'}{0}},~&\text{if }l\leq \frac{(m-k)\eta}{1-\eta},\\
\displaystyle  e^{lp_0(\eta (\C{C'}{0}+c)-c)}e^{(m-k)\eta p_0(\C{C'}{0}+c)},~&\text{if }l> \frac{(m-k)\eta}{1-\eta},
\end{cases}\nonumber\\
\leq&\begin{cases}
\displaystyle  e^{lp_0\C{C'}{0}},~&\text{if }l\leq \frac{(m-k)\eta}{1-\eta},\\
\displaystyle  e^{-lcp_0/2}e^{2(m-k)\eta p_0\C{C'}{0}},~&\text{if }l> \frac{(m-k)\eta}{1-\eta},
\end{cases}\label{eqn:density.prod.of.lambda}\\
\leq &  e^{2(m-k)\eta p_0\C{C'}{0}}.          \label{eqn:density.prod.of.lambda.stupid}
\end{align}
 
Recall that $\eta\leq c/16\C{C'}{0}< 1/16$ and $n=mp_0$. Applying \eqref{eqn:density.prod.of.lambda} and \eqref{eqn:density.prod.of.lambda.stupid} to \eqref{eqn:log.lip.est}, we obtain that
\begin{align}\label{eqn:density.p0.main}
L_{(kp_0)}=&\frac{L}{\prod_{s=1}^k\lambda_s}+\C{K'}{2}(p_0;c)(K+1)\sum_{l=1}^k\frac{e^{8(n-(k-l)p_0)\eta\C{C'}{0}}}{\prod_{s=0}^{l-2}\lambda_{k-s}}\nonumber\\
\leq &Le^{2(m-k)\eta p_0\C{C'}{0}}+\C{K'}{2}(p_0;c)(K+1)\sum_{l=1}^{\left[\frac{(m-k)\eta}{1-\eta}\right]+1}{e^{8(n-(k-l)p_0)\eta\C{C'}{0}}}e^{(l-1)p_0\C{C'}{0}}\nonumber\\
&+\C{K'}{2}(p_0;c)(K+1)\sum_{l=\left[\frac{(m-k)\eta}{1-\eta}\right]+2}^k{e^{8(n-(k-l)p_0)\eta\C{C'}{0}}}e^{-(l-1)cp_0}e^{2(m-k)\eta p_0\C{C'}{0}}\nonumber\\
= &Le^{2(n-kp_0)\eta\C{C'}{0}}+\C{K'}{2}(p_0;c)(K+1)e^{8(n-kp_0)\eta\C{C'}{0}}\sum_{l=1}^{\left[\frac{(m-k)\eta}{1-\eta}\right]+1}e^{(8l\eta+l-1)p_0\C{C'}{0}}\nonumber\\
&+\C{K'}{2}(p_0;c)(K+1)e^{10(n-kp_0)\eta\C{C'}{0}}\sum_{l=\left[\frac{(m-k)\eta}{1-\eta}\right]+2}^ke^{cp_0}e^{lp_0(8\eta\C{C'}{0}-c)}\nonumber\\
\leq &Le^{2(n-kp_0)\eta\C{C'}{0}}+\C{K'}{2}(p_0;c)(K+1)e^{8(n-kp_0)\eta\C{C'}{0}}\sum_{l=1}^{\left[\frac{(m-k)\eta}{1-\eta}\right]+1}e^{2lp_0\C{C'}{0}}\nonumber\\
&+\C{K'}{2}(p_0;c)(K+1)e^{10(n-kp_0)\eta\C{C'}{0}}\sum_{l=\left[\frac{(m-k)\eta}{1-\eta}\right]+2}^ke^{cp_0}e^{-lp_0c/2}\nonumber\\
\leq &Le^{2(n-kp_0)\eta\C{C'}{0}}+\C{K'}{2}(p_0;c)(K+1)e^{8(n-kp_0)\eta\C{C'}{0}}\cdot\frac{e^{4p_0\C{C'}{0}}e^{\frac{2\eta}{1-\eta}(m-k)p_0\C{C'}{0}}}{e^{2p_0\C{C'}{0}}-1}\nonumber\\
&+\C{K'}{2}(p_0;c)(K+1)e^{10(n-kp_0)\eta\C{C'}{0}}\cdot\frac{e^{cp_0}}{1-e^{-p_0c/2}}\nonumber\\
\leq &Le^{11(n-kp_0)\eta\C{C'}{0}}+\C{K'}{2}(p_0;c)(K+1)e^{11(n-kp_0)\eta\C{C'}{0}}\cdot\frac{e^{4p_0\C{C'}{0}}}{e^{2p_0\C{C'}{0}}-1}\nonumber\\
&+\C{K'}{2}(p_0;c)(K+1)e^{11(n-kp_0)\eta\C{C'}{0}}\cdot\frac{e^{cp_0}}{1-e^{-p_0c/2}}\nonumber\\
\leq& \C{K''}{2}(p_0;c)(L+1)(K+1)e^{11(n-kp_0)\eta\C{C'}{0}},
\end{align}
where 
$$\CS{K''}{2}(p_0;c):=1+\C{K'}{2}(p_0;c)\left(\frac{e^{4p_0\C{C'}{0}}}{e^{2p_0\C{C'}{0}}-1}+\frac{e^{cp_0}}{1-e^{-cp_0/2}}\right).$$
It remains for us to consider a suitable choice of $L_{(j)}$ when $j\not\in p_0\ZZ$. Since the assumptions on $\eta,c,p_0$ are stronger than the corresponding assumptions in Lemma \ref{lem:curv.est}, we observe that for any $j\in\{0,\dots,n-1\}$, by \eqref{eqn:unif.C2.norm} in \hyperlink{A1}{(\textbf{A1})} and Lemma \ref{lem:curv.est}, we have
\begin{align}\label{eqn:density.est.3}
&\left|\ln\left\|Df_{j}(\gamma_j(t_1))|_{\dot\gamma_j(t_1)}\right\|-\ln\left\|Df_{j}(\gamma_j(t_2))|_{\dot\gamma_j(t_2)}\right\|\right|\nonumber\\
\leq &e^{2\C{C'}{0}}d(\dot\gamma_j(t_1),\dot\gamma_j(t_2))\nonumber\\
\leq &e^{2\C{C'}{0}}\sqrt{1+\left(\C{K}{1}(p_0;c)(K+1)e^{8(n-j)\eta\C{C'}{0}}\right)^2}d_{\gamma_j}(\gamma_j(t_1),\gamma_j(t_2))\nonumber\\
\leq &2e^{2\C{C'}{0}}\C{K}{1}(p_0;c)(K+1)e^{8(n-j)\eta\C{C'}{0}}d_{\gamma_j}(\gamma_j(t_1),\gamma_j(t_2))\nonumber\\
\leq &2e^{3\C{C'}{0}}\C{K}{1}(p_0;c)(K+1)e^{8(n-j)\eta\C{C'}{0}}d_{\gamma_{j+1}}(\gamma_{j+1}(t_1),\gamma_{j+1}(t_2))
\end{align}
Write $j=kp_0+d$, where $k\in\{0,\dots,m-1\}$ and $d\in\{1,\dots,p_0-1\}$. Then by \eqref{eqn:unif.C2.norm} in \hyperlink{A1}{(\textbf{A1})}, \eqref{eqn:density.ineq} and \eqref{eqn:density.est.3}, for any $1\leq d\leq p_0-1$, one can always choose $L_{(kp_0+1)},\dots ,L_{(kp_0+d)}$ such that for any $s\in\{0,\dots, d-1\}$, we have
\begin{align}\label{eqn:density.est.4}
L_{(kp_0+s+1)}=e^{\C{C'}{0}}L_{(kp_0+s)}+2e^{3\C{C'}{0}}\C{K}{1}(p_0;c)(K+1)e^{8(n-kp_0-s)\eta\C{C'}{0}}
\end{align}
By \eqref{eqn:density.est.4}, \eqref{eqn:density.p0.main} and the assumption that $\eta<1/18$, we have 
\begin{align*}
L_{(j)}=&L_{(kp_0+d)}\\
=&e^{d\C{C'}{0}}L_{(kp_0)}+\sum_{s=0}^{d-1}e^{(d-s-1)\C{C'}{0}}2e^{3\C{C'}{0}}\C{K}{1}(p_0;c)(K+1)e^{8(n-kp_0-s)\eta\C{C'}{0}}\\
\leq & e^{d\C{C'}{0}}\C{K''}{2}(p_0;c)(L+1)(K+1)e^{11(n-kp_0)\eta\C{C'}{0}},\\
&+\sum_{s=0}^{d-1}e^{d\C{C'}{0}}2e^{3\C{C'}{0}}\C{K}{1}(p_0;c)(K+1)e^{8(n-kp_0)\eta\C{C'}{0}}\\
=  &e^{d\C{C'}{0}} \C{K''}{2}(p_0;c)(L+1)(K+1)e^{11(n-kp_0)\eta\C{C'}{0}},\\
&+2de^{(3+d )\C{C'}{0}}\C{K}{1}(p_0;c)(K+1)e^{8(n-kp_0)\eta\C{C'}{0}}\\
\leq &e^{11d\eta\C{C'}{0}} \C{K''}{2}(p_0;c)(L+1)(K+1)e^{11(n-kp_0-d)\eta\C{C'}{0}},\\
&+2de^{(3+d+8d\eta)\C{C'}{0}}\C{K}{1}(p_0;c)(K+1)e^{8(n-kp_0-d)\eta\C{C'}{0}}\\
\leq &e^{p_0\C{C'}{0}} \C{K''}{2}(p_0;c)(L+1)(K+1)e^{11(n-kp_0-d)\eta\C{C'}{0}},\\
&+{2p_0e^{(3+2p_0)\C{C'}{0}}}\C{K}{1}(p_0;c)(K+1)e^{11(n-kp_0-d)\eta\C{C'}{0}}\\
\leq& \C{K}{2}(p_0;c)(L+1)(K+1)e^{11(n-kp_0-d)\eta\C{C'}{0}}=\C{K}{2}(p_0;c)(L+1)(K+1)e^{11(n-j)\eta\C{C'}{0}},
\end{align*}
where
$$\CS{K}{2}(p_0;c):=e^{p_0\C{C'}{0}} \C{K''}{2}(p_0;c)+{2p_0e^{(3+2p_0)\C{C'}{0}}}\C{K}{1}(p_0;c)\geq \C{K''}{2}(p_0;c).$$
This completes the proof.
\end{proof}

\section{Distortion bounds}\label{Sec.Distortion}

The goal of this section is to establish some basic estimates on distortion. These estimates are quite general. Indeed, for this section we suppose that $\mathcal{U}' \subset \mathcal{U}$ are $C^2$-open sets in $\mathrm{Diff}^2(\mathbb{T}^2)$ and $\mu$ a probability measure verifying assumption \hyperlink{A1}{(\textbf{A1})}.  \Cref{lem:bded.distortion.1} below gives estimates to the effect of the derivative on the same vector for points that are very close. 

\begin{lem}\label{lem:bded.distortion.1}
Let $z\in\TT^2$, $\omega\in\Sigma_+$ and $n\in\ZZ_+$. Assume that two points $x,y\in\TT^2$ satisfy $f^n_\omega(x),f^n_\omega(y)\in B(f^n_\omega(z),e^{-6n\C{C'}{0}})$, then the following holds:
\begin{enumerate}
\item $d(f^{n-k}_\omega(x),f^{n-k}_\omega(y))\leq 2e^{(k-6n)\C{C'}{0}}$.
\item For any unit vector $v\in T^1\TT^2$, we have
$$\sphericalangle(Df^n_\omega(x)v,Df^n_\omega(y)v)\leq e^{-3n\C{C'}{0}} \text{ and }\|Df^n_\omega(x)v-Df^n_\omega(y)v\|\leq e^{-3n\C{C'}{0}}.$$
\end{enumerate}
\end{lem}
\begin{proof}
Set $\omega=(f_{0},f_{1},\dots )$. Recall that $f^{n}_{\omega}=f_{n-1} \circ \dots \circ f_{0}$.
The first assertion can be deduce directly from \eqref{eqn:unif.C2.norm} in \hyperlink{A1}{(\textbf{A1})}.

Set $x_m = f^{m}_{\omega}(x)$ and $y_m = f^{m}_{\omega}(y)$ for all $m\in \Zbb_+$. For any $k\in\{1,\dots,n\}$, the first assertion and \hyperlink{A1}{(\textbf{A1})} imply that 
\begin{align*}
\sphericalangle(Df_{\omega}^{k}(x)v, Df_{\omega}^{k}(y)v)  \le & \sphericalangle(Df_{\omega}^{k}(x)v, Df_{k-1}(x_{k-1})\circ Df_{\omega}^{k-1}(y)v) \\
&+ \sphericalangle(Df_{k-1}(x_{k-1})\circ Df_{\omega}^{k-1}(y)v, Df^{k}_{\omega}(y)v)\\
\le & e^{\C{C'}{0}} \sphericalangle(Df_{\omega}^{k-1}(x)v, Df_{\omega}^{k-1}(y)v)+e^{\C{C'}{0}}d(x_{k-1},y_{k-1})\\
\le & e^{\C{C'}{0}}\sphericalangle(Df_{\omega}^{k-1}(x)v, Df^{k-1}_{\omega}(y)v)+2e^{(2-k-5n)\C{C'}{0}}.
\end{align*}
We can apply the above inequality inductively so that we have
\[\sphericalangle(Df_{\omega}^{n}(x)v,Df_{\omega}^{n}(y)v) \le 2e^{(2-6n)\C{C'}{0}}+2e^{(4-6n)\C{C'}{0}}+\dots+2e^{(2n-6n)\C{C'}{0}}\le \frac{2e^{-4n\C{C'}{0}}}{1-e^{-2\C{C'}{0}}}<e^{-3n\C{C'}{0}}.\] Here, we assumed $\C{C'}{0}$ is large enough so that $\frac{2e^{-n\C{C'}{0}}}{1-e^{-2\C{C'}{0}}}<1$ for all $n\in\ZZ_+$. This shows the first inequality in the second item.

Notice that for any $k\in\{1,\dots,n\}$, the first assertion and \hyperlink{A1}{(\textbf{A1})} imply that 
\begin{align*}
&\|Df_{\omega}^{k}(x)v-Df_{\omega}^{k}(y)v\|\\
\le & \|Df_{\omega}^{k}(x)v-Df_{k-1}(x_{k-1})\circ Df_{\omega}^{k-1}(y)v\|+\|Df_{k-1}(x_{k-1})\circ Df_{\omega}^{k-1}(y)v - Df_{\omega}^{k}(y)v \|  \\
\le & e^{\C{C'}{0}}\|Df_{\omega}^{k-1}(x)v-Df_{\omega}^{k-1}(y)v\|+e^{\C{C'}{0}}d(x_{k-1},y_{k-1})\|Df_{\omega}^{k-1}(y)v\|\\
\le & e^{\C{C'}{0}}\|Df_{\omega}^{k-1}(x)v-Df_{\omega}^{k-1}(y)v\|+e^{\C{C'}{0}}e^{(1-k-5n)\C{C'}{0}}e^{(k-1)\C{C'}{0}}\\
\le & e^{\C{C'}{0}}\|Df_{\omega}^{k-1}(x)v-Df_{\omega}^{k-1}(y)v\|+e^{(1-5n)\C{C'}{0}}.
\end{align*}
We use the above inequality inductively so that we have
\[\|Df_{\omega}^{n}(x)v-Df_{\omega}^{n}(y)v\|\le e^{(1-5n)\C{C'}{0}}+e^{(2-5n)\C{C'}{0}}+\dots+e^{-4n\C{C'}{0}}\le \frac{e^{-4n\C{C'}{0}}}{e^{\C{C'}{0}}-1} <e^{-3n\C{C'}{0}}.\]
Here, we assumed $\C{C'}{0}$ is large enough so that $\frac{e^{-n\C{C'}{0}}}{e^{\C{C'}{0}}-1} <1$ for all $n\in\ZZ_+$. This proves the second assertion in \Cref{lem:bded.distortion.1}.
\end{proof}

As a consequence of \Cref{lem:bded.distortion.1}, we obtain the following result. 

\begin{cor}\label{cor:bded.distortion.2}
Under the same assumptions as in Lemma \ref{lem:bded.distortion.1}, let $v_x\in T^1_x\TT^2$ and $v_y\in T^1_y\TT^2$ be unit vectors such that
$$\|Df^n_\omega(x)v_x\|=\inf_{v\in T^1_x\TT^2}\|Df^n_\omega(x)v\|\text{ and }\|Df^n_\omega(y)v_y\|=\inf_{v\in T^1_y\TT^2}\|Df^n_\omega(y)v\|.$$
Then, we have
$$\left\|Df^n_\omega(x)v_x\right\|-\left\|Df^n_\omega(y)v_y\right\|\leq e^{-3n\C{C'}{0}}$$
and
$$\left|\left\|Df^n_\omega(x)v_y\right\|-\left\|Df^n_\omega(x)v_x\right\|\right|\leq 2e^{-3n\C{C'}{0}}.$$
In particular, following the notations in \eqref{notation:EuEs}, if $E^s_{f^n_\omega}(x)$ and $E^s_{f^n_\omega}(y)$ are well-defined, then
$$\left\|Df^n_\omega(x)|_{E^s_{f^n_\omega}(x)}\right\|-\left\|Df^n_\omega(y)|_{E^s_{f^n_\omega}(y)}\right\|\leq e^{-3n\C{C'}{0}}$$
and
$$\left|\left\|Df^n_\omega(x)|_{E^s_{f^n_\omega}(y)}\right\|-\left\|Df^n_\omega(x)|_{E^s_{f^n_\omega}(x)}\right\|\right|\leq 2e^{-3n\C{C'}{0}}.$$
\end{cor}
\begin{proof}
Using \Cref{lem:bded.distortion.1}, we can deduce the first inequality as follows:
\[ \|Df_{\omega}^{n}(x)v_x\|\le \|Df_{\omega}^{n}(x)v_y\| \le \|Df_{\omega}^{n}(y)v_y\|+e^{-3n\C{C'}{0}}.\]  
Using the above inequality and \Cref{lem:bded.distortion.1}, 
\[\|Df_{\omega}^{n}(x)v_x\|\le \|Df_{\omega}^{n}(y)v_y\|+e^{-3n\C{C'}{0}}\le \|Df_{\omega}^{n}(x)v_y\|+2e^{-3n\C{C'}{0}}.\]
This shows the second inequality.
\end{proof}

Let us now obtain some quantitative estimate on the angle of the mostly contracting directions for points that are close to each other. This is obtained in the following lemma.

\begin{lem}\label{lem:bded.distortion.3}
Under the same assumptions as in Lemma \ref{lem:bded.distortion.1}, assume in addition that there are constants $\CS{C}{0}>0$ and $\CS{\epsilon}{0}$,  with $0<\C{\epsilon}{0}\ll1$, such that
\begin{itemize}
\item $\Jac f^n_\omega(q)\in(e^{-\C{C}{0}-2n\C{\epsilon}{0}},e^{\C{C}{0}+2n\C{\epsilon}{0}} )$ for any $q\in\TT^2$.
\item There exists $c\in(8{\C{\epsilon}{0}},\C{C'}{0})$ such that $\|Df^n_\omega(z)\|\geq e^{cn}$.
\end{itemize}
Then there exists some $\C{n'}{0}=\C{n'}{0}(c)$ such that $E^s_{f^n_\omega}(x),E^s_{f^n_\omega}(y)$ are well-defined and
$$\sphericalangle(E^s_{f^n_\omega}(x),E^s_{f^n_\omega}(y))<\frac{1}{2}e^{-n\C{C'}{0}}.$$ 
whenever $n\geq \C{n'}{0}$.
\end{lem}

\begin{remark}
When we apply \Cref{lem:bded.distortion.3} in this work, we will use the constants given by \hyperlink{A5}{(\textbf{A5})}. However, \Cref{lem:bded.distortion.3} itself holds without adding more assumptions on $\epsilon_0$. 
\end{remark}

\begin{proof}
We first assume that $\C{n'}{0}(c)$ is large enough such that 
\begin{align}\label{eqn:bded.distortion.3-0}
e^{\C{C}{0}-(c\C{n'}{0}/2)}+2e^{-3\C{n'}{0}\C{C'}{0}}< e^{c\C{n'}{0}}.
\end{align}
Then by the two bullet points, \Cref{lem:bded.distortion.1}, \Cref{cor:bded.distortion.2}, the fact that $c>8\C{\epsilon}{0}$ and \eqref{eqn:bded.distortion.3-0}, for any $n\geq\C{n'}{0}$ and for any $q\in\TT^2$ such that $f^n_\omega(q)\in B(f^n_\omega(z),e^{-6n\C{C'}{0}})$, we have 
\begin{align*}
\inf_{v\in T^1_q\TT^2}\|Df^n_{\omega}(q)v\|\leq& e^{-3n\C{C'}{0}}+ \inf_{v\in T^1_z\TT^2}\|Df^n_{\omega}(z)v\|\\
=&e^{-3n\C{C'}{0}}+\frac{\Jac f^n_\omega(z)}{\|Df^n_{\omega}(z)\|}\\
\leq& e^{-3n\C{C'}{0}}+e^{\C{C}{0}+2n\C{\epsilon}{0}-cn}\\
\leq& e^{-3n\C{C'}{0}}+e^{\C{C}{0}-(cn/2)}<  e^{cn}-e^{-3n\C{C'}{0}}\leq \|Df^n_\omega(z)\|-e^{-3n\C{C'}{0}}\leq \|Df^n_\omega(q)\|.
\end{align*}
This proves that $E^{s}_{f^{n}_{\omega}}(q)$ and $E^{u}_{f^{n}_{\omega}}(q)$ are well-defined. In particular, $E^{*}_{f^{n}_{\omega}}(x),E^{*}_{f^{n}_{\omega}}(y)$ and $E^{*}_{f^{n}_{\omega}}(z)$ are all well-defined for any $*=s,u$.

Let $\theta_{x}=\sphericalangle(E^{s}_{f^{n}_{\omega}}(x),E^{s}_{f^{n}_{\omega}}(z))$ and $\theta_{y}=\sphericalangle(E^{s}_{f^{n}_{\omega}}(y), E^{s}_{f^{n}_{\omega}}(z))$. Let $\lambda^{s}(z)=\|Df^{n}_{\omega}(z)|_{E^{s}_{f^{n}_{\omega}(z)}}\|$.
\begin{align}\label{eqn:bded.distortion.3-1}
\|Df^{n}_{\omega}(z)|_{E^{s}_{f^{n}_{\omega}(x)}}\|^{2}  \ge & (\lambda^{s}(z))^{2}\cos^{2}\theta_{x}+\frac{e^{-4{\C{\epsilon}{0}} n - 2 \C{C}{0}}}{(\lambda^{s}(z))^{2}}\sin^{2}\theta_{x} \quad(\textrm{using the first bullet point})\nonumber\\
\ge & (\lambda^{s}(z))^{2}(1-\theta_{x}^{2})+\frac{e^{-4{\C{\epsilon}{0}} n - 2 \C{C}{0}}}{(\lambda^{s}(z))^{2}} \frac{4}{\pi^{2}}\theta_{x}^{2}\nonumber\\
= & (\lambda^{s}(z))^{2}+\theta_{x}^{2}\left(\frac{4}{\pi^{2}} \frac{e^{-4{\C{\epsilon}{0}} n - 2 \C{C}{0}}}{(\lambda^{s}(z))^{2}} - (\lambda^{s}(z))^{2}\right)
\end{align}
Since $\|Df_{\omega}^{n}(z)\| \ge e^{cn}$ and $c>8\C{\epsilon}{0}$, we have
 \[\lambda^{s}(z)\le \frac{e^{{2\C{\epsilon}{0}} n +\C{C}{0}}}{e^{cn}}\le e^{-\frac{3cn}{4}+\C{C}{0}}.\]
 Therefore, \[\left(\frac{4}{\pi^{2}} \frac{e^{-4{\C{\epsilon}{0}} n - 2 \C{C}{0}}}{(\lambda^{s}(z))^{2}} - (\lambda^{s}(z))^{2}\right) \ge \frac{4}{\pi^{2}} e^{cn-4\C{C}{0}}-e^{\frac{-3cn}{2}+2\C{C}{0}}.\]
In particular, if we choose $\CS{n'}{0}=\C{n'}{0}(c)$ large enough such that \eqref{eqn:bded.distortion.3-0} holds and 
$$\frac{2}{\pi^{2}} e^{c\C{n'}{0}-4\C{C}{0}}> e^{\frac{-3c\C{n'}{0}}{2}+2\C{C}{0}} \text{ and } 2e^{-\frac{3c\C{n'}{0}}{4}+\C{C}{0}}+2e^{-3\C{n'}{0}\C{C'}{0}}<\frac{ e^{-4\C{C}{0}}}{16\pi^2},$$
then for any $n\geq \C{n'}{0}$, we have
\begin{align}\label{eqn:bded.distortion.3-2}
2\lambda^s(z)+2e^{-3n\C{C'}{0}}<\frac{ e^{-4\C{C}{0}}}{16\pi^2}\text{ and } \left(\frac{4}{\pi^{2}} \frac{e^{-4{\C{\epsilon}{0}} n - 2 \C{C}{0}}}{(\lambda^{s}(z))^{2}} - (\lambda^{s}(z))^{2}\right) >\frac{2}{\pi^{2}} e^{cn-4\C{C}{0}}.
\end{align}
On the other hand, by \Cref{cor:bded.distortion.2} and the above, we have
\begin{align*}
&\|Df_{\omega}^{n}(z)|_{E^{s}_{f^{n}_{\omega}(x)}}\|^{2}-(\lambda^{s}(z))^{2}\\
\le& \left(2\lambda^{s}(z)+\left|\|Df_{\omega}^{n}(z)|_{E^{s}_{f^{n}_{\omega}(x)}}\|-\lambda^{s}(z)\right|\right)\left|\|Df_{\omega}^{n}(z)|_{E^{s}_{f^{n}_{\omega}(x)}}\|-\lambda^{s}(z)\right|\\
\leq& (2\lambda^s(z)+2e^{-3n\C{C'}{0}})\cdot  2e^{-3n\C{C'}{0}}\leq \frac{ e^{-4\C{C}{0}}}{8\pi^2}e^{-3n\C{C'}{0}}.
\end{align*}
By \eqref{eqn:bded.distortion.3-1} and the above, we have
\[\theta_{x}^{2}\left(\frac{4}{\pi^{2}} \frac{e^{-4{\C{\epsilon}{0}} n - 2 \C{C}{0}}}{(\lambda^{s}(z))^{2}} - (\lambda^{s}(z))^{2}\right) < \frac{ e^{-4\C{C}{0}}}{8\pi^2}e^{-3n\C{C'}{0}}.\] 
Apply \eqref{eqn:bded.distortion.3-2} to the above, we obtain  
  \begin{align*}
  \theta_{x}^{2} \le \frac{\frac{ e^{-4\C{C}{0}}}{8\pi^2}e^{-3n\C{C'}{0}}}{\left(\frac{4}{\pi^{2}} \frac{e^{-4{\C{\epsilon}{0}} n - 2 \C{C}{0}}}{(\lambda^{s}(z))^{2}} - (\lambda^{s}(z))^{2}\right)}<\frac{\frac{e^{-4\C{C}{0}}}{8\pi^2}e^{-3n\C{C'}{0}}}{\frac{2}{\pi^{2}}e^{cn-4\C{C}{0}}}\leq \frac{1}{16}e^{-2n\C{C'}{0}}.
  \end{align*}
This proves that $\theta_x<e^{-n\C{C'}{0}}/4$. One can prove $\theta_y<e^{-n\C{C'}{0}}/4$ using the same argument. Hence 
$$\sphericalangle(E^s_{f^n_\omega}(x),E^s_{f^n_\omega}(y))\leq \theta_x+\theta_y<\frac{1}{2}e^{-n\C{C'}{0}}.\qedhere$$
  
\end{proof}

\begin{lem}\label{lem:bded.distortion.4}
Under the same assumptions as in Lemma \ref{lem:bded.distortion.3}, for any $\eta\in(0,\C{C'}{0})$ and any line $F$ such that
$$\sphericalangle(F,E^s_{f^n_\omega}(z))\geq e^{-n\eta},$$
we have
$$e^{-n({2\C{\epsilon}{0}}+\eta)-\C{C}{0}-2}\leq \left\|Df^n_\omega(x)|_F\right\|\cdot\left\|Df^n_\omega(y)|_{E^s_{f^n_\omega}(y)}\right\|\leq e^{{2\C{\epsilon}{0}} n+\C{C}{0}+2},$$
whenever $n\geq \C{n}{0}(c)$ for some $\C{n}{0}(c)\geq \C{n'}{0}(c)$.
\end{lem}
\begin{proof}
By \Cref{lem:bded.distortion.3}, we have $\sphericalangle(F,E^s_{f^n_\omega}(y))\geq \frac{1}{2}e^{-n\eta}$ whenever $n\geq \C{n'}{0}(c)$. Therefore, we have 
\begin{align}\label{eqn:bded.distortion.4-1}
\left\|Df^n_\omega(y)|_F\right\|\geq \sin(\sphericalangle(F,E^s_{f^n_\omega}(y)))\left\|Df^n_{\omega}(y)|_{E^u_{f^n_{\omega}}(y)}\right\|\geq \frac{e^{-n\eta}}{\pi}\left\|Df^n_{\omega}(y)|_{E^u_{f^n_{\omega}}(y)}\right\|.
\end{align}

Recall that $\eta\in(0,\C{C'}{0})$, $8\C{\epsilon}{0}<c\leq \C{C'}{0}$ and $\Jac f^n_\omega(q)\in(e^{-\C{C}{0}-2n\C{\epsilon}{0}},e^{\C{C}{0}+2n\C{\epsilon}{0}} )$ for any $q\in\TT^2$. It follows from \eqref{eqn:unif.C2.norm} in \hyperlink{A1}{(\textbf{A1})}, \Cref{lem:bded.distortion.1} and \eqref{eqn:bded.distortion.4-1} that for any $n\geq \C{n'}{0}(c)$, we have
\begin{align}\label{eqn:bded.distortion.4-lower}
&\left\|Df^n_\omega(x)|_F\right\|\cdot\left\|Df^n_\omega(y)|_{E^s_{f^n_\omega}(y)}\right\|\nonumber\\
\geq &\left(\left\|Df^n_\omega(y)|_F\right\|-e^{-3n\C{C'}{0}}\right)\cdot\left\|Df^n_\omega(y)|_{E^s_{f^n_\omega}(y)}\right\|\nonumber\\
\geq & \left(\frac{e^{-n\eta}}{\pi}\left\|Df^n_{\omega}(y)|_{E^u_{f^n_{\omega}}(y)}\right\|-e^{-3n\C{C'}{0}}\right)\cdot\left\|Df^n_{\omega}(y)|_{E^s_{f^n_{\omega}(y)}}\right\|\nonumber\\
\geq& \frac{e^{-n\eta}}{\pi}\left\|Df^n_{\omega}(y)|_{E^u_{f^n_{\omega}}(y)}\right\|\cdot \left\|Df^n_{\omega}(y)|_{E^s_{f^n_{\omega}}(y)}\right\|-e^{-2n\C{C'}{0}}\nonumber\\
=& \frac{e^{-n\eta}}{\pi}\Jac(f^n_{\omega}(y))-e^{-2n\C{C'}{0}}\nonumber\\
\geq& \frac{e^{-n\eta}}{\pi}e^{-\C{C}{0}-2n\C{\epsilon}{0}}-e^{-2n\C{C'}{0}}\geq  \left(\frac{1}{\pi}-e^{\C{C}{0}-(n\C{C'}{0}/2)}\right)e^{-n(\eta+2\C{\epsilon}{0})-\C{C}{0}}.
\end{align}
On the other hand, by  \Cref{lem:bded.distortion.1} and \Cref{cor:bded.distortion.2}, we have 
\begin{align}\label{eqn:bded.distortion.4-upper}
&\left\|Df^n_\omega(x)|_F\right\|\cdot\left\|Df^n_\omega(y)|_{E^s_{f^n_\omega}(y)}\right\|\nonumber\\
\leq &\left(\left\|Df^n_\omega(y)|_F\right\|+e^{-3n\C{C'}{0}}\right)\cdot\left\|Df^n_\omega(y)|_{E^s_{f^n_\omega}(y)}\right\|\nonumber\\
\leq &\Jac f^n_{\omega}(y)+e^{-3n\C{C'}{0}}\|Df^n_{\omega}(y)\|
\leq e^{\C{C}{0}+2n\C{\epsilon}{0}}+e^{-2n\C{C'}{0}}
\leq 2e^{\C{C}{0}+2n\C{\epsilon}{0}}.
\end{align}
Choose $\CS{n}{0}(c)\geq \C{n'}{0}(c)$ such that
$$\frac{1}{\pi}-e^{\C{C}{0}-(\C{n}{0}\C{C'}{0}/2)}<e^{-2}.$$
The lemma then follows from \eqref{eqn:bded.distortion.4-lower} and \eqref{eqn:bded.distortion.4-upper}.
\end{proof}

\section{Semi-norm of a measure}\label{Sec:semi}

Given two finite measures $\nu$ and $\nu'$ on $\mathbb{T}^2$, and a number $\rho>0$, we define the $\rho$-inner product between $\nu$ and $\nu'$ by
\begin{align*}
\langle \nu, \nu'\rangle_{\rho} := \frac{1}{\rho^{4}}\int_{\mathbb{T}^2} \nu(B(z,\rho))\nu'(B(z,\rho))  d\Leb(z),
\end{align*}
where $B(z,\rho)$ denotes the ball (with respect to the standard product metric on $\TT^2$) centered at $z$ with radius $\rho$. Define the $\rho$-semi-norm of $\nu$ by $\|\nu\|_{\rho} = \sqrt{\langle \nu, \nu\rangle_{\rho}}$. One can obviously see that
\begin{align}\label{eqn:obvious.norm.upper.bd}
\|\nu\|_\rho\leq\frac{\nu(\TT^2)}{\rho^2}.
\end{align}

\begin{lem}[\cite{Tsujii-bigpaper}, Lemma 6.2]\label{lem.liminfnorm}
If $\liminf_{\rho \to 0} \|\nu\|_{\rho} < +\infty$ then $\nu$ is absolutely continuous with respect to the standard Lebesgue measure $\Leb$ and $\lim_{\rho \to 0} \|\nu\|_{\rho} = \left\| \frac{d\nu}{d\Leb}\right\|_{L^2(\Leb)}$.
\end{lem}

We will also need the following lemma.

\begin{lem}[\cite{Tsujii-bigpaper}, Lemma 6.3]
\label{lem.normconvergence}
If a sequence of Borel finite measures $\nu_k$ converges weakly to a measure $\nu_{\infty}$, then for any $\rho>0$, we have $\|\nu_\infty\|_\rho = \lim_{k\to +\infty} \|\nu_k\|_\rho$.
\end{lem}
We introduce the following generalization of the above semi-norm: For any $\nu\in\mathrm{Prob}(\TT^2)$ and any Lebesgue measurable, strictly positive function $\delta:\TT^2\to\RR_+$, we define
\begin{align}\label{generalized r-norm}
\|\nu\|_\delta^2:=\int_{\TT^2}\frac{(\nu(B(x,\delta(x))))^2}{(\delta(x))^4}d\Leb(x).
\end{align}
In particular, for any $\rho\in\RR_+$, $\|\nu\|_\rho$ is given by \eqref{generalized r-norm} with $\delta(x)\equiv \rho$. We also have the following generalization of \cite[Lemma 6.1]{Tsujii-bigpaper}.

\begin{lem}\label{generalized large<small}
There exist a constant $\CS{C}{4}>1$ such that the following holds. For any positive numbers $0<\rho\leq \delta_-\leq \delta_+\leq 1$ and any Lebesgue measurable function $\delta:\TT^2\to \RR$ such that $\delta(\TT^2)\subset [\delta_-,\delta_+]$, we have 
$$\|\nu\|^2_\delta\leq \C{C}{4}(1+\ln(\delta_+/\delta_-))\|\nu\|^2_\rho.$$
\end{lem}
\begin{proof}
Let $A_\rho:=\{z_1,...,z_k\}$ be a maximal $(\rho/5)$-separated subset, that is, a maximal subset of $\TT^2$ (with respect to inclusion) such that for any $x\neq y\in A_\rho$, $d(x,y)>(\rho/5)$. Then there exist some $N_0>0$, independent of the choice of $\rho$, such that 
\begin{enumerate}
\item $\bigcup_{z\in A_\rho}B(z,\rho/4)=\TT^2$;
\item For any $x\in \TT^2$, there are at most $N_0$ points $z\in A_\rho$ such that $x\in B(z,\rho/4)$.
\end{enumerate}
Indeed, if there exists some $x\in \left(\bigcup_{z\in A_\rho}B(z,\rho/4)\right)\setminus \TT^2$, then $A_\rho\sqcup\{x\}$ is a strictly larger $(\rho/5)$-separated subset of $\TT^2$. This contradicts the maximality of $A_\rho$. For any $x\in M$, $A_\rho\cap B(x,\rho/4)$ is $(\rho/5)$-separated. Therefore $\{B(z,\rho/13)\}_{z\in A_\rho\cap B(x,\rho/4)}$ is a collection of disjoint subsets of $B(x, \rho/3)$. Hence $|A_\rho\cap B(x,\rho/4)|\cdot(\pi\rho^2/169)\leq \pi\rho^2/9$. In particular, there are at most $N_0:=19=[169/9]+1$ points $z\in A_\rho$ such that $x\in B(z,\rho/4)$.

By the second property of $A_\rho$ above, we have
\begin{align*}
\|\nu\|^2_\rho=&\frac{1}{N_0^2\rho^4}\int_{\TT^2}\left(N_0\nu(B(x,\rho))\right)^2d\Leb(x)\\
\geq &\frac{1}{N_0^2\rho^4}\int_{\TT^2}\left(\sum_{z\in A_\rho\cap B(x,\rho/2)}\nu(B(z,\rho/4))\right)^2d\Leb(x)\\
\geq &\frac{1}{N_0^2\rho^4}\int_{\TT^2}\sum_{z\in A_\rho\cap B(x,\rho/2)}\left(\nu(B(z,\rho/4))\right)^2d\Leb(x)\\
=&\frac{1}{N_0^2\rho^4}\sum_{z\in A_\rho}\left(\nu(B(z,\rho/4))\right)^2\left(\int_{B(z,\rho/2)}d\Leb(x)\right)\geq\frac{\pi}{4N_0^2\rho^2}\sum_{z\in A_\rho}\left(\nu(B(z,\rho/4)))\right)^2.\\
\end{align*}

Therefore it suffices to show that there exists $\C{C}{4}>1$ independent of the choice of $\delta$ and $\rho$, such that 
\begin{align}\label{rho-discretize}
\|\nu\|^2_\delta\leq \frac{\C{C}{4}\pi(1+\ln(\delta_+/\delta_-))}{4N_0^2\rho^2}\sum_{z\in A_\rho}\left(\nu(B(z,\rho/4)))\right)^2.
\end{align}

By the first property of $A_\rho$, we have 
\begin{align}\label{delta<rho-1}
\|\nu\|^2_\delta=&\int_{\TT^2}\frac{(\nu(B(x,\delta(x))))^2}{(\delta(x))^4}d\Leb(x) \nonumber\\
\leq&\int_{\TT^2}\frac{1}{(\delta(x))^4}\left(\sum_{z\in A_\rho\cap B(x,2\delta(x))}\nu(B(z,\rho/4))\right)^2d\Leb(x) \nonumber\\
\leq &\int_{\TT^2}\frac{|A_\rho\cap B(x,2\delta(x))|}{(\delta(x))^4}\left(\sum_{z\in A_\rho\cap B(x,2\delta(x))}(\nu(B(z,\rho/4)))^2\right)d\Leb(x) \nonumber\\
=&\sum_{z\in A_\rho}(\nu(B(z,\rho/4)))^2\int_{\{x\in \TT^2| d(x,z)<2\delta(x)\}}\frac{|A_\rho\cap B(x,2\delta(x))|}{(\delta(x))^4}d\Leb(x).
\end{align}

Since $A_\rho$ is $(\rho/5)$-separated, $\{B(z,\rho/11)\}_{z\in A_\rho\cap B(x, 2\delta(x))}$ is a collection of pairwise disjoint subsets of $B(x,3\delta(x))$. Therefore $|A_\rho\cap B(x,2\delta(x))|\cdot(\pi\rho^2/121)\leq 9\pi(\delta(x))^2$. Hence \eqref{delta<rho-1} implies that
\begin{align}\label{delta<rho-2}
\|\nu\|^2_\delta\leq &\sum_{z\in A_\rho}(\nu(B(z,\rho/4)))^2\int_{\{x\in M| d(x,z)<2\delta(x)\}}\frac{1}{(\delta(x))^4}\cdot \frac{9\cdot 121(\delta(x))^2}{\rho^2}d\Leb(x) \nonumber\\
=&\sum_{z\in A_\rho}(\nu(B(z,\rho/4)))^2\int_{\{x\in M| d(x,z)<2\delta(x)\}}\frac{1089}{(\delta(x))^2\rho^2} d\Leb(x) \nonumber\\
=&\sum_{z\in A_\rho}(\nu(B(z,\rho/4)))^2\int_{B(z,\delta_-)}\frac{1089}{(\delta(x))^2\rho^2} d\Leb(x) \nonumber\\
&+\sum_{z\in A_\rho}(\nu(B(z,\rho/4)))^2\int_{\{x\in M| d(x,z)<2\delta(x)\}\setminus B(z,\delta_-)}\frac{1089}{(\delta(x))^2\rho^2} dm(x) \nonumber\\
\leq &\sum_{z\in A_\rho}(\nu(B(z,\rho/4)))^2\int_{B(z,\delta_-)}\frac{1089}{\delta_-^2\rho^2} d\Leb(x) \nonumber\\
&+\sum_{z\in A_\rho}(\nu(B(z,\rho/4)))^2\int_{\{x\in M| d(x,z)<2\delta(x)\}\setminus B(z,\delta_-)}\frac{1089}{(d(x,z)/2)^2\rho^2} d\Leb(x) \nonumber\\
\leq &\sum_{z\in A_\rho}(\nu(B(z,\rho/4)))^2\left(\frac{1089\pi}{\rho^2}+\int_{B(z,2\delta_+)\setminus B(z,\delta_-)}\frac{1089}{(d(x,z)/2)^2\rho^2} d\Leb(x)\right)\nonumber\\
=&\sum_{z\in A_\rho}(\nu(B(z,\rho/4)))^2\left(\frac{1089\pi}{\rho^2}+2\pi\int_{\delta_-}^{2\delta_+}\frac{4356}{ r^2\rho^2}rdr\right)\nonumber\\
=&\pi\left(\frac{1089}{\rho^2}+\frac{8712}{\rho^2}\ln(2\delta_+/\delta_-)\right)\sum_{z\in A_\rho}(\nu(B(z,\rho/4)))^2.
\end{align}
Choose $\C{C}{4}=4N_0^2\cdot(1089+8712(1+\ln(2)))$ and the lemma follows directly from \eqref{rho-discretize} and \eqref{delta<rho-2}.
\end{proof}

\section{Admissible measures}\label{sect:ACAM}

In this section we introduce the notion of admissible measures. This is the way we formalize many of the computations we make involving measures obtained by doing some combination of measures absolutely continuous along curves, which is the case of SRB measures.  In this section, we will suppose that $\mathcal{U}' \subset \mathcal{U}$ are $C^2$-open sets in $\mathrm{Diff}^2(\mathbb{T}^2)$ and $\mu$ a probability measure verifying assumptions \hyperlink{A1}{(\textbf{A1})} - \hyperlink{A6}{(\textbf{A6})}.  

Recall that assumption \hyperlink{A2}{(\textbf{A2})} states that there is a forward invariant cone $\cC$, where one can also consider the case $\cC =T \TT^2/\{0\}$. The set $\cC$ will contain all of the possible unstable directions, by the forward invariance.  

Let $\cA$ be the collection of all oriented, arclength-parametrized $C^2$-curves with finite length in $\TT^2$ which are everywhere tangent to $\cC$.  Recall that for such a curve $\gamma$,  we use $m_\gamma$ for its arclength measure.  For any $K>0$ and any $a>0$, we define
$$\cA_K(a):=\{\gamma\in\cA:|\curv(\cdot;\gamma)|\leq K\text{ and }m_\gamma(\gamma)=a\}$$
and
$$\AA_K(a):=\cA_K(a)\times [0,a].$$
For simplicity, for any $J\subset \RR_+$ and any $K>0$, we write
$$\cA_K(J)=\bigsqcup_{a\in J}\cA_K(a)\text{, } \AA_K(J)=\bigsqcup_{a\in J}\AA_K(a)\text{ and } \AA:=\bigsqcup_{K>0}\AA_K(\RR^+).$$
Then we have the following natural projection maps $\Pi:\AA\to\cA$ and $\sP:\AA\to \TT^2$ given by
$$\Pi(\gamma, t)=\gamma\text{ and }\sP(\gamma, t)=\gamma(t).$$
In particular, $\Pi(\AA_K(J))=\cA_K(J)$.

Recall that given a finite measure $\widetilde{\nu}$ on $\AA$, the partition by preimages of $\Pi$ is $\widetilde{\nu}$-measurable. In particular,  we can write 
$$\widetilde{\nu}=\int_{\cA}\nu_\gamma d\alpha(\gamma).$$ 
\begin{dfn}[Admissible measures]\label{dfn:adm.msr}
Given $K,L>0$,  a finite measure $\widetilde{\nu}$ on $\AA$ is \emph{$(K,L)$-admissible on $\AA$} if the following holds:
\begin{enumerate}
\item $\widetilde{\nu}$ is supported on $\AA_K(\RR^+)$.
\item For $\alpha$-a.e. $\gamma$, the measure $\nu_\gamma$ is absolutely continuous with respect to $m_\gamma$. Moreover, $d\nu_\gamma/dm_\gamma$ is $L$-log-Lipschitz.
\end{enumerate} 

We say that a $(K,L)$-admissible measure $\widetilde{\nu}$ is \emph{supported on curves of length at most $a$} if $\widetilde{\nu}$ is supported on $\AA_K([0,a])$.

A finite measure $\nu$ on $\TT^2$ is \emph{$(K,L)$-admissible on $\TT^2$} if there exists a $(K,L)$-admissible measure $\widetilde{\nu}$ on $\AA$ such that $\sP_*\widetilde{\nu}=\nu$. The measure $\widetilde{\nu}$ is called \emph{an admissible lift} of $\nu$. We say that  a $(K,L)$-admissible measure  $\nu$ is \emph{supported on curves of length at most $a$} if $\widetilde{\nu}$ is supported on $\AA_K([0,a])$. 
\end{dfn}

\begin{example}
If $\nu$ is a $\mu$-stationary SRB measure, then $\nu$ is $(K,L)$-admissible for some constants $K$ and $L$.
\end{example}

For any $C^2$ diffeomorphism $F:\TT^2\to\TT^2$ and $C^2$ curve $\gamma:[0,a]\to \TT^2$, we define $F_*\gamma$ as the curve obtained by applying an orientation-preserving arclength reparametrization of $F\circ\gamma$. When $DF(\cC)\subset \cC$, this defines a map $F_*:\cA\to\cA$. We also define $F_{\#}:\AA\to \AA$ as the map 
\begin{equation}\label{eq.Fstar}
F_{\#}(\gamma,t)=\left(F_*\gamma, d_{F_*\gamma}(F(\gamma(0)),F(\gamma(t)))\right)=\left(F_*\gamma,\int_0^t\|DF(\gamma(\tau))\dot\gamma(\tau)\|d\tau\right).
\end{equation}
One can easily verify that $\sP\circ F_{\#}=F\circ\sP$. Also, it is clear that for any $F,G\in\Diff^2(\TT^2)$, we have 
\begin{align}\label{eqn:functorial}
F_*\circ G_*=(F\circ G)_*\text{ and }F_{\#}\circ G_{\#}=(F\circ G)_{\#}.
\end{align}

Recall that we defined $(p_0, \eta)$-NCT and $(p_0,c,\eta)$-ET in \Cref{dfn:tempered}.  Using the above language, we obtain the following immediate corollary of \hyperlink{A2}{(\textbf{A2})}, Lemma \ref{lem:curv.est} and Lemma \ref{lem:density.est}.
\begin{cor}\label{cor:BUEF.adm}
Let $\widetilde{\nu}$ be a $(K,L)$-admissible measure on $\AA$ with $K,L\geq 1$ (in particular, it is supported on $\AA_K(\RR_+)$). Let $n=mp_0$, where $m,p_0\in\ZZ_+$. Fix a word $\omega^n\in\Sigma_n$. Assume that for $\Pi_*\widetilde{\nu}$-a.e. curve $\gamma\in\cA_K(\RR_+)$, $\omega^n$ is $(p_0,\eta)$-NCT and $(p_0; c,\eta)$-ET along $\gamma$ (in the sense of Definition \ref{dfn:tempered}) with $p_0,c,\eta$ satisfying 
$$\C{\epsilon}{0}<\frac{c}{2}<\frac{\C{C'}{0}}{2}\text{ and }0<\frac{\C{C}{0}}{\C{C'}{0}p_0}<\eta<\frac{c}{2c+16\C{C'}{0}}<\frac{1}{18},$$  
then for any $j\in\{1,\dots,n\}$, the measure $((f^{j}_{\omega^n})_{\#})_*\widetilde{\nu}$ is $(K',L')$-admissible, where
$$K'=\C{K}{1}(p_0;c)(K+1)e^{8(n-j)\eta\C{C'}{0}}.$$
and 
$$L'= \C{K}{2}(p_0;c)(L+1)(K+1)e^{11(n-j)\eta\C{C'}{0}}.$$ 
\end{cor}
In \Cref{subsect:good.words}, we will construct a notion of ``good convolution'' for admissible measures, which depends on the disintegration in item (2) of \Cref{dfn:adm.msr}.  Roughly speaking, the ``good convolution'' is contructed by discarding some unwanted words from the regular convolution for each conditional measure suppported on a curve. However, if the support of the admissible measure contains long curves, we may end up discarding too much information. In order to avoid this, we need to cut short all long curves in the support of an admissble measure. This is realized by the following map: for any $a>0$, we define the map $\Delta_a:\AA\to\AA$ such that for any curve $\gamma\in\cA$ with $m_\gamma(\gamma)=l$, and for any $t\in[0.l]$, we have
\begin{align}\label{eqn:cutshort.map}
\Delta_a(\gamma,t)=\begin{cases}
\displaystyle (\gamma,t)~&\text{if }l\leq 2a,\\
\displaystyle \left(\gamma_{[t[l/a]/l]}, t-[t[l/a]/l]\cdot(l/[l/a])\right)~&\text{if }l> 2a,
\end{cases}
\end{align}
where $\gamma_i:[0,l/[l/a]]:\to\TT^2$ is defined as $\gamma_i(t)=\gamma(t+(il/[l/a]))$.
One can easily check that $\sP=\sP\circ\Delta_a$ for any $a\in\RR_+$. The lemma below is a summary of some other basic properties.
\begin{lem}\label{lem:refinement.lemma}
Let $\widetilde{\nu}$ be an admissible measure. The following holds:
\begin{enumerate}
\item For any $a\in\RR_+$ and for any $F\in\Diff^2(\TT^2)$, the following are equivalent:
\begin{enumerate}
\item[(i)] $\widetilde{\nu}$ is $(K,L)$-admissible.
\item[(ii)] $(\Delta_a)_*\widetilde{\nu}$ is $(K,L)$-admissible.
\item[(iii)] $\left(F_{\#}\circ\Delta_a\circ((F)^{-1})_{\#}\right)_*\widetilde{\nu}$ is $(K,L)$-admissible.
\end{enumerate} 
\item For any $a\in\RR_+$, if $\widetilde{\nu}$ is admissible, then the measure $(\Delta_a)_*\widetilde{\nu}$ is an admissible measure supported on curves of length at most $2a$.
\end{enumerate}
\end{lem}
\begin{proof}
This follows directly from \Cref{dfn:adm.msr} and \eqref{eqn:cutshort.map}.
\end{proof}

\section{A Lasota-Yorke inequality}\label{sec.lasotayorke}

The main goal of this section is to prove a Lasota-Yorke type of estimate for certain convolutions of measures. This is given by \Cref{lem:LY} below. In this section,  suppose that $\mathcal{U}' \subset \mathcal{U}$ are $C^2$-open sets in $\mathrm{Diff}^2(\mathbb{T}^2)$ and $\mu$ a probability measure verifying assumptions \hyperlink{A1}{(\textbf{A1})} - \hyperlink{A6}{(\textbf{A6})}.

\subsection{Notations}\label{subsection:notation.LY}
\begin{enumerate}
\item Recall that the constants $\delta$, $\chi$ and $\overline{\chi}$ were obtained, or fixed,  in \Cref{lem.DK.trick}, \Cref{cor.mostly expand/contract} and \Cref{eqn:epsilon.cond.1}. To simplify notations, we let 
$$\lambda=\chi/\delta, \text{  }\overline{\lambda}=\overline{\chi}/\delta\text{ and } \widehat{\lambda}:=\min\{\overline{\lambda}/2,1\}.$$
We assume in addition that $8{\C{\epsilon}{0}}<\overline{\lambda}$.
\item Recall that $\C{C'}{0}$ is the constant given by \hyperlink{A1}{(\textbf{A1})}. For any $z\in\TT^2$, for any $n\in\ZZ_+$ and for any $\omega^n\in\Sigma_n$, we define
$$J^s_{{\omega^n}}(z)=\left\|Df^n_{\omega^n}(z)|_{E^s_{{\omega^n}}(z)}\right\|\text{ and }\widehat{J}^s_{{\omega^n}}(z)=\inf_{y\in B\left(z,e^{-7n\C{C'}{0}}\right)}J^s_{{\omega^n}}(y).$$
\item Recall that in \Cref{Sec:semi} we defined the $\rho$-norm of a measure and in \Cref{sect:ACAM} we defined the notion of admissible measures.  For any $(K,L)$-admissible measure $\widetilde{\nu}$ on $\AA$ and for any $\rho>0$, we define $\|\widetilde{\nu}\|_\rho:=\|\sP_*\widetilde{\nu}\|_\rho$. One can also replace $\rho$ by any positive measurable function $h:\TT^2\to \RR_+$ in the above definition.
\item For any $n\in\ZZ_+$ and for any $(K,L)$-admissible measure $\widetilde{\nu}$ on $\AA$, we define the $n$-fold convolution of $\widetilde{\nu}$ by $\mu$ as 
$$\mu^{*n}*\widetilde{\nu}:=\int_{\Sigma_n}((f^n_{\omega^n})_\#)_*\widetilde{\nu}d\mu^n(\omega^n),$$ 
where the operation$(f^n_{\omega^n})_\#)_*\widetilde{\nu}$ is defined in \Cref{eq.Fstar}. 
Since $\sP\circ(f^n_{\omega^n})_{\#} =f^n_{\omega^n}\circ\sP$, we have 
$$\sP_*(\mu^{*n}*\widetilde{\nu})=\mu^{*n}*(\sP_*\widetilde{\nu}).$$
\item For any $n\in\ZZ_+$, we let 
\begin{align}\label{eqn:LY.VAP}
\Sigma_{n,\text{NC}}:=\{\omega^n\in\Sigma_n: \Jac f^n_{\omega^n}(x)\in(e^{-\C{C}{0}-2n\C{\epsilon}{0}},e^{\C{C}{0}+2n\C{\epsilon}{0}})\text{ for any }x\in\TT^2.\}
\end{align}
\end{enumerate}

\subsection{Overview on the strategy of the proof of \Cref{lem:LY}}

\subsection{Norm estimates}\label{subsect:norm.est}

Recall that $\C{C}{4}$ is the constant obtained in \Cref{generalized large<small}. 

\begin{lem}\label{lem:stupid.norm}
For any finite measure $\nu$ on $\TT^2$, for any $n\in\ZZ_+$,  any $\theta\in(0,1)$ and  $\omega^n\in\Sigma_n$, we have
$$\|(f^n_{\omega^n})_*\nu\|_\rho^2\leq \C{C}{4}e^{6n\C{C'}{0}}\|\nu\|^2_{e^{n\C{C'}{0}}\rho\theta}.$$
\end{lem}
\begin{proof}
By \hyperlink{A1}{(\textbf{A1})}, we have 
\begin{align*}
\|(f^n_{\omega^n})_*\nu\|_\rho^2\leq&\frac{1}{\rho^4}\int_{\TT^2}\left(\nu((f^n_{\omega^n})^{-1}(B(z,\rho)))\right)^2d\Leb(z)\\
= &\frac{1}{\rho^4}\int_{\TT^2}\left(\nu(B((f^n_{\omega^n})^{-1}(z),e^{n\C{C'}{0}}\rho)\right)^2d\Leb(z)\\
\leq &\frac{1}{\rho^4}\int_{\TT^2}\left(\nu(B(y,e^{n\C{C'}{0}}\rho))\right)^2e^{2n\C{C'}{0}}d\Leb(y)=e^{6n\C{C'}{0}}\|\nu\|^2_{e^{n\C{C'}{0}}\rho}.
\end{align*}
The lemma then follows directly from \Cref{generalized large<small}.
\end{proof}
\begin{lem}\label{lem:clever.norm}
Let $\nu=\int_\cA\nu_\gamma d\alpha(\gamma)$ be a $(Ke^C,Le^C)$-admissible measure on $\TT^2$ for some $K>1$ and $L>1$ and $C>0$.  Let $\omega^n\in\Sigma_{n,\text{NC}}$ be a fixed element (see \eqref{eqn:LY.VAP}). 
Assume in addition that $(f^n_{\omega^n})_*\nu$ is $(Ke^C,Le^C)$-admissible and there exists a constant $\eta\in(0,\C{C'}{0})$ such that the following holds:
\begin{enumerate}
\item For any $\gamma\in\cA$ in the support of $d\alpha$ and for any $t$ in the domain of $\gamma$, we have 
$$\|Df^n_{\omega^n}|_{\dot\gamma(t)}\|\geq 2e^{\overline{\lambda}n}.$$
\item For any $\gamma\in\cA$ in the support of $d\alpha$ and for any $t$ in the domain of $\gamma$, we have 
$$\sphericalangle(\dot\gamma(t),E^s_{{\omega^n}}(\gamma(t)))\geq 2e^{-\eta n}.$$
\item Let $\rho$ be an arbitrary constant in $(0,(16KL)^{-1}e^{-C-8n\C{C'}{0}})$. For any $\gamma\in\cA$ in the support of $d\alpha$, we have $m_{\gamma}(\gamma)\geq 4\rho e^{n\C{C'}{0}}$.
\end{enumerate}
 Then there exists a constant $\C{C}{5}$ such that for any $\theta\in(0,1)$, we have 
$$\|(f^n_{\omega^n})_*\nu\|_\rho^2\leq (1+n)\C{C}{5}e^{(6{\C{\epsilon}{0}}+2\eta)n}\|\nu\|^2_{\frac{2}{3}\theta e^{(\overline{\lambda}n/2)-\C{C}{0}}\rho},$$
whenever $n\geq\C{n}{0}(\overline{\lambda})$.
\end{lem}
\begin{proof}
Since 
$$\|(f^n_{\omega^n})_*\nu\|_\rho^2=\frac{1}{\rho^4}\int_{\TT^2}\left(\int_{\cA}\nu_\gamma((f^n_{\omega^n})^{-1}(B(z,\rho)))d\alpha(\gamma)\right)^2d\Leb(z).$$
we first estimate $\nu_\gamma((f^n_{\omega^n})^{-1}(B(z,\rho)))$ for any $\gamma\in\cA$ in the support of $d\alpha$. By the assumption that $\nu$ and $(f^n_{\omega^n})_*\nu$ are $(Ke^C,Le^C)$-admissible, we can assume WLOG that 
\begin{align}\label{eqn:clever.norm.curv.K}
|\curv(\cdot,\gamma)|\leq Ke^C, ~ |\curv(\cdot,f^n_{\omega^n}\circ\gamma)|\leq Ke^C
\end{align}
and
\begin{align}\label{eqn:clever.norm.curv.L}
\frac{d\nu_{\gamma}}{dm_{\gamma}}, \frac{d(f^n_{\omega^n})_*\nu_{\gamma}}{dm_{f^n_{\omega^n}\circ\gamma}} \text{ are }Le^C\text{-log Lipschitz}.
\end{align}
Let $I$ be any connected component of $\gamma\cap(f^n_{\omega^n})^{-1}(B(z,\rho))$. We observe the following:
\begin{enumerate}
\item[(i)] By the first assertion in \Cref{lem:length.bd.curv}, \eqref{eqn:clever.norm.curv.K} and the assumption on $\rho$, the length of $f^n_{\omega^n}(I)$ is at most $4\rho$. Hence by \eqref{eqn:unif.C2.norm},
$$
m_{\gamma}(I)\leq e^{n\C{C'}{0}}\cdot\mathrm{length}(f^n_{\omega^n}(I))\leq 4\rho e^{n\C{C'}{0}}.
$$
In particular, by the assumptions on $\rho$, we have 
\begin{align}\label{eqn:clever.norm.l-1}
Le^Cm_{\gamma}(I)\leq Le^C\cdot 4\rho e^{n\C{C'}{0}}\leq 1.
\end{align}
\item[(ii)] For any $y\in I$, $y\in B((f^n_{\omega^n})^{-1}(z),e^{-7n\C{C'}{0}})$. This follows from \hyperlink{A1}{(\textbf{A1})} and the assumption on $\rho$. 
\item[(iii)] When $n\geq \C{n}{0}(\overline{\lambda})$, we have 
$$
\left(\inf_{\gamma(t)\in I}\left\|Df^n_{\omega^n}(\gamma(t))|_{\dot\gamma(t)}\right\|\right)\cdot\widehat{J}^s_{{\omega^n}}((f^n_{\omega^n})^{-1}(z))\geq e^{-n({2\C{\epsilon}{0}}+\eta)-\C{C}{0}-2}.
$$
To see why the above holds, notice that
\begin{itemize}
\item By item (1), item (ii) and Lemma \ref{lem:bded.distortion.1}, we have $\|Df^n_{\omega^n}(z)\|\geq e^{\overline{\lambda}n}$. 
\item By \eqref{eqn:epsilon.cond.1} and the assumptions on ${\C{\epsilon}{0}}$ in \Cref{subsection:notation.LY}, we have $\overline{\lambda}\in(8\C{\epsilon}{0},\C{C'}{0})$. It follows from item (ii) and Lemma \ref{lem:bded.distortion.3} that
$$\sphericalangle(E^s_{\omega^n}(\gamma(t)),E^s_{\omega^n}(z))<\frac{1}{2}e^{-n\C{C'}{0}},\text{ for any }\gamma(t)\in I.$$ 
(Here, we also used the fact that $\C{n}{0}(\overline{\lambda})\geq \C{n'}{0}(\overline{\lambda})$, which is stated in Lemma \ref{lem:bded.distortion.4}) Hence by item (2) and the above, for any $\gamma(t)\in I$, we have 
\begin{align*}
\sphericalangle(\dot\gamma(t),E^s_{{\omega^n}}(z))\geq &\sphericalangle(\dot\gamma(t),E^s_{{\omega^n}}(\gamma(t)))-\sphericalangle(E^s_{\omega^n}(\gamma(t)),E^s_{\omega^n}(z))\\
\geq& 2e^{-\eta n}-\frac{1}{2}e^{-n\C{C'}{0}}>e^{-\eta n}.
\end{align*}
\end{itemize}
Item (iii) then follows from \Cref{lem:bded.distortion.4} and item (ii).
\end{enumerate}
As a corollary of the above, when $n\geq\C{n}{0}(\overline{\lambda})$, we have
\begin{align}\label{eqn:clever.norm.I-2}
\begin{split}
m_{\gamma}(I)\leq&\left(\inf_{\gamma(t)\in I}\left\|Df^n_{\omega^n}(\gamma(t))|_{\dot\gamma(t)}\right\|\right)^{-1}\cdot\mathrm{length}(f^n_{\omega^n}(I)) \\
\leq& 4\rho e^{n({2\C{\epsilon}{0}}+\eta)+\C{C}{0}+2}\cdot\widehat{J}^s_{{\omega^n}}((f^n_{\omega^n})^{-1}(z)).
\end{split}
\end{align}
Let $\zeta_\gamma(y):=d\nu_\gamma/dm_{\gamma}(y)$ for any $y\in\gamma$. Fix some $x\in I$. Then \eqref{eqn:clever.norm.l-1} and \eqref{eqn:clever.norm.curv.L} imply that 
$$|\ln\zeta_\gamma(y)-\ln\zeta_\gamma(x)|\leq Le^Cd_\gamma(x,y)\leq Le^Cm_\gamma(I)\leq 1,\text{ for any }y\in I.$$
It follows from \eqref{eqn:clever.norm.I-2} and the above that
\begin{align}\label{eqn:clever.norm.I}
\nu_\gamma(I)\leq e^1 \cdot \zeta_\gamma(x)m_{\gamma}(I)\leq  \zeta_\gamma(x)\cdot 4\rho e^{n({2\C{\epsilon}{0}}+\eta)+\C{C}{0}+3}\cdot\widehat{J}^s_{{\omega^n}}((f^n_{\omega^n})^{-1}(z)).
\end{align}
On the other hand, notice that by \hyperlink{A1}{(\textbf{A1})} and item (2) in \Cref{subsection:notation.LY}, we have 
\begin{align}\label{eqn:clever.norm.J-0}
(f^n_{\omega^n})^{-1}(B(z,\rho))\subset& B\left((f^n_{\omega^n})^{-1}(z), \left(\widehat{J}^s_{{\omega^n}}((f^n_{\omega^n})^{-1}(z))\right)^{-1}\rho\right)\nonumber\\
\subset& B\left((f^n_{\omega^n})^{-1}(z), 2\left(\widehat{J}^s_{{\omega^n}}((f^n_{\omega^n})^{-1}(z))\right)^{-1}\rho\right)\nonumber\\
\subset& B\left((f^n_{\omega^n})^{-1}(z),2e^{n\C{C'}{0}}\rho\right)\subset (f^n_{\omega^n})^{-1}(B(z,2e^{2n\C{C'}{0}}\rho))
\end{align}
Let $J$ be the connected component of $\gamma\cap B\left((f^n_{\omega^n})^{-1}(z), 2\left(\widehat{J}^s_{{\omega^n}}((f^n_{\omega^n})^{-1}(z))\right)^{-1}\rho\right)$ containing $I$. We first claim that
\begin{align}\label{eqn:clever.norm.localize}
I=J\cap (f^n_{\omega^n})^{-1}(B(z,\rho)).
\end{align}
Indeed, if we let $\overline{I}$ be the connected component of $\gamma\cap(f^n_{\omega^n})^{-1}(B(z,2e^{2n\C{C'}{0}}\rho))$ containing $I$. By \eqref{eqn:clever.norm.J-0}, we have $I\subset J\subset \overline{I}$. Since $4Ke^C\rho e^{2n\C{C'}{0}}<1$, Apply the second assertion in \Cref{lem:length.bd.curv} to $f^n_{\omega^n}\circ\gamma$ and $B(z,2e^{2n\C{C'}{0}}\rho)$, we have that
$J\cap (f^n_{\omega^n})^{-1}(B(z,\rho))=\overline{I}\cap(f^n_{\omega^n})^{-1}(B(z,\rho))$ is connected. (Here we also used \eqref{eqn:clever.norm.curv.K}.) Since $I$ is assumed to be a connected component of $\gamma\cap(f^n_{\omega^n})^{-1}(B(z,\rho))$, this verifies \eqref{eqn:clever.norm.localize}.

Since $8Ke^C\rho e^{n\C{C'}{0}}<1$, by \eqref{eqn:clever.norm.curv.K} and the first assertion in \Cref{lem:length.bd.curv}, we have
$$m_{\gamma}(J)\leq  8\left(\widehat{J}^s_{{\omega^n}}((f^n_{\omega^n})^{-1}(z))\right)^{-1}\rho\leq 8e^{n\C{C'}{0}}\rho.$$
In particular,
\begin{align}\label{eqn:clever.norm.J-1}
Le^Cm_{\gamma}(J)\leq 8Le^{C+n\C{C'}{0}}\rho\leq 1.
\end{align}
Recall that 
$$m_{\gamma}(\gamma)\geq 4\rho e^{n\C{C'}{0}}\geq 4\rho\cdot \left(\widehat{J}^s_{{\omega^n}}((f^n_{\omega^n})^{-1}(z))\right)^{-1}.$$ 
By \eqref{eqn:clever.norm.curv.K}, the first assertion in \Cref{lem:length.bd.curv} and the above, we have 
$$\gamma\not\subset  B\left((f^n_{\omega^n})^{-1}(z), 2\left(\widehat{J}^s_{{\omega^n}}((f^n_{\omega^n})^{-1}(z))\right)^{-1}\rho\right).$$ 
This and \eqref{eqn:clever.norm.J-0} imply that
\begin{align}\label{eqn:clever.norm.J-2}
m_{\gamma}(J)\geq \left(\widehat{J}^s_{{\omega^n}}((f^n_{\omega^n})^{-1}(z))\right)^{-1}\rho.
\end{align}
It follows from \eqref{eqn:clever.norm.curv.L}, \eqref{eqn:clever.norm.J-1} and \eqref{eqn:clever.norm.J-2} that
\begin{align}\label{eqn:clever.norm.J}
\nu_\gamma(J)\geq \zeta_\gamma(x)\cdot e^{-1}m_{\gamma}(J)\geq \zeta_\gamma(x)\cdot e^{-1} \left(\widehat{J}^s_{{\omega^n}}((f^n_{\omega^n})^{-1}(z))\right)^{-1}\rho.
\end{align}
By \eqref{eqn:clever.norm.I} and \eqref{eqn:clever.norm.J}, we have 
$$\frac{\nu_\gamma(I)}{\nu_\gamma(J)}\leq 4 e^{n({\C{2\epsilon}{0}}+\eta)+\C{C}{0}+4}\cdot\left(\widehat{J}^s_{{\omega^n}}((f^n_{\omega^n})^{-1}(z))\right)^2$$
whenever $n\geq \C{n}{0}(\overline{\lambda})$. \eqref{eqn:clever.norm.localize} and the above then imply that 
\begin{align*}
&\frac{\nu_\gamma((f^n_{\omega^n})^{-1}(B(z,\rho)))}{\nu_\gamma\left(B\left((f^n_{\omega^n})^{-1}(z), 2\left(\widehat{J}^s_{{\omega^n}}((f^n_{\omega^n})^{-1}(z))\right)^{-1}\rho\right)\right)}\\
\leq& 4 e^{n({2\C{\epsilon}{0}}+\eta)+\C{C}{0}+4}\cdot\left(\widehat{J}^s_{{\omega^n}}((f^n_{\omega^n})^{-1}(z))\right)^2
\end{align*}
Since the right-hand-side of the above does not depend on $\gamma$, we have 
\begin{align}\label{eqn:clever.norm.key}
\begin{split}
&\frac{\nu((f^n_{\omega^n})^{-1}(B(z,\rho)))}{\nu\left(B\left((f^n_{\omega^n})^{-1}(z), 2\left(\widehat{J}^s_{{\omega^n}}((f^n_{\omega^n})^{-1}(z))\right)^{-1}\rho\right)\right)}\\
\leq& 4 e^{n({2\C{\epsilon}{0}}+\eta)+\C{C}{0}+4}\cdot\left(\widehat{J}^s_{{\omega^n}}((f^n_{\omega^n})^{-1}(z))\right)^2
\end{split}
\end{align}
It follows from \eqref{eqn:LY.VAP} and \eqref{eqn:clever.norm.key} that
\begin{align}\label{eqn:clever.norm.1}
&\|(f^n_{\omega^n})_*\nu\|_\rho^2\nonumber\\
=&\frac{1}{\rho^4}\int_{\TT^2}\left(\nu((f^n_{\omega^n})^{-1}(B(z,\rho)))\right)^2d\Leb(z)\nonumber\\
\leq&\frac{256e^{2n({\C{2\epsilon}{0}}+\eta)+2\C{C}{0}+8}}{\left(2\widehat{J}^s_{{\omega^n}}((f^n_{\omega^n})^{-1}(z))\right)^{-4}\rho^4}\int_{\TT^2}\left(\nu\left(B\left((f^n_{\omega^n})^{-1}(z), 2\left(\widehat{J}^s_{{\omega^n}}((f^n_{\omega^n})^{-1}(z))\right)^{-1}\rho\right)\right)\right)^2d\Leb(z)\nonumber\\
\leq &\frac{256e^{2n({2\C{\epsilon}{0}}+\eta)+2\C{C}{0}+8}e^{\C{C}{0}+2n\C{\epsilon}{0}}}{\left(2\widehat{J}^s_{{\omega^n}}(y)\right)^{-4}\rho^4}\int_{\TT^2}\left(\nu\left(B\left(y, 2\left(\widehat{J}^s_{{\omega^n}}(y)\right)^{-1}\rho\right)\right)\right)^2d\Leb(y)\nonumber\\
\leq &256e^{3\C{C}{0}+8}e^{n(6{\C{\epsilon}{0}}+2\eta)}\|\nu\|^2_{2\left(\widehat{J}^s_{{\omega^n}}(\cdot)\right)^{-1}\rho}.
\end{align}
Notice that for any $y\in\TT^2$, the assumption on $\rho$ implies that $2\left(\widehat{J}^s_{{\omega^n}}(y)\right)^{-1}\rho<2e^{n\C{C'}{0}}\rho<e^{-7n\C{C'}{0}}$/8. If we assume in addition that $\nu\left(B\left(y, 2\left(\widehat{J}^s_{{\omega^n}}(y)\right)^{-1}\rho\right)\right)>0$, then there exists some $\gamma\in\cA$ and some $y'\in\gamma$ such that $d(y,y')<e^{-7n\C{C'}{0}}/8$. By \Cref{cor:bded.distortion.2}, item (1) and the assumption that $\omega^n\in\Sigma_{n,\text{NC}}$, we have 
\begin{align*}
\widehat{J}^s_{{\omega^n}}(y)\leq &\widehat{J}^s_{{\omega^n}}(y')+e^{-3n\C{C'}{0}}\\
\leq &{J}^s_{f^n_{\omega^n}}(y')+2e^{-3n\C{C'}{0}}\\
\leq & \frac{\Jac f^n_{\omega^n}(y')}{\|Df^n_{\omega^n}(y')|_{\dot\gamma|_{y'}}\|}+2e^{-3n\C{C'}{0}}\leq \frac{1}{2}e^{{2\C{\epsilon}{0}}  n+\C{C}{0}-\overline{\lambda}n}+2e^{-3n\C{C'}{0}}\leq 3e^{{2\C{\epsilon}{0}} n+\C{C}{0}-\overline{\lambda}n}.
\end{align*}
In particular
\begin{align}\label{eqn:clever.norm.lower.index}
\begin{split}
2\left(\widehat{J}^s_{{\omega^n}}(y)\right)^{-1}\geq \frac{2}{3}e^{-{2\C{\epsilon}{0}} n-\C{C}{0}+\overline{\lambda}n}.
\end{split}
\end{align}
For simplicity, we let $h:\TT^2\to\RR$ such that 
$$h(y)=\max\left\{\frac{2}{3}e^{-{2\C{\epsilon}{0}} n-\C{C}{0}+\overline{\lambda}n}, 2\left(\widehat{J}^s_{{\omega^n}}(y)\right)^{-1}\right\}.$$ 
By the assumptions on ${\C{\epsilon}{0}}$ in \Cref{subsection:notation.LY}, we have 
$$\frac{2}{3}e^{-n\C{C'}{0}-\C{C}{0}}\leq \frac{2}{3}e^{(\overline{\lambda}n/2)-\C{C}{0}}\leq \frac{2}{3}e^{-{2\C{\epsilon}{0}} n-\C{C}{0}+\overline{\lambda}n}=:h_-\leq h(y)\leq h_+:= 2e^{n\C{C'}{0}}.$$
Then following \eqref{eqn:clever.norm.1}, \eqref{eqn:clever.norm.lower.index}, \Cref{generalized large<small} and the above, when $n\geq\C{n}{0}(\overline{\lambda})$, we have 
\begin{align*}
&\|(f^n_{\omega^n})_*\nu\|_\rho^2\\
\leq &256e^{3\C{C}{0}+8}e^{n(6{\C{\epsilon}{0}}+2\eta)}\|\nu\|^2_{2\left(\widehat{J}^s_{{\omega^n}}(\cdot)\right)^{-1}\rho}\\
=&256e^{3\C{C}{0}+8}e^{n(6{\C{\epsilon}{0}}+2\eta)}\|\nu\|^2_{h(\cdot)\rho}\\
\leq &256e^{3\C{C}{0}+8}\C{C}{4}(1+\ln(h_+/h_-))e^{n(6{\C{\epsilon}{0}}+2\eta)}\|\nu\|^2_{h_-\rho}\\
\leq &\C{C}{5}(1+n)e^{n(6{\C{\epsilon}{0}}+2\eta)}\|\nu\|^2_{\frac{2}{3}\theta e^{(\overline{\lambda}n/2)-\C{C}{0}}\rho},
\end{align*}
where
$$\CS{C}{5}:=256e^{3\C{C}{0}+8}(\C{C}{4})^2(1+\C{C}{0}+\ln(3)+2\C{C'}{0}).\qedhere$$
\end{proof}

\subsection{Good words and good convolution}\label{subsect:good.words}

In the process of our proof, our goal is to understand the $\rho$-norm of the convolution of a measure supported on a curve which is absolutely continuous with respect to the arclength measure on the curve.  Similar to Pesin-Sinai's construction of SRB measures, it is important for us to control several information on the pushforward of this measure, including expansion and distortion estimates. Since we are in a nonuniformly hyperbolic setting, it is reasonable to expect that we can't get such control for every word. For this purpose, we introduce the notion of good words and good convolution. In an informal way, for a given curve, these will be the words for which we can get good estimates for the pushforward. The good convolution will be doing the convolution with the driving measure $\mu$, but only considering the good words.  Let us formalize these ideas.

\begin{dfn}[Good words]\label{dfn:good.word}
Let $\gamma:[0,a]\to\TT^2$ be a $C^2$ curve. A word $\omega^n\in\Sigma_n$ is \emph{$\eta$-good} for $\gamma$ if the following holds:
\begin{enumerate}
\item For any $t\in[0,a]$, we have 
$$\|Df^n_{\omega^n}|_{\dot\gamma(t)}\|\geq 2e^{\overline{\lambda}n}.$$
\item For any $t\in[0,a]$, we have 
$$\sphericalangle(\dot\gamma(t),E^s_{f^n_{\omega^n}}(\gamma(t)))\geq 2e^{-\eta n}.$$
\item $\omega^n\in\Sigma_{n,\text{NC}}$.
\end{enumerate}
We denote by $\cG_{n,\eta}(\gamma)\subset \Sigma_n$ the collection of all $\eta$-good words for $\gamma$ with length $n$.
\end{dfn}
\begin{lem}\label{lem:good.count-1}
Let $\gamma$ be a $C^2$ curve on $\TT^2$. Let $\eta\in (0, \min\{(\lambda-\overline{\lambda})/\C{\beta}{1},1\})$. Assume that there exist constants $K>1$ and $C>0$ such that 
$$|\curv(\cdot;\gamma)|\leq Ke^C\text{ and }m_{\gamma}(\gamma)\leq \frac{1}{2K}e^{-7n\C{C'}{0}-C}.$$
Then there exist constants $\C{C}{6}>0$ and $\C{n}{1}=\C{n}{1}(\overline{\lambda})$,  independent of the choice of $K$, $C$ and $\eta$, such that for any $n\geq \C{n}{1}$, we have
$$\mu^n(\cG_{n,\eta}(\gamma))\geq 1-\C{C}{6}e^{-\C{\beta}{1}\eta n}.$$
\end{lem}
\begin{proof}
Without loss of generality, we assume that $\gamma$ is parametrized by arclength. Choose $x\in\gamma$. 
Consider the set
\begin{align}\label{eqn:alt.good.word}
\cG_{n,\eta,x}:=\left\{\omega^n\in\Sigma_{n,\text{NC}}\left|\begin{aligned}
&\|Df^n_{\omega^n}(x)\|\geq 4e^{\overline{\lambda} n}\text{ and }\\
&\sphericalangle(\dot\gamma|_x,E^s_{{\omega^n}}(x))\geq 4e^{-\eta n}
\end{aligned}\right.\right\}
\end{align}
Let 
$$\CS{C}{6}:=\C{C'}{1}+e^{\C{C}{2}+(4/\delta)}+\C{C}{3}4^{\C{\beta}{1}}+1,$$
where these constants were obtained in \Cref{lem:VAP}, \Cref{cor.mostly expand/contract} and \Cref{lem.cone.trans}, we have 
\begin{align*}
\mu^n(\cG_{n,\eta,x})
\geq&1-\mu^n(\Sigma_n\setminus\Sigma_{n,\text{NC}})-\mu^n\left(\left\{\omega^n\in\Sigma_{n}: \|Df^n_{\omega^n}(x)\|<4e^{\overline{\lambda}n}    \right\}\right)\\
&-\mu^n\left(\left\{\omega^n\in\Sigma_{n}: E^{s}_{\omega^n}(x)\text{ is not well-defined, or } \sphericalangle(\dot\gamma|_x,E^s_{{\omega^n}}(x))< 4e^{-\eta n}    \right\}\right)\\
\geq& 1-\C{C'}{1}e^{-\C{\beta}{0}n}-e^{\C{C}{2}+\frac{4}{\delta}-n(\lambda-\overline{\lambda})}-\C{C}{3}\left(4e^{-\eta n}\right)^{\C{\beta}{1}}\geq 1-\C{C}{6}e^{-\C{\beta}{1}\eta n}.
\end{align*}

Recall that $\C{n'}{0}(\overline{\lambda})$ is the time obtained in \Cref{lem:bded.distortion.4}. 

\begin{claim}\label{claim.inclusion}
If $n\geq \C{n'}{0}$ then $\cG_{n,\eta,x}\subset \cG_{n,\eta}(\gamma)$. 
\end{claim}

The lemma follows immediately from \Cref{claim.inclusion} by taking $n_1(\overline{\lambda}) := \C{n'}{0}(\overline{\lambda})$. 

\begin{proof}[Proof of \Cref{claim.inclusion}]
Let $\omega^n$ be an arbitrary element in $\cG_{n,\eta,x}$. Notice that for any $y\in\gamma$, by the assumptions on $\gamma$, we have 
\begin{align}\label{eqn:good.count-1.1}
d(\dot\gamma|_x,\dot\gamma|_y)\leq \sqrt{1+(Ke^C)^2}\cdot d_\gamma(x,y)\leq 2Ke^Cd_\gamma(x,y)\leq e^{-7n\C{C'}{0}}.
\end{align}
In particular, by \eqref{eqn:unif.C2.norm} and \eqref{eqn:good.count-1.1}, for any $y\in\gamma$, we have 
\begin{align}\label{eqn:good.count-1.2}
\left\|Df^n_{\omega^n}(y)|_{\dot\gamma|_x}-Df^n_{\omega^n}(y)|_{\dot\gamma|_y}\right\|\leq \|Df^n_{\omega^n}(y)\|\cdot \|\dot\gamma|_x-\dot\gamma|_y\|\leq e^{n\C{C'}{0}}\cdot e^{-7n\C{C'}{0}}.
\end{align}
\Cref{lem:bded.distortion.1} and \eqref{eqn:good.count-1.1} imply that
\begin{align}\label{eqn:good.count-1.3}
\left\|Df^n_{\omega^n}(x)|_{\dot\gamma|_x}-Df^n_{\omega^n}(y)|_{\dot\gamma|_x}\right\|\leq e^{-3n\C{C'}{0}},~\text{ for any }y\in\gamma.
\end{align}
Hence by \eqref{eqn:alt.good.word}, \eqref{eqn:good.count-1.2} and \eqref{eqn:good.count-1.3}, for any $y\in\gamma$, we have 
\begin{align*}
\left\|Df^n_{\omega^n}(y)|_{\dot\gamma|_y}\right\|\geq& \left\|Df^n_{\omega^n}(x)|_{\dot\gamma|_x}\right\|-\left\|Df^n_{\omega^n}(x)|_{\dot\gamma|_x}-Df^n_{\omega^n}(y)|_{\dot\gamma|_x}\right\|\\
&-\left\|Df^n_{\omega^n}(y)|_{\dot\gamma|_x}-Df^n_{\omega^n}(y)|_{\dot\gamma|_y}\right\|\\
\geq &4e^{\overline{\lambda}n}-e^{-3n\C{C'}{0}}-e^{n\C{C'}{0}}\cdot e^{-7n\C{C'}{0}}\geq 2e^{\overline{\lambda}n}.
\end{align*}
This verifies item (1) in Definition \ref{dfn:good.word} for the word $\omega^n$. Item (2) in Definition \ref{dfn:good.word} for the word $\omega^n$ can be verified via \Cref{lem:bded.distortion.3} and \eqref{eqn:good.count-1.1}: when $n\geq\C{n}{1}(\overline{\lambda})$, we have
\begin{align*}
\sphericalangle(\dot\gamma|_y, E^s_{{\omega^n}}(y))\geq& \sphericalangle(\dot\gamma|_x, E^s_{{\omega^n}}(x))-\sphericalangle(\dot\gamma|_x, \dot\gamma|_y)-
\sphericalangle(E^s_{{\omega^n}}(x), E^s_{{\omega^n}}(y))\\
\geq &4e^{-\eta n}-e^{-7n\C{C'}{0}}-\frac{1}{2}e^{-n\C{C'}{0}}\geq 2e^{-\eta n},~\text{ for any }y\in\gamma.
\end{align*}
Here, Item (1) in \Cref{subsection:notation.LY} and \eqref{eqn:good.count-1.1} allows us to apply \Cref{lem:bded.distortion.3} when estimating $\sphericalangle(E^s_{{\omega^n}}(x), E^s_{{\omega^n}}(y))$. This completes the proof of $\cG_{n,\eta,x}\subset \cG_{n,\eta}(\gamma)$. 
\end{proof}

The proof of \Cref{lem:good.count-1} is complete. 
\end{proof}

\begin{dfn}[Good convolution]\label{dfn:good.convolution}
Let $\widetilde{\nu}=\int_\cA \nu_\gamma d\alpha(\gamma)$ be a $(K,L)$-admissible measure on $\AA$, where $\alpha=\Pi_* \widetilde{\nu}$. We define the \emph{$n$-step $\eta$-good convolution of $\nu$ by $\mu$} as
$$\mu^{*n,\good,\eta}*\widetilde{\nu}:=\int_{\cA\times \Sigma_n}\mathbbm{1}_{\cG_{n,\eta}(\gamma)}(\omega^n)\cdot((f^n_{\omega^n})_{\#})_*\nu_\gamma d\mu^n(\omega^n)d\alpha(\gamma).$$
For any $a\in\RR^+$, we let 
$$\mu^{*n,\good,\eta}_{(a)}*\widetilde{\nu}:=\mu^{*n,\good,\eta}*((\Delta_a)_*\widetilde{\nu}).$$
Correspondingly, we write
$$\mu^{*n,\bad,\eta}*\widetilde{\nu}=\mu^{*n}*\widetilde{\nu}-\mu^{*n,\good,\eta}*\widetilde{\nu}\text{ and }\mu^{*n,\bad,\eta}_{(a)}*\widetilde{\nu}=\mu^{*n}*\widetilde{\nu}-\mu^{*n,\good,\eta}_{(a)}*\widetilde{\nu}.$$
\end{dfn}
\begin{rmk}
One can immediately see that for any $\sigma\in\{\text{good},\text{bad}\}$, for any $a\in\RR_+$ and for any $n\in\ZZ_+$, the maps $\widetilde{\nu}\to\mu^{*n,\sigma,\eta}*\widetilde{\nu}$ and $\widetilde{\nu}\to\mu^{*n,\sigma,\eta}_{(a)}*\widetilde{\nu}$ are linear. Since the pushforward of nonnegative measures is nonnegative,  for any $(K,L)$-admissible measure $\widetilde{\nu}$ on $\AA$, we  have 
$$\mu^{*n,\sigma,\eta}*\widetilde{\nu}(\AA)\geq 0\text{ and }\mu^{*n,\sigma,\eta}_{(a)}*\widetilde{\nu}(\AA)\geq 0.$$ 
\end{rmk}

For any $n\in\ZZ_+$, we define 
\begin{align*}
\cW_n:=\{\text{good},\text{bad}\}^n
\end{align*}
One can verify inductively that for any $m,p_0\in\ZZ_+$,  any $(K,L)$-admissible measure $\widetilde{\nu}$ and  any sequence of positive numbers $l_1,\dots, l_m$, we have 
\begin{align}\label{eqn:goodbad.decomp}
\sum_{(\sigma_1,\dots,\sigma_m)\in\cW_m}\mu^{*p_0,\sigma_1,\eta}_{(l_1)}*\mu^{*p_0,\sigma_2,\eta}_{(l_2)}*\dots*\mu^{*p_0,\sigma_m,\eta}_{(l_m)}*\widetilde{\nu}=\mu^{*mp_0}*\widetilde{\nu}.
\end{align}
For any $n\in\ZZ_+$ and for any $\eta\in(0,1)$, we define
\begin{align}\label{eqn:goodbad.eta}
\cW_n(\eta):=\left\{(\sigma_1,\dots,\sigma_n)\in\cW_n\left|\begin{aligned}
& \text{For any }k\in\{1,\dots,n\},\text{ we have }\\
&\#\{1\leq j\leq k|\sigma_j=\text{bad}\}\leq k\eta
\end{aligned}\right.\right\}.
\end{align}

The proof of our main theorem relies on understanding measures of the following form:
\begin{align}\label{eqn:measure.we.care.full}
\sum_{(\sigma_1,\dots,\sigma_m)\in\cW_m(\eta)}\mu^{*p_0,\sigma_1,\eta}_{(l_1)}*\mu^{*p_0,\sigma_2,\eta}_{(l_2)}*\dots*\mu^{*p_0,\sigma_m,\eta}_{(l_m)}*\widetilde{\nu}=\mu^{*mp_0}*\widetilde{\nu}.
\end{align}
In order to do so, we will need to  understand measures of the form
\begin{align}\label{eqn:measure.we.care}
\mu^{*p_0,\sigma_{m-k+1},\eta}_{(l_{m-k+1})}*\dots*\mu^{*p_0,\sigma_m,\eta}_{(l_m)}*\widetilde{\nu}.
\end{align}
The rest of this section verifies some quatitative properties of the above as an admissible measure. To achieve this, we first introduce some notations.
\begin{dfn}[Good/Bad pushforward, long/short parts]\label{dfn:good/bad.and.refined/unrefined}
Let $\widetilde{\nu}=\int_\cA \nu_\gamma d\alpha(\gamma)$ be a $(K,L)$-admissible measure on $\AA$, where $\alpha=\Pi_* \widetilde{\nu}$. Let $\omega^n\in\Sigma_n$ and $a\in\RR_{\geq 0}$. We define the following measures:
\begin{align*}
\widetilde{\nu}^{a,\llong}:=&\int_{\cA}\mathbbm{1}_{\{\gamma:m_{\gamma}(\gamma)\geq a\}}(\gamma)\nu_\gamma d\alpha(\gamma),~ \widetilde{\nu}^{a,\sshort}:=\widetilde{\nu}-\widetilde{\nu}^{a,\llong},\\
 (f^{n}_{\omega^n})^{\good,\eta}_{*}\widetilde{\nu}:=&\int_{\cA}\mathbbm{1}_{\{\gamma:\omega^n\in\cG_{n,\eta}(\gamma)\}}(\gamma)((f^{n}_{\omega^n})_{\#})_*\nu_\gamma d\alpha(\gamma),\\
 (f^{n}_{\omega^n})^{\bad,\eta}_{*}\widetilde{\nu}:=&((f^{n}_{\omega^n})_{\#})_*\widetilde{\nu}-(f^{n}_{\omega^n})^{\good,\eta}_{*}\widetilde{\nu},\\
(f^{n}_{\omega^n})^{\sigma,\eta}_{*,a,r}\widetilde{\nu}:=& (f^{n}_{\omega^n})^{\sigma,\eta}_{*}((\Delta_a)_*(\widetilde{\nu}^{a,r})),~\sigma\in\{\good,\bad\}, r\in\{\llong,\sshort\}.
\end{align*}
\end{dfn}
In other words, $\widetilde{\nu}^{a,\llong}$ is obtained by discarding all $\nu_\gamma$ which is supported on curves of length less than $a$ from $\widetilde{\nu}$. In particular, $(\Delta_a)_*\widetilde{\nu}^{a,\sshort}=\widetilde{\nu}^{a,\sshort}$. $(f^{n}_{\omega^n})^{\good,\eta}_{*}\widetilde{\nu}$ is obtained by discarding all $((f^{n}_{\omega^n})_{\#})_*\nu_\gamma$ such that $\omega^n$ is not $\eta$-good for $\gamma$ in the sense of \Cref{dfn:good.word} from $((f^{n}_{\omega^n})_{\#})_*\widetilde{\nu}$.

We would like to first show that if $\widetilde{\nu}$ is $(K,L)$-admissible, then the measures of the form \eqref{eqn:measure.we.care} are $(K',L')$-admissible with $K'/K,L'/L$ not too large. Moreover, if $\widetilde{\nu}$ is supported on curves that are not too short, then the measures of the form \eqref{eqn:measure.we.care} are also supported on curves that are not too short. The precise statement is in \Cref{cor:curv.length.est}, which is a direct consequence of \Cref{lem:quant.adm.est}. Before that, we need some technical preparations. \Cref{lem:admissibility.of.parts} and \Cref{lem:BUEF.extension} are used to prove that the admissibility constants of \eqref{eqn:measure.we.care} do not explode. \Cref{lem:curv.length.est} is used to show that \eqref{eqn:measure.we.care} is supported on curves that are not too short.

\begin{lem}\label{lem:admissibility.of.parts}
Let $\widetilde{\nu}=\int_\cA \nu_\gamma d\alpha(\gamma)$ be a $(K,L)$-admissible measure on $\AA$, where $\alpha=\Pi_* \widetilde{\nu}$. For any $\omega^n\in\Sigma_n$,  any $\sigma\in\{\good,\bad\}$,  any $r\in\{\llong,\sshort\}$ and  $a,b\in\RR_+$, there exist $(K,L)$-admissible measures $\widetilde{\nu}'$ and $\widetilde{\nu}''$ such that
$$(f^n_{\omega^n})^{\sigma,\eta}_{*,a,r}\widetilde{\nu}=((f^n_{\omega^n})_{\#})_*\widetilde{\nu}' \text{ and } (\Delta_b)_*\left((f^n_{\omega^n})^{\sigma,\eta}_{*,a,r}\widetilde{\nu}\right)=((f^n_{\omega^n})_{\#})_*\widetilde{\nu}''.$$
Moreover, for $\Pi_*\widetilde{\nu}'$-a.e $\gamma$, or for $\Pi_*\widetilde{\nu}''$-a.e $\gamma$, respectively, we have $\omega^n\in\cG_{n,\eta}(\gamma)$. 
\end{lem}
\begin{proof}
It is clear from \Cref{dfn:good/bad.and.refined/unrefined} that $\widetilde{\nu}^{a,r}$ is $(K,L)$-admissible, Hence, by the first assertion in \Cref{lem:refinement.lemma}, $(\Delta_a)_*\widetilde{\nu}^{a,r}$ is $(K,L)$-admissible. By \Cref{dfn:good/bad.and.refined/unrefined} again, we have 
\begin{align}\label{eqn:admissibility.of.parts.1}
(\Delta_a)_*\widetilde{\nu}^{a,r}=(((f^n_{\omega^n})^{-1})_{\#})_*((f^n_{\omega^n})^{\good}_{\#,a,r}\widetilde{\nu})+(((f^n_{\omega^n})^{-1})_{\#})_*((f^n_{\omega^n})^{\bad}_{\#,a,r}\widetilde{\nu}).\end{align}
Notice that
$$\{\gamma\in\cA:\omega^n\in\cG_{n,\eta}(\gamma)\}=\cA\setminus\{\gamma\in\cA:\omega^n\not\in\cG_{n,\eta}(\gamma)\},$$
$$(((f^n_{\omega^n})^{-1})_{\#})_*((f^n_{\omega^n})^{\bad}_{\#,a,r}\widetilde{\nu})(\{\gamma\in\cA:\omega^n\in\cG_{n,\eta}(\gamma)\})=0$$
and
$$(((f^n_{\omega^n})^{-1})_{\#})_*((f^n_{\omega^n})^{\good}_{\#,a,r}\widetilde{\nu})(\{\gamma\in\cA:\omega^n\not\in\cG_{n,\eta}(\gamma)\})=0.$$
By the above and the fact that the left hand side \eqref{eqn:admissibility.of.parts.1} is $(K,L)$-admissible, the measures $(((f^n_{\omega^n})^{-1})_{\#})_*((f^n_{\omega^n})^{\good}_{\#,a,r}\widetilde{\nu})$ and $(((f^n_{\omega^n})^{-1})_{\#})_*((f^n_{\omega^n})^{\bad}_{\#,a,r}\widetilde{\nu})$ are also $(K,L)$-admissible (whenever they are non-zero). In particular, $\widetilde{\nu}'$ is $(K,L)$-admissible. Notice that by \eqref{eqn:functorial}, we have 
$$\widetilde{\nu}''= \left(((f^n_{\omega^n})^{-1})_{\#}\circ\Delta_b\circ(f^n_{\omega^n})_{\#}\right)_*\widetilde{\nu}'.$$
Therefore, by the first assertion in \Cref{lem:refinement.lemma}, $\widetilde{\nu}''$ is also $(K,L)$-admissible. The rest follows directly from \Cref{dfn:good.word} and \Cref{dfn:good/bad.and.refined/unrefined}.
\end{proof}
\begin{lem}\label{lem:curv.length.est}
Let $\gamma:[0,a]\to\TT^2$ be a $C^2$ curve with $m_\gamma(\gamma)=a$. Let $p_0,m\in\ZZ_+$ and $n=mp_0$.  Fix some $\omega^n=(\omega^{p_0}_1,\dots,\omega^{p_0}_m)\in\Sigma_n$ and write $\gamma_k:=f^k_{\omega^n}\circ\gamma=f_{k-1}\circ\dots\circ f_0\circ \gamma$.  For simplicity, we write $\gamma_0=\gamma$.

Assume, in addition, that there exist $k\in\{1,\dots,m\}$ and $\eta\in(0,1)$ such that 
$$\left|\left\{j\in\{1,\dots,k\}:\omega^{p_0}_j\not\in\cG_{p_0,\eta}(\gamma_{(j-1)p_0})\right\}\right|\leq m\eta.$$
Then, 
$$m_{\gamma_{kp_0}}(\gamma_{kp_0})\geq ae^{-\C{C'}{0}p_0m\eta+(k-m\eta)p_0\overline{\lambda}}\geq ae^{kp_0\overline{\lambda}-2\C{C'}{0}mp_0\eta}.$$
\end{lem}
\begin{proof}
This lemma follows immediately from \eqref{eqn:unif.C2.norm} and \Cref{dfn:good.word}.
\end{proof}

Recall that $\C{C}{6}$ and $\C{n}{1}(\overline{\lambda})$ were given by \Cref{lem:good.count-1}. Also, $\C{\beta}{1}$ is given by \Cref{lem.cone.trans}.  We have the following lemma. 

\begin{lem}\label{lem:BUEF.extension}
Let $\gamma:[0,a]\to \TT^2$ be a $C^2$ curve. Let $p_0,m\in\ZZ_+$ and $k\in\{0,\dots, m\}$. Let $f_0,\dots,f_{kp_0-1}\in\cU$. For any $j\in\{1,\dots, k\}$, we write $\omega_j^{p_0}=(f_{(j-1)p_0},\dots,f_{jp_0-1})$ and $\gamma_{jp_0}:=f_{jp_0-1}\circ\dots\circ f_0\circ\gamma=f^{p_0}_{\omega_j^{p_0}}\circ\dots\circ f^{p_0}_{\omega_1^{p_0}}\circ\gamma$. For simplicity, we write $\gamma_0:=\gamma$, $\omega^{kp_0}=(f_0,\dots, f_{kp_0-1})$ and $n=mp_0$.

Let $0<\eta<\min\{(\lambda-\overline{\lambda})/\C{\beta}{1},1)\}$. Assume, in addition, that 
\begin{align}\label{lem:BUEF.extension.assump}
p_0\geq\CS{n'}{1}(\eta,\overline{\lambda}):=\max\left\{\C{n}{1}(\overline{\lambda}),\frac{\ln(2\C{C}{6})}{\C{\beta}{1}\eta}\right\},
\end{align}
and that for any $i\in\{0,\dots,k-1\}$, we have 
$$\#\left\{j\in\{i+1,\dots,k\}\left|\exists t\in[0,a]\text{ s.t. } \left|\ln \frac{\Jac f^{jp_0}_{\omega^{kp_0}}(\gamma(t))}{\Jac f^{(j-1)p_0}_{\omega^{kp_0}}(\gamma(t))}\right|>{\C{C}{0}+2\C{\epsilon}{0}p_0}\right.\right\}<(m-i)\eta$$
and
$$\#\left\{j\in\{i+1,\dots,k\}\left|\exists t\in[0,a]\text{ s.t. } \frac{\|\dot\gamma_{jp_0}(t)\|}{\|\dot\gamma_{(j-1)p_0}(t)\|}<e^{\overline{\lambda}p_0}\right.\right\}<(m-i)\eta.$$
Then for any $x\in\gamma$ and for any subsegment $\gamma_x$ of $\gamma$ containing $x$ with length at most $(4K)^{-1}e^{-14n\C{C'}{0}}$, there exist diffeomorphisms $f_{kp_0},\dots,f_{mp_0-1}\in\cU$, such that $\omega^n=(f_0,\dots, f_{kp_0-1},f_{kp_0},\dots,f_{mp_0-1})$ has $(p_0;\eta)$-NCT and $(p_0;\overline{\lambda},\eta)$-ET along $\gamma_x$.
\end{lem}
\begin{proof}
Let $K>1$ be a constant such that $|\curv(\cdot;\gamma)|\leq K$. Then by \Cref{lem:basic.curv.est}, \eqref{eqn:unif.C2.norm} and the assumption that $\C{C'}{0}>1$, for any $f_0',\dots, f_{n-1}'\in\cU$ and for any $i\in\{1,\dots, n\}$, we have 
\begin{align}\label{eqn:BUEF.extension-1}
|\curv(t;f_{i-1}'\circ\dots\circ f_0'\circ\gamma)|\leq &e^{5\C{C'}{0}}(1+|\curv(t,f_{i-2}'\circ\dots\circ f_1'\circ\gamma)|)\nonumber\\
\leq& e^{10\C{C'}{0}}(2+|\curv(t,f_{i-3}'\circ\dots\circ f_1'\circ\gamma)|)\nonumber\\
\leq&\dots\leq e^{5i\C{C'}{0}}(i+|\curv(t,\gamma)|)<2e^{6i\C{C'}{0}}K.
\end{align}
Let $\gamma_x$ be an arbitrary subsegment of $\gamma$ containing $x$ with length less than $(4K)^{-1}e^{-14n\C{C'}{0}}$. By \Cref{dfn:good.word}, it suffices to show that one can find 
\[\omega^{p_0}_{k+1}=(f_{kp_0},\dots, f_{(k+1)p_0-1}),\dots,\omega^{p_0}_{m}=(f_{(m-1)p_0},\dots, f_{mp_0-1})\in\Sigma_{p_0,\text{NC}}\] such that for any $j\in\{k,\dots,m-1\}$, we have
\begin{align*}
\omega^{p_0}_{j+1}\in\cG_{p_0,\eta}(f_{jp_0-1}\circ \dots\circ f_{0}\circ \gamma).
\end{align*}
It is then enough to show that for any $j\in\{k,\dots,m-1\}$ and any given $f_{kp_0},\dots,f_{jp_0-1}$, we have 
\begin{align}\label{eqn:BUEF.extension-2}
\cG_{p_0,\eta}(f_{jp_0-1}\circ \dots\circ f_{0}\circ \gamma)\neq \emptyset.
\end{align}
Assume that we have constructed $f_{kp_0},\dots,f_{ip_0-1}$ such that $\eqref{eqn:BUEF.extension-2}$ holds for any $j\in\{k,\dots,i-1\}$. Then by \eqref{eqn:BUEF.extension-1} and the fact that $m_{\gamma_x}(\gamma_x)\leq (4K)^{-1}e^{-14n\C{C'}{0}}$, we have 
$$|\curv(\cdot; f_{ip_0-1}\circ \dots\circ f_{0}\circ \gamma_x)|\leq 2e^{6ip_0\C{C'}{0}} K $$
and
$$\text{length}( f_{ip_0-1}\circ \dots\circ f_{0}\circ \gamma_x)\leq e^{ip_0\C{C'}{0}}\cdot\frac{1}{4K}e^{-14n\C{C'}{0}}\leq\frac{1}{4e^{6ip_0\C{C'}{0}} K}e^{-7n\C{C'}{0}}.$$
By \Cref{lem:good.count-1} and \eqref{lem:BUEF.extension.assump}, we have 
$$\mu^{p_0}(\cG_{p_0,\eta}(f_{ip_0-1}\circ \dots\circ f_{0}\circ \gamma_x))\geq 1-\C{C}{6}e^{-\C{\beta}{1}\eta p_0}\geq\frac{1}{2}.$$
In particular, $\cG_{p_0,\eta}(f_{ip_0-1}\circ \dots\circ f_{0}\circ \gamma_x)\neq \emptyset$. This verifies \eqref{eqn:BUEF.extension-2} when $j=i$, and hence finishes the proof.
\end{proof}
\begin{rmk}
One can prove a stronger version of \Cref{lem:BUEF.extension} with $\C{n'}{1}$ not depending on $\eta$. To be specific, when we reduce the lemma to proving \eqref{eqn:BUEF.extension-2}, the second item in \Cref{dfn:good.word} is unnecessary for this lemma. If we remove the second item in \Cref{dfn:good.word}, one can prove a similar version of \Cref{lem:good.count-1} which does not depend on $\eta$.
\end{rmk}

Recall that $\C{K}{1}(p_0;\overline{\lambda})$ is obtained in \Cref{lem:curv.est}, and that $\C{K}{2}(p_0;\overline{\lambda})$ is obtained in \Cref{lem:density.est}.

\begin{lem}\label{lem:quant.adm.est}
Let $\widetilde{\nu}=\int_\cA \nu_\gamma d\alpha(\gamma)$ be a $(K,L)$-admissible measure on $\AA$, where $\alpha=\Pi_* \widetilde{\nu}$. Let $0<\eta\leq \min\left\{1/\C{\beta}{1},(\lambda-\overline{\lambda})/\C{\beta}{1},1/3\right\}.$ Let $p_0,m\in\ZZ_+$ and $k\in\{0,\dots, m\}$. Let $\omega^{kp_0}:=(f_0,\dots,f_{kp_0-1})\in\Sigma_{kp_0}$. For any $j\in\{1,\dots, k\}$, we write \[\omega_j^{p_0}=(f_{(j-1)p_0},\dots,f_{jp_0-1}).\]

Assume, in addition, that $p_0\geq \max\{\C{n'}{1}(\eta,\overline{\lambda}),\C{C}{0}/(\C{C'}{0}\eta)\}$. Then, for any $(\sigma_1,\dots,\sigma_m)\in\cW_m(\eta)$,  any $(r_1,\dots,r_m)\in\{\llong,\sshort\}^m$,  any $l_1,\dots,l_m\in\RR_+$ and  $k\in\{1,\dots,m\}$, the following holds:
\begin{enumerate}
\item The measure 
\begin{align}\label{eqn:part.in.middle.ingredient}
\widetilde{\nu}_{(k)}:=(f^{p_0}_{\omega^{p_0}_k})^{\sigma_{m-k+1},\eta}_{*,l_{m-k+1},r_k}\left(\dots\left((f^{p_0}_{\omega^{p_0}_1})^{\sigma_m,\eta}_{*,l_m,r_1}\widetilde{\nu}\right)\dots\right)
\end{align}
is $(K',L')$-admissible, where 
$$K'=\C{K}{1}(p_0;\overline{\lambda})(K+1)e^{8(m-k)p_0\eta\C{C'}{0}}$$
and
$$L'=\C{K}{2}(p_0;\overline{\lambda})(L+1)(K+1)e^{11(m-k)p_0\eta\C{C'}{0}}.$$
\item If $r_1=\dots=r_k=\sshort$ and $\widetilde{\nu}$ is supported on curves of length at least $l$, then $\widetilde{\nu}_{(k)}$
is supported on curves of length at least $le^{kp_0\overline{\lambda}-2\C{C'}{0}mp_0\eta}$.

Otherwise, if $r_{j+1}=\dots=r_k=\sshort$ and $r_{j}=\llong$ for some $j\in\{1,\dots,k\}$, then $\widetilde{\nu}_{(k)}$ is supported on curves of length at least $l_{(m-j+1)}e^{(k-j)p_0\overline{\lambda}-2\C{C'}{0}(m-j)p_0\eta}$.
\end{enumerate}
\end{lem}
\begin{proof}
\begin{enumerate}
\item Let $a=(8K)^{-1}e^{-15n\C{C'}{0}}$. Apply \Cref{dfn:good/bad.and.refined/unrefined} and \Cref{lem:admissibility.of.parts} iteratively to \eqref{eqn:part.in.middle.ingredient}, we have 
\begin{align}\label{eqn:lem:quant.adm.est.1-1}
\widetilde{\nu}':=\left(((f^{p_0}_{\omega^{p_0}_k}\circ\dots\circ f^{p_0}_{\omega^{p_0}_1})^{-1})_{\#}\circ\Delta_a\right)_*\widetilde{\nu}_{(k)}
\end{align}
is $(K,L)$-admissible. Moreover, by the fact that $(\sigma_1,\dots,\sigma_m)\in\cW_m(\eta)$ and \eqref{eqn:goodbad.eta}, for $\Pi_*\widetilde{\nu}'$-a.e. $\gamma$ and for any $i\in\{0,\dots,k-1\}$, the measure
$$\#\left\{j\in\{i+1,\dots,k\}\left|\exists t\in[0,a]\text{ s.t. } \left|\ln\frac{\Jac f^{jp_0}_{\omega^{kp_0}}(\gamma(t))}{\Jac f^{(j-1)p_0}_{\omega^{kp_0}}(\gamma(t))}\right|>{\C{C}{0}+2\C{\epsilon}{0}p_0}\right.\right\}<(m-i)\eta$$
and
$$\#\left\{j\in\{i+1,\dots,k\}\left|\exists t\in[0,a]\text{ s.t. } \frac{\|\dot\gamma_{jp_0}(t)\|}{\|\dot\gamma_{(j-1)p_0}(t)\|}<e^{\overline{\lambda}p_0}\right.\right\}<(m-i)\eta,$$
where $\gamma_{jp_0}:=f_{jp_0-1}\circ\dots\circ f_{0}\circ\gamma$. Recall that $a=(8K)^{-1}e^{-15n\C{C'}{0}}$. By \eqref{eqn:unif.C2.norm} and the second assertion in \Cref{lem:refinement.lemma}, $\widetilde{\nu}$ is supported on curves of length at most $(4K)^{-1}e^{-14n\C{C'}{0}}$. Hence by \Cref{lem:BUEF.extension}, for $\Pi_*\widetilde{\nu}$-a.e. $\gamma$, there exists $f_{kp_0},\dots, f_{n-1}\in\cU$ such that the word $\omega^n:=(f_0,\dots ,f_{n-1})$ has $(p_0;\eta)$-NCT and $(p_0;\overline{\lambda},\eta)$-ET along $\gamma$. 

Write $\widetilde{\nu}'=\int_{\cA}\nu'_{\gamma}d\alpha'(\gamma)$, where $\alpha':=\Pi_*\widetilde{\nu}'$. Then \Cref{cor:BUEF.adm} and the above imply that for $\alpha'$-a.e $\gamma$, $(f^{kp_0}_{\omega^n})_{\#})_*\nu'_\gamma$ is $(K',L')$-admissible. (See \eqref{lem:BUEF.extension}, item (1) in \Cref{subsection:notation.LY} and the assumption on $p_0,\eta$ for why we can apply \Cref{cor:BUEF.adm}.) Therefore, by \eqref{eqn:functorial}, \eqref{eqn:lem:quant.adm.est.1-1} and the first assertion in \Cref{lem:refinement.lemma}, $\widetilde{\nu}_{(k)}$ is $(K',L')$-admissible on $\AA$.
\item By \eqref{eqn:cutshort.map} and \Cref{dfn:good/bad.and.refined/unrefined}, for any admissible measure $\widetilde{\nu}'$ and for any $a\in\RR_+$, the following holds:
\begin{itemize}
\item $(\Delta_a)_*(\widetilde{\nu}')^{a,\good}$ is supported on curves of length at least $a$.
\item $(\Delta_a)_*(\widetilde{\nu}')^{a,\bad}=(\widetilde{\nu}')^{a,\bad}$.
\end{itemize}
The second assertion in \Cref{lem:quant.adm.est} then follows directly from \Cref{lem:curv.length.est} and the above.
\end{enumerate}
\end{proof}

As a direct consequence of \Cref{lem:quant.adm.est}, we have the following result. 

\begin{cor}\label{cor:curv.length.est}
Let $\widetilde{\nu}=\int_{\cA}\nu_\gamma d\alpha(\gamma)$ be a $(K,L)$-admissible measure on $\AA$ supported on curves of length at least $l>0$, where $\alpha=\Pi_*\widetilde{\nu}$. Let $0<\eta\leq \min\left\{1/\C{\beta}{1},(\lambda-\overline{\lambda})/\C{\beta}{1},1/3\right\}.$ Let $n=mp_0$, where $m,p_0\in\ZZ_+$. Let $l_1,\dots,l_m>0$ be a sequence of positive numbers. 

Assume in addition that $p_0\geq \max\{\C{n'}{1}(\eta,\overline{\lambda}),\C{C}{0}/(\C{C'}{0}\eta)\}$. Then for any $k\in\{1,\dots,m\}$, the measure
$$\mu^{*p_0,\sigma_{m-k+1},\eta}_{(l_{m-k+1})}*\dots*\mu^{*p_0,\sigma_m,\eta}_{(l_m)}*\widetilde{\nu}$$
is a $(K',L')$-admissible measure supported on curves of length at least 
$$\min\left\{le^{kp_0\overline{\lambda}-2\C{C'}{0}mp_0\eta},\min_{1\leq j\leq k}\left\{l_{(m-j+1)}e^{(k-j)p_0\overline{\lambda}-2\C{C'}{0}(m-j)p_0\eta}\right\}\right\},$$
where
$$K'=\C{K}{1}(p_0;\overline{\lambda})(K+1)e^{8(m-k)p_0\eta\C{C'}{0}}$$
and
$$L'=\C{K}{2}(p_0;\overline{\lambda})(L+1)(K+1)e^{11(m-k)p_0\eta\C{C'}{0}}.$$
\end{cor}
\begin{proof}
By \Cref{dfn:good.convolution} and \Cref{dfn:good/bad.and.refined/unrefined}, for any $\sigma\in\{\good,\bad\}$ and for any $a\in\RR_+$, we have 
$$\mu^{*p_0,\sigma,\eta}_{(a)}*\widetilde{\nu}=\int_{\Sigma_{p_0}} (f^{p_0}_{\omega^{p_0}})^{\sigma,\eta}_{*,a,\llong}\widetilde{\nu}d\mu^{p_0}(\omega^{p_0})+\int_{\Sigma_{p_0}} (f^{p_0}_{\omega^{p_0}})^{\sigma,\eta}_{*,a,\sshort}\widetilde{\nu}d\mu^{p_0}(\omega^{p_0}).$$
Hence,
\begin{align*}
&\mu^{*p_0,\sigma_{m-k+1},\eta}_{(l_{m-k+1})}*\dots*\mu^{*p_0,\sigma_m,\eta}_{(l_m)}*\widetilde{\nu}\\
=&\sum_{(r_1,\dots, r_k)\in\{\llong,\sshort\}^k}\int_{\Sigma_{p_0}}\dots\int_{\Sigma_{p_0}}(f^{p_0}_{\omega^{p_0}_k})^{\sigma_{m-k+1},\eta}_{*,l_{m-k+1},r_k}\left(\dots\left((f^{p_0}_{\omega^{p_0}_1})^{\sigma_m,\eta}_{*,l_m,r_1}\widetilde{\nu}\right)\dots\right)d\mu^{p_0}(\omega^{p_0}_1)\dots d\mu^{p_0}(\omega^{p_0}_k).
\end{align*}
\Cref{cor:curv.length.est} then directly follows from \Cref{lem:quant.adm.est} and the above.
\end{proof}

We end this section by showing that the total variation of the measure in \eqref{eqn:measure.we.care.full} can be  arbitrarily close to the total variation of $\widetilde{\nu}$. The precise statement can be found in \Cref{lem:good.count-2}. Before proving \Cref{lem:good.count-2}, we first prove \Cref{lem:bad.convol.small}, which is a straightforward consequence of \Cref{lem:good.count-1}.
\begin{lem}\label{lem:bad.convol.small}
Let $\widetilde{\nu}=\int_\cA \nu_\gamma d\alpha(\gamma)$ be a $(K,L)$-admissible measure on $\AA$ and $0<\eta\leq \min\{(\lambda-\overline{\lambda})/\C{\beta}{1},1\}$. Then for any $n\geq \C{n}{1}(\overline{\lambda})$ and for any $0<a\leq \frac{1}{4K}e^{-7n\C{C'}{0}}$, we have 
$$\frac{\mu^{n,\bad,\eta}_{(a)}*\widetilde{\nu}(\AA)}{\widetilde{\nu}(\AA)}\leq \C{C}{6}e^{-\C{\beta}{1}\eta n}.$$
\end{lem}
\begin{proof}
Without loss of generality, we assume that $\widetilde{\nu}$ is supported on curves of length at most $2a$. One only needs to replace $\widetilde{\nu}$ by $(\Delta_a)_*\widetilde{\nu}$.  In particular, $\mu^{n,\bad,\eta}_{(a)}*\widetilde{\nu}=\mu^{n,\bad,\eta}*\widetilde{\nu}$.
Since $2a< \frac{1}{2K}e^{-7n\C{C'}{0}}$, by \Cref{lem:good.count-1} and \Cref{dfn:good.convolution}, we have 
\begin{align*}
\mu^{n,\bad,\eta}*\widetilde{\nu}(\AA)=&\int_{\cA\times \Sigma_n}(1-\mathbbm{1}_{\cG_{n,\eta}(\gamma)}(\omega^n))\cdot((f^n_{\omega^n})_{\#})_*\nu_\gamma(\AA) d\mu^n(\omega^n)d\alpha(\gamma)\\
=&\int_{\cA\times \Sigma_n}(1-\mathbbm{1}_{\cG_{n,\eta}(\gamma)}(\omega^n))\cdot\nu_\gamma(\AA) d\mu^n(\omega^n)d\alpha(\gamma)\\
=&\int_{\cA}(1-\mu^n(\cG_{n,\eta}(\gamma)))\nu_\gamma(\AA) d\alpha(\gamma)\\
\leq&\int_{\cA}\C{C}{6}e^{-\C{\beta}{1}\eta n}\nu_\gamma(\AA) d\alpha(\gamma)=\C{C}{6}e^{-\C{\beta}{1}\eta n}\widetilde{\nu}(\AA).\qedhere
\end{align*}
\end{proof}
\begin{lem}\label{lem:good.count-2}
Let $\widetilde{\nu}=\int_\cA \nu_\gamma d\alpha(\gamma)$ be a $(K,L)$-admissible measure on $\AA$, where $\alpha=\Pi_* \widetilde{\nu}$. Let $0<\eta\leq \min\left\{1/\C{\beta}{1},(\lambda-\overline{\lambda})/\C{\beta}{1},1/3\right\}.$ Let $n=mp_0$, where $m,p_0\in\ZZ_+$. Let $l_1, l_2,\dots, l_m>0$ be a sequence of positive numbers such that
$$l_{k}\leq (4\C{K}{1}(p_0;\overline{\lambda})(K+1)e^{8kp_0\eta\C{C'}{0}})^{-1}e^{-7n\C{C'}{0}} \text{ for any }k\in\{1,\dots,m\}.$$
For any $c\in(0,1)$, assume in addition that $p_0\geq \CS{p}{*}(\eta,\overline{\lambda},c)$ for some $\C{p}{*}=\C{p}{*}(\eta,\overline{\lambda},c)\geq  \max\{\C{n'}{1}(\eta,\overline{\lambda}),\C{C}{0}/(\C{C'}{0}\eta)\}$ such that
$$ \beta_1\eta \C{p}{*}(\eta,\overline{\lambda},c)-\ln(\C{C}{6})\geq \C{a}{*}(\eta,c)-\frac{1}{\eta}\ln(\eta^\eta(1-\eta)^{1-\eta}),$$
where $\CS{a}{*}(\eta,c)>0$ is a constant such that
$$\frac{\C{C''}{4}e^{-\C{a}{*}(\eta,c)\eta}}{(1-e^{-\C{a}{*}(\eta,c)})(1-e^{-\C{a}{*}(\eta,c)\eta})}<c.$$
Then, we have
$$\sum_{(\sigma_1,\dots,\sigma_m)\in\cW_m(\eta)}\mu^{*p_0,\sigma_1,\eta}_{(l_1)}*\mu^{*p_0,\sigma_2,\eta}_{(l_2)}*\dots*\mu^{*p_0,\sigma_m,\eta}_{(l_m)}*\widetilde{\nu}(\AA)> (1-c)\widetilde{\nu}(\AA).$$
\end{lem}
\begin{proof}
For any $k\in\{1,\dots,m\}$, we define the projection map $\mathrm{pr}^m_k:
\cW_m\to\cW_k$ as 
$$\mathrm{pr}^m_k(\sigma_1,\dots,\sigma_m)=(\sigma_{m-k+1},\dots, \sigma_m).$$
Let $\cW^m_k(\eta):=\mathrm{pr}^m_k(\cW_m(\eta))$. For any $k\in\{1,\dots,m\}$ and any $(\sigma_{m-k+1},\dots, \sigma_m)\in \cW_k$, we define $f_k:\cW_k\to\RR_{\geq 0}$ as follows: when $k=1$, we define
\begin{align}\label{eqn:percentage}
f_1(\sigma_m)=\frac{\mu^{*p_0,\sigma_m,\eta}_{(l_m)}*\widetilde{\nu}(\AA)}{\widetilde{\nu}(\AA)}. 
\end{align}
When $k\in\{2,\dots,m\}$, we define
\begin{align*}
&f_k(\sigma_{m-k+1};\sigma_{m-k+2},\dots, \sigma_m)\\
=&\begin{cases}
\displaystyle \frac{\mu^{*p_0,\sigma_{m-k+1},\eta}_{(l_{m-k+1})}*\dots*\mu^{*p_0,\sigma_m,\eta}_{(l_m)}*\widetilde{\nu}(\AA)}{\mu^{*p_0,\sigma_{m-k+2},\eta}_{(l_{m-k+2})}*\dots*\mu^{*p_0,\sigma_m,\eta}_{(l_m)}*\widetilde{\nu}(\AA)},~&\text{if well-defined and }(\sigma_{m-k+2},\dots, \sigma_m)\in\cW_{k-1}^m(\eta), \\
\displaystyle 1-\C{C}{6}e^{-\C{\beta}{1}\eta p_0}, &\text{if not the first case and }\sigma_{m-k+1}=\good, \\
\displaystyle \C{C}{6}e^{-\C{\beta}{1}\eta p_0}, &\text{if not the first case and }\sigma_{m-k+1}=\bad.
\end{cases}
\end{align*}
By the assumption on $\eta$, $f_1,\dots, f_m$ are non-negative functions on their domain. 

Using the above notations, we have
\begin{align}\label{eqn:good.count-2-1}
\begin{split}
&\frac{1}{\widetilde{\nu}(\AA)}\sum_{(\sigma_1,\dots,\sigma_m)\in\cW_m(\eta)}\mu^{*p_0,\sigma_1,\eta}_{(l_1)}*\mu^{*p_0,\sigma_2,\eta}_{(l_2)}*\dots*\mu^{*p_0,\sigma_m,\eta}_{(l_m)}*\widetilde{\nu}(\AA)\\
=&\sum_{(\sigma_1,\dots,\sigma_m)\in\cW_m(\eta)}\prod_{k=1}^mf_k(\sigma_{m-k+1};\sigma_{m-k+2},\dots, \sigma_m).
\end{split}
\end{align}
Notice that for any $k\in\{1,\dots,m\}$ and for any $(\sigma_1,\dots,\sigma_m)\in\cW_m$, Definition \ref{dfn:good.convolution} implies that 
$$f_k(\good;\sigma_{m-k+2},\dots, \sigma_m)+f_k(\bad;\sigma_{m-k+2},\dots, \sigma_m)=1.$$
Hence for any $k\in\{0,\dots,m-1\}$ and any $\sigma_{m-k+1},\dots,\sigma_m\in\{\good,\bad\}$
\begin{align}\label{eqn:good.count-2-2}
&\sum_{(\sigma_{1},\dots,\sigma_{m-k})\in\cW_{m-k}}\prod_{i=k+1}^{m}f_i(\sigma_{m-i+1};\sigma_{m-i+2},\dots, \sigma_m)\nonumber\\
=&\sum_{(\sigma_{2},\dots,\sigma_{m-k})\in\cW_{m-k-1}}\prod_{i=k+1}^{m-1}f_i(\sigma_{m-i+1};\sigma_{m-i+2},\dots, \sigma_m)\nonumber\\
=&\dots=\sum_{(\sigma_{m-k-1},\sigma_{m-k})\in\cW_{2}}\prod_{i=k+1}^{k+2}f_i(\sigma_{m-i+1};\sigma_{m-i+2},\dots, \sigma_m)\nonumber\\
=&\sum_{\sigma_{m-k}\in\{\good,\bad\}}f_{k+1}(\sigma_{m-k};\sigma_{m-k+1},\dots, \sigma_m)
=1.
\end{align}
It follows from  \eqref{eqn:good.count-2-1} and \eqref{eqn:good.count-2-2} that
\begin{align}\label{eqn:good.count-2-3}
&\frac{1}{\widetilde{\nu}(\AA)}\sum_{(\sigma_1,\dots,\sigma_m)\in\cW_m(\eta)}\mu^{*p_0,\sigma_1,\eta}_{(l_1)}*\mu^{*p_0,\sigma_2,\eta}_{(l_2)}*\dots*\mu^{*p_0,\sigma_m,\eta}_{(l_m)}*\widetilde{\nu}(\AA)\nonumber\\
=&1-\left(\sum_{(\sigma_1,\dots,\sigma_m)\in\cW_m\setminus\cW_m(\eta)}\prod_{k=1}^mf_k(\sigma_{m-k+1};\sigma_{m-k+2},\dots, \sigma_m)\right).
\end{align}
For any $k\in\{1,\dots,m\}$, let
\begin{align}\label{eqn:bad.string.decomp.parts}
\cW_{m;k}(\eta)=\left\{(\sigma_1,\dots,\sigma_m)\in\cW_m:\#\{j\in\{1,\dots,k\}:\sigma_j=\bad\}> k\eta\right\}
\end{align}
By \eqref{eqn:goodbad.eta}, we have 
\begin{align}\label{eqn:bad.string.decomp}
\cW_m\setminus\cW_m(\eta)\subset \bigcup_{k=1}^m\cW_{m;k}(\eta).
\end{align}
Apply \eqref{eqn:bad.string.decomp} to \eqref{eqn:good.count-2-3}, we have 
\begin{align}\label{eqn:good.count-2-4}
&\frac{1}{\widetilde{\nu}(\AA)}\sum_{(\sigma_1,\dots,\sigma_m)\in\cW_m(\eta)}\mu^{*p_0,\sigma_1,\eta}_{(l_1)}*\mu^{*p_0,\sigma_2,\eta}_{(l_2)}*\dots*\mu^{*p_0,\sigma_m,\eta}_{(l_m)}*\widetilde{\nu}(\AA)\nonumber\\
\geq& 1-\sum_{k=1}^m\left(\sum_{(\sigma_1,\dots,\sigma_m)\in\cW_{m;k}(\eta)}\prod_{j=1}^mf_j(\sigma_{m-j+1};\sigma_{m-j+2},\dots, \sigma_m)\right).
\end{align}
To estimate the right-hand-side of \eqref{eqn:good.count-2-4}, we notice that by \Cref{cor:curv.length.est} and the assumption on $p_0$ and $\eta$, for any $k\in\{2,\dots,m\}$, the measure $\mu^{*p_0,\sigma_{m-k+2},\eta}_{(l_{m-k+2})}*\dots*\mu^{*p_0,\sigma_m,\eta}_{(l_m)}*\widetilde{\nu}$ is $(K(k),L(k))$-admissible, where
$$K(k)=\C{K}{1}(p_0;\overline{\lambda})(K+1)e^{8(m-k+1)p_0\eta\C{C'}{0}}$$
and
$$L(k)=\C{K}{2}(p_0;\overline{\lambda})(L+1)(K+1)e^{11(m-k+1)p_0\eta\C{C'}{0}}.$$
By the assumption on $l_1,\dots l_m$, we have $l_m\leq \frac{1}{4K}e^{-7n\C{C'}{0}}$ and 
$$l_{m-k+1}\leq \frac{1}{4K(k)}e^{-7n\C{C'}{0}},\text{ for any }k\in\{2,\dots, m\}.$$
It follows from the definition of $f_1,\dots, f_m$, \Cref{lem:bad.convol.small} and the above that
\begin{align}\label{eqn:bad.small.msr}
f_k(\sigma_{m-k+1};\sigma_{m-k+2},\dots, \sigma_m)\leq \C{C}{6}e^{-\C{\beta}{1}\eta p_0}\text{ when }\sigma_{m-k+1}=\bad.
\end{align}
Recall that $f_1,\dots ,f_m$ are non-negative functions on their domain. By \eqref{eqn:good.count-2-2}, \eqref{eqn:bad.string.decomp.parts} and \eqref{eqn:bad.small.msr}, for any $k\in\{1,\dots,m\}$, we have 
\begin{align}\label{eqn:good.count-2-5}
&\sum_{(\sigma_1,\dots,\sigma_m)\in\cW_{m;k}(\eta)}\prod_{j=1}^mf_j(\sigma_{m-j+1};\sigma_{m-j+2},\dots, \sigma_m)\nonumber\\
\leq&\sum_{(\sigma_1,\dots,\sigma_m)\in\cW_{m;k}(\eta)}( \C{C}{6}e^{-\C{\beta}{1}\eta p_0})^{\#\{i\in\{1,\dots,k\}:\sigma_i=\bad\}}\prod_{j=1}^{m-k}f_j(\sigma_{m-j+1};\sigma_{m-j+2},\dots, \sigma_m)\nonumber\\
\leq& \sum_{(\sigma_{k+1},\dots, \sigma_{m})\in\cW_{m-k}}\sum_{i=[k\eta]+1}^k{k\choose i}( \C{C}{6}e^{-\C{\beta}{1}\eta p_0})^i\prod_{j=1}^{m-k}f_j(\sigma_{m-j+1};\sigma_{m-j+2},\dots, \sigma_m)\nonumber\\
=& \left(\sum_{i=[k\eta]+1}^k{k\choose i}( \C{C}{6}e^{-\C{\beta}{1}\eta p_0})^i\right)\left(\sum_{(\sigma_{1},\dots, \sigma_{m})\in\cW_{m}}\prod_{j=1}^{m-k}f_j(\sigma_{m-j+1};\sigma_{m-j+2},\dots, \sigma_m)\right)\nonumber\\
=&\sum_{i=[k\eta]+1}^k{k\choose i}( \C{C}{6}e^{-\C{\beta}{1}\eta p_0})^i.
\end{align}
Apply \eqref{eqn:good.count-2-5} to \eqref{eqn:good.count-2-4}, we obtain
\begin{align}\label{eqn:good.count-2-6}
&\frac{1}{\widetilde{\nu}(\AA)}\sum_{(\sigma_1,\dots,\sigma_m)\in\cW_m(\eta)}\mu^{*p_0,\sigma_1,\eta}_{(l_1)}*\mu^{*p_0,\sigma_2,\eta}_{(l_2)}*\dots*\mu^{*p_0,\sigma_m,\eta}_{(l_m)}*\widetilde{\nu}(\AA)\nonumber\\
\geq& 1-\left(\sum_{k=1}^m\sum_{i=[k\eta]+1}^k{k\choose i}( \C{C}{6}e^{-\C{\beta}{1}\eta p_0})^i\right).
\end{align}
By Lemma \ref{lem:stirling-2} and the assumptions on $p_0,\C{p}{*},\C{a}{*}$, we have 
\begin{align}\label{eqn:good.count-2-7}
&1-\left(\sum_{k=1}^m\sum_{i=[k\eta]+1}^{k}{k\choose i}( \C{C}{6}e^{-\C{\beta}{1}\eta p_0})^i\right)\nonumber\\
\geq& 1-\left(\sum_{k=1}^m\frac{\C{C''}{4}e^{-\C{a}{*}(\eta,c)k\eta}}{1-e^{-\C{a}{*}(\eta,c)}}\right)
>1- \frac{\C{C''}{4}e^{-\C{a}{*}(\eta,c)\eta}}{(1-e^{-\C{a}{*}(\eta,c)})(1-e^{-\C{a}{*}(\eta,c)\eta})}>1-c.
\end{align}
\Cref{lem:good.count-2} then follows immediately from \eqref{eqn:good.count-2-6} and \eqref{eqn:good.count-2-7}.
\end{proof}

\subsection{Transversality}\label{subsect:transversality}

In this subsection we introduce our notion of transversality for words.  The idea will be that we can control much better the norm of the measure corresponding to transverse pairs of words.  Recall that in \Cref{subsection:notation.LY} we introduced the notation  $\widehat{\lambda}:=\min\{\overline{\lambda}/2,1\}$.  Also, in \eqref{eqn:Es.past}, we defined the notation $V^u_{-\omega^n}(x)$. 

\begin{dfn}[Transversality for words]\label{dfn:transversality}
Let $\omega^n_1,\omega^n_2\in\Sigma_n$ and $z\in\TT^2$. We say that $\omega^n_1$ and $\omega^n_2$ are \emph{transverse} at $z$ if the following holds:
\begin{enumerate}
\item $V^u_{-{\omega^n_1}}(z)$ and $V^u_{-{\omega^n_2}}(z)$ are well-defined.  
\item $\sphericalangle\left(V^u_{-{\omega^n_1}}(z),
V^u_{-{\omega^n_2}}(z)\right)\geq 5e^{\C{C}{0}-\widehat{\lambda}n}.$
\end{enumerate}
For simplicity, we denote by $\omega^n_1\pitchfork_z\omega^n_2$ if $\omega^n_1$ and $\omega^n_2$ are {transverse} at $z$. Otherwise, we use the notation $\omega^n_1\parallel_z\omega^n_2$. Recall that $\C{n'}{0}(\overline{\lambda})$ is obtained in \Cref{lem:bded.distortion.3}, and that $\C{n}{0}(\overline{\lambda})$ is obtained in \Cref{lem:bded.distortion.4}. Also, $\C{n}{0}(\overline{\lambda}) \geq \C{n'}{0}(\overline{\lambda})$. 
\end{dfn}
\begin{lem}\label{lem:nearby.transversality}
Let $z\in\TT^2$, $n\geq \C{n}{0}(\overline{\lambda})$ and $\omega^n\in\Sigma_{n,\text{NC}}$. 
For any $x\in B((f^n_{\omega^n})^{-1}(z),e^{-7n\C{C'}{0}})$, for any $\eta\in(0,\widehat{\lambda})$ and for any $v\in T_{x}\TT^2\setminus\{0\}$, assume in addition that
\begin{enumerate}
\item $\|Df^n_{\omega^n}(x)\|\geq 2e^{\overline{\lambda}n}$.
\item $\sphericalangle(v, V^u_{{\omega^n}}(x))\geq 2e^{-\eta n}.$
\end{enumerate}
Then we have 
$$\sphericalangle(Df^n_{\omega^n}(x)v,V^u_{-\omega^n}(z))\leq 2e^{\C{C}{0}-\widehat{\lambda}n}.$$
\end{lem}
\begin{proof}
By \eqref{eqn:unif.C2.norm} and \Cref{lem:bded.distortion.1}, for any $i\in\{1,2\}$, we have
\begin{align}\label{eqn:nearby.transversality-1}
\sphericalangle(Df^n_{\omega^n}(x)v, Df^n_{\omega^n}((f^n_{\omega^n})^{-1}(z))v)\leq e^{-3n\C{C'}{0}}\leq \min\{e^{-\eta n},e^{-\widehat{\lambda}n}\}
\end{align}
Write $v=\cos(\theta)v^s+\sin(\theta)v^u$ for some $\theta\in[0,\pi/2]$, where $v^*$ is a unit vector in $E^*_{\omega^n}((f^n_{\omega^n})^{-1}(z))$. From \Cref{subsection:notation.LY},  $\overline{\lambda}>8\C{\epsilon}{0}$.  By \Cref{lem:bded.distortion.3} and item (2) above, we have
$$\tan\theta\geq \theta\geq 2e^{-\eta n}-\sphericalangle\left(E^s_{{\omega^n}}(x), E^s_{{\omega^n}}((f^n_{\omega^n})^{-1}(z))\right)\geq 2e^{-\eta n}-\frac{1}{2}e^{-n\C{C'}{0}}\geq e^{-\eta n}.$$ 
Recall that $\widehat{\lambda}=\min\{\overline{\lambda}/2,1\}$ and that $\omega^n\in\Sigma_{n,\text{NC}}$. By item (1) and \eqref{eqn:LY.VAP} and the above, we have
\begin{align}\label{eqn:nearby.transversality-2}
&\sphericalangle(Df^n_{\omega^n}((f^n_{\omega^n})^{-1}(z))v, V^u_{-\omega^n}(z))\nonumber\\
=&\sphericalangle(Df^n_{\omega^n}((f^n_{\omega^n})^{-1}(z))v, Df^n_{\omega^n}((f^n_{\omega^n})^{-1}(z))E^u_{\omega^n}((f^n_{\omega^n})^{-1}(z)))\nonumber\\
\leq& \arctan\left(\frac{\Jac f^n_{\omega^n}((f^n_{\omega^n})^{-1}(z))}{\left\| Df^n_{\omega^n}((f^n_{\omega^n})^{-1}(z))\right\|^2}\cdot\cot(\theta)\right)\nonumber\\
\leq &\arctan \left(e^{\C{C}{0}+(\eta+{2\C{\epsilon}{0}} -2\overline{\lambda})n}\right)
\leq e^{\C{C}{0}+(\eta+2{\C{\epsilon}{0}} -2\overline{\lambda})n}\leq e^{\C{C}{0}-\widehat{\lambda}n}.
\end{align}
The lemma then follows immediately from \eqref{eqn:nearby.transversality-1} and \eqref{eqn:nearby.transversality-2}.
\end{proof}
\begin{lem}\label{lem:trans.counting}
For any $z\in\TT^2$, for any $n\in\ZZ^+$ and for any $\omega^n\in\Sigma_n$, we have 
$$\mu^n\left(\{\widetilde{\omega}^n\in\Sigma_n:\widetilde{\omega}^n\parallel_z\omega^n\right)\leq  \CS{C}{7}e^{-\widehat{\lambda}\C{\beta}{1}n},$$
where $\CS{C}{7} := \C{C}{3}\left(5e^{\C{C}{0}}\right)^{\C{\beta}{1}}$.
\end{lem}
\begin{proof}
This is an immediate corollary of the second assertion in \Cref{lem.cone.trans}.
\end{proof}
\subsection{A Lasota-Yorke inequality}\label{subsect:LY}

The heart of this section, and of this paper, is to show the following Lasota-Yorke type of inequality.  

\begin{prop}[A Lasota-Yorke inequality]\label{lem:LY}
Let $K,L>1$, $\widetilde{C}\geq C>0$ and $\eta>0$ be constants. Let $\widetilde{\nu}=\int_{\cA}\nu_\gamma d\alpha(\gamma)$ be a $(Ke^C,Le^C)$-admissible measure on $\AA$. Assume that the following holds:
\begin{enumerate}
\item Let $\CS{\rho}{1}=\C{\rho}{1}(n,K,L,\widetilde{C}):=(16KL)^{-1}e^{-\widetilde{C}-8n\C{C'}{0}}$. For any $\gamma$ in the support of $\alpha$, the length of $\gamma$ is at least $12\C{\rho}{1} e^{n\C{C'}{0}}$.
\item $\mu^{*n,\good,\eta}*\nu$ is a $(Ke^C,Le^C)$-admissible measure on $\AA$.
\end{enumerate}
Let $l=\frac{1}{4K}e^{-7n\C{C'}{0}-C}$. Then there exist constants $\C{n}{3}=\C{n}{3}(\overline{\lambda})>0$, $\C{C}{8}>0$ and $\C{C}{9}=\C{C}{9}(n,K,L)>0$ such that for any $\rho\in(0,\frac{1}{4}\C{\rho}{1}e^{-2n\C{C'}{0}})$ and for any $\theta\in(0,1)$, we have 
\begin{align*}
&\left\|\mu^{*n,\good,\eta}_{(l)}*\widetilde{\nu}\right\|^2_\rho
\leq\C{C}{8}\cdot(1+n)e^{-(\widehat{\lambda}\C{\beta}{1}-6{\C{\epsilon}{0}}-2\eta)n}\|\widetilde{\nu}\|^2_{\rho\theta}+\C{C}{9}(n,K,L)e^{4\widetilde{C}}(\widetilde{\nu}(\AA))^2,
\end{align*}
whenever $n\geq \C{n}{3}(\overline{\lambda})$.
\end{prop}
\begin{proof}
By Definition \ref{dfn:good.convolution}, we assume, without loss of generality,  that for $\alpha$-a.e curve $\gamma$, the length of $\gamma$ is at least $12\C{\rho}{1}e^{n\C{C'}{0}}$ and at most $2l=\frac{1}{2K}e^{-7n\C{C'}{0}-C}$.  This can be done by replacing $\widetilde{\nu}$ by $(\Delta_l)_*\widetilde{\nu}$.  In particular, $\mu^{*n,\good,\eta}_{(l)}*\widetilde{\nu}=\mu^{*n,\good,\eta}*\widetilde{\nu}$. 

 Let $\Gamma_{n,\widetilde{C}}$ be the collection of points $(x_1,x_2)$ in $\mathbb{T}^2$, where $x_i$ is of the form $j.  ([2/\C{\rho}{1}]+1)^{-1}$ for $j\in \{0, \cdots, [2/\C{\rho}{1}]\}$, and for $i=1,2$.  In particular, 
\begin{align}\label{eqn:lattice}
{\bigcup_{x\in \Gamma_{n,\widetilde{C}}} B(x,\C{\rho}{1}) = \mathbb{T}^2\text{~and~}\sup_{z\in\TT^n} \#\{x\in\Gamma_{n,\widetilde{C}}|z\in B(x,3\C{\rho}{1})\} \leq 21^2.}
\end{align}
For any $x\in\TT^2$ and any $r\in\RR_+$, we let 
$$\widetilde{B}(x,r)=\sP^{-1}(B(x,r))\subset \AA.$$
By \eqref{eqn:lattice}, we have 
\begin{align}\label{eqn:chopping.1}
\|\mu^{*n,\good,\eta}* \widetilde{\nu}\|^2_{\rho} \leq 21^2 \displaystyle \sum_{p\in \Gamma_{n,\widetilde{C}}} \|(\mu^{*n,\good,\eta}* \widetilde{\nu})|_{\widetilde{B}(x,\C{\rho}{1})}\|^2_{\rho} 
\end{align}
Fix $x\in\Gamma_{n,\widetilde{C}}$. We first estimate $ \|\mu^{*n,\good,\eta}* \widetilde{\nu}|_{\widetilde{B}(x,\C{\rho}{1})}\|^2_{\rho}$. Notice that 
\begin{align}\label{eqn:local.LY.split}
& \|(\mu^{*n,\good,\eta}* \widetilde{\nu})|_{\widetilde{B}(x,\C{\rho}{1})}\|^2_{\rho}\nonumber \\
=&\left\|\int_{\cA\times \Sigma_n}\mathbbm{1}_{\cG_{n,\eta}(\gamma)}(\omega^n)((f^n_{\omega^n})_*\nu_\gamma)|_{B(x,\C{\rho}{1})}(B(z,\rho)) d\mu^n(\omega^n)d\alpha(\gamma)\right\|^2_{\rho}\nonumber\\
=&\int_{(\cA\times \Sigma_n)^2}\mathbbm{1}_{\cG_{n,\eta}(\gamma_{1})\times \cG_{n,\eta}(\gamma_{2})}(\omega^n_1,\omega^n_2)\left\langle (f^n_{\omega^n_1})_*\nu_{\gamma_1}|_{B(x,\C{\rho}{1})},(f^n_{\omega_2^n})_*\nu_{\gamma_2}|_{B(x,\C{\rho}{1})} \right\rangle_\rho\nonumber\\
&\,\,\,\,\,\,\,\,\,\,\,\,\,\,\,\,\,\,\,\,\,\,\cdot d\alpha(\gamma_1)d\mu^n(\omega^n_1)d\alpha(\gamma_2)d\mu^n(\omega^n_2)\nonumber\\
=&\int_{\cA\times\cA}\int_{\omega^n_1 \parallel_x \omega^n_2}\left\langle ((f^n_{\omega^n_1})_*\nu_{\gamma_1})|_{B(x,\C{\rho}{1})},((f^n_{\omega_2^n})_*\nu_{\gamma_2})|_{B(x,\C{\rho}{1})} \right\rangle_\rho \nonumber d\mu^n(\omega^n_1)d\mu^n(\omega^n_2)d\alpha(\gamma_1)d\alpha(\gamma_2)\nonumber\\
&+\int_{\cA\times\cA}\int_{\omega^n_1 \pitchfork_x \omega^n_2}\left\langle ((f^n_{\omega^n_1})_*\nu_{\gamma_1})|_{B(x,\C{\rho}{1})},((f^n_{\omega_2^n})_*\nu_{\gamma_2})|_{B(x,\C{\rho}{1})} \right\rangle_\rho \nonumber d\mu^n(\omega^n_1)d\mu^n(\omega^n_2)d\alpha(\gamma_1)d\alpha(\gamma_2)\nonumber\\
=&I_x+II_x,
\end{align}
where
\[
\omega^n_1 \parallel_x \omega^n_2 = \{\omega^n_1 \parallel_x \omega^n_2\}^{\gamma_1, \gamma_2}_{n,\eta} = \left\{(\omega^n_1,\omega^n_2)\in\cG_{n,\eta}(\gamma_{1})\times \cG_{n,\eta}(\gamma_{2}):  \omega^n_1\parallel_x\omega^n_2\right\}
\]
and
\[
\omega^n_1 \pitchfork_x \omega^n_2 = \{\omega^n_1 \pitchfork_x \omega^n_2\}_{n,\eta}^{\gamma_1, \gamma_2} = \left\{(\omega^n_1,\omega^n_2)\in\cG_{n,\eta}(\gamma_{1})\times \cG_{n,\eta}(\gamma_{2}) \omega^n_1\pitchfork_x\omega^n_2\right\}.
\]

\textbf{Estimating $I_x$}: To simplify notations, for any $\omega^n\in \Sigma_n$, we introduce the following finite measure on $\TT^2$:
\begin{align}\label{eqn:nu.omega.n}
{\nu}_{\omega^n}:=\int_{\{\gamma\in\cA:\omega^n\in\cG_{n,\eta}(\gamma)\}}\nu_\gamma d\alpha(\gamma).
\end{align}
In particular, $(f^n_{\omega^n})_*\nu_{\omega^n}=\sP_*((f^n_{\omega})_*^{\good,\eta}\widetilde{\nu})$. Since ${\nu}_{\omega^n}$ is by definition only a part of $\sP_*\widetilde{\nu}$, for any measurable subset $U\subset \TT^2$, we have 
\begin{align}\label{eqn:wordwise-1}
\left\|{\nu}_{\omega^n}|_U\right\|_{\rho\theta}\leq \left\|(\sP_*\widetilde{\nu})|_U\right\|_{\rho\theta}=\left\|\widetilde{\nu}_{\sP^{-1}(U)}\right\|_{\rho\theta}.
\end{align}
By \Cref{lem:trans.counting} and Cauchy-Schwarz, we have 
\begin{align}\label{eqn:Ix.1}
I_x
=&\int_{\omega^n_1\parallel_x\omega^n_2}\left\langle ((f^n_{\omega^n_1})_*{\nu}_{\omega_1^n})|_{{B}(x,\C{\rho}{1})},((f^n_{\omega^n_2})_*{\nu}_{\omega_2^n})|_{{B}(x,\C{\rho}{1})}\right\rangle_\rho d\mu^n(\omega^n_1)d\mu^n(\omega^n_2)\nonumber\\
\leq&\frac{1}{2}\int_{\omega^n_1\parallel_x\omega^n_2}\left( \left\|((f^n_{\omega^n_1})_*\nu_{\omega_1^n})|_{B(x,\C{\rho}{1})}\right\|^2_\rho+\left\|((f^n_{\omega^n_2})_*\nu_{\omega_2^n})|_{B(x,\C{\rho}{1})}\right\|^2_\rho\right) d\mu^n(\omega^n_1)d\mu^n(\omega^n_2)\nonumber\\
\leq&\int_{\Sigma_n}\mu^n\left(\{\widetilde{\omega}^n\in\Sigma_n:\widetilde{\omega}^n\parallel_x\omega^n\right)\left\|((f^n_{\omega^n})_*\nu_{\omega^n})|_{B(x,\C{\rho}{1})}\right\|^2_\rho d\mu^n(\omega^n)\nonumber\\
\leq&\C{C}{7}e^{-\widehat{\lambda}\C{\beta}{1}n}\int_{\Sigma_n}\left\|((f^n_{\omega^n})_*\nu_{\omega^n})|_{B(x,\C{\rho}{1})}\right\|^2_\rho d\mu^n(\omega^n)
\end{align}
For any $r>0$, for any $y\in\TT^2$ and for any $\omega^n\in\Sigma_n$, we write
$$D^n_{\omega^n}(y,r):=(f^n_{\omega^n})^{-1}(B(y,r)) \text{ and } \widetilde{D}^n_{\omega^n}(y,r):=\sP^{-1}(D^n_{\omega^n}(y,r)).$$
Let $\gamma$ be a $C^2$ curve in $\TT^2$ satisfying the following properties:
\begin{enumerate}
\item The length of $\gamma$ is at least $12\C{\rho}{1} e^{n\C{C'}{0}}$.
\item $|\curv(\cdot;\gamma)|\leq Ke^C$.
\end{enumerate} 
Let $\gamma'$ be a connected component of $\gamma\cap D^n_{\omega^n}(x,3\C{\rho}{1})$ with length smaller than $4\rho e^{n\C{C'}{0}}$. Notice that $f^n_{\omega^n}\circ\gamma$ has length at least $12\C{\rho}{1}$ (due to \eqref{eqn:unif.C2.norm}). By the first assertion in \Cref{lem:length.bd.curv} and the definition of $\C{\rho}{1}$, $f^n_{\omega^n}\circ\gamma'$ must intersect the boundary of $B(x,3\C{\rho}{1})$. This is equivalent to the fact that $\gamma'$ must intersect the boundary of $D^n_{\omega^n}(x,3\C{\rho}{1})$. Since the length of $f^n_{\omega^n}\circ\gamma'$ is at most $4\rho e^{2n\C{C'}{0}}$ (due to \eqref{eqn:unif.C2.norm}) and $4\rho e^{2n\C{C'}{0}}<\C{\rho}{1}$, we have $f^n_{\omega^n}\circ\gamma'\cap B(x,\C{\rho}{1})=\emptyset$, which is equivalent to $\gamma'\cap D^n_{\omega^n}(x,\C{\rho}{1}) =\emptyset$.

By item (2), for $\mu^n$-a.e. $\omega^n\in\Sigma_n$, the measure $\nu_{\omega^n}$ is $(Ke^C,Le^C)$-admissible. For any such $\omega^n$, we let 
$$\widehat{\nu}_{\omega^n}:=\int_{\{\gamma\in\cA:\omega^n\in\cG_{n,\eta}(\gamma)\text{ and }m_{\gamma}(\gamma\cap D^n_{\omega^n}(x,3\C{\rho}{1}))\geq 4\rho e^{n\C{C'}{0}}\}}\nu_\gamma|_{D^n_{\omega^n}(x,3\C{\rho}{1})} d\alpha(\gamma).$$
In other words, $\widehat{\nu}_{\omega^n}$ is the measure obtained by discarding the part of the measure supported on curves of length smaller than $4\rho e^{n\C{C'}{0}}$ from $\nu_{\omega^n}|_{D^n_{\omega^n}(x,3\C{\rho}{1})}$ (see \eqref{eqn:nu.omega.n} for the definition of $\widehat{\nu}_{\omega^n}$.
Recall that $n_0(\overline{\lambda})$ was obtained in \Cref{lem:bded.distortion.4}. By \Cref{lem:clever.norm},  for any $n\geq \CS{n'}{3}(\overline{\lambda}):= \max\left\{\C{n}{0}(\overline{\lambda}),\frac{2\C{C}{0}+2\ln(3/2)}{\overline{\lambda}}\right\}$, we have
\begin{align}\label{eqn:Ix.2}
\begin{split}
\left\|(f^n_{\omega^n})_*\nu_{\omega^n}|_{B(x,\C{\rho}{1})}\right\|^2_\rho
\leq& \|(f^n_{\omega^n})_*\widehat{\nu}_{\omega^n}\|^2_\rho\\
\leq& (1+n)\C{C}{5}e^{(6{\C{\epsilon}{0}}+2\eta)n}\|\widehat{\nu}_{\omega^n}\|^2_{\frac{2}{3}e^{(\overline{\lambda}n/2)-\C{C}{0}}\rho\theta}
\end{split}
\end{align}
Apply \eqref{eqn:wordwise-1} and \eqref{eqn:Ix.2} to \eqref{eqn:Ix.1}, we obtain
\begin{align}\label{eqn:lx}
I_x\leq& \C{C'}{8}\cdot(1+n)e^{-(\widehat{\lambda}\C{\beta}{1}-6{\C{\epsilon}{0}}-2\eta)n}\int_{\Sigma_n}\|\nu_{\omega^n}|_{D^n_{\omega^n}(x,3\C{\rho}{1})}\|^2_{\rho\theta}d\mu^n(\omega^n)\nonumber\\
\leq & \C{C'}{8}\cdot(1+n)e^{-(\widehat{\lambda}\C{\beta}{1}-6{\C{\epsilon}{0}}-2\eta)n}\int_{\Sigma_n}\|\widetilde{\nu}|_{\widetilde{D}^n_{\omega^n}(x,3\C{\rho}{1})}\|^2_{\rho\theta}d\mu^n(\omega^n),
\end{align}
where $\CS{C'}{8}:=\C{C}{5}\C{C}{7}$.

\textbf{Estimating $II_x$}: We would like to show that for any $\omega^n_1\in\cG_{n,\eta}(\gamma_{1})$ and for any $\omega^n_2\in\cG_{n,\eta}(\gamma_{2})$ such that $\omega^n_1\pitchfork_x\omega^n_2$, there exist constants $\C{C'}{9}:=\C{C'}{9}(n,K,L)$ and $\C{n''}{3}>0$ such that for any $n\geq \C{n''}{3}$, we have 
\begin{align}\label{eqn:IIx.1}
\begin{split}
&\left\langle ((f^n_{\omega^n_1})_*\nu_{\gamma_1})|_{B(x,\C{\rho}{1})},((f^n_{\omega_2^n})_*\nu_{\gamma_2})|_{B(x,\C{\rho}{1})} \right\rangle_\rho\\
\leq &\C{C'}{9}(n,K,L)e^{2\widetilde{C}}\nu_{\gamma_1}(D^n_{\omega^n_1}(x,3\rho_1))\nu_{\gamma_2}(D^n_{\omega^n_2}(x,3\rho_1)),
\end{split}
\end{align}
whenever the measures
\begin{align}\label{eqn:get.rid.of.stupidity}
\nu_{\gamma_1},~\nu_{\gamma_2},~(f^n_{\omega^n_1})_*\nu_{\gamma_1}\text{ and }(f^n_{\omega^n_2})_*\nu_{\gamma_2}
\end{align}
are all $(Ke^C,Le^C)$-admissible. By \Cref{dfn:good.word} and the bilinearity of \eqref{eqn:IIx.1} in $\nu_{\gamma_1}$ and $\nu_{\gamma_2}$, we can assume, without loss of generality,  that $\gamma_i':=\gamma_{i}|_{B(x,3\C{\rho}{1})}$ is connected for any $i\in\{1,2\}$. 

Since all measures in \eqref{eqn:get.rid.of.stupidity} are $(Ke^C,Le^C)$-admissible, we have 
\begin{align*}
|\curv(\cdot;f^n_{\omega^n_i}\circ\gamma_i)|\cdot 3\C{\rho}{1}\leq Ke^C\cdot 3\C{\rho}{1}<\frac{1}{4}.
\end{align*}
Notice that $m_{\gamma_i}(\gamma_i)\geq 12\C{\rho}{1}e^{n\C{C}{0}}$ from the assumption of this lemma, by \eqref{eqn:unif.C2.norm}, the length of $f^n_{\omega^n_i}\circ\gamma_i$ is at least $12\C{\rho}{1}$. It follows from the first assertion in \Cref{lem:length.bd.curv} that $\gamma_i'$ must intersect the boundary of $D^n_{\omega^n_i}(x,3\C{\rho}{i})$. Equivalently, $f^n_{\omega^n_i}\circ \gamma_i'$ must intersect the boundary of $B(x,3\C{\rho}{1})$. 

If $\gamma_i'$ has length less than $\C{\rho}{1}e^{-n\C{C'}{0}}$, then by \eqref{eqn:unif.C2.norm}, $f^n_{\omega^n_i}\circ\gamma_i'$ has length less than $\C{\rho}{1}$. In particular, $f^n_{\omega^n_i}\circ\gamma_i'$ does not intersect $B(x,\C{\rho}{1})$. This is equivalent to $\gamma_i'$ not intersecting $D^n_{\omega_i^n}(x,\C{\rho}{1})$. Hence, \eqref{eqn:IIx.1} holds trivially. Therefore, we assume, without loss of generality, that 
\begin{align}\label{eqn:IIx.2}
m_{\gamma_i'}(\gamma_i')\geq \C{\rho}{1}e^{-n\C{C'}{0}}.
\end{align}
Notice that by \eqref{eqn:unif.C2.norm}, we have
$$D^n_{\omega_i^n}(x,3\C{\rho}{1})\subset B((f^n_{\omega^n_i})^{-1}(x),3\C{\rho}{1}e^{n\C{C'}{0}})\text{ and } 3\C{\rho}{1}e^{n\C{C'}{0}}\cdot Ke^C<\frac{1}{4}.$$
By \Cref{lem:length.bd.curv} and the above, we have 
\begin{align}\label{eqn:IIx.3}
Le^C\cdot m_{\gamma_i'}(\gamma_i')\leq Le^C\cdot 4\cdot 3\C{\rho}{1}e^{n\C{C'}{0}}=12L\C{\rho}{1}e^{n\C{C'}{0}+C}<1.
\end{align}
As a direct corollary of \eqref{eqn:get.rid.of.stupidity}, \eqref{eqn:IIx.2} and \eqref{eqn:IIx.3}, for any $z\in\TT^2$, we have 
\begin{align*}
\frac{\nu_{\gamma_i}(D^n_{\omega^n_i}(z,\rho)\cap D^n_{\omega^n_{i}}(x,\C{\rho}{1}))}{\nu_{\gamma_i}(D^n_{\omega^n_i}(x,3\C{\rho}{1}))}
=&\frac{\nu_{\gamma_i}(D^n_{\omega^n_i}(z,\rho)\cap \gamma_i')}{\nu_{\gamma_i}(D^n_{\omega^n_i}(x,3\C{\rho}{1}))}\\
\leq&  e^2\frac{m_{\gamma_i'}(D^n_{\omega^n_i}(z,\rho))}{m_{\gamma_i'}(D^n_{\omega^n_i}(x,3\C{\rho}{1}))}\leq \frac{e^{2+n\C{C'}{0}}}{\C{\rho}{1}}m_{\gamma_i'}(D^n_{\omega^n_i}(z,\rho)).
\end{align*}

Define the function
\[
\mathbbm{1}_r(y_1,y_2) = 
\begin{cases}
1, \textrm{  if }d(y_1,y_2) \leq r,\\
0, \textrm{ otherwise.}
\end{cases}
\]

Therefore,

\begin{align}\label{eqn:IIx.4}
&\frac{\left\langle ((f^n_{\omega^n_1})_*\nu_{\gamma_1})|_{B(x,\C{\rho}{1})},((f^n_{\omega_2^n})_*\nu_{\gamma_2})|_{B(x,\C{\rho}{1})} \right\rangle_\rho}{\nu_{\gamma_1}(D^n_{\omega^n_1}(x,3\rho_1))\nu_{\gamma_2}(D^n_{\omega^n_2}(x,3\rho_1))}\nonumber\\
\leq&\frac{1}{\rho^4}\int_{\TT^2} \frac{\nu_{\gamma_1}(D^n_{\omega^n_1}(z,\rho) \cap D^n_{\omega^n_{1}}(x,\C{\rho}{1}))}{\nu_{\gamma_1}(D^n_{\omega^n_1}(x,3\C{\rho}{1}))}\cdot\frac{\nu_{\gamma_2}(D^n_{\omega^n_2}(z,\rho)\cap  D^n_{\omega^n_{2}}(x,\C{\rho}{1}))}{\nu_{\gamma_2}(D^n_{\omega^n_2}(x,3\C{\rho}{1}))}d\Leb(z)\nonumber\\
\leq& \left(\frac{e^{2+n\C{C'}{0}}}{\C{\rho}{1}}\right)^2\cdot\frac{1}{\rho^4}\int_{\TT^2}m_{\gamma_1'}(D^n_{\omega^n_1}(z,\rho))m_{\gamma_2'}(D^n_{\omega^n_2}(z,\rho))d\Leb(z)\nonumber\\
\leq &\left(\frac{e^{2+n\C{C'}{0}}}{\C{\rho}{1}}\right)^2\int_{\mathbb{T}^2} \frac{1}{\rho^{4}} \int_{\gamma_1' \times \gamma_2'} \mathbbm{1}_{\rho}(f^n_{\omega^n_1}(y_1), z) \mathbbm{1}_\rho(f^n_{\omega^n_2}(y_2),z) dm_{\gamma_1'}(y_1) dm_{\gamma_2'}(y_2) d\Leb(z) \nonumber\\
\leq &\left(\frac{e^{2+n\C{C'}{0}}}{\C{\rho}{1}}\right)^2\int_{\mathbb{T}^2} \frac{1}{\rho^{4}} \int_{\gamma_1' \times \gamma_2'} \mathbbm{1}_{\rho}(f^n_{\omega^n_1}(y_1), z) \mathbbm{1}_{2\rho}(f^n_{\omega^n_1}(y_1) , f^n_{\omega^n_2}(y_2)) dm_{\gamma_1'}(y_1) dm_{\gamma_2'}(y_2) d\Leb(z) \nonumber\\
= &\frac{\C{C''}{9}(n,K,L)e^{2\widetilde{C}}}{\rho^2}\int_{\gamma_1' \times \gamma_2'} \mathbbm{1}_{2\rho}(f^n_{\omega^n_1}(y_1) , f^n_{\omega^n_2}(y_2))dm_{\gamma_1'}(y_1) dm_{\gamma_2'}(y_2),
\end{align}
where 
$$\CS{C''}{9}(n,K,L)e^{2\widetilde{C}}=\pi\left(\frac{e^{2+n\C{C'}{0}}}{\C{\rho}{1}}\right)^2.$$ 
Recall that $\C{\rho}{1}=\C{\rho}{1}(n,K,L,\widetilde{C})=(16KL)^{-1}e^{-\widetilde{C}-8n\C{C'}{0}}$. Therefore,
\[
\C{C''}{9}(n,K,L)e^{2\widetilde{C}}=256\pi K^2L^2e^{4+18n\C{C'}{0}}e^{2\widetilde{C}}.
\]

For any $i=1,2$, we write $\gamma_i^n = f_{\omega^n_i}^n\circ \gamma_i'$. Recall that $\gamma'_i=\gamma_i|_{D^n_{\omega^n_i}(x,3\C{\rho}{1})}$, we have $\gamma_i^n=(f_{\omega^n_i}^n\circ \gamma_i')|_{B(x,3\C{\rho}{1})}$. Since $Ke^{C}\cdot 3\C{\rho}{1}<1/4$, by \eqref{eqn:get.rid.of.stupidity} and the first assertion in \Cref{lem:length.bd.curv}, we have 
\begin{align}\label{eqn:IIx.5.0}
m_{\gamma_1^n}(\{y_1^n\in\gamma_1^n:d(y_1^n,y_2^n)<2\rho\})\leq 8\rho,~\text{ for any }y_2^n\in\gamma_2^n,
\end{align}
and 
\begin{align}\label{eqn:IIx.5.repetitive}
m_{\gamma_i^n }(\gamma_i^n )\leq 4\cdot 3\C{\rho}{1}=12\C{\rho}{1}<e^{-8n\C{C'}{0}}.
\end{align}
As a direct consequence of \eqref{eqn:unif.C2.norm} and \eqref{eqn:IIx.5.0}, we have 
\begin{align}\label{eqn:IIx.5.1}
m_{\gamma_1'}((f^n_{\omega_1^n})^{-1}\{y_1^n\in\gamma_1^n:d(y_1^n,y_2^n)<2\rho\})\leq 8\rho e^{n\C{C'}{0}}.
\end{align}
Applying \eqref{eqn:IIx.5.1} to \eqref{eqn:IIx.4}, we have
\begin{align}\label{eqn:IIx.5.2}
&\frac{\left\langle (f^n_{\omega^n_1})_*\nu_{\gamma_1}|_{B(x,\C{\rho}{1})},(f^n_{\omega_2^n})_*\nu_{\gamma_2}|_{B(x,\C{\rho}{1})} \right\rangle_\rho}{\nu_{\gamma_1}(D^n_{\omega^n_1}(x,3\rho_1))\nu_{\gamma_2}(D^n_{\omega^n_2}(x,3\rho_1))}\nonumber\\
\leq&\frac{\C{C''}{9}(n,K,L)e^{2\widetilde{C}}}{\rho^2}\int_{ \gamma_2} m_{\gamma_1'}((f^n_{\omega_1^n})^{-1}\{y_1^n\in\gamma_1^n:d(y_1^n,f^n_{\omega^n_2}(y_2))<2\rho\}) dm_{\gamma_2'}(y_2)\nonumber\\
\leq& \frac{\C{C''}{9}(n,K,L)e^{2\widetilde{C}}}{\rho^2}\cdot8\rho e^{n\C{C'}{0}}\cdot m_{\gamma_2'}((f^n_{\omega^n_2})^{-1}\{z^n_2\in\gamma_2^n:d(z^n_2,\gamma_1^n)<2\rho\}).
\end{align}

Hence it remains to estimate $m_{\gamma_2'}((f^n_{\omega^n_2})^{-1}\{z^n_2\in\gamma_2^n:d(z^n_2,\gamma_1^n)<2\rho\})$. Recall that $\gamma^n_i=f^n_{\omega^n_i}\circ\gamma'_i\subset f^n_{\omega^n_i}\circ\gamma_i$ for any $i\in\{1,2\}$. Then for any $i\in\{1,2\}$ and for any $z_i\in\gamma'_i$, by \Cref{dfn:good.word} and the assumption that $\omega_i^n\in\cG_{n,\eta}(\gamma_i)\subset\cG_{n,\eta}(\gamma_i')$, we have
$$\|Df^n_{\omega^n_i}|_{\dot\gamma'|_{z_i}}\|\geq 2e^{\overline{\lambda}n}\text{, }\sphericalangle(\dot\gamma'|_{z_i},E^s_{\omega^n_i}(z_i))\geq 2e^{-\eta n}\text{ and }\omega^n_i\in\Sigma_{n,\text{NC}}.$$
Since $\gamma^n_1,\gamma^n_2\subset B(x,3\C{\rho}{1})$ and $3\C{\rho}{1}<e^{-8n\C{C'}{0}}$, by \eqref{eqn:unif.C2.norm}, we have $\gamma'_i\subset B((f^n_{\omega^n_i})^{-1}(x),e^{-7n\C{C'}{0}})$ for any $i\in\{1,2\}$. Therefore, by \Cref{lem:nearby.transversality}, for any $n\geq \C{n}{0}$, for any $i\in\{1,2\}$ and for any $z_i^n\in\gamma_i^n$, we have 
$$\sphericalangle(\dot\gamma^n_i|_{z_i^n},V^u_{-\omega^n_i}(x))\leq 2e^{\C{C}{0}-\widehat{\lambda}n}.$$
Recall that $\omega^n_1\pitchfork_x\omega^n_2$, by \Cref{dfn:transversality} and the above, for any $z_1^n\in\gamma_1^n$ and for any $z_2^n\in\gamma_2^n$, we have 
\begin{align}\label{eqn:IIx.5.3}
&\sphericalangle(\dot\gamma^n_1|_{z_1^n},\dot\gamma^n_2|_{z_2^n})\nonumber\\
\geq& \sphericalangle(V^u_{-\omega^n_1}(x),V^u_{-\omega^n_2}(x))-\sphericalangle(\dot\gamma^n_1|_{z_1^n}, V^u_{-\omega^n_1}(x))-\sphericalangle(\dot\gamma^n_2|_{z_2^n},V^u_{-\omega^n_2}(x))\nonumber\\
\geq&5e^{\C{C}{0}-\widehat{\lambda}n} - 2e^{\C{C}{0}-\widehat{\lambda}n}-2e^{\C{C}{0}-\widehat{\lambda}n}=e^{\C{C}{0}-\widehat{\lambda}n}.
\end{align}
Assume that there exist $y_1^n\in\gamma_1^n$ and $y_2^n\in\gamma_2^n$ such that $d(y_1^n,y_2^n)<2\rho$. Then by \eqref{eqn:IIx.5.3} and the Mean Value Theorem (applied to $\gamma_1^n$ and $\gamma_2^n$), for any $z_1^n\in\gamma_1^n\setminus\{y_1^n\}$ and for any $z_2^n\in\gamma_2^n\setminus\{y_2^n\}$, the angle between the geodesic segments $\overline{y_1^nz_1^n}$ and $\overline{y_2^nz_2^n}$ is at least $e^{\C{C}{0}-\widehat{\lambda}n}$. Choose $\CS{n''}{3}(\overline{\lambda})>\C{n}{0}(\overline{\lambda})$ such that $2\sin(t)>t$ whenever $0\leq t\leq e^{\C{C}{0}-\widehat{\lambda}\C{n''}{3}}$. Then for any $n\geq \C{n''}{3}(\overline{\lambda})$, we have
$$\{z^n_2\in\gamma_2^n:d(z^n_2,\gamma_1^n)<2\rho\}\subset B(y_2^n,4\rho/\sin(e^{\C{C}{0}-\widehat{\lambda}n}))\subset B(y_2^n,8e^{-\C{C}{0}+\widehat{\lambda}n}\rho).$$
Since , by the first assertion in \Cref{lem:length.bd.curv}, we have 
$$m_{\gamma_2^n}(\{y_2\in\gamma_2^n:d(y_2,\gamma_1^n)<2\rho\})\leq 32e^{-\C{C}{0}+\widehat{\lambda}n}\rho\leq 32e^{-\C{C}{0}+\C{C'}{0}n}\rho,$$
which implies that
\begin{align}\label{eqn:IIx.5.4}
m_{\gamma_2'}((f^n_{\omega^n_2})^{-1}\{y_2\in\gamma_2^n:d(y_2,\gamma_1^n)<2\rho\})\leq 32e^{-\C{C}{0}+2\C{C'}{0}n}\rho.
\end{align}
Equation \eqref{eqn:IIx.1} then follows from \eqref{eqn:IIx.5.4} and \eqref{eqn:IIx.5.2} with 
$$\CS{C'}{9}(n,K,L):=\C{C''}{9}(n,K,L)\cdot 8e^{n\C{C'}{0}}\cdot 32e^{-\C{C}{0}+2\C{C'}{0}n}.$$
By applying item (2) and \eqref{eqn:IIx.1} to $II_x$, we obtain
\begin{align}\label{eqn:IIx}
II_x\leq& \int_{(\cA\times \Sigma_n)^2} \C{C'}{9}(n,K,L)e^{2\widetilde{C}}\nu_{\gamma_1}(D^n_{\omega^n_1}(x,3\rho_1))\nu_{\gamma_2}(D^n_{\omega^n_2}(x,3\rho_1))d\alpha(\gamma_1)d\mu^n(\omega^n_1)d\alpha(\gamma_2)d\mu^n(\omega^n_2)\nonumber\\
=&  \C{C'}{9}(n,K,L)e^{2\widetilde{C}}\int_{\Sigma_n}\widetilde{\nu}(\widetilde{D}^n_{\omega^n_1}(x,3\rho_1))d\mu^n(\omega^n_1)\int_{\Sigma_n}\widetilde{\nu}(\widetilde{D}^n_{\omega^n_2}(x,3\rho_1))d\mu^n(\omega^n_2)\nonumber\\
\leq &\C{C'}{9}(n,K,L)e^{2\widetilde{C}}(\nu(\AA))^2.
\end{align}

\textbf{Finishing the proof by summing over $x\in\Gamma_{n,\widetilde{C}}$}: By the definition of $\C{\rho}{1}$, there exists some constant $\CS{C'''}{9}=\C{C'''}{9}(n,K,L)$ such that
$$|\Gamma_{n,\widetilde{C}}|=\C{C'''}{9}(n,K,L)e^{2\widetilde{C}}.$$
Choose 
$$\CS{n}{3}(\overline{\lambda}):=\max\{\C{n}{0}(\overline{\lambda}),\C{n'}{3}(\overline{\lambda}),\C{n''}{3}(\overline{\lambda})\}$$
By \eqref{eqn:lattice}, \eqref{eqn:chopping.1}, \eqref{eqn:local.LY.split}, \eqref{eqn:lx}, \eqref{eqn:IIx} and the above, for any $n\geq \C{n}{3}(\overline{\lambda})$, we have 
\begin{align*}
&\left\|\mu^{*n,\good,\eta}*\nu\right\|^2_\rho\\
\leq&21^2\sum_{x\in\Gamma_{n,\widetilde{C}}} \|\mu^{*n,\good,\eta}* \nu|_{B(x,\C{\rho}{1})}\|^2_{\rho}\\
=&21^2\sum_{x\in\Gamma_{n,\widetilde{C}}}(I_x+II_x)\\
\leq&21^2\C{C'}{8}\cdot(1+n)e^{-(\widehat{\lambda}\C{\beta}{1}-6{\C{\epsilon}{0}}-2\eta)n}\int_{\Sigma_n}\sum_{x\in\Gamma_{n,\widetilde{C}}}\|\widetilde{\nu}|_{\widetilde{D}^n_{\omega^n}(x,3\C{\rho}{1})}\|^2_{\rho\theta}d\mu^n(\omega^n)\\
&+21^2\sum_{x\in\Gamma_{n,\widetilde{C}}}\C{C'}{9}(n,K,L)e^{2\widetilde{C}}(\widetilde{\nu}(\AA))^2\\
\leq&21^4\C{C'}{8}\cdot(1+n)e^{-(\widehat{\lambda}\C{\beta}{1}-6{\C{\epsilon}{0}}-2\eta)n}\int_{\Sigma_n}\|\widetilde{\nu}\|^2_{\rho\theta}d\mu^n(\omega^n)\\
&+21^2|\Gamma_{n,\widetilde{C}}|\C{C'}{9}(n,K,L)e^{2\widetilde{C}}(\widetilde{\nu}(\AA^2))^2\\
\leq&21^4\C{C'}{8}\cdot(1+n)e^{-(\widehat{\lambda}\C{\beta}{1}-6{\C{\epsilon}{0}}-2\eta)n}\int_{\Sigma_n}\|\nu\|^2_{\rho\theta}d\mu^n(\omega^n)+\C{C}{9}(n,K,L)e^{4\widetilde{C}}(\widetilde{\nu}(\AA))^2\\
=&\C{C}{8}\cdot(1+n)e^{-(\widehat{\lambda}\C{\beta}{1}-6{\C{\epsilon}{0}}-2\eta)n}\|\widetilde{\nu}\|^2_{\rho\theta}+\C{C}{9}(n,K,L)e^{4\widetilde{C}}(\widetilde{\nu}(\AA))^2,
\end{align*}
where $\CS{C}{8}:=21^4\C{C'}{8}$ and $\CS{C}{9}(n,K,L):=21^2\C{C'''}{9}(n,K,L)\C{C'}{9}(n,K,L)$. This finishes the proof.
\end{proof}

\section{Absolute continuity of admissible measures}\label{sec.abs}

In this section, we apply \Cref{lem:LY} to prove the absolute continuity of limits of convolutions of admissible measures.  This is obtained in the following theorem.  

\begin{thm}\label{thm:main.detailed}
Let $K,L>1$ be constants. Let $\nu$ be a $(K,L)$-admissible measure on $\TT^2$ supported on curves of length at least $e^{-l}>0$ for some $l\geq 0$. Let $\nu_\infty$ be any limit measure of a subsequence of 
\begin{align}\label{eqn:erg.avg}
\left\{\overline{\nu}_n:=\frac{1}{n}\sum_{k=0}^{n-1}\mu^{*k}*\nu\right\}_{n\in\ZZ_+}.
\end{align}
Assume that the constant ${\C{\epsilon}{0}}$ in \eqref{eqn:partial.VAP} satisfies 
\begin{align}\label{eqn:epsilon0}{\C{\epsilon}{0}}\leq \min\left\{\frac{\widehat{\lambda}\C{\beta}{1}}{7},\C{C}{0},\C{C'}{0},\overline{\chi}/2\delta,\frac{\overline{\lambda}}{8}\right\}.
\end{align}

Then, for any $c\in(0,1)$, there exists a finite measure $\nu_{\infty,c}$ satisfying the following properties:
\begin{enumerate}
\item $\nu_{\infty,c}(\TT^2)\geq (1-c)\nu(\TT^2)$.
\item $\nu_{\infty,c}$ and $\nu-\nu_{\infty,c}$ are finite (possibly zero, unsigned) measures.
\item The measure $\nu_{\infty,c}$ is absolutely continuous with respect to $\Leb$. Moreover, $d\nu_{\infty,c}/d\Leb\in L^2(\TT^2,\Leb)$.
\end{enumerate}
\end{thm}
\begin{proof}
Fix $\eta>0$ such that 
$$\eta< \min\left\{\frac{(\lambda-\overline{\lambda})}{\C{\beta}{1}},\frac{1}{\C{\beta}{1}},\frac{1}{13},\frac{\widehat{\lambda}\C{\beta}{1}}{112},\frac{\widehat{\lambda}\C{\beta}{1}}{3328\C{C'}{0}+2\widehat{\lambda}\C{\beta}{1}},\frac{\overline{\lambda}}{13\C{C'}{0}}\right\}.$$
In particular, by the assumptions on $\eta$ and ${\C{\epsilon}{0}}$, we have 
$$
\widehat{\lambda}\C{\beta}{1}-6{\C{\epsilon}{0}}-2\eta> \frac{\widehat{\lambda}\C{\beta}{1}}{8}=:2\CS{\beta}{2}
$$
and
\begin{align}\label{eqn:eta.cor-2}
(1-\eta)\C{\beta}{2}-7\C{C'}{0}\eta\geq \frac{1}{2}\C{\beta}{2}.
\end{align}
We also fix $p_0\geq \max\{\C{p}{**}(\eta,\overline{\lambda},c),\C{n}{3}(\overline{\lambda})\}$, where 
$$\CS{p}{**}(\eta,\overline{\lambda},c)=\max\left\{\C{p}{*}(\eta,\overline{\lambda},c),\frac{|\ln(\C{C}{4})|}{7\C{C'}{0}},\frac{-24(\eta\ln(\eta)+(1-\eta)\ln(1-\eta))}{\C{\beta}{2}}\right\}$$
and
\begin{align}\label{eqn:p**cond.2}
\C{C}{8}(1+n)e^{-(\widehat{\lambda}\C{\beta}{1}-6{\C{\epsilon}{0}}-2\eta)n}<e^{-\C{\beta}{2}n},~\forall n\geq \C{p}{**}(\eta,\overline{\lambda},c).
\end{align}
It is not hard to see that there exists constant $\CS{K'}{}=\C{K'}{}(p_0)>1$ and $\CS{L'}{}=\C{L'}{}(p_0)>1$ such that $\mu^{*k}*\nu$ are $(\C{K'}{},\C{L'}{})$-admissible for any $0\leq k\leq p_0-1$.  

The main idea of the proof is to show that there exists some $\C{\rho}{2}=\C{\rho}{2}(p_0,l)$ such that for any $d\in\{0,\dots,p_0-1\}$ and for any $m\in\ZZ_{\geq 0}$, there exists a measure $\nu_{d,m,c}$ on $\TT^2$ such that the following holds
\begin{enumerate}
\item[\hypertarget{P1}{(P1)}] $\nu_{d,m,c}(\TT^2)\geq (1-c)\nu(\TT^2)$.
\item[\hypertarget{P2}{(P2)}]  The measures $\nu_{d,m,c}$ and $\mu^{*(d+mp_0)}*\nu-\nu_{d,m,c}$ are finite.
\item[\hypertarget{P3}{(P3)}] There exist constants $\C{C'''}{10}>\max\{1,\C{C}{4}\}$ and $\C{C''''}{10}(p_0,\C{K''}{},\C{L''}{},l)>0$ such that for any $\rho\in(0,\C{\rho}{2})$, we have
$$\displaystyle \|\nu_{d,m,c}\|_{\rho}\leq \C{C'''}{10}e^{-\frac{1}{6}\C{\beta}{2}p_0m}\|\mu^{*d}*{\nu}\|_{\rho e^{-13p_0m\C{C'}{0}\eta}}+\C{C''''}{10}(p_0,\C{K''}{},\C{L''}{},l){\nu}(\TT^2).$$
\end{enumerate}
Recall that the constant $\C{C}{4}$ above is obtained in \Cref{generalized large<small}.  We show that once the existence of these measures $\nu_{d,m,c}$ have been verified, the theorem will follow. Indeed, for any $m\in\ZZ_+$, any limit measure $\nu_\infty$ of a subsequence of \eqref{eqn:erg.avg} is also a limit measure of a subsequence of $\{\overline{\nu}_{mp_0}\}_{m\in\ZZ_+}$. Choose $m_1<m_2<\dots$ such that 
$$\lim_{j\to\infty }\overline{\nu}_{m_jp_0}=\nu_\infty.$$
Let $\nu_{\infty,c}$ be a limit of a subsequence of 
$$\left\{\overline{\nu}_{d,m_j,c}:=\frac{1}{m_jp_0}\sum_{d=0}^{p_0-1}\sum_{k=0}^{m_j-1}\nu_{d,k,c}\right\}_{j\in\ZZ_+}.$$
By \hyperlink{P1}{(P1)}, $\nu_{\infty,c}$ is a finite measure with $\nu_{\infty,c}(\TT^2)\geq (1-c)\nu(\TT^2)$. This verifies item (1) in the theorem. \hyperlink{P2}{(P2)} implies that $\overline{\nu}_{m_jp_0}-\overline{\nu}_{d,m_j,c}$ and $\overline{\nu}_{d,m_j,c}$ are finite measures on $\TT^2$ or equal to $0$ when viewed as elements in the dual space of all continuous functions on $\TT^2$. Item (2) in the theorem then follows after letting $j\to\infty$. To see that item (3) in the theorem holds, we first observe that by the assumptions on $\eta$ and the definition of $\C{\beta}{2}$, we have $26\C{C'}{0}\eta<\widehat{\lambda}\C{\beta}{1}/128=\C{\beta}{2}/16$.
By \eqref{eqn:obvious.norm.upper.bd}, \hyperlink{P3}{(P3)} and the above discussions, for any $\rho\in(0,\C{\rho}{2})$ and any $j\in\ZZ_+$, we have 
\begin{align*}
&\left\|\overline{\nu}_{d,m_j,c}\right\|_\rho\\
\leq &\frac{1}{m_jp_0}\sum_{d=0}^{p_0-1}\sum_{k=0}^{m_j-1}\|\nu_{d,k,c}\|_\rho\\
\leq &\frac{1}{m_jp_0}\sum_{d=0}^{p_0-1}\sum_{k=0}^{m_j-1}\left(\C{C'''}{10}e^{-\frac{1}{6}\C{\beta}{2}p_0k}\|\mu^{*d}*{\nu}\|_{\rho e^{-13p_0k\C{C'}{0}\eta}}+\C{C''''}{10}(p_0,\C{K''}{},\C{L''}{},l){\nu}(\TT^2)\right)\\
\leq &\C{C''''}{10}(p_0,\C{K''}{},\C{L''}{},l){\nu}(\TT^2)+\frac{1}{m_jp_0}\sum_{d=0}^{p_0-1}\sum_{k=0}^{m_j-1}\left(\C{C'''}{10}e^{-\frac{1}{6}\C{\beta}{2}p_0k}\frac{\nu(\TT^2)}{\rho^2 e^{-26p_0k\C{C'}{0}\eta}}\right)\\
\leq &\C{C''''}{10}(p_0,\C{K''}{},\C{L''}{},l){\nu}(\TT^2)+\frac{1}{m_jp_0}\sum_{d=0}^{p_0-1}\sum_{k=0}^{m_j-1}\left(\C{C'''}{10}e^{-\frac{5}{48}\C{\beta}{2}p_0k}\frac{\nu(\TT^2)}{\rho^2 }\right)\\
\leq &\C{C''''}{10}(p_0,\C{K''}{},\C{L''}{},l){\nu}(\TT^2)+\frac{1}{m_j}\cdot\frac{\C{C'''}{10}}{\rho^2}\cdot\frac{1}{1-e^{-\frac{5}{48}\C{\beta}{2}p_0}}.
\end{align*}
Let $j\to\infty$, by \Cref{lem.normconvergence}, we have
\begin{align*}
\|\nu_{\infty,c}\|_\rho=\lim_{j\to\infty}\left\|\overline{\nu}_{d,m_j,c}\right\|_\rho\leq \CS{C}{10}(p_0,\C{K''}{},\C{L''}{},l):=\C{C''''}{10}(p_0,\C{K''}{},\C{L''}{},l){\nu}(\TT^2).
\end{align*}
Item (3) in the theorem then follows from \Cref{lem.liminfnorm} and the above. This proves the theorem assuming the existence of $\nu_{d,m,c}$ satisfying \hyperlink{P1}{(P1)}, \hyperlink{P2}{(P2)} and \hyperlink{P3}{(P3)}.

It remains for us to construct $\nu_{d,m,c}$ satisfying \hyperlink{P1}{(P1)}, \hyperlink{P2}{(P2)} and \hyperlink{P3}{(P3)}. We choose $\nu_{d,m,c}=\mu^{*d}*\nu$ when $m=0$. They obviously satisfy all three properties. From now on, we assume that $m\geq 1$.

Let $\widetilde{\nu}$ be an admissible lift of $\nu$ supported on curves of length at least $e^{-l}$.  Condition \eqref{eqn:unif.C2.norm} implies that for any $d\in\{0,\dots,p_0-1\}$, $\mu^{*d}*\widetilde{\nu}$ is supported on curves of length at least $e^{-l-p_0\C{C'}{0}}$. 

To simplify notations, we write
$$\CS{K''}{}=\C{K''}{}(p_0;\overline{\lambda}):=\C{K}{1}(p_0;\overline{\lambda})(\C{K'}{}+1)e^{8p_0\C{C'}{0}}$$
and
$$\CS{L''}{}=\C{L''}{}(p_0,\overline{\lambda}):=\C{K}{2}(p_0;\overline{\lambda})(\C{K'}{}+1)(\C{L'}{}+1)e^{11p_0\C{C'}{0}},$$
where $\C{K}{1}(p_0;\overline{\lambda})$ and $\C{K}{2}(p_0;\overline{\lambda})$ are given by \Cref{lem:curv.est} and  \Cref{lem:density.est}.

For all  $k\in\ZZ_+$, we define
\begin{align}\label{eqn:convol.cut.length}
l_k:=\frac{1}{4\C{K''}{}}e^{-7p_0\C{C'}{0}-(11p_0k\C{C'}{0}\eta+l+p_0\C{C'}{0})}.
\end{align}
Given any  $\bw=(\sigma_1,\dots,\sigma_m)\in\cW_m(\eta)$ and any $k\in\{0,\dots,m\}$, we define
\begin{align}\label{eqn:part.of.part}\widetilde{\nu}_{d,m,c;\bw,k}:=\begin{cases}\displaystyle \mu_{(l_{m-k+1})}^{*p_0,\sigma_{m-k+1},\eta}*\dots,*\mu_{(l_{m})}^{*p_0,\sigma_{m},\eta}*(\mu^{*d}*\widetilde{\nu}),~&\text{if }k\geq 1,\\
\displaystyle \mu^{*d}*\widetilde{\nu},&\text{if }k=0
\end{cases}
\end{align}
The measure $\nu_{d,m,c}$ is defined as
\begin{align}\label{eqn:part}
\widetilde{\nu}_{d,m,c}:=\sum_{\bw=(\sigma_1,\dots,\sigma_m)\in\cW_m(\eta)}\widetilde{\nu}_{d,m,c;\bw,m}\text{ and } {\nu}_{d,m,c}=\sP_*\widetilde{\nu}_{d,m,c}
\end{align}
It remains for us to verify properties \hyperlink{P1}{(P1)}, \hyperlink{P2}{(P2)} and \hyperlink{P3}{(P3)} for the above construction of $\nu_{d,m,c}$.

\textbf{Verifying }\hyperlink{P1}{(P1)}\textbf{ for }$\nu_{d,m,c}$: We would like to apply \Cref{lem:good.count-2} to $\widetilde{\nu}_{d,m,c}$. We first verify the hypothesis of \Cref{lem:good.count-2} for $\widetilde{\nu}_{d,m,c}$ as follows:
\begin{enumerate}
\item $\mu^{*d}*\widetilde{\nu}$ is $(\C{K'}{},\C{L'}{})$-admissible,  where the definitions of $\C{K'}{}$ and $\C{L'}{}$ shortly after \eqref{eqn:p**cond.2}.
\item $0<\eta<\min\{1/\C{\beta}{1},(\lambda-\overline{\lambda})/\C{\beta}{1},1/3\}$ and $p_0\geq \C{p}{*}(\eta,\overline{\lambda},c)$, where $\C{p}{*}(\eta, \overline{\lambda}, c)$ is given in \Cref{lem:good.count-2} (see also the assumptions of \Cref{thm:main.detailed}).
\item $l_k\leq (4\C{K}{1}(p_0;\overline{\lambda})(\C{K'}{1}+1)e^{8kp_0\eta\C{C'}{0}})^{-1}e^{-7n\C{C'}{0}}$. This follows from \eqref{eqn:convol.cut.length}, which relies on the definition of $\C{K''}{}$ shortly before \eqref{eqn:convol.cut.length}.
\end{enumerate}
Then by \eqref{eqn:part.of.part}, \eqref{eqn:part} and \Cref{lem:good.count-2}, we have $\overline{\nu}_{d,m,c}\geq (1-c)\mu^{*d}*\overline{\nu}(\AA)=(1-c)\overline{\nu}(\AA)$, which is equivalent to \hyperlink{P1}{(P1)}.

\textbf{Verifying }\hyperlink{P2}{(P2)}\textbf{ for }$\nu_{d,m,c}$: By \eqref{eqn:goodbad.decomp}, \eqref{eqn:part.of.part} and \eqref{eqn:part}, we have 
\begin{align*}
\mu^{*(d+mp_0)}*\nu-\nu_{d,m,c}=&\sP_*(\mu^{*(d+mp_0)}*\widetilde{\nu}-\widetilde{\nu}_{d,m,c})\\
=&\sP_*\left(\sum_{(\sigma_1,\dots,\sigma_m)\in\cW_m\setminus\cW_m(\eta)}\mu^{*p_0,\sigma_1,\eta}_{(l_1)}*\dots\mu^{*p_0,\sigma_m,\eta}_{(l_m)}*(\mu^d*\widetilde{\nu})\right).
\end{align*}
Following the remark after \Cref{dfn:good.convolution}, $\nu_{d,m,c}$ and the above are finite measures or equal to zero. This verifies \hyperlink{P2}{(P2)} for the measure $\nu_{d,m,c}$.

\textbf{Verifying }\hyperlink{P3}{(P3)}\textbf{ for }$\nu_{d,m,c}$: This is the most techincal part of the proof. We split the verification into the three steps. Below is a summary.
\begin{itemize}
\item \emph{Step 1}: For any $k\in\{0,\dots,m-1\}$ and for any $\bw\in\cW_m(\eta)$, we estimate $\left\|\widetilde{\nu}_{d,m,c;\bw,k+1}\right\|^2_{\rho e^{-13p_0(m-k-1)\C{C'}{0}\eta}}$ from the above by using $\left\|\widetilde{\nu}_{d,m,c;\bw,k}\right\|^2_{\rho e^{-13p_0(m-k)\C{C'}{0}\eta}}$ with $0<\rho\ll 1$.
\item \emph{Step 2}:  Combine the estimates in \emph{Step 1} to give an estimate of $\left\|\widetilde{\nu}_{d,m,c;\bw,m}\right\|_{\rho}^2$ from the above.
\item \emph{Step 3}: Estimate $\left\|\widetilde{\nu}_{d,m,c}\right\|_{\rho}^2$ from the above using \emph{Step 2}.
\end{itemize}

\emph{Step 1}: We compare $\left\|\widetilde{\nu}_{d,m,c;\bw,k+1}\right\|^2_{\rho e^{-13p_0(m-k-1)\C{C'}{0}\eta}}$ and $\left\|\widetilde{\nu}_{d,m,c;\bw,k}\right\|^2_{\rho e^{-13p_0(m-k)\C{C'}{0}\eta}}$ for any $k\in\{0,\dots,m-1\}$, for any $\bw\in\cW_m(\eta)$ and for any $0<\rho\ll 1$. The main tools are \Cref{lem:stupid.norm} and \Cref{lem:LY}, depending on the value of $\sigma_{m-k}$.

\textbf{Case 1}: When $\sigma_{m-k}=\bad$, by \Cref{lem:stupid.norm} and the assumptions that $p_0\geq \frac{|\ln\C{C}{4}|}{7\C{C'}{0}}$, we have 
\begin{align}\label{eqn:bad.interation}
&\left\|\widetilde{\nu}_{d,m,c;\bw,k+1}\right\|^2_{\rho e^{-13p_0(m-k-1)\C{C'}{0}\eta}}\nonumber\\
=&\left\|\mu^{*p_0,\bad,\eta}_{(l_{m-k})}*\widetilde{\nu}_{d,m,c;\bw,k}\right\|^2_{\rho e^{-13p_0(m-k-1)\C{C'}{0}\eta}}\nonumber\\
=&\left\|\mu^{*p_0}*\widetilde{\nu}_{d,m,c;\bw,k}\right\|^2_{\rho e^{-13p_0(m-k-1)\C{C'}{0}\eta}}\nonumber\\
\leq&\left(\int_{\Sigma_{p_0}}\left\|\left(f^{p_0}_{\omega^{p_0}}\right)_{\#}\widetilde{\nu}_{d,m,c;\bw,k}\right\|_{\rho e^{-13p_0(m-k-1)\C{C'}{0}\eta}}d\mu^{p_0}(\omega^{p_0})\right)^2\nonumber\\
\leq&\C{C}{4}e^{6\C{C'}{0}p_0}\left\|\widetilde{\nu}_{d,m,c;\bw,k}\right\|^2_{\rho e^{-13p_0(m-k)\C{C'}{0}\eta}}\leq  e^{7\C{C'}{0}p_0}\left\|\widetilde{\nu}_{d,m,c;\bw,k}\right\|^2_{\rho e^{-13p_0(m-k)\C{C'}{0}\eta}}
\end{align}

\textbf{Case 2}: When $\sigma_{m-k}=\good$, we choose 
\begin{align}\label{eqn:choice.of.C.LY}
\begin{split}
&C=11(m-k)p_0\eta\C{C'}{0}+l+p_0\C{C'}{0}\\
\text{ and }&\widetilde{C}=\ln(4)+13(m-k)p_0\eta\C{C'}{0}+l+2p_0\C{C'}{0}.
\end{split}
\end{align}
By \eqref{eqn:choice.of.C.LY} and the definition of $\C{\rho}{1}$ in \Cref{lem:LY}, we have 
\begin{align}\label{eqn:rho.1.LY}
\begin{split}
\C{\rho}{1}(p_0,\C{K''}{},\C{L''}{},\widetilde{C})=&\frac{1}{16\C{K''}{}\C{L''}{}}e^{-\widetilde{C}-8p_0\C{C'}{0}}\\
=&\frac{1}{64\C{K''}{}\C{L''}{}}e^{-8p_0\C{C'}{0}-(13(m-k)p_0\eta\C{C'}{0}+l+2p_0\C{C'}{0})}.
\end{split}
\end{align}
We also define
\begin{align}\label{eqn:rho.2.LY}
\CS{\rho}{2}(p_0,l):=({256\C{K''}{}\C{L''}{}})^{-1}e^{-13\C{C'}{0}p_0-l}.
\end{align}
For any $\rho\in(0,\C{\rho}{2})$, we would like to apply \Cref{lem:LY} in the following setting:
\begin{itemize}
\item In the statement of \Cref{lem:LY}, $n$, $K$, $L$, $\widetilde{\nu}$, $l$, $\rho$ and $\theta$ are chosen to be $p_0$, $\C{K''}{}$, $\C{L''}{}$, $\widetilde{\nu}_{d,m,c;\bw,k}$, $l_{m-k}$, $\rho e^{-13p_0(m-k-1)\C{C'}{0}\eta}$ and $e^{-13p_0\C{C'}{0}\eta}$ respectively.
\item $C,\widetilde{C}$ are chosen in \eqref{eqn:choice.of.C.LY}.
\end{itemize}  
The hypothesis of \Cref{lem:LY} is verified in the following items (1)-(5):
\begin{enumerate}
\item For any $\bw=(\sigma_1,\dots,\sigma_m)\in\cW_m(\eta)$ and for any $k\in\{0,\dots,m-1\}$, the measures $\widetilde{\nu}_{d,m,c;\bw,k}$ and $\widetilde{\nu}_{d,m,c;\bw,k+1}$ are $(\C{K''}{}e^{8(m-k)p_0\eta\C{C'}{0}},\C{L''}{}e^{11(m-k)p_0\eta\C{C'}{0}})$-admissible. (A straightforward corollary of this is that the measures $\widetilde{\nu}_{d,m,c;\bw,k}$ and $\widetilde{\nu}_{d,m,c;\bw,k+1}$ are $(\C{K''}{}e^{C},\C{L''}{}e^{C})$-admissible.)

\textbf{Proof of this fact}: We would like to apply \Cref{cor:curv.length.est} to $\widetilde{\nu}_{d,m,c;\bw,k}$ and $\widetilde{\nu}_{d,m,c;\bw,k+1}$. We verify the hypothesis of \Cref{cor:curv.length.est} as follows:
\begin{itemize}
\item $\mu^{*d}*\widetilde{\nu}$ is $(\C{K'}{},\C{L'}{})$-admissible. (See the definitions of $\C{K'}{}$ and $\C{L'}{}$ shortly after \eqref{eqn:p**cond.2}.)
\item $0<\eta<\min\{1/\C{\beta}{1},(\lambda-\overline{\lambda})/\C{\beta}{1},1/3\}$ and $p_0\geq \C{p}{*}(\eta,\overline{\lambda},c)\geq \max\{\C{n'}{1}(\eta,\overline{\lambda}),\C{C}{0}/(\C{C'}{0}\eta)\}$. (See the assumptions of \Cref{thm:main.detailed} and the definition of $\C{p}{*}(\eta,\overline{\lambda},c)$ in \Cref{lem:good.count-2}.)
\end{itemize}
By \eqref{eqn:part.of.part}, \Cref{cor:curv.length.est} and the definitions of $\C{K''}{},\C{L''}{}$ shortly before \eqref{eqn:convol.cut.length}, $\widetilde{\nu}_{d,m,c;\bw,k}$ and $\widetilde{\nu}_{d,m,c;\bw,k+1}$ are $(\C{K''}{}e^{8(m-k)p_0\eta\C{C'}{0}},\C{L''}{}e^{11(m-k)p_0\eta\C{C'}{0}})$-admissible.
\item $l_{m-k}=\frac{1}{4\C{K''}{}}e^{-7p_0\C{C'}{0}-C}$. This follows directly from \eqref{eqn:convol.cut.length} and \eqref{eqn:choice.of.C.LY}.
\item For any $\bw=(\sigma_1,\dots,\sigma_m)\in\cW_m(\eta)$ and for any $k\in\{0,\dots,m-1\}$, the measure $\widetilde{\nu}_{d,m,c;\bw,k}$ is supported on curves with length at least $12\C{\rho}{1}(p_0,\C{K''}{},\C{L''}{},\widetilde{C})e^{p_0\C{C'}{0}}$.

\textbf{Proof of this fact}: Recall that $\mu^{*d}*\widetilde{\nu}$ is supported on curves with length at least $e^{-l-p_0\C{C'}{0}}$. By \eqref{eqn:part.of.part} and \Cref{cor:curv.length.est}, we know that $\widetilde{\nu}_{d,m,c;\bw,k}$ is supported on curves with length at least 
$$\min\left\{e^{-l-p_0\C{C'}{0}}e^{kp_0\overline{\lambda}-2\C{C'}{0}mp_0\eta},\min_{1\leq j\leq k}\left\{l_{m-j+1}e^{(k-j)p_0\overline{\lambda}-2\C{C'}{0}(m-j)p_0\eta}\right\}\right\}.$$
Recall that $\C{K''}{},\C{L''}{}\geq 1$ and that $\eta\leq  \min\left\{\frac{\overline{\lambda}}{13\C{C'}{0}},\frac{1}{13}\right\}$. By \eqref{eqn:convol.cut.length} and \eqref{eqn:rho.1.LY}, we have 
\begin{align*}
&\min\left\{e^{-l-p_0\C{C'}{0}}e^{kp_0\overline{\lambda}-2\C{C'}{0}mp_0\eta},\min_{1\leq j\leq k}\left\{l_{m-j+1}e^{(k-j)p_0\overline{\lambda}-2\C{C'}{0}(m-j)p_0\eta}\right\}\right\}\\
>&\min\left\{e^{-l-p_0\C{C'}{0}}e^{kp_0\overline{\lambda}-2\C{C'}{0}mp_0\eta},\min_{1\leq j\leq k}\left\{l_{m-j+1}e^{(k-j)p_0\overline{\lambda}-2\C{C'}{0}(m-j+1)p_0\eta}\right\}\right\}\\
= &\min\left\{e^{-l-p_0\C{C'}{0}}e^{kp_0\overline{\lambda}-2\C{C'}{0}mp_0\eta},\min_{1\leq j\leq k}\left\{\frac{1}{4\C{K''}{}}e^{-8p_0\C{C'}{0}-l+(k-j)p_0\overline{\lambda}-13\C{C'}{0}(m-j+1)p_0\eta}\right\}\right\}\\
\geq &\min\left\{e^{-l-p_0\C{C'}{0}}e^{2k\C{C'}{0}p_0\eta-2\C{C'}{0}mp_0\eta},\min_{1\leq j\leq k}\left\{\frac{1}{4\C{K''}{}}e^{-8p_0\C{C'}{0}-l+13(k-j)\C{C'}{0}p_0\eta-13\C{C'}{0}(m-j+1)p_0\eta}\right\}\right\}\\
= &\min\left\{e^{-l-p_0\C{C'}{0}}e^{-2\C{C'}{0}(m-k)p_0\eta},\min_{1\leq j\leq k}\left\{\frac{1}{4\C{K''}{}}e^{-8p_0\C{C'}{0}-l-13\C{C'}{0}(m-k+1)p_0\eta}\right\}\right\}\\
= &\frac{1}{4\C{K''}{}}e^{-7p_0\C{C'}{0}-(13\C{C'}{0}(m-k+1)p_0\eta+l+p_0\C{C'}{0})}\\
\geq &\frac{1}{4\C{K''}{}}e^{-7p_0\C{C'}{0}-(13\C{C'}{0}(m-k)p_0\eta+l+2p_0\C{C'}{0})}\\
> &\frac{12}{64\C{K''}{}\C{L''}{}}e^{-7p_0\C{C'}{0}-(13p_0(m-k)\C{C'}{0}\eta+l+2p_0\C{C'}{0})}=12\C{\rho}{1}(p_0,\C{K''}{},\C{L''}{},\widetilde{C})e^{p_0\C{C'}{0}}.
\end{align*}
Hence, $\widetilde{\nu}_{d,m,c;\bw,k}$ is supported on curves of length at least $12\C{\rho}{1}(p_0,\C{K''}{},\C{L''}{},\widetilde{C})e^{p_0\C{C'}{0}}$.
\item For any $\rho\in(0,\C{\rho}{2})$, by \eqref{eqn:rho.1.LY}, \eqref{eqn:rho.2.LY}, and the assumption that $\eta<1/13$, we have 
\begin{align*}
0<\rho e^{-13p_0(m-k-1)\C{C'}{0}\eta}
<&\C{\rho}{2} e^{-13p_0(m-k-1)\C{C'}{0}\eta}\\
=&\frac{1}{256\C{K''}{}\C{L''}{}}e^{-13\C{C'}{0}p_0-l-13p_0(m-k-1)\C{C'}{0}\eta}\\
\leq&\frac{1}{256\C{K''}{}\C{L''}{}}e^{-12\C{C'}{0}p_0-l-13p_0(m-k)\C{C'}{0}\eta}\\
=&\frac{1}{4}\cdot \frac{1}{64\C{K''}{}\C{L''}{}}e^{-8p_0\C{C'}{0}-(13(m-k)p_0\eta\C{C'}{0}+l+2p_0\C{C'}{0})}e^{-2p_0\C{C'}{0}}\\
=&\frac{1}{4}\C{\rho}{1}(p_0,\C{K''}{},\C{L''}{},\widetilde{C})e^{-2p_0\C{C'}{0}}.
\end{align*}
\item $p_0\geq\C{n}{3}(\overline{\lambda})$. (See the assumptions of \Cref{thm:main.detailed}).
\end{enumerate}
Therefore, \Cref{lem:LY} implies that for any $\rho\in(0,\C{\rho}{2})$, we have 
\begin{align}\label{eqn:good.interation.unfinished}
&\left\|\widetilde{\nu}_{d,m,c;\bw,k+1}\right\|^2_{\rho e^{-13p_0(m-k-1)\C{C'}{0}\eta}}\nonumber\\
=&\left\|\mu^{*p_0,\good,\eta}_{(l_{m-k})}*\widetilde{\nu}_{d,m,c;\bw,k}\right\|^2_{\rho e^{-13p_0(m-k-1)\C{C'}{0}\eta}}\nonumber\\
\begin{split}
\leq&\C{C}{8}\cdot(1+p_0)e^{-(\widehat{\lambda}\C{\beta}{1}-6{\C{\epsilon}{0}}-2\eta)p_0}\|\widetilde{\nu}_{d,m,c;\bw,k}\|^2_{\rho e^{-13p_0(m-k)\C{C'}{0}\eta}}\\
&+\C{C}{9}(p_0,K,L)e^{4\widetilde{C}}(\widetilde{\nu}_{d,m,c;\bw,k}(\AA))^2.
\end{split}
\end{align}
Choose
$$\C{C'}{10}(p_0,\C{K''}{},\C{L''}{},l)=256\C{C}{9}(p_0,\C{K''}{},\C{L''}{})e^{4l+8p_0\C{C'}{0}}.$$
Apply \eqref{eqn:p**cond.2}, \eqref{eqn:choice.of.C.LY} and the assumption that $p_0\geq \C{p}{**}(\eta,\overline{\lambda},c)$ to \eqref{eqn:good.interation.unfinished}, for any $\rho\in(0,\C{\rho}{2})$, we have
\begin{align}\label{eqn:good.interation}
&\left\|\widetilde{\nu}_{d,m,c;\bw,k+1}\right\|^2_{\rho e^{-13p_0(m-k-1)\C{C'}{0}\eta}}\nonumber\\
\begin{split}
\leq&e^{-\C{\beta}{2}p_0}\left\|\widetilde{\nu}_{d,m,c;\bw,k}\right\|^2_{\rho e^{-13p_0(m-k)\C{C'}{0}\eta}}\\
&+\C{C'}{10}(p_0,\C{K''}{},\C{L''}{},l)e^{52p_0(m-k)\C{C'}{0}\eta}(\widetilde{\nu}_{d,m,c;\bw,k}(\AA))^2.
\end{split}
\end{align}

\emph{Step 2}: Now we use \eqref{eqn:bad.interation} and \eqref{eqn:good.interation} iteratively to estimate $\|\widetilde{\nu}_{d,m,c;\bw,m}\|_\rho^2$ from the above, where $\bw$ is an arbitrary element in $\cW_m(\eta)$. 
To simplify notations, we define functions $\bC:\{\good,\bad\}\to\RR_+$ and $\bD_m:\{\good,\bad\}\times\{0,\dots, m-1\}$ as   
$$\bC(\sigma)=\begin{cases}
e^{-\C{\beta}{2}p_0},~&\text{if }\sigma=\good,\\
e^{7\C{C'}{0}p_0},~&\text{if }\sigma=\bad
\end{cases}$$
and
$$
\bD_m(\sigma,k)=\begin{cases}
\C{C'}{10}(p_0,\C{K''}{},\C{L''}{},l)e^{52p_0(m-k)\C{C'}{0}\eta}(\widetilde{\nu}_{d,m,c;\bw,k}(\AA))^2,~&\text{if }\sigma=\good,\\
0,~&\text{if }\sigma=\bad.
\end{cases}
$$
Then for any $\bw=(\sigma_1,\dots,\sigma_m)\in\cW_m(\eta)$ and for any $k\in\{0,\dots,m-1\}$, \eqref{eqn:good.interation} and \eqref{eqn:bad.interation} can be simplified to
\begin{align}\label{eqn:iteration}
\begin{split}
&\left\|\widetilde{\nu}_{d,m,c;\bw,k+1}\right\|^2_{\rho e^{-13p_0(m-k-1)\C{C'}{0}\eta}}\\
\leq& \bC(\sigma_{m-k})\left\|\widetilde{\nu}_{d,m,c;\bw,k}\right\|^2_{\rho e^{-13p_0(m-k)\C{C'}{0}\eta}}+\bD_m(\sigma_{m-k},k).\\
\end{split}
\end{align}
Iterate \eqref{eqn:iteration}, we have 
\begin{align}\label{eqn:part.est-1}
\left\|\widetilde{\nu}_{d,m,c;\bw,m}\right\|^2_{\rho}\leq &\left(\prod_{k=1}^{m}\bC(\sigma_k)\right)\|\mu^{*d}*\widetilde{\nu}\|^2_{\rho e^{-13p_0m\C{C'}{0}\eta}}\nonumber\\
&+\sum_{k=0}^{m-1}\left(\bD_m(\sigma_{m-k},k)\prod_{j=1}^{m-k-1}\bC(\sigma_{j})\right).
\end{align}
Notice that for any $\bw\in\cW_m(\eta)$ and for any $k\in\{0,\dots,m-1\}$, by \eqref{eqn:goodbad.eta} and \eqref{eqn:eta.cor-2}, we have
\begin{align}\label{eqn:part.coeff.est-1}
\prod_{j=1}^{k}\bC(\sigma_{j})\leq e^{(7\C{C'}{0}p_0\eta-(1-\eta)\C{\beta}{2}p_0)k}\leq e^{-\frac{1}{2}\C{\beta}{2}p_0k}.
\end{align} 
Apply \eqref{eqn:part.coeff.est-1} to \eqref{eqn:part.est-1}, we obtain
\begin{align*}
&\left\|\widetilde{\nu}_{d,m,c;\bw,m}\right\|^2_{\rho}\nonumber\\
\leq &e^{-\frac{1}{2}\C{\beta}{2}p_0m}\|\mu^{*d}*\widetilde{\nu}\|^2_{\rho e^{-13p_0m\C{C'}{0}\eta}}+\sum_{k=0}^{m-1}\left(\bD_m(\sigma_{m-k},k)e^{-\frac{1}{2}\C{\beta}{2}p_0(m-k-1)}\right).
\end{align*}
Hence
\begin{align}\label{eqn:part.est-2}
&\left\|\widetilde{\nu}_{d,m,c;\bw,m}\right\|_{\rho}\nonumber\\
\leq &e^{-\frac{1}{4}\C{\beta}{2}p_0m}\|\mu^{*d}*\widetilde{\nu}\|_{\rho e^{-13p_0m\C{C'}{0}\eta}}+\sum_{k=0}^{m-1}\left(\bD_m(\sigma_{m-k},k)e^{-\frac{1}{2}\C{\beta}{2}p_0(m-k-1)}\right)^{1/2}.
\end{align}

\emph{Step 3}: The rest of the proof uses \eqref{eqn:part.est-2} to varify the desired estimation of $\|\widetilde{\nu}_{d,m,c}\|_{\rho}$ in \hyperlink{P3}{(P3)}. In view of the notations in \Cref{subsection:notation.LY}, \hyperlink{P3}{(P3)} is equivalent to
$$\displaystyle \|\widetilde{\nu}_{d,m,c}\|_{\rho}\leq \C{C'''}{10}(p_0)e^{-\frac{1}{6}\C{\beta}{2}p_0m}\|\mu^{*d}*\widetilde{\nu}\|_{\rho e^{-13p_0m\C{C'}{0}\eta}}+\C{C''''}{10}(p_0,\C{K''}{},\C{L''}{},l)\widetilde{\nu}(\AA).$$
By \eqref{eqn:part.est-2} and \eqref{eqn:part}, we have
\begin{align}\label{eqn:part.est.3}
&\left\|\widetilde{\nu}_{d,m,c}\right\|_{\rho}\nonumber\\
\leq& \left\|\sum_{\bw=(\sigma_1,\dots,\sigma_m)\in\cW_m(\eta)}\widetilde{\nu}_{d,m,c;\bw,m}\right\|_{\rho}\nonumber\\
\leq&\sum_{\bw=(\sigma_1,\dots,\sigma_m)\in\cW_m(\eta)}\left\|\widetilde{\nu}_{d,m,c;\bw,m}\right\|_{\rho}\nonumber\\
\leq &|\cW_m(\eta)|e^{-\frac{1}{4}\C{\beta}{2}p_0m}\|\mu^{*d}*\widetilde{\nu}\|_{\rho e^{-13p_0m\C{C'}{0}\eta}}\nonumber\\
&+\sum_{\bw=(\sigma_1,\dots,\sigma_m)\in\cW_m(\eta)}\sum_{k=0}^{m-1}\left(\bD_m(\sigma_{m-k},k)e^{-\frac{1}{2}\C{\beta}{2}p_0(m-k-1)}\right)^{1/2}\nonumber\\
=&|\cW_m(\eta)|e^{-\frac{1}{4}\C{\beta}{2}p_0m}\|\mu^{*d}*\widetilde{\nu}\|_{\rho e^{-13p_0m\C{C'}{0}\eta}}\\
&+\C{C''}{10}(p_0,\C{K''}{},\C{L''}{},l)\sum_{k=0}^{m-1}\left(\sum_{\bw=(\sigma_1,\dots,\sigma_m)\in\cW_m(\eta)}\widetilde{\nu}_{d,m,c;\bw,k}(\AA)\right)e^{\left(26\C{C'}{0}\eta-\frac{1}{4}\C{\beta}{2}\right)p_0(m-k)}\nonumber
\end{align} 
where $\CS{C''}{10}(p_0,\C{K''}{},\C{L''}{},l)=(\C{C'}{10}(p_0,\C{K''}{},\C{L''}{},l))^{\frac{1}{2}}e^{\frac{1}{4}\C{\beta}{2}p_0}$.
By \eqref{eqn:goodbad.eta}, for any $\bw=(\sigma_1,\dots,\sigma_m)\in\cW_m(\eta)$ and for any $k\in\{0,\dots,m-1\}$, we have $(\sigma_1,\dots,\sigma_{m-k})\in\cW_{m-k}(\eta)$. Hence by \eqref{eqn:goodbad.decomp} and \eqref{eqn:part.of.part}, we have
\begin{align}\label{eqn:part.est.4}
&\sum_{\bw=(\sigma_1,\dots,\sigma_m)\in\cW_m(\eta)}\widetilde{\nu}_{d,m,c;\bw,k}(\AA)\nonumber\\
\leq&\sum_{\bw=(\sigma_1,\dots,\sigma_m)\in\cW_m(\eta)}\mu_{(l_{m-k+1})}^{*p_0,\sigma_{m-k+1},\eta}*\dots,*\mu_{(l_{m})}^{*p_0,\sigma_{m},\eta}*(\mu^{*d}*\widetilde{\nu})(\AA)\nonumber\\
\leq &|\cW_{m-k}(\eta)|\sum_{(\sigma_{m-k+1},\dots,\sigma_m)\in\cW_k}\mu_{(l_{m-k+1})}^{*p_0,\sigma_{m-k+1},\eta}*\dots,*\mu_{(l_{m})}^{*p_0,\sigma_{m},\eta}*(\mu^{*d}*\widetilde{\nu})(\AA)\nonumber\\
\leq& |\cW_{m-k}(\eta)|\mu^{*mp_0}(\mu^{*d}*\widetilde{\nu})(\AA)=|\cW_{m-k}(\eta)|\widetilde{\nu}(\AA).
\end{align}
Apply \eqref{eqn:part.est.4} to \eqref{eqn:part.est.3}, we have 
\begin{align}\label{eqn:part.est.4.5}
&\|\widetilde{\nu}_{d,m,c}\|_\rho\nonumber\\
\begin{split}
\leq & |\cW_m(\eta)|e^{-\frac{1}{4}\C{\beta}{2}p_0m}\|\mu^{*d}*\widetilde{\nu}\|_{\rho e^{-13p_0m\C{C'}{0}\eta}}\\
&+\C{C''}{10}(p_0,\C{K''}{},\C{L''}{},l)\sum_{k=0}^{m-1}|\cW_{m-k}(\eta)|\widetilde{\nu}(\AA)e^{\left(26\C{C'}{0}\eta-\frac{1}{4}\C{\beta}{2}\right)p_0(m-k)}.
\end{split}
\end{align}
Observe the following
\begin{itemize}
\item For any $k\in\{0,\dots,m-1\}$, by \eqref{eqn:goodbad.eta} and \Cref{lem:stirling-1}, we have 
\begin{align}\label{eqn:part.est.5}
|\cW_{m-k}(\eta)|\leq \sum_{j=0}^{[(m-k)\eta]}{m-k\choose j}\leq \C{C''}{4}(m-k)\eta \left(\frac{1}{\eta^\eta(1-\eta)^{1-\eta}}\right)^{m-k}
\end{align}
\item By the assumptions on $\eta$ and the definition of $\C{\beta}{2}$, we have $26\C{C'}{0}\eta<\widehat{\lambda}\C{\beta}{1}/128=\C{\beta}{2}/8$. Hence, by the assumption that $\eta<1/3$ and the assumption that 
$$p_0\geq \C{p}{**}>\frac{-24(\eta\ln(\eta)+(1-\eta)\ln(1-\eta))}{\C{\beta}{2}},$$
we have 
\begin{align}\label{eqn:part.est.6}
 &\left(\frac{1}{\eta^\eta(1-\eta)^{1-\eta}}\right)^{m-k}e^{\left(26\C{C'}{0}\eta-\frac{1}{4}\C{\beta}{2}\right)p_0(m-k)}\nonumber\\
 \leq& \left(\frac{1}{\eta^\eta(1-\eta)^{1-\eta}}\right)^{m-k}e^{\left(-\frac{1}{8}\C{\beta}{2}\right)p_0(m-k)}\nonumber\\
 \leq& e^{\frac{1}{24}\C{\beta}{2}p_0(m-k)}e^{-\frac{1}{8}\C{\beta}{2}p_0(m-k)}=e^{-\frac{1}{12}\C{\beta}{2}p_0(m-k)}.
\end{align}
\item Recall that $\eta<1/13$. Then there exists a constant $\C{C'''}{10}>\{1,\C{C}{4}\}$ (depending only on $\C{\beta}{2}$) such that for any $k\in\ZZ_{\geq 0}$, we have
\begin{align}\label{eqn:part.est.7}
\C{C''}{4}k\eta<\C{C''}{4}k<\CS{C'''}{10}e^{\frac{1}{24}\C{\beta}{2}k}\leq\C{C'''}{10}e^{\frac{1}{24}\C{\beta}{2}p_0k}.
\end{align}
\end{itemize}
Apply \eqref{eqn:part.est.5}, \eqref{eqn:part.est.6} and \eqref{eqn:part.est.7} to \eqref{eqn:part.est.4.5}, we have 
\begin{align*}
&\|\nu_{d,m,c}\|_\rho\\
\leq &\C{C'''}{10}e^{-\frac{1}{6}\C{\beta}{2}p_0m}\|\mu^{*d}*\widetilde{\nu}\|_{\rho e^{-13p_0m\C{C'}{0}\eta}}\\
&+\C{C''}{10}(p_0,\C{K''}{},\C{L''}{},l)\C{C'''}{10}(p_0)\sum_{k=0}^{m-1}e^{-\frac{1}{24}\C{\beta}{2}p_0(m-k)}\widetilde{\nu}(\AA)\\
\leq&\C{C'''}{10}e^{-\frac{1}{6}\C{\beta}{2}p_0m}\|\mu^{*d}*\widetilde{\nu}\|_{\rho e^{-13p_0m\C{C'}{0}\eta}}+\C{C''''}{10}(p_0,\C{K''}{},\C{L''}{},l)\widetilde{\nu}(\AA),
\end{align*}
where 
$$\CS{C''''}{10}(p_0,\C{K''}{},\C{L''}{},l)=\frac{\C{C''}{10}(p_0,\C{K''}{},\C{L''}{},l)\C{C'''}{10}}{1-e^{-\frac{1}{24}\C{\beta}{2}p_0}}.$$
This verifies \hyperlink{P3}{(P3)} and hence completes the proof. 
\end{proof}

\section{Proof of \Cref{thm:maintechnicaltheorem}}\label{sec.mainthm}

\begin{proof}[Proof of \Cref{thm:maintechnicaltheorem}]
Let $K,L$ be constants.  Consider $\nu$ any $(K,L)$ admissible measure supported on a single curve of length at least $e^{-l}$, for some $l\geq 0$.  Consider 
\[
\overline{\nu}_n:= \displaystyle \frac{1}{n} \sum_{k=0}^{n-1} \mu^{*k}* \nu.
\]

Then, \Cref{thm:main.detailed} implies that for any $c \in (0,1)$, any limit measure $\nu_{\infty}$  of $\overline{\nu}_n$ has a part $\nu_{\infty,c}$ which is absolutely continuous with respect to $m$.  Recall that any such limit $\nu_{\infty}$ is a $\mu$-stationary measure. 

Observe that the UEF property implies that every stationary measure has one positive Lyapunov exponent.  Every hyperbolic $\mu$-stationary measure that is absolutely continuous with respect to $m$ is SRB.  Indeed, one does not need the measure to be hyperbolic, but to have a positive Lyapunov exponent, see \cite{LYRandom}.  This implies the existence of $\mu$-stationary SRB measures. 

Suppose that $\nu$ is an SRB $\mu$-stationary measure.  By the SRB property, there exists a family of conditional measures $\{\nu^u_{(\omega^-,x)}\}_{(\omega^-,x) \in \Sigma^- \times \TT^2}$ each of which is supported on the local unstable manifold $W^u_{\omega^-}(x)$ and it is absolutely continuous with respect to the arclength measure of the corresponding unstable manifold.  Moreover, we know that the density of such conditional measures is $\log$-Lipschitz.  Take a typical conditional measure $\nu^u_{(\omega^-,x)}$.  We have that
\[
\displaystyle  \lim_{n\to +\infty} \frac{1}{n} \sum_{k=0}^{n-1} \mu^{*k}_* \nu^u_{(\omega^-,x)} = \nu.
\]
\Cref{thm:main.detailed} implies that $\nu$ is absolutely continuous.  By \Cref{lem:uniformsizenbd}, there exists $\rho>0$ such that any ergodic SRB measure $\nu'$, the support of no other $\mu$-stationary SRB measure $\widehat{\nu}$ can intersect the $\rho$-neighborhood of the support of $\nu'$. By compactness of $\TT^2$, we obtain that there are finitely many SRB measures.  \qedhere

\end{proof}

\section{Proof of the other main theorems } \label{Sec.proofmaintheorems}

In this section, we will explain how the main theorems stated in the introduction follow from \Cref{thm:maintechnicaltheorem}.

\subsection{Some preliminary results and estimates}

Let us state the main measure rigidity result from \cite{Brown-Hertz}

\begin{theorem}[\cite{Brown-Hertz}]\label{thm:brownhertz}
Let $S$ be a closed surface and $\mu$ be a probability measure in $\mathrm{Diff}^2(S)$ such that 
\[
\displaystyle \int \ln \max\{\|f\|_{C^2}, \|f^{-1}\|_{C^2}\} d\mu(f) <+\infty.
\]
Suppose that $\nu$ is an ergodic hyperbolic $\mu$-stationary probability measure with Lyapunov exponents $\lambda^-(\nu) < 0 < \lambda^+(\nu)$.  Then, either
\begin{enumerate}
\item $\nu$ is atomic; or
\item for $\nu$-almost every $x$, the stable oseledets direction $E^{s}_\omega(x)$ is non-random, i.e., it does not depend on the choice of $\omega$; or
\item $\nu$ is SRB. 
\end{enumerate}
\end{theorem}

\Cref{thm:brownhertz} implies that under the assumption of UEF, taking $\mathcal{C} = T\TT$,  ergodic stationary measures are either atomic or SRB.  Indeed, this follows from the following characterization of the UEF property, among volume preserving random dynamical systems (or near volume preserving). 

\begin{theorem}[\cite{chung}]
\label{thm:equivalentconditionUE}
Let $\mu$ be a probability measure supported on $\mathrm{Diff}^2_m(S)$ such that
\[
\displaystyle \int \ln\max\{\|f\|_{C^1}, \|f^{-1}\|_{C^1}\} d\mu(f) < +\infty.
\]
The measure $\mu$ is UEF if and only if every $\mu$-stationary measure $\nu$ is hyperbolic and the stable direction of $\nu$ is random.
\end{theorem}
Applying \Cref{thm:equivalentconditionUE} for the past, we obtain that $\mu$ is UEP if and only if every $\mu^{-1}$-stationary measure $\nu$ is hyperbolic and the unstable direction is random (it is actually the stable direction of the backwards dynamical systems). 

For convenience, in this section, we will write $\mu^+$ the Bernoulli measure generated by $\mu$ in $\Sigma^+$ and $\mu^-$ the Bernoulli measure generated by $\mu$ in $\Sigma^-$.  The proof of the next few lemmas are essentially contained in \cite{DK, Zhang, JonDimaExp} (see \cite{LOP} for similar estimates for endomorphisms).

\begin{lem}\label{lem:existencestable}
Suppose that $\mathcal{U}' \subset \mathcal{U}$ are $C^2$-open sets in $\mathrm{Diff}^2(\mathbb{T}^2)$ and $\mu$ a probability measure verifying assumptions \hyperlink{A1}{(\textbf{A1})} - \hyperlink{A6}{(\textbf{A6})}.   Then, for every point $x \in \TT^2$, the following properties hold:
\begin{enumerate}
\item For $\mu^+$-almost every $\omega$, there exists a direction $\widehat{E}^s_\omega(x)$ in $T_x \TT^2$ which contracts exponentially fast under the action of $Df^n_\omega(x)$. Moreover, $Df_\omega^n(x) \widehat{E}^s_\omega(x) = \widehat{E}^s_{\sigma^n(\omega)}(f^n_\omega(x))$. In other words, for $\mu^+$-almost every $\omega$, there is a well defined stable direction on $x$ for $f_\omega$.  
\item For $\mu^-$-almost every $\omega^-$, there exists a direction $\widehat{E}^u_{\omega^-}(x)$ in $T_x\TT^2$ which contracts exponentially fast under the action of $Df^{-n}_{\omega^-}(x)$.  Moreover, $Df^{-n}_{\omega^-}(x) \widehat{E}^u_{\omega^-}(x ) = \widehat{E}^u_{\sigma^{-n}(\omega^-)}(f^{-n}_{\omega^-}(x)).$ In other words, for $\mu^-$-almost every $\omega^-$ there is a well defined unstable direction. 
\end{enumerate}
\end{lem}

\begin{proof}[Sketch of the proof]
The proof is essentially contained in the proof of Proposition $4.8$ from \cite{JonDimaExp}. Let us explain the main steps in this proof.  Given constants $K, \hat{\chi}, \varepsilon>0$, we say that $(\omega, x) \in \Sigma^+ \times \TT^2$ has $(K, \hat{\chi}, \varepsilon)$-subtempered norm if for any $i\geq 1$ and $j>0$
\begin{equation}\label{eq.subtempered}
\|Df_\omega^{i + j}(x)\|  \geq e^{-K+\hat{\chi}i - \varepsilon j } \|Df^j_\omega(x)\|.
\end{equation}

Consider $\overline{\chi}>0$ be as in \Cref{remark:constants}. Proposition $4.7$ in \cite{JonDimaExp} and by the Borel-Cantelli Lemma, for $\mu^+$-almost every $\omega$,  for $\varepsilon>0$ sufficiently small, there exists a constant $K = K(\omega)$,  such that $(\omega,x)$ has $(K, \overline{\chi}, \varepsilon)$-subtempered norm. 

Assumption \hyperlink{A6}{(\textbf{A6})}, \Cref{lem:VAP} and by Borel-Cantelli Lemma, for $\mu^+$-almost every $\omega$,  there exists $n' = n'(\omega,x)$ such that for any $n\geq n'$,
\begin{equation}\label{eq.jac}
\Jac f^n_{\omega^n}(x)\in(e^{-\C{C}{0}-2n\C{\epsilon}{0}}, e^{\C{C}{0}+2n\C{\epsilon}{0}}). 
\end{equation}

Suppose that $(\omega,x)$ verifies \Cref{eq.subtempered} for some large $K$ and small $\varepsilon$, and that verifies \Cref{eq.jac} for every $n\geq n'$.  Let us show that there is a well defined stable direction $\widehat{E}^s_\omega(x)$.  

Write $\omega^n$ to be the first $n$-coordinates of $\omega \in \Sigma^+$.  Recall that, in our notation,  $E^s_{\omega^n}(x)$ is the most contracting direction of $Df^n_\omega(x)$, and that $E^u_{\omega^n}(x)$ is the most expanding direction.  Since they correspond to directions coming from the singular decomposition, they are orthogonal.  Write $s_n$ and $u_n$ unit vectors in $E^s_{\omega^n}(x)$ and $E^u_{\omega^n}(x)$, respectively.  For each $n$, let $\theta_n = \theta_n(\omega,x)$ be defined by 
\[
s_n = \cos(\theta_n)s_{n+1} + \sin(\theta_n) u_{n+1}.  
\]
\begin{claim}
$\theta_n$ converges to $0$ exponentially fast. In particular,  the directions $E^s_{\omega^n}(x)$ forms a Cauchy sequence in the projective space. 
\end{claim}
\begin{proof}
Observe that
\[
\|Df^{n+1}_\omega(x) s_n\|  \geq |\sin(\theta_n)| \|Df^{n+1}_\omega(x)\|,
\]
and
\[
\|Df^{n+1}_\omega(x) s_n\| = \|Df_{\sigma^n(\omega)}(f^n_\omega(x)) Df^n_\omega(x) s_n\| \leq e^{\C{C'}{0}} m(Df^n_\omega(x)). 
\]
Therefore,
\[
|\sin(\theta_n)| \leq e^{\C{C'}{0}} \frac{m(Df^n_\omega(x))}{\|Df^{n+1}_\omega(x)\|} \leq e^{\C{C'}{0} +\C{C}{0} +2K - \overline{\chi}} +  2n(\C{\epsilon}{0}  - 2\overline{\chi}),
\]
which decays to $0$ exponentially fast in $n$.  Hence, the sequence $s_n$ is Cauchy and converges to some vector $s = s(\omega,x)$.
\end{proof}
Let us show that the unit vector $s$ contracts exponentially fast under the action of $Df^n_\omega(x)$.  For each $n\in \mathbb{N}$, write $s = \cos(\alpha_n) s_n + \sin(\alpha_n) u_n$.  Let us estimate $\|Df^n_\omega(x) (\sin(\alpha_n) u_n\|$.  From our previous estimate, we have that
\[
\displaystyle |\sin(\alpha_n)| \leq \sum_{j=1}^{+\infty} |\sin(\theta_{n+j})| \leq e^{\C{C'}{0}} \sum_{j=1}^{+\infty}  \frac{m(Df^{n+j}_\omega(x))}{\|Df^{n+j +1}_\omega(x)\|}.
\]
By $(K, \overline{\chi}, \varepsilon)$-subtemperedness of the norm, we have that
\[
\|Df^{n+j+1}_\omega(x)\| \geq e^{-K +\overline{\chi}(j+1) -\varepsilon (j+1)} \|Df^n_\omega(x)\|.
\]
Hence,
\[
\displaystyle \|Df^n_\omega(x) (\sin(\alpha_n) u_n\| \leq e^{\C{C'}{0}} \sum_{j=1}^{+\infty}  m(Df^{n+j}_\omega(x))e^{K -\overline{\chi}(j+1) +\varepsilon (j+1)}. 
\]
Combining with our estimate on $m(Df^{n+j}_\omega(x))$, we obtain
\[
\displaystyle \|Df^n_\omega(x) (\sin(\alpha_n) u_n\| \leq e^{\C{C'}{0}}   \sum_{j=1}^{+\infty}  e^{\C{C}{0} - n(\overline{\chi} - 2\C{\epsilon}{0} ) + 2 K - (\overline{\chi}- \varepsilon)  - j(2\overline{\chi} - 2\C{\epsilon}{0} - \varepsilon)}.
\]
Observe that this goes to $0$ exponentially fast as $n$ increases.  Since $\|Df^n_\omega(x) (\cos(\alpha_n) s_n)\|$ goes to $0$ exponentially fast as $n$ increases, we conclude that the direction generated by $s$ is a stable direction.  The same argument works for the unstable manifold,  but working with past iterates.  \qedhere 
\end{proof}

Observe that, in the proof of \Cref{lem:existencestable}, we conclude that the directions $E^s_{\omega^n}(x)$ converge exponentially fast to the real stable direction. Similar conclusion holds for the unstable direction. This allows one to adapt the proof of \Cref{lem.cone.trans} to conclude the following lemma. 

\begin{lem}\label{lem.cone.trans2}
Suppose that $\mathcal{U}' \subset \mathcal{U}$ are $C^2$-open sets in $\mathrm{Diff}^2(\mathbb{T}^2)$ and $\mu$ a probability measure verifying assumptions \hyperlink{A1}{(\textbf{A1})} - \hyperlink{A6}{(\textbf{A6})}.  Then, there exist positive constants $\CS{\beta'}{1}\in(0,1)$ and $\CS{C'}{3}>0$ such that for any $\eta\geq e^{-n}$, the following holds:
\begin{enumerate}
\item For any $x\in \TT^2$ and for any $v\in \cC^1_x$, we have
$$\mu^+\left(\left\{\omega\in\Sigma^+: \sphericalangle(\widehat{E}^s_{\omega}(x),v)<\eta\right\}\right)\leq \C{C'}{3} \eta^{\C{\beta'}{1}}$$
\item For any $(x,v)\in T^1\TT^2$, we have
$$\mu^-\left(\left\{\omega^-\in\Sigma^-: \sphericalangle(\widehat{E}^u_{\omega^-}(x),v)<\eta\right\}\right)\leq \C{C'}{3} \eta^{\C{\beta'}{1}}$$
\end{enumerate}
\end{lem}

Let us define a version of Pesin blocks that will be suited for our setting.  Let us first introduce some notation.  Let $x\in \TT^2$.  Given a direction $E$ on $T_x\TT^2$, $\omega \in \Sigma^+$ and $n\in \mathbb{N}$, write $E_{\omega,n}(x) := Df^n_\omega(x) E$.  Similarly, given a word $\omega^- \in \Sigma^-$, and $n\in \mathbb{N}$, we write $E_{\omega^-, -n}(x):= Df^{-n}_{\omega^-}(x) E$. 

\begin{definition}[Pesin future and past events]

Given a point $x\in \TT^2$, a direction $E$ and constants $C \geq 1$, $\varepsilon' >0$ and $\overline{\chi}>0$, we define the \emph{$(C,\overline{\chi},\varepsilon',E)$-Pesin future events of  $x$ } as the set $\Lambda^+_{C,\overline{\chi}, \varepsilon',E}(x)$, which is defined by the set of words $\omega\in \Sigma^+$ such that for every $n,k \geq 0$ the following properties hold:
 \begin{enumerate}
 \item $\|Df^n_{\sigma^k(\omega)}(f^k_\omega(x))|_{E^s_{\sigma^k(\omega)}(f^k_\omega(x))}\| \leq C e^{-\overline{\chi}n + k\varepsilon'}$;
 \item $\|Df^n_{\sigma^k(\omega)}(f^k_\omega(x))|_{E_{\omega,k}(x)}\| \geq C^{-1} e^{\overline{\chi}n - k\varepsilon'}$;
 \item $\sphericalangle(E_{\omega,k}(x), E^s_{\sigma^k(\omega)}(f^k_\omega(x))) \geq C^{-1} e^{-k\varepsilon'}. $
 \end{enumerate}
 
 Similarly, we defined the set \emph{$(C, \overline{\chi}, \varepsilon', E)$-Pesin past events of $x$}, and denote it by  $\Lambda^-_{C, \overline{\chi},\varepsilon',E}(x)$,  as the set of past words $\omega^- \in \Sigma^-$, where the same estimates above holds in the past replacing $\widehat{E}^s$ by $\widehat{E}^u$. 
\end{definition}

\begin{prop}
Suppose that $\mathcal{U}' \subset \mathcal{U}$ are $C^2$-open sets in $\mathrm{Diff}^2(\mathbb{T}^2)$ and $\mu$ a probability measure verifying assumptions \hyperlink{A1}{(\textbf{A1})} - \hyperlink{A6}{(\textbf{A6})}.  Then, for every $\varepsilon>0$,  for any constant line field $E$,  there are constants $C \geq 1$ and $\varepsilon', \overline{\chi}>0$ such that for any $x\in \TT^2$,
\[
\mu^+\left( \Lambda^+_{C,\overline{\chi},\varepsilon',E}(x) \right) >1-\varepsilon,
\]
and
\[
\mu^- \left( \Lambda^-_{C,\overline{\chi},\varepsilon',E}(x) \right) > 1 - \varepsilon.
\]

\end{prop}

\begin{proof}[Sketch of the poof]

The proof is essentially contained in \cite{DK, Zhang, JonDimaExp, LOP}. Let us give the main steps for just the sake of completeness.  Let us estimate the measure of the set of points not verifying conditions $(1)-(3)$ in the definition of $\Lambda^+_{C,\overline{\chi}, \varepsilon', E}(x)$. The same arguments in the past will give the result for the Pesin past events. 

\textbf{Estimating (3):}  Using that our $\mu^+$ is a Bernoulli measure, by \Cref{lem.cone.trans2}, for each $k$
\[
\mu^+\left(\omega : \sphericalangle(E_{\omega,n}(x), E^s_{\sigma^k(\omega)}(f^k_\omega(x))) \leq C^{-1} e^{-k\varepsilon'}\right) <\C{C'}{3} C^{-\C{\beta'}{1}} e^{-k\C{\beta'}{1} \varepsilon'}. 
\]
Summing over $k\in \mathbb{N}$, we get that the measure of events that violate $(3)$ is bounded from above by $C^{-\C{\beta'}{1}} \times \mathrm{Constant}$. We can make the measure of such set smaller than $\frac{\varepsilon}{3}$ by taking $C$ large. 

\textbf{Estimating (1) and (2):} To estimate the remainder two conditions, we need the following claim.

\begin{claim}\label{claim.claim}
Suppose that 
\[\|Df^n_{\sigma^k(\omega)}(f^k_\omega(x))|_{\widehat{E}^s_{\sigma^k(\omega)}(f^k_\omega(x))}\| > Ce^{-n\overline{\chi} + k\varepsilon'} \textrm{ and } \Jac f^n_{\sigma^k(\omega)}(f^k_\omega(x)) >  e^{-\C{C}{0} - 2n \C{\epsilon}{0}}.\]  Then,  for any $\hat{\varepsilon} \leq \varepsilon'$, at least one of the following holds:
\begin{enumerate}
\item[(A)] $\sphericalangle (\widehat{E}^s_{\sigma^{n+k}(\omega)}(f^{n+k}_\omega(x)),  E_{\omega, n+k}(x)) < 2C^{-\frac{1}{2}}e^{-(n+k) \hat{\varepsilon}},$ or
\item[(B)] $\|Df^n_{\sigma^k(\omega)}|_{E_{\omega, k}(x)}\|^{-1}> \frac{2}{\pi}  C^{\frac{1}{2}}e^{-n\overline{\chi} -\C{C}{0} - 2n \C{\epsilon}{0} + k\varepsilon'-(n+k) \hat{\varepsilon}}.$ 
\end{enumerate}
\end{claim}
\begin{proof}
Given a subspace $L$, let $P_L$ be the orthogonal projection on $L$.  Let $v^s_k$ be a unitary vector in $E^s_{\sigma^k(\omega)}(f^k_\omega(x))$.  If $(A)$ does not hold, we have
\[
\begin{array}{l}
\|P_{E_{\omega, n+k}^\perp(x)} \circ Df^n_{\sigma^k(\omega)}(f^k_\omega(x)) \circ P_{E_{\omega, k}(x)} v^s_k\| = \|P_{E_{\omega, n+k}^\perp(x)} \circ Df^n_{\sigma^k(\omega)}(f^k_\omega(x))  v^s_k\|\\
 \geq  \| Df^n_{\sigma^k(\omega)}(f^k_\omega(x))  v^s_k\| |\sin \sphericalangle((\widehat{E}^s_{\sigma^{n+k}(\omega)}(f^{n+k}_\omega(x)),  E_{\omega, n+k}(x))|\\
\geq  \frac{2}{\pi}  Ce^{-n\overline{\chi} + k\varepsilon'} C^{-\frac{1}{2}}e^{-(n+k) \hat{\varepsilon}}.
\end{array}
\]
Observe that 
\[
\|P_{E_{\omega, n+k}^\perp(x)} \circ Df^n_{\sigma^k(\omega)}(f^k_\omega(x)) \circ P_{E_{\omega, k}(x)}\|= \frac{\Jac f^n_{\sigma^k(\omega)}(f^k_\omega(x))}{\|Df^n_{\sigma^k(\omega)}|_{E_{\omega, k}(x)}\|}.
\]
Therefore,
\[
\begin{array}{l}
\|Df^n_{\sigma^k(\omega)}|_{E_{\omega, k}(x)}\|^{-1} > \left(\Jac f^n_{\sigma^k(\omega)}(f^k_\omega(x))\right)^{-1}  \frac{2}{\pi}  C^{\frac{1}{2}}e^{-n\overline{\chi} + k\varepsilon'-(n+k) \hat{\varepsilon}}\\
\geq  \frac{2}{\pi}  C^{\frac{1}{2}}e^{-n\overline{\chi} -\C{C}{0} - 2n \C{\epsilon}{0} + k\varepsilon'-(n+k) \hat{\varepsilon}}.
\end{array}
\]
This concludes the proof of the claim.  
\end{proof}

Take $\hat{\varepsilon} = \varepsilon'/2$. One can then use Lemmas \ref{lem:VAP}, \ref{lem.DK.trick} and \ref{lem.cone.trans2} to conclude that the measure of points not verifying $(A)$ or $(B)$ in  \Cref{claim.claim} is of the form 
\[
C^{-\gamma} \times (\textrm{something exponentially small in $n$ and $k$}),
\]
for some $\gamma >0$.  Adding over $n$ and $k$,  we can conclude that $(A)$ or $(B)$ don't happen in a set of the form $C^{-\gamma} \times \mathrm{Constant}$. We can then make the measure of this set smaller than $\frac{\varepsilon}{3}$ by taking $C$ sufficiently large.  This say that $(1)$ happens in a set of events with large measure.  $(2)$ is a consequence of a similar argument using \Cref{lem.DK.trick} (see Lemma $3$ in \cite{Zhang} for more details). \qedhere
\end{proof}

The estimates in the definition of Pesin events are what is needed to apply a graph transform, and uniform neighborhoods using Pesin charts, to obtain stable/unstable manifolds of uniform size. In particular, we obtain the following corollary.

\begin{cor}\label{cor:uniformsizemanifolds}
Suppose that $\mathcal{U}' \subset \mathcal{U}$ are $C^2$-open sets in $\mathrm{Diff}^2(\mathbb{T}^2)$ and $\mu$ a probability measure verifying assumptions \hyperlink{A1}{(\textbf{A1})} - \hyperlink{A6}{(\textbf{A6})}.  Then, for every $\varepsilon>0$,  there exists $l>0$ such that for any $x\in \TT^2$,  for any constant vector field $E$,  any word $\omega \in \Lambda^+_{C,\overline{\chi},\varepsilon, E}(x)$, there is a stable manifold, $W^s_{\omega}(x)$, centered in $x$ for $f_\omega$ having length bounded from below by $l$. Similarly, for any $\omega^-\in \Lambda^-_{C,\overline{\chi},\varepsilon',E}(x)$, there is an unstable manifold, $W^u_{\omega^-}(x)$, centered at $x$ having length bounded from below by $l$. 
\end{cor}

Pesin future (or past) events of a point are related to words.  Let us change the perspective a little bit and fix a word an look at the points verifying the good estimates. This is done in the following definition.  

\begin{definition}[Pesin block for a given event]

Given a word $\omega \in \Sigma^+$, a direction $E$ and constants $C \geq 1$, $\varepsilon' >0$ and $\overline{\chi}>0$, we define the \emph{$(C,\overline{\chi},\varepsilon',E)$-Pesin block for  $\omega$ } as the set $\Lambda^+_{C,\overline{\chi}, \varepsilon', E}(\omega)$, which is defined by the set of points $x\in \TT^2$ such that $\omega \in \Lambda^+_{C,\overline{\chi}, \varepsilon', E}(x)$. 
 
 Similarly,  given $\omega^- \in \Sigma^-$, we defined the set \emph{$(C, \overline{\chi}, \varepsilon', E)$-Pesin block for $\omega^-$}, and denote it by  $\Lambda^-_{C, \overline{\chi},\varepsilon',E}(x)$,  as the set of points $x\in \TT^2$ such that $\omega^- \in \Lambda^-_{C, \overline{\chi},\varepsilon',E}(x)$.

\end{definition}

Given  $\omega \in \Sigma^+$, a constant line field $E$,  and constants $C\geq 1$ and $\varepsilon',\overline{\chi}>0$,  let $l>0$ be a uniform lower bound on the size of the stable manifolds of points $x\in \Lambda^+_{C,\overline{\chi},\varepsilon', E}(\omega)$.  Consider the following lamination
 \[
 \mathcal{L}^s_{C,\overline{\chi},\varepsilon',E, \omega} := \displaystyle \bigcup_{y\in \Lambda_{C,\overline{\chi},\varepsilon',E}(\omega)} W^s_{\omega,l}(y).
 \]

\begin{lem}\label{lem:absolutecontinuity}
The stable lamination $\mathcal{L}^s_{C,\overline{\chi},\varepsilon, E, \omega}$ is absolutely continuous, that is, given any pair of transverses $D_1$ and $D_2$, the holonomy map of $\mathcal{L}^s_{C,\overline{\chi},\varepsilon',E,\omega}$ sends sets of zero Lebesgue measure into sets of zero Lebesgue measure. 
\end{lem}

The proof is the same as for a single diffeomorphism, but using $f_\omega$. 

\subsection{Proof of the Theorems}

\begin{proof}[Proof of \Cref{simplifiedmaintheorem1} assuming \Cref{simplifiedmaintheorem2}]

Suppose that $\mu$ is a probability measure supported on $\mathrm{Diff}^2_m(\mathbb{T}^2)$.  By \Cref{thm:equivalentconditionUE}, the uniform expansion on average in the future property for $\mu$ is equivalent to every $\mu$-stationary measure $\nu$ has one positive and one negative exponent, and the stable direction is random, meaning that for $\mu^{+} \otimes \nu$-almost every $(\omega,x)$, the Oseledets stable direction exists $E^s_\omega(x)$ and  depends on the choice of word $\omega$ (see \cite{chung}).  If there is a stable direction that is nonrandom, then there is a direction that is invariant under the action of the elements in $\mathrm{supp}(\mu)$.
 Similarly, the uniform expansion on average in the past property is equivalent to every $\mu^{-1}$-stationary measure $\nu$ (stationary measures in the past) have one postive and one negative exponent, and the stable distribution in the past (or the unstable distribution) is random. 

If $\mu$ is a measure verifying the assumptions of \Cref{simplifiedmaintheorem1}, the Zariski density assumption implies that every $\mu$-stationary measure is hyperbolic and that the stable direction is random. Similar conclusion holds for the unstable direction and past stationary measures. In particular, $\mu$ is uniformly expanding on average in the future and past.  Let $L$ be as in the statement of \Cref{simplifiedmaintheorem1}.  For any $\mu'$ sufficiently weak*-close to $\mu$ and that 
\[
\displaystyle \max_{f\in \mathrm{supp}(\mu')} \{\|f\|_{C^2}, \|f^{-1}\|_{C^2}\} < L,
\]
by \Cref{simplifiedmaintheorem2}, there exists a unique $\mu'$-stationary SRB measure $\nu$ and this measure is absolutely continuous with respect to the Lebesgue measure. In particular, the support of $\nu$ has positive Lebesgue measure.

\begin{claim}\label{claimabove}
$\nu$ is fully supported. 
\end{claim}
\begin{proof}
Let $\Gamma_{\mu'}$ be the sub-semigroup generated by the support of $\mu'$. Observe that, since $\nu$ is $\mu'$-stationary, then for any $g\in \Gamma_{\mu'}$ we have 
\begin{equation}\label{eq.invariancesupport}
g(\mathrm{supp}(\nu)) \subset \mathrm{supp}(\nu).
\end{equation}

Since $\Gamma_{\mu}$ is Zariski dense in $\mathrm{SL}(2,\mathbb{R})$, there exists an Anosov element $F$ in $\Gamma_{\mu}$. If $\mu'$ is sufficiently weak*-close to $\mu$, then there exists $F' \in \Gamma_{\mu'}$ which is $C^1$-close to $F$ and, in particular, it is Anosov. 

Let $m$ be the Lebesgue measure on $\mathbb{T}^2$ and let $F'$ be the Anosov element fixed above and consider $m_{\nu}:= m|_{\mathrm{supp}(\nu)}$ be the normalized restriction of $m$ to the support of $\nu$.  Recall that every Anosov diffeomorphism in $\TT^2$ is transitive and, in particular,  has a unique SRB measure and this measure is fully supported. By Pesin-Sinai's construction, we know that 
\[
\frac{1}{n} \sum_{j=0}^{n-1} (F')^j_* m_{\nu}
\]
converges to the unique SRB measure for $F'$,  the measure$\eta$.  The support of $\eta$ is contained in 
\[
\displaystyle \bigcap_{j=1}^\infty \bigcup_{k \geq j} (F')^k(\mathrm{supp}(m_{\nu})).
\]
Using that $\mathrm{supp}(m_{\nu}) = \mathrm{supp}(\nu)$,  the support of $\eta$ is $\TT^2$, and using \eqref{eq.invariancesupport}, we obtain that $\mathrm{supp}(\nu) = \TT^2$.  \qedhere
\end{proof}
This concludes the proof of \Cref{simplifiedmaintheorem1}.\qedhere
\end{proof}

\begin{lem}\label{lem:uniformsizenbd}
Suppose that $\mathcal{U}' \subset \mathcal{U}$ are $C^2$-open sets in $\mathrm{Diff}^2(\mathbb{T}^2)$ and $\mu$ a probability measure verifying assumptions \hyperlink{A1}{(\textbf{A1})} - \hyperlink{A6}{(\textbf{A6})}.  Then, there exists $\rho>0$, depending only on the choices of the constants involved in the assumptions \hyperlink{A1}{(\textbf{A1})} - \hyperlink{A6}{(\textbf{A6})}, such that for any two ergodic $\mu$-stationary SRB measures $\nu_1$ and $\nu_2$,  we have that
\[
\displaystyle \left(\bigcup_{x\in \mathrm{supp}(\nu_1)} B(x, \rho) \right) \cap \mathrm{supp}(\nu_2) = \emptyset.
\]
Moreover, for any measure $\mu'$ supported on $\mathcal{U}$ and sufficiently weak*-close to $\mu$, the same conclusion holds for $\rho$. 
\end{lem}

\begin{proof}

Suppose that $\nu_1$ and $\nu_2$ are two different ergodic $\mu$-stationary SRB measures.  
Let us show that there exists a constant $\rho>0$ such that for every $x\in \mathrm{supp}(\nu_1)$  the support of the measure $\nu_2$ cannot intersect $B(x,\rho)$.

Since $\nu_1 \neq \nu_2$,  there exists a continuous function $\varphi:\TT^2 \to \mathbb{R}$ such that $\int \varphi d\nu_1 \neq \int \varphi d\nu_2$.

Fix $\varepsilon>0$ small and let $C, \overline{\chi}, \varepsilon'$ to be constants such that $\mu^+ ( \Lambda^+_{C, \overline{\chi}, \varepsilon', E}(y)) > 1-\varepsilon$ for every $y\in \TT^2$.    Fix $\rho>0$ much smaller than the size of the stable and unstable manifolds for words in $\Lambda^+_{C,\overline{\chi},\varepsilon', E}(y)$ and $\Lambda^-_{C,\overline{\chi},\varepsilon', E}(y)$, for every $y\in \TT^2$.  

Define the set  $\Lambda^+_{C,\overline{\chi}, \varepsilon',E}$ to be the set of points $(\omega,y) \in \Sigma^+ \times \TT^2$ such that $\omega \in \Lambda^+_{C,\overline{\chi},\varepsilon',E}(y)$.   Observe that
 \[
 \begin{array}{ll}
 \mu^+ \times \nu_1 \left(\Lambda^+_{C,\overline{\chi},\varepsilon', E }\cap \left(\Sigma^+ \times B(x,\rho)\right)\right) &=  \displaystyle \int_{\Sigma^+} \nu_1\left(\Lambda^+_{C,\overline{\chi},\varepsilon', E}(\omega) \cap B(x,\rho)\right)d\mu^+(\omega)\\
 &=  \displaystyle \int_{B(x,\rho)} \mu^+(\Lambda^+_{C,\overline{\chi},\varepsilon', E}(y)d\nu_1(y)\\
 & \geq (1-\varepsilon)\nu_1(B(x,\rho)) >0.
 \end{array}
 \]
 Therefore, there is a set $\mathcal{D} \subset \Sigma^+$, having positive $\mu^+$-measure, such that for any $\omega\in \mathcal{D}$,  we have that
 \[
 \nu_1\left(\Lambda^+_{C,\overline{\chi},\varepsilon',E}(\omega) \right) > (1-\varepsilon) \nu_1(B(x,\rho))>0. 
 \] 
 
 Given a word $\omega \in \Sigma^+$ and a point $y$, define 
 \[
 \varphi^+_\omega(y):= \displaystyle \lim_{n\to +\infty} \frac{1}{n} \sum_{j=0}^{n-1} \varphi(f^j_\omega(y)),
 \]
 if the limit exists. Similarly, given $\omega^- \in \Sigma^-$, define
 \[
 \varphi^-_{\omega^-}(y) := \displaystyle \lim_{n\to +\infty} \frac{1}{n} \sum_{j=0}^{n-1} \varphi( f^{-j}_{\omega^-}(y)),
 \]
 if the limit exists. 
 
Fix $\omega \in \mathcal{D}$ and consider the stable lamination $\mathcal{L}^s_{C,\overline{\chi},\varepsilon',E,\omega}$.  Fix $z_1$, a $\nu_1$-typical point in $B(x,\rho)$.  Since $\nu_1$ is SRB, and by Birkhoff ergodic theorem, we have that for $\mu^-$-almost every $\omega^-\in \Sigma^-$, for $m^u_{\omega^-,z_1}$-almost every $z_1^u$ and for $\mu^+$-almost every $\omega'$, we have that 
 \[
 \displaystyle \varphi^-_{\omega^-}(z^u_1)  = \varphi^+_\omega(z_1^u) = \int \varphi d\nu_1,
 \]
where $m^u_{\omega^-, z_1}$ is the arclength measure along the unstable manifold $W^u_{\omega^-}(z_1)$.  Since $\varphi$ is continuous and $z^u_1 \in W^u_{\omega^-}(z_1)$, we also have
\[
 \displaystyle \varphi^-_{\omega^-}(z^u_1) = \varphi^-_{\omega^-}(z_1). 
\]

By \Cref{lem.cone.trans2}, and up to reducing $\rho$ if necessary,  we can choose a word $\omega^-_1 \in \Lambda^-_{C,\overline{\chi},\varepsilon',E}(z_1)$ such that $W^u_{\omega^-_1, l}(z_1)$ intersects and it is transverse to $\mathcal{L}^s_{C,\overline{\chi},\varepsilon',E, \omega} \cap B(x,\rho)$.  Suppose that $\nu_2 (B(x,\rho))>0$. Fix a $\nu_2$-typical point $z_2$ and choose $\omega^-_2 \in \Lambda^-_{C,\overline{\chi},\varepsilon',E}(z_2)$ such that $W^u_{\omega^-_2, l}(z_2)$ intersects transversally every leaf of $\mathcal{L}^s_{C,\overline{\chi},\varepsilon',E,\omega} \cap B(x,\rho)$.  By the absolute continuity of the stable lamination, we can find points $z_i^u \in W^u_{\omega^-_i,l}(z_i)$, for $i=1,2$, such that $z_1^u$ and $z_2^u$ are in the same stable manifold of the lamination $\mathcal{L}^s_{C,\overline{\chi},\varepsilon', E, \omega}$ and 
\[
\varphi_{\omega^-_1}^-(z_1^u) = \varphi_{\omega}^+(z_1^u) \textrm{ and } \varphi^-_{\omega_2^-}(z_2^u) = \varphi^+_\omega(z_2^u). 
\]
Therefore, since $z_1^u$ and $z_2^u$ are in the same stable manifold, we obtain
\[
\displaystyle \int \varphi d\nu_1 = \varphi^-_{\omega^-_1}(z_1) = \varphi^-_{\omega_1^u}(z_1^u) = \varphi^+_{\omega}(z^u_1) = \varphi^+_\omega(z_2^u) = \varphi^-_{\omega^-_2}(z_2) = \int \varphi d\nu_2,
\]
which is a contradiction since $\int \varphi d\nu_1 \neq \int \varphi d\nu_2$. 

Observe that the estimates on the constants $C, \overline{\chi}, \varepsilon'$ also hold if $\mu'$ is sufficiently weak*-close to $\mu$ and supported in $\mathcal{U}$. \qedhere 

\end{proof}

\begin{proof}[Proof of \Cref{simplifiedmaintheorem2}]

Assume that $\mu$ is a probability measure verifying the hypothesis of \Cref{simplifiedmaintheorem2}.  Fix $L>0$ as in the statement and take $C_0' = \log L$.  This $L$ determines $\mathcal{U}$. It is immediate to see that $\mu$ verifies assumptions \hyperlink{A1}{\textbf{(A1)}} - \hyperlink{A4}{\textbf{(A4)}}, where we consider $\cC = T\mathbb{T}^2$ in \hyperlink{A2}{\textbf{(A2)}}. 

Let $\C{C}{0}$ and $\C{\epsilon}{0}$ be as in \hyperlink{A5}{\textbf{(A5)}}. These constants will determine $\mathcal{U}'$, since our starting assumption is that $\mu$ is supported on volume preserving diffeomorphisms. 

Let  $\C{\epsilon}{1}$ and $\C{\epsilon}{2}$ be as in \hyperlink{A6}{\textbf{(A6)}}.  Since UEF and UEP are open properties, we have that if $\mu'$ is a probability measure supported on $\mathcal{U}$ and sufficiently weak*-close to $\mu$, then the UEF and UEP estimates of $\mu$ also hold for $\mu'$ and  $\mu'(\mathcal{U}')>1-\C{\epsilon}{2}$. Then \Cref{thm:maintechnicaltheorem} implies that every ergodic SRB $\mu'$-stationary measure $\nu$ is absolutely continuous. Combining this with \Cref{thm:brownhertz} we obtain that every $\mu'$-stationary measure is either finitely supported or absolutely continuous.   Our goal now is to show the uniqueness of the SRB measure.

\begin{claim}\label{claim:continuityofSRB}
Suppose that $(\mu_n)_{n\in \mathbb{N}}$ is a sequence of probability measures converging to $\mu$, and that for each $n\in \mathbb{N}$ the measure $\nu_n$ is an ergodic $\mu_n$-stationary SRB measure. Then $\nu_n$ converges to $m$.
\end{claim}

Observe that \Cref{claim:continuityofSRB} implies that if $\mu'$ is sufficiently close to $\mu$, then any ergodic $\mu'$-stationary measure $\nu$ is weak*-close to $m$, which is $\mu$-ergodic (see \cite{chung}). Let $\rho>0$ be the constant given by  \Cref{lem:uniformsizenbd},  if $\mu'$ is sufficiently close to $\mu$, then an ergodic $\mu'$-stationary SRB measure would have its support $\rho$-dense in the manifold, so there is only one SRB measure.

\begin{proof}[Proof of \Cref{claim:continuityofSRB}]
Let $\mu_n$ and $\nu_n$ be as in the claim.  Suppose that $\nu_n$ converges to $\nu$. Since $\nu_n$ are $\mu_n$-stationary measures, then $\nu$ is a $\mu$-stationary measure.  By \Cref{thm:brownhertz}, we have that 
\[
\displaystyle \nu = \alpha m + \sum_{p} \alpha_p \delta_p,
\]
where $\alpha + \sum_p \alpha_p = 1$.  We need to show that $\alpha = 1$.  Suppose $p$ is a point such that $\alpha_p>0$. 

Observe that the estimates in \Cref{thm:main.detailed} are uniform.  Fix $c= \frac{\alpha_p}{2}$, as an application of \Cref{thm:main.detailed}, we have that for each $n$ large enough, the SRB measure $\nu_n$ can be decomposed as
\[
\nu_n = \nu_{n,good, c} + \nu_{n,bad,c},
\]
where
\begin{enumerate}
\item $\nu_{n,good,c}(S) \geq (1-c) \nu_n(S)$;
\item there exists a constant $\C{C''''}{10}$, after fixing some uniform constants as in the proof of \Cref{thm:main.detailed}, so that for any $\rho>0$ small enough
\[
\|\nu_{n,good,c}\|_\rho \leq \C{C''''}{10}.
\]
\end{enumerate}

Up to taking a subsequence, suppose that $\nu_{n,good,c}$ converges to $\nu_{good,c}$, as $n \to +\infty$. By \Cref{lem.normconvergence}, we have that for any $\rho$ sufficiently small
\[
\|\nu_{good,c} \|_\rho\leq \C{C''''}{10}.
\]
In particular, $\nu_{good,c}$ is absolutely continuous with respect to $m$.  Moreover, $\nu = \nu_{good,c} + \nu_{bad,c}$ and $\nu_{good,c}(S) \geq (1-c)  = 1- \frac{\alpha_p}{2}$.  Observe that $\nu_{good,c}$ cannot have any atoms. Hence
\[
 \alpha_p = \nu(\{p\}) = \nu_{bad,c}(\{p\}) \leq \frac{\alpha_p}{2},
\]
which is a contradiction. We conclude that $\nu = m$.\qedhere

\end{proof}

Fix $\rho>0$ obtained by \Cref{lem:uniformsizenbd} and fix a ball of radius $\rho>0$, $B$.  By \Cref{claim:continuityofSRB},  on top of the requirements we had before for the measure  $\mu'$, assume that $\mu'$ is sufficiently close to $\mu$, so that any $\mu'$-stationary SRB $\nu'$, we have $\nu'(B) >0$. By \Cref{lem:uniformsizenbd}, $\mu'$ cannot have two different ergodic SRB $\mu'$-stationary measures, and we conclude the uniqueness.  \qedhere

\end{proof}

\begin{proof}[Proof of \Cref{simplifiedmaintheorem3}]

Suppose that $\mu$ is a probability measure on $\mathrm{Diff}^2_m(\TT^2)$ which is UEF and UEP.  Let $L>0$ be as in the statement.   Suppose that $\mathcal{V}$ is the weak*-neigbhorhood obtained in \Cref{simplifiedmaintheorem2}. In particular, for any $\mu' \in \mathcal{V}$, we have uniqueness of the $\mu'$-stationary SRB measure, $\nu_{\mathrm{abs}}'$.

Recall that $\Gamma_{\mu'}$ is the sub-semigroup generated by the support of $\mu'$.  
Below, we will state several results from \cite{chung} that hold in more generality, under some integrability condition.

\begin{prop}[Proposition $4.6$ from \cite{chung}]\label{proposition.countablefiniteorbits}
The number of points with finite $\Gamma_{\mu'}$-orbit is countable. 
\end{prop}

\begin{lem}[Lemma $4.7$ in \cite{chung}]\label{lemma.omegaset}
Let $\mathcal{N}$ be a finite $\Gamma_{\mu'}$-orbit in $S$. For any $\varepsilon >0$, there exists an open set $\Omega_{\mathcal{N},\varepsilon}$ containing $\mathcal{N}$, such that for any compact set $H \subset S / \mathcal{N}$, there exists a positive integer $n_H$, such that for all $x\in H$, and $n> n_H$, we have
\[
\displaystyle \left(\frac{1}{n} \sum_{i=0}^{n-1} (\mu')^{*i}* \delta_x\right)(\Omega_{\mathcal{N},\varepsilon}) < \varepsilon. 
\]
\end{lem}

The proof of Proposition \ref{proposition.countablefiniteorbits} and Lemma \ref{lemma.omegaset} uses a Margulis function (see Lemma $4.3$ from \cite{chung}).

The conclusion of the proof of \Cref{simplifiedmaintheorem3} is exactly the same as the proof of Proposition $4.1$ from \cite{chung}, where the unique $\mu'$-stationary SRB measure $\nu_{\mathrm{abs}}'$ takes the role of the smooth measure $m$ in the proof.   \qedhere

\end{proof}

\begin{proof}[Proof of \Cref{cor.orbitclassification}]

Let $\mu'$ be as in the statement of \Cref{cor.orbitclassification}.  Suppose that $x$ is a point with infinite $\Gamma_{\mu'}$-orbit. By \Cref{simplifiedmaintheorem3}, we have that 
\[
\displaystyle \lim_{n \to +\infty} \frac{1}{n} \sum_{i=0}^{n-1} (\mu')^{*i}* \delta_x = \nu_{\mathrm{abs}}'.
\]
This implies that 
\[
\overline{\Gamma_{\mu'}(x)} \supset \mathrm{supp}(\nu').
\]
Since we assumed that $\nu_{\mathrm{abs}}'$ is fully supported, we conclude that every $\Gamma_{\mu'}$-orbit is either finite or dense.  \qedhere

\end{proof}

\begin{proof}[Proof of \Cref{cor.genericminimality}]

Suppose that $f,g \in \mathrm{Diff}^2_m(\mathbb{T}^2)$ are two diffeomorphisms verifying the assumptions of \Cref{cor.genericminimality}.  By assumption, there is an element $F = f_{\omega_{M-1}} \circ f_{\omega_{M-2}} \circ \cdots \circ f_{\omega_0} \in \Gamma_{(f,g)}$ which is an Anosov diffeomorphism.  In particular, there are $C^2$-neighborhoods of $f$ and $g$ in $\mathrm{Diff}^2(\mathbb{T}^2)$,  $\mathcal{U}_f$ and $\mathcal{U}_g$, respectively, such that for any $(\widehat{f}, \widehat{g}) \in \mathcal{U}_f \times \mathcal{U}_g$, the sub-semigroup $\Gamma_{(\widehat{f}, \widehat{g})}$ contains an Anosov element $\widehat{F}$.  

Let $\mu$ be a probability measure supported on $\{f,g\}$ which is UEF and UEP and suppose that, up to reducing the size of $\mathcal{U}_f$ and $\mathcal{U}_g$ if needed,  $\hat{\mu}$ is a probability measure supported on $\{\widehat{f}, \widehat{g}\}$ which is sufficiently weak*-close to $\mu$, so that the conclusions of \Cref{cor.orbitclassification} and \Cref{simplifiedmaintheorem2} hold. 

Let $\hat{\nu}$ be the unique $\hat{\mu}$-stationary SRB measure.  By \Cref{simplifiedmaintheorem2}, $\hat{\nu}$ is a probability measure which is absolutely continuous with respect to $m$.  In particular, $m(\mathrm{supp}(\hat{\nu}))>0$. By the same argument as in \Cref{simplifiedmaintheorem1}, we obtain that $\hat{\nu}$ is fully supported.

By \Cref{cor.orbitclassification}, we obtain that every $\Gamma_{(\widehat{f}, \widehat{g})}$-orbit is either finite or dense.   For each $n\in \mathbb{N}$, let $\mathcal{U}_*^n \subset \mathcal{U}_*$ be the subset of diffeomorphisms in $\mathcal{U}_*$ such that all of the periodic points with period up to $n$ are hyperbolic, for $*=f,g$.  By the same arguments as in the proof of Kupka-Smale's Theorem, the set $\mathcal{U}_*^n$ is $C^2$-open and $C^2$-dense in $\mathcal{U}_*$, for $*=f,g$.  Consider 
\[
\mathcal{U}^n_{f,g} \subset \mathcal{U}^n_f \times \mathcal{U}_g^n
\]
the subset of pairs $(\widehat{f}, \widehat{g}) \in \mathcal{U}^n_f \times \mathcal{U}^n_g$ such that $\mathrm{Per}^n(\widehat{f}) \cap \mathrm{Per}^n(\widehat{g}) = \emptyset$, so they have no common periodic points up to period $n$.  It is easy to see that the set $\mathcal{U}^n_{f,g}$ is open and dense.  Take 
\[
\mathcal{R} = \displaystyle \bigcap_{n\in \mathbb{N}} \mathcal{U}^n_{f,g},
\]
this is a dense $G_\delta$-subset of $\mathcal{U}_f \times \mathcal{U}_g$.  If $(\widehat{f}, \widehat{g}) \in \mathcal{R}$, then $\mathrm{Per}(\widehat{f}) \cap \mathrm{Per}(\widehat{g}) = \emptyset$. 
\end{proof}

\begin{proof}[Proof of \Cref{thm:UHcase}]

Let $f$ and $g$ be diffeomorphisms verifying conditions \textbf{(C1)} - \textbf{(C3)} from  \Cref{thm:UHcase}.  Fix $\beta \in (0,1)$ and consider $\mu_\beta := \beta \delta_f + (1-\beta) \delta_g$.  Take $\mathcal{C} = \mathcal{C}^u$, the unstable cone given by condition \textbf{(C2)}.  

Fix $\mathcal{U}_f$ and $\mathcal{U}_g$, $C^2$-neighborhoods of $f$ and $g$ such that any $f\in \mathcal{U}_f$ and $g \in \mathcal{U}_g$, we have that conditions \textbf{(C1)} - \textbf{(C3)} are verified, and that the $C^2$-norm of $f$ and $g$ are close to the norm of $f$ and $g$.  Write $\mathcal{U} = \mathcal{U}_f \cup \mathcal{U}_g$. 

By the uniform expansion of vectors inside $\mathcal{C}$,  we obtain for free \hyperlink{A2}{\textbf{(A2)}}  and \hyperlink{A3}{\textbf{(A3)}}. Condition \textbf{(C3)} implies that the unstable direction is ``random", which gives \hyperlink{A4}{\textbf{(A4)}}.  Fix constants $\delta, \chi >0$ which verify \Cref{lem.DK.trick}, and fix $\overline{\chi}$ that verifies \eqref{eqn:epsilon.cond.1}.

Take $\C{\epsilon}{0}$ as in \hyperlink{A5}{\textbf{(A5)}}. Up to reducing the size of $\mathcal{U}_f$ and $\mathcal{U}_g$, and taking $\mathcal{U}' = \mathcal{U}_f \cup \mathcal{U}_g = \mathcal{U}$, we may suppose that $\mathcal{U}'$ verifies \hyperlink{A5}{(\textbf{A5})}.  Observe that we are not considering any ``very dissipative'' system,  as everything is happening near $f$ and $g$.  By \Cref{thm:maintechnicaltheorem}, there is a neighborhood $\mathcal{V}$ of $\mu_\beta$ such that for any $\mu'$, there  are finitely many  $\mu'$-stationary SRB measures. Moreover,  each of which is absolutely continuous with respect to $m$.  For the same reason as \Cref{claimabove}, every $\mu'$-stationary SRB measure is fully supported, in particular, there is only one such measure. 

Assuming that $f$ and $g$ also verify \textbf{(C4)} we obtain that the random dynamics is UEF for $\mathcal{C} = T\TT$. Applying \Cref{thm:brownhertz}, we conclude the measure classification portion of \Cref{thm:UHcase}. \qedhere

\end{proof}

\appendix
\section{Two estimates involving the Stirling's formula}
The Stirling's formula implies that there exists some $\CS{C''}{4}>1$ such that for any $n\geq 1$ and any $k\in\{0,\dots, n\}$, we have 
\begin{align}\label{eqn:stirling}
{n\choose k}\leq \C{C''}{4}\left(\frac{n}{k}\right)^k\cdot \left(\frac{n}{n-k}\right)^{n-k}.
\end{align}
We need the following two elementary estimates which uses \eqref{eqn:stirling}.
\begin{lem}\label{lem:stirling-1}
Let $\eta\in(0,1/2)$. Then for any $n$, we have 
$$\sum_{k=0}^{[n\eta]}{n\choose k}\leq \C{C''}{4}n\eta\cdot \left(\frac{1}{\eta^\eta(1-\eta)^{1-\eta}}\right)^n$$
\end{lem}
\begin{proof}
Let $f(t)=t^{-t}(1-t)^{t-1}$. One can easily check that $f(t)$ can be continuously extended onto $[0,1]$ with $f(0)=f(1)=1$.  When $t\in(0,1/2)$ we have 
\begin{align}\label{eqn:stirling.assist}
\frac{d}{dt}\ln f(t)=\ln\left(\frac{1-t}{t}\right)>0
\end{align}
Hence by \eqref{eqn:stirling} and \eqref{eqn:stirling.assist}, we have 
\begin{align*}
\sum_{k=0}^{[n\eta]}{n\choose k}\leq &\C{C''}{4}\sum_{k=0}^{[n\eta]}\left(\frac{n}{k}\right)^k\cdot \left(\frac{n}{n-k}\right)^{n-k}\\
=&\C{C''}{4}\sum_{k=0}^{[n\eta]}\left(f\left(\frac{k}{n}\right)\right)^n\\
\leq& \C{C''}{4}n\eta\cdot (f(\eta))^n\leq  \C{C''}{4}n\eta\cdot \left(\frac{1}{\eta^\eta(1-\eta)^{1-\eta}}\right)^n\qedhere
\end{align*}
\end{proof}
\begin{lem}\label{lem:stirling-2}
Let $\eta\in(0,1/2)$ and $a>0$. Assume that $b\geq a-\frac{1}{\eta}\ln(\eta^\eta(1-\eta)^{1-\eta})$ Then for any $n$, we have 
$$\sum_{k=[n\eta]+1}^{n}{n\choose k}e^{-b k}\leq \frac{\C{C''}{4}e^{-an\eta}}{1-e^{-a}}.$$
\end{lem}
\begin{proof}
We adopt the same $f(t)$ as in the proof of Lemma \ref{lem:stirling-1}. Notice that for any $t\in (0,1)$, we have 
$$\frac{d^2}{dt^2}\ln f(t)=-\frac{1}{t}-\frac{1}{1-t}<0.$$
Hence for any $t\in [\eta, 1)$, we have 
$$\ln f(t)\leq \ln f(0)+\frac{t}{\eta}\left(\ln f(\eta)-\ln f(0)\right)=\frac{t}{\eta} \ln f(\eta).$$
Therefore by \eqref{eqn:stirling} and the above, we have
\begin{align*}
\sum_{k=[n\eta]+1}^{n}{n\choose k}e^{-b k}\leq& \C{C''}{4}\sum_{k=[n\eta]+1}^{n}\left(\frac{n}{k}\right)^k\cdot \left(\frac{n}{n-k}\right)^{n-k}e^{-b k}\\
=&\C{C''}{4}\sum_{k=[n\eta]+1}^{n}\left(f\left(\frac{k}{n}\right)\right)^ne^{-b k}\\
\leq&\C{C''}{4}\sum_{k=[n\eta]+1}^{n} e^{\frac{k}{\eta}\ln f(\eta)}e^{-b k}\\
\leq&\C{C''}{4}\sum_{k=[n\eta]+1}^{n}e^{-a k}\leq\frac{\C{C''}{4}e^{-an\eta}}{1-e^{-a}}.\qedhere
\end{align*}
\end{proof}


\medskip

\begin{thebibliography}{alpha}

\bibitem[AGT06]{AGT}
A.~Avila, S.~Gouëzel, and M.~Tsujii,
\newblock Smoothness of solenoidal attractors,
\newblock {\em Discrete Contin. Dyn. Syst.} 15 (2006), no.~1, 21--35.







\bibitem[BQ11]{MR2831114}
Y.~Benoist and J.-F. Quint,
\newblock Mesures stationnaires et fermés invariants des espaces homogènes,
\newblock {\em Ann. of Math. (2)} 174, no.~2, 1111--1162, 2011.

\bibitem[BQ13]{MR3037785}
Y.~Benoist and J.-F. Quint,
\newblock Stationary measures and invariant subsets of homogeneous spaces (II),
\newblock {\em J. Amer. Math. Soc.} 26, no.~3, 659--734, 2013.

\bibitem[Bor19]{Bortolotti}
R.~T. Bortolotti,
\newblock Physical measures for certain partially hyperbolic attractors on 3-manifolds,
\newblock {\em Ergodic Theory Dynam. Systems} 39, no.~1, 74--104, 2019.

\bibitem[BFQM07]{MR2340439}
J.~Bourgain, A.~Furman, E.~Lindenstrauss, and S.~Mozes,
\newblock Invariant measures and stiffness for non-abelian groups of toral automorphisms,
\newblock {\em C. R. Math. Acad. Sci. Paris} 344, no.~12, 737--742, 2007.

\bibitem[BEFH25]{befh}
A.~Brown, A.~Eskin, S.~Filip and F.~Rodriguez~Hertz,
\newblock Measure rigidity for generalized u-Gibbs states and stationary measures via the factorization method,
\newblock 	arXiv:2502.14042, 2025. 

\bibitem[BH17]{Brown-Hertz}
A.~Brown and F.~Rodriguez~Hertz,
\newblock Measure rigidity for random dynamics on surfaces and related skew products,
\newblock {\em J. Amer. Math. Soc.} 30, no.~4, 1055--1132, 2017.

\bibitem[BLOR24]{BLOR}
A.~Brown, H.~Lee, D.~Obata, and Y.~Ruan,
\newblock Absolute continuity of stationary measures,
\newblock arXiv:2409.18252, 2024.

\bibitem[Chu20]{chung}
P.~N. Chung,
\newblock Stationary measures and orbit closures of uniformly expanding random dynamical systems on surfaces,
\newblock arXiv:2006.03166, 2020.

\bibitem[DD24]{JonDimaExp}
J.~DeWitt and D.~Dolgopyat,
\newblock Expanding on average diffeomorphisms of surfaces: exponential mixing,
\newblock arXiv:2410.08445, 2024.

\bibitem[DD25]{JonDimaCo}
J.~DeWitt and D.~Dolgopyat,
\newblock Conservative coexpanding on average diffeomorphisms,
\newblock arXiv:2503.06855, 2025.

\bibitem[DK07]{DK}
D.~Dolgopyat and R.~Krikorian,
\newblock On simultaneous linearization of diffeomorphisms of the sphere,
\newblock {\em Duke Math. J.} 136, no.~3, 475--505, 2007.



\bibitem[Ell23]{RoseUE}
R.~Elliott~Smith,
\newblock Uniformly expanding random walks on manifolds,
\newblock {\em Nonlinearity} 36, no.~11, 5955--5972, 2023.

\bibitem[EL1]{el1}
A.~Eskin and E.~Lindenstrauss,
\newblock Random walks on locally homogeneous spaces,
\newblock Preprint on A. Eskin's webpage.

\bibitem[EL2]{el2}
A.~Eskin and E.~Lindenstrauss,
\newblock Zariski dense random walks on homogeneous spaces,
\newblock Preprint on A. Eskin's webpage. 

\bibitem[EM18]{MR3814652}
A.~Eskin and M.~Mirzakhani,
\newblock Invariant and stationary measures for the ${\rm SL}(2,\mathbb{R})$ action on moduli space,
\newblock {\em Publ. Math. Inst. Hautes Études Sci.} 127, 95--324, 2018.

\bibitem[EMM15]{emm}
A. ~Eskin, M. ~Mirzakhani and A. ~Mohammadi,
\newblock Isolation, equidistribution, and orbit closures for the {${\rm SL}(2,\Bbb R)$} action on moduli space,
\newblock {\em Ann. of Math. (2)}, 182, 673--721, 2015. 

\bibitem[Fur63]{Furstenberg}
H.~Furstenberg,
\newblock Noncommuting random products,
\newblock {\em Trans. Amer. Math. Soc.} 108, 377--428, 1963.


\bibitem[LY88]{LYRandom}
F.~Ledrappier and L.-S. Young,
\newblock Entropy formula for random transformations,
\newblock {\em Probab. Theory Related Fields} 80, no.~2, 217--240, 1988.

\bibitem[LOP24]{LOP}
Y.~Lima, D.~Obata, and M.~Poletti,
\newblock Measures of maximal entropy for non-uniformly hyperbolic maps,
\newblock arXiv:2405.04676, 2024.

\bibitem[LQ95]{Liu-Qian-book}
P.-D. Liu and M.~Qian,
\newblock {\em Smooth ergodic theory of random dynamical systems},
\newblock Lecture Notes in Mathematics, Vol. 1606, Springer-Verlag, Berlin, 1995.


\bibitem[Pot22]{PotrieUE}
R.~Potrie,
\newblock A remark on uniform expansion,
\newblock {\em Rev. Un. Mat. Argentina} 64, no.~1, 11--21, 2022.

\bibitem[Tsu01]{Tsujii-Fat}
M.~Tsujii,
\newblock Fat solenoidal attractors,
\newblock {\em Nonlinearity} 14, no.~5, 1011--1027, 2001.

\bibitem[Tsu05]{Tsujii-bigpaper}
M.~Tsujii,
\newblock Physical measures for partially hyperbolic surface endomorphisms,
\newblock {\em Acta Math.} 194, no.~1, 37--132, 2005.

\bibitem[Zha19]{Zhang}
Z.~Zhang,
\newblock On stable transitivity of finitely generated groups of volume-preserving diffeomorphisms,
\newblock {\em Ergodic Theory Dynam. Systems} 39, no.~2, 554--576, 2019.

\end{thebibliography}
\end{document}